\documentclass[10pt]{amsart}
\pdfoutput=1
\usepackage{amsmath, amsthm, amssymb,slashed,stmaryrd}
\usepackage{ifpdf}
\usepackage[pdftex]{graphicx}
\usepackage{tikz}
\usetikzlibrary{matrix,arrows,calc}

\usepackage[pdftex,plainpages=false,hypertexnames=false,pdfpagelabels]{hyperref}
\long\def\todo#1{{\color{red} {#1}}}

 \theoremstyle{plain}
 \newtheorem{thm}{Theorem}[section]
    
 \newtheorem{cor}[thm]{Corollary}
  
 \newtheorem{lem}[thm]{Lemma}
 \newtheorem{prop}[thm]{Proposition}
 \newtheorem{conj}[thm]{Conjecture}
  \newtheorem{hyp}[thm]{Hypothesis}
 \theoremstyle{definition}
 \newtheorem{defn}[thm]{Definition}
 \newtheorem{notation}[thm]{Notation}
 \newtheorem{ex}[thm]{Example}

 \newtheorem*{thm*}{Theorem}
 \theoremstyle{remark}
 \newtheorem{rmk}[thm]{Remark}
  \newtheorem{terminology}[thm]{Terminology}
 \newcommand{\colim}{\mathop{\rm colim}\limits}

\def\beq{\begin{eqnarray}}
\def\eeq{\end{eqnarray}}
 \newcommand{\bp}{\begin{proof}[Proof]}
 \newcommand{\ep}{\end{proof}}

\DeclareSymbolFont{bbold}{U}{bbold}{m}{n}
\DeclareSymbolFontAlphabet{\mathbbold}{bbold}

\def\Ell{{\rm Ell}}
\def\wz{{\rm wz}}

\def\dEll{\widehat{{\rm Ell}}{}}

\def\TMF{{\rm TMF}}

\def\eigen{{\epsilon}}
\def\Mst{\mc{M}}
\def\Fcq{\mc{F}}
\def\Mstil{\widetilde{\Mst}}
\def\Ftil{\widetilde{\Fcq}}
\def\SMQ{\mathcal{S}_{\rm MQ}}
\def\x{{\bf x}}
\def\deg{{\rm deg}}
\def\aux{{a}}
\def\omegas{{\omega}}
\def\omegasbar{{\overline{\omega}}}

\def\Mell{\mathcal{M}_{\rm ell}}

\def\O{{\mathcal{O}}}
\def\H{{\rm H}}
\def\HH{{\mathbb H}}
\def\Spin{{\rm Spin}}

\def\Ch{{\rm C}}

\def\Det{{\rm Det}}
\def\Ad{{\rm Ad}}
\def\Map{{\sf Map}}
\def\pf{{\rm pf}}

\def\EE{{\mathcal{E}}}
\def\Euni{\mathcal{E}}

\def\red{{\rm red}}

\def\pt{{\rm pt}}

\def\Sym{{\rm Sym}}

\def\pf{{\rm pf}}
\def\MF{{\rm MF}}
\def\vol{{\rm vol}}

\def\Pic{{\sf Pic}}
\def\ev{{\rm ev}}
\def\odd{{\rm odd}}

\def\Fer{{\rm Fer}}
\def\pre{{\rm pre}}

\def\R{{\mathbb{R}}}
\def\fg{{\mathfrak{g}}}

\def\E{{\mathbb{E}}}

\def\T{{\mathbb{T}}}

\def\CP{{\mathbb{CP}}}

\def\bS{{\mathbb{S}}}
\def\id{{{\rm id}}}

\def\C{{\mathbb{C}}}

\def\Z{{\mathbb{Z}}}
\def\X{{\mathcal{X}}}
\def\Y{{\mathcal{Y}}}
\def\F{{\mathcal{F}}}

\def\Vect{{\sf Vect}}
 
\def\End{{\sf End}}
\def\Aut{{\sf Aut}}
\def\Bun{{\sf Bun}}

\def\SL{{\rm SL}}
\def\GL{{\rm GL}}
\def\MP{{\rm MP}}
\def\Hom{{\sf Hom}}
\def\Fun{{\sf Fun}}
\def\SM{{\sf Map}}

\def\Conf{{\sf Conf}}

\def\Lat{{\sf Lat}}
\def\cL{{\mathcal{F}}}

\newcommand{\op}{{\sf{op}}}   
\newcommand{\sq}{{ / \!\! / }}
\newcommand{\cq}{/}
\newcommand{\gq}{{ / \!\! / }}
\newcommand{\nsq}{{\sq^{\!\nabla}}\!}
\newcommand{\ncq}{{/^\nabla}\! }

\def\twocommute{\ensuremath{\rotatebox[origin=c]{30}{$\Rightarrow$}}}
\def\downin{\ensuremath{\rotatebox[origin=c]{90}{$\in$}}}

\newcommand\nc{\newcommand}
\nc\mf\mathfrak
\nc\mc\mathcal
\nc\mb\mathbb

\begin{document}

\title[A model for equivariant elliptic cohomology from quantum field theory]{A model for complex analytic equivariant elliptic cohomology from quantum field theory}

\author{Daniel Berwick-Evans and Arnav Tripathy}

\date{\today}

\begin{abstract}
We construct a global geometric model for complex analytic equivariant elliptic cohomology for all compact Lie groups. 
Cocycles are specified by functions on the space of fields of the two-dimensional sigma model with background gauge fields and $\mc{N} = (0, 1)$ supersymmetry. We also consider a theory of free fermions valued in a representation whose partition function is a section of a determinant line bundle. We identify this section with a cocycle representative of the (twisted) equivariant elliptic Euler class of the representation. Finally, we show that the moduli stack of $U(1)$-gauge fields carries a multiplication compatible with the complex analytic group structure on the universal (dual) elliptic curve, with the Euler class providing a choice of coordinate. This provides a physical manifestation of the elliptic group law central to the homotopy-theoretic construction of elliptic cohomology. 
\end{abstract}

\maketitle 
\setcounter{tocdepth}{1}
\tableofcontents

\section{Introduction}\label{intro}

Elliptic cohomology grew out of a combination of Quillen's deep work on the role of formal group laws in stable homotopy theory~\cite{Quillen} and Ochanine's construction of elliptic genera~\cite{Ochanine,Segal_Elliptic,LRS}. In the late 1980s, Witten showed that certain 2-dimensional supersymmetric quantum field theories have deformation invariants valued in elliptic genera~\cite{Witten_Elliptic,Witten_Dirac}. Shortly thereafter, Segal suggested the existence of a map from a (yet to be constructed) space of 2-dimensional supersymmetric field theories to elliptic cohomology, realizing these quantum field theories as geometric cocycles~\cite{Segal_Elliptic,SegalElliptic}. This spurred some dazzling developments in homotopy theory led by Hopkins and collaborators~\cite{HopkinsICM94,AHSI,HopkinsICM2002,GoerssHopkins,AHSII,Goerss,Lurie,AHR,LurieI,LurieII,LurieIII}. One highlight is the construction of the universal elliptic cohomology of topological modular forms (TMF) with its string orientation. This orientation lifts the universal elliptic genus (now called the \emph{Witten genus}) to a map of ring spectra, analogous to how the Spin orientation in real K-theory lifts the $\hat{A}$-genus. 

Despite all this progress, Segal's proposed connection between field theories and elliptic cohomology remains unrealized. The original ideas have been refined considerably by Stolz and Teichner, leading to a conjecture~\cite[Conjecture~1.17]{ST11} which we loosely summarize as follows.

\begin{conj}[Segal, Stolz--Teichner]\label{conj}
For any smooth manifold $M$, there exists a natural isomorphism
\beq
&&\left\{\begin{array}{c} 2{\rm -dimensional \ quantum\ field\ theories} \\ {\rm with} \ \mathcal{N}=(0,1)\  {\rm supersymmetry\ over}\ M\end{array}\right\}/{\rm deformation} \stackrel{\sim}{\to} \TMF(M)\label{eq:TMFconj}
\eeq
that realizes deformation classes of field theories as classes in $\TMF(M)$. 
\end{conj}

Taken at face value, this is extremely surprising: quantum field theories are objects that belong to the world of analysis and differential geometry, whereas elliptic cohomology is constructed using sophisticated algebraic and homotopy-theoretic techniques. The statement is powerful in either direction. Using the left-hand side to understand the right-hand side would provide geometric cocycles for elliptic cohomology, as was Segal's original motivation. 
Conversely, the right-hand side gives a homotopy theoretic characterization of deformation classes of two-dimensional field theories, which would be a tool of great interest to high-energy theorists, e.g., see~\cite{GJFW,VafaTMF}. 

To date, the primary evidence for~\eqref{eq:TMFconj} comes from analogous statements relating K-theory and \emph{one}-dimensional supersymmetric quantum field theories, i.e., supersymmetric quantum mechanics. Indeed, Stolz and Teichner~\cite{ST04,HST} constructed a space of supersymmetric quantum mechanical theories that represents the spectrum KO, giving a parallel to~\eqref{eq:TMFconj}. The physics argument for the index theorem~\cite{WittenMorse,Alvarez} has also been deeply influential, where the index map in K-theory is identified with a quantization map in supersymmetric quantum mechanics. This point of view leads to the slogan ``the Witten genus is the index of the Dirac operator on loop space." To date, this can only be understood at a physics level of rigor~\cite{Witten_Dirac}. However, considerable insights were made early on~\cite{Taubes,BottTaubes} and progress on the analytic underpinnings continues~\cite{KottkeMelroseI,KottkeMelroseII,Melrosecourse}. 

In short, Conjecture~\ref{conj} asks for an upgrade of certain structures in supersymmetric quantum mechanics to ones in supersymmetric quantum field theory. The mathematics of quantum field theory has been fraught with a myriad of technical difficulties, which is largely why~\eqref{eq:TMFconj} is so difficult. Indeed, there are swaths of quantum field theory for which there is currently no hope for rigorous foundations. Other pieces (e.g., free theories) are well-understood by mathematicians but usually not very interesting to physicists. Conjecture~\ref{conj} identifies a Goldilocks zone: a class of quantum field theories for which one might hope that rigorous foundations exist, and yet---in view of the complexities in TMF---must be highly nontrivial.

\subsection*{An equivariant refinement of Conjecture~\ref{conj}} Supposing one could construct the space of quantum field theories on the left side of~\eqref{eq:TMFconj}, there remains another significant challenge: writing down a map. Fortunately, there is a universal property characterizing TMF. Roughly, it is the global sections of a sheaf of elliptic cohomology theories on the moduli stack of elliptic curves. 
For a fixed family of elliptic curves~$E$, this sheaf defines an elliptic cohomology theory $h_E$ that is characterized by its theory of Chern classes. For~$c$ the Chern class of the tautological line, consider pullbacks 
\beq
&&c\in h_E(\CP^\infty)\stackrel{p_1^*,p_2^*,m^*}{\longrightarrow} h_E(\CP^\infty\times \CP^\infty),\quad p_1,p_2,m\colon \CP^\infty\times \CP^\infty\to \CP^\infty,\label{eq:FGL}
\eeq
along projections $p_1,p_2$ and the multiplication map $m$ that classifies the tensor product of tautological lines. This gives a universal formula for the Chern class of a tensor product of line bundles in terms of a power series~$F$,
$$
m^*c=F(p_1^*c,p_2^*c)\in h_E(\CP^\infty\times \CP^\infty),\qquad F\in h_E(\pt)[[x,y]].
$$
By definition of the elliptic cohomology theory~$h_E$,~$F$ is the formal group law of the elliptic curve with choice of coordinate determined by~$c$. This leads to an important question: where in 2-dimensional physics do we find elliptic (formal) group laws?

One place is $U(1)$-gauge theory, where classical fields are principal $U(1)$-bundles with connection. Such bundles can be tensored, and this endows moduli spaces of $U(1)$-gauge fields with a multiplication. In particular, the flat $U(1)$-bundles on a genus~1 Riemann surface can be identified with degree zero holomorphic line bundles with group operation given by tensor product. This complex analytic group is canonically isomorphic to an elliptic curve. 
The appearance of elliptic curves as groups---rather than merely their elliptic formal group laws---suggests a connection with \emph{equivariant} elliptic cohomology. 

The coefficients for Grojnowski's complex analytic $U(1)$-equivariant elliptic cohomology is the sheaf of functions on the dual elliptic curve, $\Ell_{U(1)}(\pt)\simeq \mathcal{O}_{\EE^\vee}$~\cite{Grojnowski}. Naturality for the homomorphisms $m,p_1,p_2\colon U(1)\times U(1)\to U(1)$ then lifts the formal group law~\eqref{eq:FGL} to multiplication on the curve itself, expressed in terms of its sheaf of functions. An important structure in the theory is a decompletion of~\eqref{eq:FGL} in the sense of Atiyah--Segal completion~\cite{Atiyah_Segal_complete}. Namely, Ando constructed a $\ell$-twisted class
\beq
\begin{array}{ccl}
\Ell^\ell_{U(1)}(\pt)&\stackrel{{\rm completion}}{\longrightarrow}& h_E(BU(1))=h_E(\CP^\infty)\\
\stackrel{\downin}{c_{U(1)}} &\stackrel{}{\mapsto} & \phantom{blah}\stackrel{\downin}{c} 
\end{array}\label{eq:ellcomplete}
\eeq
that refines the non-equivariant Chern classes $c\in h_E(\CP^\infty)$ for each elliptic curve~\cite{Ando}. 
Geometrically, \emph{twisted} means that $c_{U(1)}$ is a section of a line bundle over $\EE^\vee$. Restricting to an elliptic curve curve $E$ and trivializing this line bundle near the identity $e\in E$ identifies $c_{U(1)}$ with a choice of coordinate~$c$ as in~\eqref{eq:FGL}. 


Returning to the discussion of $U(1)$-gauge theory, the sheaf $\Ell_{U(1)}(\pt)\simeq \mathcal{O}_{\EE^\vee}$ can be identified with functions on the moduli space of flat $U(1)$-bundles on a genus~1 Riemann surface. In physics, this sheaf arises as functions on a moduli space of \emph{classical vacua} or \emph{dimensionally reduced} fields in $U(1)$-gauge theory on genus~1 Riemann surfaces. There is a well-known line bundle over this moduli space that encodes the \emph{anomaly} of a family of free fermion theories. These theories are built out of the $\bar\partial$-operator on an elliptic curve twisted by a degree zero holomorphic line bundle, and the anomaly line bundle is the Quillen determinant line bundle of these $\bar\partial$-operators~\cite[\S4]{Freed_Det}. Furthermore, this line bundle has a holomorphic section determined by the \emph{partition function} of the quantum field theory. This partition function can be identified with the Weierstrass $\sigma$-function, which is a particular $\theta$-function that features prominently in the string orientation of topological modular forms~\cite{AHSI,AHR}. 

This discussion leads to a refinement of Conjecture~\ref{conj}. In the following, $\TMF_G$ denotes the derived global sections of a refinement of the sheaf $\Ell_G$ to one over~$\Z$; see~\cite{LurieIII,LenartDavid}.




\begin{conj}\label{Gconj}
For a smooth manifold $M$ with the action of a compact Lie group~$G$, there exists a natural isomorphism
\beq
&&\left\{\begin{array}{c} 2{\rm -dimensional \ quantum\ field\ theories} \\ {\rm with} \ \mathcal{N}=(0,1)\  {\rm supersymmetry\ and }\\  {\rm background\ } G \ {\rm flavor \ symmetry\ over} \ M\end{array}\right\}/{\rm deformation} \stackrel{\sim}{\to} \TMF_G(M)\label{eq:GTMFconj}
\eeq
that realizes deformation classes of field theories as classes in $\TMF_G(M)$. Furthermore, the partition function of the free fermion with background $U(1)$-gauge fields gives a cocycle refinement of the twisted equivariant Chern class, $c_{U(1)}\in \TMF^\ell_{U(1)}(\pt)$. 
\end{conj}

The relationship between gauge theories and equivariant refinements has been hinted at in the literature for years, e.g., for finite groups in~\cite{Segal_Elliptic,LibgoberSzczesny,GanterHecke}, in relation to loop spaces and loop groups in~\cite{Witten_Dirac,Brylinski,BottTaubes,KitchlooII,Rezk,Spong}, and for lower dimensional field theories in~\cite{GuilleminSternberg} and~\cite{ST11}. Indeed, the discussion in~\cite[\S1.7]{ST11} hints at a version of Conjecture~\ref{Gconj}. Although $\TMF_G$ has yet to be constructed in complete generality, the complex analytic theory has been around since the `90s~\cite{Grojnowski}, and many salient features of the general theory are known~\cite{Lurie,LurieIII,LenartDavid}.

Upgrading Conjecture~\ref{conj} to Conjecture~\ref{Gconj} serves four purposes. First, \emph{equivariant} elliptic cohomology satisfies an even stronger universal property, making a map~\eqref{eq:GTMFconj} easier to construct than~\eqref{eq:TMFconj} as outlined in~\cite[\S5.5]{Lurie}. Second, equivariant elliptic cohomology incorporates positive energy representations of loop groups, which allows one to utilize various tools from representation theory that have well-known connections with 2-dimensional field theory~\cite{PressleySegal,SegalCFT}. This runs in analogy to the situation in K-theory, wherein the representation theoretic content of equivariant refinements clarify many seemingly non-equivariant structures, such as the spin orientation and power operations. Third, equivariant TMF has significantly more structure, making it harder to guess the ``wrong" definition on the left hand side in~\eqref{eq:GTMFconj}. 
This idea of adding structure is parallel to a critical step in the construction of the derived sheaf whose global sections are~TMF: upgrading  the sheaf to one of $E_\infty$-algebras is what makes the obstruction theory tractable~\cite[\S6.2]{HopkinsICM2002},\cite[pg.~9]{Lurie}. 
Fourth and finally, the geometry of elliptic curves and their group laws is more apparent in the equivariant context. Clarifying the structures in physics that produce (formal) group laws is important not only for Conjecture~\ref{conj}, but also in getting to the root of the many parallels between structures in chromatic homotopy theory at height~$d$ and quantum field theory in dimension~$d$, e.g., see~\cite{AndoMorava,MoravaHKR,Ganterpower,ST11,Powerops}.

There are many hints from physics that a resolution of Conjecture~\ref{Gconj} will lead to genuinely new mathematics. 
Recent invariants of quantum field theories are offering hints for how to recover certain torsion classes in TMF from the geometry of manifolds~\cite{GJF2,GJFW}. In a different direction, equivariant elliptic cohomology is already playing a role in the study of 3d $\mathcal{N} = 4$ gauge theories via their 2d supersymmetric boundary conditions~\cite{AganagicOkounkov}. Finally, taking twisted compactifications of 6d $\mc{N} = (0, 1)$ field theories on $4$-manifolds yield 2d $\mc{N} = (0, 1)$-supersymmetric field theories that (by Conjecture~\ref{Gconj}) should define classes in TMF. These have an equivariant refinement by residual flavor symmetry in the 2d compactification. Gukov, Pei, Putrov, and Vafa~\cite{VafaTMF} have recently explored these intriguing $\TMF_G$-valued 4-manifold invariants, which notably generalize classical invariants such as (modified) Donaldson and Vafa--Witten invariants. Surely, these applications will continue to develop and expand as the theory does. 

\subsection*{Summary of results}

The main result of this paper is a proof of Conjecture~\ref{Gconj} over the complex numbers. For any reasonable definition of the left hand side in Conjectures~\ref{conj} and~\ref{Gconj}, there is a character map called the \emph{partition function}, which gives the downward arrow on the left:
\beq
\begin{tikzpicture}[baseline=(basepoint)];
\node (A) at (0,0) {$\left\{\begin{array}{c} 2{\rm -dimensional \ quantum\ field\ theories} \\ {\rm with} \ \mathcal{N}=(0,1)\  {\rm supersymmetry\ and }\\ G\ {\rm flavor\ symmetry\ over} \ M\end{array}\right\}$};
\node (B) at (9,0) {$\TMF_G(M)$};
\node (C) at (0,-2) {$\mathcal{O}(\cL_0(M\ncq G))$};
\node (D) at (9,-2) {$\R\Gamma(\Bun_G(\EE);\Ell_G(M))$}; 
\draw[->,dashed] (A) to node [above] {Conjecture~\ref{Gconj}}  (B);
\draw[->,dashed] (A) to node [left] {${\rm partition\ function}$}  (C);
\draw[->] (C) to node [below] {Theorem~\ref{thm:comparison}} (D);
\draw[->,dashed] (B) to node [right] {$\otimes \C$} (D);
\path (0,-.75) coordinate (basepoint);
\end{tikzpicture}\nonumber
\eeq
where $\cL_0(M\ncq G)$ is a certain moduli stack of \emph{inertia fields} constructed below. 
The upper horizontal arrow is the content of Conjecture~\ref{Gconj}, though the source and target have yet to be defined at their desired levels of generality. However, complexification on the right is expected to recover (derived global sections of) Grojnowski's complex analytic equivariant elliptic cohomology ${\Ell}_G(M)$. We construct a lower horizontal arrow as a morphism of sheaves of commutative differential graded algebras that extracts from the sheaf of functions on $\cL_0(M\ncq G)$ the de~Rham cocycle model for ${\Ell}_G(M)$ from~\cite{BET0}. In particular, this constructs Grojnowski's complex analytic equivariant elliptic cohomology, proving the first statement in Conjecture~\ref{Gconj} over $\C$.

As essentially a formal consequence of the definitions, 
 the moduli space~$\F_0(\pt\ncq U(1))$ of $U(1)$-gauge fields carries a commutative multiplication. We show that the universal dual elliptic curve $\EE^\vee$ (as a complex analytic group) is a subgroup,
\beq
\begin{tikzpicture}[baseline=(basepoint)];
\node (A) at (0,0) {$\EE^\vee$};
\node (B) at (6,0) {$\left\{\begin{array}{c} U(1) {\rm -gauge\ fields}\\ {\rm with\ multiplication }  \end{array}\right\}$};
\node (C) at (0,-1) {$c_{U(1)}$};
\node (D) at (6,-1) {$\begin{array}{c} {\rm free\ fermion}\\ {\rm partition\ function} \end{array}$}; 
\draw[->,right hook-latex] (A) to node [above] {Theorem~\ref{thm:III}}  (B);
\draw[|->] (D) to node [above] {Theorem~\ref{thm:II}} (C);
\path (0,-.75) coordinate (basepoint);
\end{tikzpicture}\nonumber
\eeq
and furthermore the free fermion partition function is a section of an \emph{anomaly line bundle} over $\F_0(\pt\ncq U(1))$ whose restriction recovers $c_{U(1)}$. The anomaly line bundle is nontrivial, but a choice of trivialization near the trivial $U(1)$-bundle (i.e., the neutral element in $\F_0(\pt\ncq U(1))$) determines a choice of coordinate on~$\EE^\vee$. This recovers the elliptic formal group law. 

We further observe that the section $c_{U(1)}$ can be identified with the super character of the vacuum representation of $LU(1)\simeq L\Spin(2)$~\cite[\S1.2]{Kefeng}, \cite[\S11.1]{Ando}; this is expanded upon in~\cite{BET0}. This tripartite description of $c_{U(1)}$---as a coordinate for a (formal) group law, a character of a loop group representation, and a partition function of a quantum field theory---gives a clear link between homotopy theory, representation theory, and physics. These three objects all have refinements to ones over~$\Z$, namely a coordinate over~$\Z$, a positive energy representation of the loop group~$LU(1)$, and the Fock space associated with the $U(1)$-gauged free fermion. This observation illuminates a path toward proving Conjecture~\ref{Gconj} over~$\Z$. 

\subsection*{Main construction: The stack of inertia fields}

Let $G$ be a compact Lie group. For a $G$-manifold~$M$, we construct a stack of fields for the classical gauged sigma model with $\mathcal{N}=(0,1)$ supersymmetry, denoted~$\cL(M\nsq G)$. Essentially, $\cL(M\nsq G)$ is an enhancement of the mapping stack~$\Map(T^{2|1},[M\nsq G])$, where $T^{2|1}\simeq \R^{2|1}/\Z^2$ is a $2|1$-dimensional super torus with conformal structure and $[M\nsq G]$ is the stack classifying principal $G$-bundles with connection and a $G$-equivariant map to $M$. There is a substack $\cL_0(M\nsq G)\subset \cL(M\nsq G)$ of \emph{inertia fields} defined analogously to the the (double) \emph{inertia} or \emph{ghost loop} stack
$$
\Map([*\sq \Z^2],[M\nsq G])\subset \Map(T^2,[M\nsq G]).
$$
We informally state the effect of the ``super" aspect as follows.

\begin{lem}[Lemma~\ref{lem:present} and \S\ref{sec:slogan}] \label{slogan1} Functions on $\F_0(M\nsq G)$ are differential forms on the ordinary double inertia stack of $[M\nsq G]$ , together with an action of symmetries,
$$
C^\infty(\F_0(M\nsq G))\simeq \Omega^\bullet(\Map([\pt\sq \Z^2],[M\nsq G]);C^\infty(\Lat))\circlearrowleft{\rm symmetries}
$$
with coefficients in functions on the moduli space $\Lat$ of conformal structures on super tori. 
\end{lem}

The action of symmetries leads to certain equivariant de~Rham complexes, using the previously-known super geometric interpretation of equivariant de~Rham cohomology reviewed in \S\ref{sec:appende}. Specifically, we recall that for ordinary double inertia stacks, \emph{twisted sectors} are specified by a pair of commuting elements~$h_1,h_2\in G$ that determine a flat $G$-bundle on~$T^2=\R^2/\Z^2$. We show that restriction to an analogously defined twisted sector determines a map from $C^\infty(\F_0(M\nsq G))$ to the equivariant de~Rham complex of $M^{h_1,h_2}$ for the action by the centralizer of $h_1,h_2$; see~\S\ref{sec:revisit}.

We define $\cL_0(M\ncq G)$ to be the coarse quotient of the stack $\F_0(M\nsq G)$ by a class of isomorphisms, namely, gauge transformations. In~\S\ref{sec:Qcoh1} we consider a subsheaf of \emph{analytic} functions on $\cL_0(M\ncq G)$, defined as smooth functions satisfying an invariance condition
motivated by ideas from physics: theories with $\mathcal{N}=(0,1)$ supersymmetry have observables that vary holomorphically under deformations of the super torus and the flat $G$-bundle; see Remark~\ref{rmk:analytic} below and~\cite[\S4.4]{DualityMock} for a closely related argument for the $\mathcal{N}=(0,1)$ supersymmetric sigma model with compact target. 
We informally state this analytic condition as follows.

\begin{defn}[Definition~\ref{defn:generalanalytic}]\label{def:informalanalytic} The sheaf $\mathcal{O}_{\cL_0(M\ncq G)}$ of analytic functions on $\F_0(M\ncq G)$ consists of gauge-invariant functions on $\F_0(M\nsq G)$ with holomorphic dependence on the moduli of flat $G$-bundles (see Remark~\ref{rmk:Hodge}). 
\end{defn}


\subsection*{Results I: A cocycle model}
Complex analytic equivariant elliptic cohomology of a $G$-manifold~$M$ is a sheaf of graded rings~$\Ell_G(M)$ on the moduli space of $G$-bundles on elliptic curves, $
\Bun_G(\EE)$~\cite{Grojnowski}. The de~Rham model for equivariant elliptic cohomology is a sheaf of commutative differential graded algebras (cdgas), $\dEll_G(M)$, on $\Bun_G(\EE)$,~\cite{BET0} (see~\S\ref{appen:ell} below for a review). There are two main steps to compare $\dEll_G(M)$ with $\mathcal{O}_{\cL_0(M\ncq G)}$:
\begin{enumerate}
\item[\underline{Step 1.}] Extract from $\mathcal{O}_{\cL_0(M\ncq G)}$ a sheaf of cdgas $\Ch^\bullet(\cL_0(M\ncq G))$ on a moduli space of $G$-bundles $\Mst_G$, with $\Mst_G$ constructed as a quotient of $\F_0(\pt\ncq G)$.
\item[\underline{Step 2.}] Compare $\dEll_G(M)$ on $\Bun_G(\EE)$ with $\Ch^\bullet(\cL_0(M\ncq G))$ on $\Mst_G$.
\end{enumerate}


\begin{rmk}Step~1 is a common procedure in supersymmetric field theory called \emph{$Q$-cohomology}: the supersymmetry generator (usually denoted~$Q$) determines a differential on the graded algebra of functions on fields. \emph{Supersymmetric localization} asserts that the restriction of $Q$-cohomology along the inclusion $\cL_0(M\nsq G)\subset \cL(M\nsq G)$ is a quasi-isomorphism; this is a formal application of equivariant localization in the infinite-dimensional setting. See~\cite{BElocalization} for a proof in the case that~$G=\{e\}$. For the case of an abelian compact Lie group, methods of~\cite{Spong} can likely be adapted to rigorously construct the localization quasi-isomorphism. 
\end{rmk}

Step~2 requires we find a comparison map between the complex analytic stack~$\Bun_G(\EE)$ and the stack $\Mst_G$ defined on the site of supermanifolds. To do this, we regard $\Bun_G(\EE)$ as a smooth stack (by forgetting the complex structure) and restrict $\Mst_G$ along the fully faithful embedding $i\colon {\sf Mfld}\to {\sf SMfld}$ of manifolds into supermanifolds. Then we construct a map between stacks on the site of smooth manifolds,
\beq
\iota\colon \Bun_G(\EE){\dashrightarrow} \Mst_G.\label{eq:iota1}
\eeq
 

\begin{thm}[Theorem~\ref{lem:proofofmain}] \label{thm:comparison}
There is an isomorphism of sheaves of commutative differential graded algebras on $\Bun_G(\EE)$,
\beq
&&\iota^{-1}\Ch^\bullet(\cL_0(M\ncq G))\stackrel{\sim}{\to} \dEll_G^\bullet(M)\label{eq:maincomparison}
\eeq
natural in~$M$ and~$G$, witnessing the sheaf of analytic functions $\mathcal{O}_{\cL_0(M\ncq G)}$ as a cocycle model for complex analytic equivariant elliptic cohomology.
\end{thm}

\begin{rmk} \label{rmk:Hodge} We recall that the Hodge bundles $\omega^{\otimes n}$ on $\Mell$ are the line bundles whose global sections are weight $n$ modular forms.  We use the same notation $\omega^{\otimes n}$ to denote the pullback of these line bundles back along the forgetful functor $\Bun_G(\EE)\to \Mell$. A consequence of Theorem~\ref{thm:comparison} are isomorphisms $\iota^{-1}\Ch^{2\bullet}(\cL_0(\pt\ncq G))\simeq \omega^{-\otimes \bullet}$, and in particular $\iota^{-1}\Ch^0(\cL_0(\pt\ncq G))\simeq \omega^{\otimes 0}\simeq \mathcal{O}_{\Bun_G(\EE)}$ is the sheaf of holomorphic functions on $\Bun_G(\EE)$.
\end{rmk}

\begin{rmk} \label{rmk:gaugeint}
Equivariant elliptic cohomology over~$\Z$ lives on a (relative) coarse moduli space of $G$-bundles over elliptic curves rather than the full moduli stack~\cite[\S5.1]{Lurie}. Theorem~\ref{thm:comparison} gives a physical explanation for this subtlety: $Q$-cohomology requires us to pass to gauge-invariant observables, which are functions on a coarse quotient of $G$-bundles. 
\end{rmk}

\subsection*{Results II: A construction of equivariant elliptic Euler classes}


Given a unitary representation~$\rho\colon G\to U(V)$, we construct a family of operators $D_\rho$ over $\F_0(\pt\nsq G)$. Physically these operators determine the free fermion classical field theory with background gauge fields. The partition function of the associated quantum theory is a regularized determinant of~$D_\rho$. Over super stacks, such regularizations become quite technical; this makes it unclear as to when different regularization procedures agree. So instead we take an approach that extends the well-developed theory of determinant lines for smooth and holomorphic families of (non-super) manifolds.


Freed~\cite[\S4]{Freed_Det} explains how the partition function of a
quantum field theory can be given rigorous meaning as a section of a determinant line bundle~\cite{Quillendet,BismutFreed1,BismutFreed2,Wittenanomaly}. In particular, twisting the $\bar\partial$-operator on an elliptic curve by a degree zero holomorphic line bundle $L$ defines a family of operators $\bar\partial_L$ over $\Bun_{U(1)}(\EE)$ whose determinant $\det(\bar\partial_L)$ is the partition function of the 2-dimensional free fermion quantum field theory with background $U(1)$-gauge fields. Indeed, formulas for the section $\det(\bar\partial_L)$ in a preferred trivialization of the determinant line bundle $\Det(\bar\partial_L)$ coincide with the formulas from physics for the partition function of the free fermion theory~\cite{Theta},~\cite[Equation~4.11]{Freed_Det}.


For $\rho=\id\colon U(1)\to U(1)$, we show that the family of operators $D_\rho$ on $\F_0(\pt\ncq  U(1))$ pulls back to the family of operators~$\bar\partial_L$ on $\Bun_{U(1)}(\EE)$; see Lemma~\ref{lem:dbarcompare}. We define a line bundle $\Det(D_\rho)$ and section $\det(D_\rho)$ on $\F_0(\pt\ncq  U(1))$ by requiring the restriction to $\Bun_{U(1)}(\EE)$ be the line bundle $\Det(\bar\partial_L)$ with section $\det(\bar\partial_L)$. 
Indeed, we show that there is a \emph{unique} analytic line bundle $\Det(D_\rho)$ with section $\det(D_\rho)$ on $\F_0(\pt\ncq U(1))$ with this property; see Proposition~\ref{prop:Unext}. Here, an analytic line is defined as a locally free rank~1 sheaf over the sheaf of analytic functions on $\F_0(\pt\ncq U(1))$ (see the informal Definition~\ref{def:informalanalytic}). By reducing to maximal tori, this further uniquely determines analytic line bundles $\Det(D_\rho)$ with sections $\det(D_\rho)$ over $\F_0(\pt\ncq G)$ when $G$ is connected with torsion free fundamental group. Under certain conditions on $\rho$ (see below) we identify $\Det(D_\rho)$ and its section $\det(D_\rho)$ with the unique analytic extension of the Looijenga line bundles $\mathcal{L}(c)$ on $\Bun_G(\EE)$, defined relative to a class $[c]\in \H^4(BG;\Z)$. 




%
%
%
%
%

\begin{thm}[Proposition~\ref{prop:Unext}, Theorem~\ref{thm:Eulercompare}]\label{thm:II} For $G$ connected with torsion-free fundamental group, the Freed--Quillen determinant line and section extend uniquely to the super moduli space $\F_0(\pt\ncq G)$, defining a line bundle $\Det(D_\rho)$ with section $\det(D_\rho)$. This determines a section of the complex of sheaves $\Ch^\bullet(\Det(D_\rho))$ on $\Mst_G$ concentrated in degree $2{\rm dim}(\rho)$, whose restriction along~\eqref{eq:iota1} recovers the twisted equivariant Euler class from~\cite{BET0}. If $c_1(\rho)=0$ mod~2, then $\Det(D_{\rho})$ is the unique analytic extension of $\mc{L}\Big(\frac{1}{2}c_1(\rho)^2 - c_2(\rho)\Big) \otimes \omega^{-{\rm dim}(\rho)}$ for $\frac{1}{2}c_1(\rho)^2 - c_2(\rho)=\frac{p_1}{2}(\rho)\in \H^4(BG;\Z).$
\end{thm}

%

\begin{rmk}
The line bundle $\Det(D_\rho)$ is the \emph{anomaly} of the quantum field theory built out of $D_\rho$. In particular, this line measures the extent to which the partition function is not actually a function. The characteristic class obstructions $c_1(\rho)\pmod{2}$ and $\frac{1}{2}c_1(\rho)^2 - c_2(\rho)=\frac{p_1}{2}(\rho)$ in Theorem~\ref{thm:II} are the expected string anomaly~\cite{Wittenanomaly,Freed_Det}. The remaining line~$\omega^{-n}$ is the conformal anomaly. 
\end{rmk}


At the root of Theorem~\ref{thm:II} is the fact that the determinant section~$\det(D_\rho)$ is essentially a product of Weierstrass $\sigma$-functions,
\beq
&&\sigma(\tau, z) := (e^{\pi i z}-e^{-\pi i z}) \prod_{n=1}^{\infty} \frac{(1-q^n e^{2\pi i z})(1 - q^n e^{-2\pi i z})}{(1 - q^n)^2}\in \mathcal{O}(\HH\times \C)\label{eq:Weierstrass}
 \eeq
 for $q=\exp(2\pi i \tau)$ with $\tau\in \HH$ the upper half plane and $z\in \C$. We caution that there are inconsistent conventions for the $\sigma$-function in the literature; we follow~\cite{HopkinsICM2002}.

\subsection*{Results III: Field theories and formal group laws}
There is a (weakly) commutative group structure on $U(1)$-gauge fields: given a pair of $U(1)$-principal bundles with connection, one can take their product, i.e., the product of $U(1)$-bundles compatible with the tensor product of the associated complex line bundles. This geometric operation coincides with the map
\beq
&&\F_0(\pt\nsq U(1))\times_{\F_0(\pt)}\F_0(\pt\nsq U(1))\simeq \cL_0(\pt\nsq U(1)\times U(1))\stackrel{m}{\to} \cL_0(\pt\nsq U(1))\label{eq:wkgrp}
\eeq
where $m$ is induced by naturality in the multiplication homomorphism~$U(1)\times U(1)\to U(1)$. We show that~\eqref{eq:wkgrp} determines a multiplication on the quotient $\Mst_{U(1)}$ which we compare with the group operation on $\EE^\vee\simeq\Bun_{U(1)}(\EE)$ under the map
\beq
\EE^\vee\simeq \Bun_{U(1)}(\EE)\stackrel{\iota}{\to} \Mst_{U(1)}.\label{Eq:iotainjective}
\eeq
In the following $\widetilde{\EE}^\vee$ and $\Mstil_{U(1)}$ denote $\SL_2(\Z)$-covers of $\EE^\vee$ and $\Mst_{U(1)}$, corresponding to the $\SL_2(\Z)$ cover $\HH\to [\HH\sq \SL_2(\Z)]\simeq \Mell$ of the moduli stack of elliptic curves. 

\begin{thm}[Proposition~\ref{prop:85} and Theorem~\ref{thm:GL}]\label{thm:III}
The map~\eqref{Eq:iotainjective} is an injective homomorphism of commutative group objects over~$\HH$. This homomorphism is compatible with the isomorphism of sheaves of analytic functions on $\Mst_{U(1)}$ and holomorphic functions on~${\EE}^\vee$. Any choice of analytic trivialization of $\Det(\mathcal{D})$ in a neighborhood of the trivial bundle $\Mstil_{\{e\}}\hookrightarrow \Mstil_{U(1)}$ determines a formal group law associated with the complex analytic group~$\widetilde{\EE}^\vee$ with coordinate given by the restriction of $\det(\mathcal{D})$ along~\eqref{Eq:iotainjective} in the chosen trivialization of $\Det(\mathcal{D})$. 
\end{thm}

The above gives a physical explanation for the appearance of elliptic group laws. These group laws are central to the construction of elliptic cohomology, and hence topological modular forms. This being the case, it is difficult to imagine a proof of Conjecture~\ref{conj} that does not involve elliptic formal group laws in some essential way. Theorem~\ref{thm:III} gives geometric means to access this information on the physics side of the proposed equivalence. 


\subsection*{Outline} Section~\ref{sec:motivated} overviews the ideas from physics that go into our constructions. Section~\ref{sec:fields} gives rigorous meaning to the objects that were informally introduced in~\S\ref{sec:motivated}, defining a stack on the site of supermanifolds of (inertia) fields. Section~\ref{sec:twisted} introduces the main technical tools we use to analyze inertia fields. One crucial construction is that of \emph{twisted sectors}, which defines a map from functions on $\F_0(M\nsq G)$ to an equivariant de~Rham complex of the fixed point sets $M^{h_1,h_2}$ for every pair of commuting elements $h_1,h_2\in G$. Section~\ref{sec:Qcoh1} defines the sheaf of analytic functions on inertia fields and the $Q$-cohomology sheaf. Section~\ref{sec:compare} compares this physical definition of equivariant elliptic cohomology with the de~Rham model from~\cite{BET0}, proving Theorem~\ref{thm:comparison}. 
Section~\ref{sec:freefer} introduces the gauged free fermion theory whose partition function is defined as a section of a determinant line. This constructs the Looijenga line bundle and cocycle refinement of the elliptic Euler class, proving Theorem~\ref{thm:II}. In Section~\ref{sec:FGL} we establish the connection between these ideas from physics and formal group laws, proving Theorem~\ref{thm:III}. Appendix~\ref{appen:equivdeRham} gives a brief introduction to supermanifolds and superstacks, and collects some technical results we require in the body of the paper. Appendix~\ref{sec:appende} reviews the previously-known super geometric interpretation of equivariant de~Rham cohomology. 
None of the results are new, but precise statements and proofs can be difficult to extract from the literature. Finally, Appendix~\ref{appen:ell} gives a quick review of the de~Rham model for equivariant elliptic cohomology from~\cite{BET0}.



\subsection*{Notation and conventions}
Throughout, $M$ will denote a smooth manifold equipped with the action of a compact Lie group~$G$. For simplicity we make the technical assumption that $M$ embeds $G$-equivariantly into a finite-dimensional $G$-representation. This is automatically satisfied when $M$ is compact by results of Mostow~\cite{Mostow} and Palais~\cite{Palais}. 
For $h \in G$, we denote by $M^h$ the submanifold of $M$ fixed by~$h$. Similarly, if $h = (h_1, h_2)$ is a pair of elements in $G$, we use the same notation $M^h=M^{h_1,h_2}$ for the submanifold fixed by both $h_1, h_2$. When $G$ acts on itself by conjugation, we observe that $G_0^h$ is the centralizer of $h$.
For a super Lie group $H$ acting on a supermanifold $N$, there are three possible quotients we will use:
\begin{enumerate}
\item the coarse quotient, denoted $N\cq H$, taken in sheaves on the site of supermanifolds;
\item the groupoid quotient, denoted $N\gq H$, taken in super Lie groupoids;
\item the stack quotient, denoted $[N\sq H]$, taken in stacks on the site of supermanifolds.
\end{enumerate} 
We caution that the double slash notation is also sometimes used to denote the~GIT quotient, but we always mean the groupoid or stack quotient. Diagrams in stacks are always assumed to be 2-commutative unless stated otherwise; 2-commutativity is extra data that we often omit in the notation. 

Unless otherwise stated, all the objects in this paper are $\Z/2$-graded, i.e., \emph{super}: vector space will mean super vector space, algebra will mean super algebra, Lie group will mean super Lie group, etc. We occasionally use the super adjective for emphasis. Tensor products of algebras of functions or spaces of sections will always be taken as the projective tensor product of Fr\'echet spaces. This is a completion of the algebraic tensor product having the key property that $C^\infty(M\times N)\simeq C^\infty(M;C^\infty(N))\simeq C^\infty(M)\otimes C^\infty(N)$ for supermanifolds~$M$ and~$N$. 

Sheaves of commutative differential graded algebras (cdgas) or sheaves of chain complexes are always taken in the strict sense: they are complexes of sheaves with additional structures. 

We define modular forms as functions on the upper half plane $\HH$ or the space of based lattices $\Lat$ with properties. The sheaf of holomorphic functions on $\HH$ and $\Lat$ will always be taken to be the one that imposes meromorphicity at infinity, so that by ``modular forms'' we always implictly mean ``weakly holomorphic modular forms.'' More precisely, for an open $U\subset \HH$, the sections $\mathcal{O}(U)$ are the holomorphic functions on $U$ with at-worst polynomial growth along any geodesic escaping to $\partial \HH$ in the standard hyperbolic metric. 
 
\subsection*{Acknowledgements} This work relies on many ideas that have been circulating in the community for years, based on both the published and unpublished work of many people. We were especially influenced by the wealth of ideas of Mike Hopkins, Graeme Segal, Stephan Stolz, and Peter Teichner. In addition, we would like to thank Mina Aganagic, Matt Ando, Ralph Cohen, Kevin Costello, Tudor Dimofte, Dan Freed, Soren Galatius, Nora Ganter, Theo Johnson-Freyd, Sheldon Katz, Nitu Kitchloo, Tom Nevins, Andrei Okounkov, Natalie Paquette, Arun Ram, and Charles Rezk for numerous insights, conversations, helpful comments on the draft, and patient encouragement. Finally, A.T. acknowledges the support of MSRI and the NSF through grants~1705008 and~1440140.

\section{Motivation: Classical fields, $Q$-cohomology, and localization}\label{sec:motivated}

The goal of this section is to explain the key ideas from physics that go into the cocycle model for equivariant elliptic cohomology (and in particular, $Q$-cohomology) while avoiding the language of supermanifolds. Precise definitions for the central objects are given in~\S\ref{sec:fields}. 

\subsection{Classical fields and the path integral}

We begin with an informal definition of the category of fields for the $\mathcal{N}=(0,1)$ supersymmetric sigma model with target $M$, background flavor symmetry for a compact Lie group~$G$, and fermions valued in a vector bundle $V\to M$. Below, a spin structure on a Riemann surface~$\Sigma$ is taken to mean a holomorphic line bundle~$\mathbb{S}$ on $\Sigma$ equipped with an isomorphism $\mathbb{S}\otimes\mathbb{S}\simeq \Omega^{1,0}_\Sigma$ as in~\cite{AtiyahSpin}.

\begin{defn}[Informal] \label{defn:informalfields} Given input data:
\begin{enumerate}
\item a compact Lie group $G$ with bi-invariant metric;
\item a Riemannian manifold $M$ with an isometric $G$-action; and
\item a $G$-equivariant vector bundle $V\to M$ with invariant metric and connection
\end{enumerate}
the \emph{groupoid of fields} has as objects quadruples 
\begin{enumerate}
\item a principal $G$-bundle $P\to \Sigma$ with connection over a spin Riemann surface $\Sigma$;
\item a section $x\in \Gamma(\Sigma;P\times_G X)$;
\item a right-moving spinor $\psi\in \Gamma(\Sigma;\overline{\mathbb{S}}\otimes x^*P\times_G TX)$; and
\item a left-moving spinor $\eta\in \Gamma(\Sigma;\mathbb{S}\otimes x^*P\times_GV)$. 
\end{enumerate}
\emph{Isomorphisms} between fields are isomorphisms of $G$-bundles with connection $P\to P'$ covering an isomorphism of spin Riemann surfaces $\Sigma\to \Sigma'$ such that $(x',\psi',\eta')$ pull back to $(x,\psi,\eta)$. The \emph{classical action} is a function on fields given by
\beq
\mathcal{S}(x,\psi,\eta)=\frac{1}{2}\int_\Sigma \big(\langle \partial x,\overline{\partial}x\rangle+\langle \psi,\partial_\nabla \psi\rangle+\langle \eta,\overline{\partial}_\nabla \eta\rangle\big)\label{eq:classicalaction}
\eeq
where the pairings are defined using the metrics on~$G,M$ and $V$, and the operators~$\partial_\nabla,\overline{\partial}_\nabla$ are twists of the Dolbeault differentials on $\Sigma$ using the $G$-connection on $P$ and the connections on~$V$ and~$TM$. The classical action is invariant under isomorphisms of fields. 
\end{defn}

\begin{ex} When $G=\{e\}$, the classical fields consist of maps $x\colon \Sigma\to M$, and sections $\psi\in \Gamma(\Sigma;\overline{\mathbb{S}}\otimes x^*TM)$ and $\eta\in \Gamma(\Sigma;\mathbb{S}\otimes x^*V)$. Morphisms between fields come from pulling back along isomorphisms of Riemann surfaces. The classical action reduces to the one in~\cite[pg.~509]{strings1}, which is the worldsheet theory of the heterotic string. 
\end{ex}

\begin{ex}\label{ex:fermions} When $M=\{\pt\}$ is the 1-point space, an equivariant vector bundle $V$ is the data of a $G$-representation. Then a field is just a spinor $\eta\in \Gamma(\Sigma;\mathbb{S}\otimes P\times_GV)$, as $x$ and $\psi$ are no additional data. Furthermore,~\eqref{eq:classicalaction} becomes the chiral free fermion classical action with background gauge fields. The superspace extension of this field theory (when~$\Sigma$ has genus~1) is the input to our construction of equivariant elliptic Euler classes in~\S\ref{sec:freefer}. 
\end{ex}

\begin{ex} \label{eq:groupstructurespinor}
If we further restrict Example~\ref{ex:fermions} to $G=U(1)$ and $V=\C$ the standard representation of $U(1)$, then fields are sections $\eta\in \Gamma(\Sigma;\mathbb{S}\otimes L)$ for complex line bundles~$L$. The tensor product of principal $U(1)$-bundles extends to a multiplication on fields if~$\mathbb{S}$ is trivialized. Indeed, a trivialization specifies an isomorphism~$(\mathbb{S}\otimes L)\otimes (\mathbb{S}\otimes L')\simeq (\mathbb{S}\otimes L\otimes L')$, so that sections $(\eta,\eta')$ can be multiplied as $(\eta\otimes\eta')$. For $\mathbb{S}$ to be trivializable, $\Sigma$ must be of genus~1 and endowed with the odd spin structure. 
\end{ex}

Quantizing the classical field theory in Definition~\ref{defn:informalfields} using the path integral looks to evaluate \emph{(quantum) expectation values}
\beq
\langle f\rangle =\int fe^{-\mathcal{S}}[dxd\psi d\eta]\label{eq:pathint}
\eeq
where $f$ is a function on fields (i.e., a \emph{classical observable}) and $[dxd\psi d\eta]$ is a hypothetical measure on fields. The result $\langle f\rangle$ is expected to be a function on the moduli space $\Bun_G^\nabla$ of $G$-bundles with connections over Riemann surfaces. This is why the gauge fields---i.e., principal bundles with connections---are called \emph{background fields} in this context: the path integral regards them as being constant. More generally, in the presence of \emph{anomalies} we might ask that~$\langle f\rangle$ be a section of a line bundle~$\mathcal{L}$, so we would like a map
\beq
\langle-\rangle\colon \{{\rm functions\ on \ fields} \}\dashrightarrow \Gamma(\Bun_G^\nabla;\mathcal{L}).\label{eq:expectation}
\eeq
Unfortunately, the measure in~\eqref{eq:pathint} is notoriously difficult to construct. 

\begin{rmk}\label{rmk:grading}
One reason that Definition~\ref{defn:informalfields} is only informal is that functions on fields carry a $\Z/2$-grading: coordinate functions dual to~$\psi$ and $\eta$ are regarded as odd, whereas coordinate functions in~$x$ are even. Supermanifolds and superstacks formalize this idea, as we will make precise in~\S\ref{sec:fields}. Below we assume that functions on fields are a $\Z/2$-graded algebra. 
\end{rmk}

\subsection{Supersymmetry and $Q$-cohomology} A common method to evaluate the path integral is to impose symmetry constraints on the map~\eqref{eq:expectation}, with the hope that any would-be path integral only depends on a finite-dimensional measure. It turns out that the field theory from Definition~\ref{defn:informalfields} permits such a treatment, as we shall explain. 

\begin{rmk}
Another reason Definition~\ref{defn:informalfields} is informal is that it only defines fields as a set, whereas fields are usually treated as a differential-geometric object. For example, the classical action is a smooth function whose critical points are classical solutions. Below we will assume that such differential-geometric constructions are well-behaved, referring to~\S\ref{sec:fields} for a rigorous definition of fields.  
\end{rmk}

\begin{defn}
An \emph{infinitesimal symmetry} on the space of fields is a vector field $X$ preserving the classical action, i.e.,~$X\mathcal{S}=0$. 
\end{defn} 

\begin{hyp}\label{hyp1}
The map~\eqref{eq:expectation} must respect infinitesimal classical symmetries, meaning $\langle Xf\rangle=0$ for any infinitesimal symmetry $X$.
\end{hyp}

\begin{ex}\label{ex:symmetry}
Given a 1-parameter family of isomorphisms of fields in the sense of Definition~\ref{defn:informalfields}, differentiation defines an infinitesimal symmetry. 
\end{ex}

There are important hidden symmetries called \emph{supersymmetries} that do not arise from differentiation of families of isomorphisms. To describe these, it is convenient to introduce an additional \emph{auxiliary field} denoted~$\aux$ defined as follows.

\begin{defn}[Informal fields, part 2]\label{defn:inffields2} Define objects of the category of fields as quintuples with data (1)-(4) as previously, and additionally 
\begin{enumerate}
\item[(5)] a section $\aux \in \Gamma(\Sigma,x^*V)$.
\end{enumerate} 
Morphisms of fields are defined as before, with the additional condition that the $G$-bundle isomorphism preserve $\aux$. The classical action is modified by adding the term $\int_\Sigma \langle \aux,\aux\rangle$. 
\end{defn}

\begin{rmk} Because the term in the classical action involving $\aux$ has no derivatives, formally integrating over $\aux$ changes the map~\eqref{eq:expectation} by an overall constant. Hence, typically the field~$\aux$ is ignored when passing to the quantum theory.\end{rmk}

\begin{ex}[A supersymmetry]\label{ex:supersym} Suppose we are given a holomorphic section $\epsilon$ of~$\mathbb{S}^\vee$. This determines a supersymmetry that is usually written as
\beq
\delta x=\overline{\epsilon} \psi,\quad \delta \psi=\overline{\epsilon} \overline{\partial} x,\quad \delta\eta=\epsilon {\partial}_\nabla \aux,\quad \delta \aux=\epsilon\eta.\label{eq:infsuper}
\eeq
This defines a vector field on fields that at the point $(x,\psi,\eta,\aux)$ is the tangent vector $(\psi,\overline{\partial}x,\overline{\partial}_\nabla \aux,\eta),$ using that $\mathbb{S}\otimes\mathbb{S}\simeq \Omega^{1,0}_\Sigma$. This determines an infinitesimal symmetry of~$\mathcal{S}$. 
\end{ex}

\begin{rmk} The data~\eqref{eq:infsuper} defining the supersymmetry are incomplete: we must also specify an infinitesimal action on the $G$-bundle~$P$ with its connection. This turns out to involve a weak action on a category. One way to get a strict action is to pass to \emph{gauge-invariant observables}, meaning functions on the (coarse) quotient $\F(M\nsq G)/{\rm gauge}$; compare Remark~\ref{rmk:gaugeint}. For this reason, below we shall restrict to gauge-invariant functions on fields. 
\end{rmk}

\begin{ex} In the case that $G=\{e\}$, the supersymmetry from the previous example is the chiral supersymmetry of the heterotic worldsheet, e.g., see~\cite[Equations~2.6-2.7]{AlvarezSinger}. 
\end{ex}

We observe that a nonvanishing global holomorphic section $\epsilon$ of~$\mathbb{S}^\vee$ exists if and only if $\Sigma$ has genus~1 and the spin structure is odd (aliases: periodic-periodic, nonbounding). Hereafter, we will restrict to this component of the category of fields, and we shall fix the trivialization of $\mathbb{S}^\vee$. For the remainder of the section, we will also set $V=\{0\}$ to simplify the discussion. 

\begin{notation}[Informal] Let $\F(M\nsq G)$ denote the component of the category of fields in which $\Sigma$ has genus~1 with nonbounding spin structure together with a nonvanishing holomorphic section $\epsilon$ of $\mathbb{S}$, and set $V=\{0\}$. Let~$Q$ denote the derivation on gauge-invariant functions on~$\F(M\nsq G)$ associated with the infinitesimal symmetry from~\eqref{eq:infsuper}. 
\end{notation}

\begin{rmk}
We observe that the conditions for the existence of a group structure on $U(1)$-gauge fields as in Example~\ref{eq:groupstructurespinor} are the same conditions required for the existence of a global supersymmetry. 
\end{rmk}

The vector field~$Q$ on $\F(M\nsq G)$ is \emph{odd}, meaning it is an odd derivation relative to the $\Z/2$-grading on functions on fields outlined in Remark~\ref{rmk:grading}. Hence, restricting to $Q$-closed functions on fields yields a $\Z/2$-graded chain complex with differential $Q$.

\begin{defn}[Informal]
The \emph{$Q$-cohomology ring}, denoted $\Ch^\bullet(\F(M\nsq G))$, is the $\Z/2$-graded chain complex for the differential~$Q$ acting on 
\beq
\{{\rm gauge- and} \ Q^2{\rm-invariant\ functions\ on}\ \F(M\nsq G)\}.\label{eq:chainQ}
\eeq
\end{defn}

By Hypothesis~\ref{hyp1}, the map~\eqref{eq:expectation} on $Q$-closed functions on $\F(M\nsq G)/{\rm gauge}$ only depends on the cohomology of the chain complex $\Ch^\bullet(\F(M\nsq G))$. The flexibility offered by replacing~\eqref{eq:chainQ} with any quasi-isomorphic chain complex will make it easier to construct a map~\eqref{eq:expectation} as
\beq
\langle-\rangle\colon \H(\Ch^\bullet(\F(M\nsq G)))\dashrightarrow \Gamma(\Bun_G^\nabla;\mathcal{L}).\label{eq:expectation2}
\eeq

\subsection{Localization} \label{sec:local}We recall the following theorem in equivariant cohomology.
 
 \begin{thm}[\cite{BerlineVergne2}, Proposition~2.1] \label{thm:BV1} Let~$M$ be a compact manifold with the action by a torus~$\T$. For $\xi\in\mathfrak{t}$, consider the complex $(\Omega(M)^\xi,Q)$ of $\xi$-invariant differential forms on~$M$ with the Cartan differential $Q=d-\iota_\xi$. Let $M_0\subset M$ be the zero set of~$\xi$. The natural restriction map 
 $$
 (\Omega(M)^\xi,Q)\to (\Omega(M_0),d)
 $$
 is a quasi-isomorphism of $\Z/2$-graded complexes. 
 \end{thm}

Supersymmetric localization looks to apply the above theorem to spaces of fields in (supersymmetric) quantum field theory. In our situation, a genus~1 Riemann surface $\Sigma$ can be described as $\Sigma\simeq \C/\Lambda$ for $\pi_1(\Sigma)=\Lambda\subset \C$ a lattice. Viewed as a group, the torus $\C/\Lambda\simeq \mathbb{T}$ acts on $\Sigma$. This endows the fields $\F(M\nsq G)$ with a $\mathbb{T}$-action. It turns out that the supersymmetry operator~$Q$ can be viewed as an equivariant differential. 

\begin{conj}[Informal]\label{conj:local} Theorem~\ref{thm:BV1} holds for the $\mathbb{T}$-action on $\F(M\nsq G)$ and the complex vector field determined by $Q$, meaning the map induced by restriction
$$
\Ch^\bullet(\F(M\nsq G))\to \H(\{{\rm zero\ set\ of\ } Q^2\},Q)=:\Ch^\bullet(\F_0(M\ncq G))
$$
is a quasi-isomorphism where $\Ch^\bullet(\F_0(M\ncq G))$ denotes the complex for $Q$ acting on gauge-invariant functions functions on~$\F_0(M\nsq G)=\{{\rm zero\ set\ of\ } Q^2\}\subset \F(M\nsq G)$. 
\end{conj}

Since the space of fields in this context is infinite-dimensional, one can not hope to apply Theorem~\ref{thm:BV1} directly. However, loop space versions of this conjecture (replacing the Riemann surface $\Sigma$ with a circle) are part of the story that explains the Atiyah--Singer index theorem in the language of supersymmetric quantum mechanics~\cite{WittenMorse,AtiyahCircular}. Furthermore, Conjecture~\ref{conj:local} has been made precise and verified when $G=\{e\}$~\cite{BElocalization}. For the purposes of this paper, we apply the above conjecture formally in the case of interest, and focus our study on the chain complex $\Ch^\bullet(\F_0(M\ncq G))$. 

\subsection{Fields in the language of stacks} To bridge the gap between the informal description of fields in Definition~\ref{defn:informalfields} and the rigorous description in the next section, we translate Definition~\ref{defn:informalfields} and~\ref{defn:inffields2} into the language of stacks on the site of (non-super) manifolds. 

The input data in Definition~\ref{defn:informalfields} can be rephrased in terms of a quotient stack $[M\nsq G]$ and a vector bundle $[V\nsq G]$ over $[M\nsq G]$. We recall that the stack $[M\nsq G]$ classifies principal $G$-bundles $P$ with connection and a $G$-equivariant map $P\to M$. We remark that the equivariant vector bundle~$TM$ defines a vector bundle $[TM\nsq G]$ over $[M\nsq G]$ in the category of stacks. Then (from Definition~\ref{defn:inffields2}) the category of fields has as objects quadruples 
\beq
\begin{array}{c}\phi\colon \Sigma\to [M\nsq G], \ \ \psi\in \Gamma(\Sigma,\overline{\mathbb{S}}\otimes \phi^*[TM\nsq G]), \\ 
\eta\in \Gamma(\Sigma,{\mathbb{S}}\otimes \phi^*[V\nsq G]), \ \ \aux\in \Gamma(\Sigma,\phi^*[V\nsq G]),
\end{array}\label{eq:fields2}
\eeq
where $\psi$ and $\eta$ are equivalent to the original data given by the same notation, and $\phi$ is equivalent to $(P,x)$. Isomorphisms between objects come from isomorphisms $\phi\stackrel{\sim}{\to}\phi'$ between maps of stacks that are compatible with the sections $\psi$ and $\eta$. In the special case that $\Sigma=\C/\Lambda$ has genus~1, this category of fields has an action by $\mathbb{T}\simeq \C/\Lambda$ acting on $\Sigma\simeq \C/\Lambda$, which changes $\phi$ by precomposition and $\psi,\eta$ by pullback. Taking the nonbounding spin structure on $\Sigma\simeq \C/\Lambda$, a nonvanishing holomorphic section of the spinor bundle determines a supersymmetry operator~$Q$. The zero set of $Q^2$ consists of maps $\phi\colon \Sigma\to [M\nsq G]$ associated to a \emph{flat} $G$-bundle over $\Sigma$ and sections $(x,\psi,\eta)$ that are (covariantly) constant over $\Sigma$. 

This stacky-type description of fields will be our approach taken in the next section, with some important modifications. First, we consider \emph{smooth families} of surfaces $\Sigma$; this allows one to make rigorous sense of the differential geometry of the stack of fields, e.g., vector fields come from differentiation along the family parameter. Second, the description~\eqref{eq:fields2} is replaced by an analogous one where $\Sigma$ is a \emph{super}manifold, making fields into a stack on the site of supermanifolds. There are two reasons for adopting this language of superstacks. First, functions on supermanifolds have a $\Z/2$-grading, so that functions on fields as a superstack have a $\Z/2$-grading as well. Second, the supersymmetry operator $Q$ comes from an infinitesimal automorphism of the supermanifold $\Sigma$ in this setup. Hence, this language of superstacks gives a precise description of fields and the $Q$-cohomology ring that was described informally above. 

\section{Classical fields as a superstack}\label{sec:fields}

The goal of this section is to give rigorous definitions for the space of fields $\F(M\nsq G)$ and some related constructions. These definitions are couched in the language of supermanifolds and superstacks; see Appendix~\ref{sec:smfld} for a brief review. 

\subsection{The moduli of super conformal tori}\label{sec:supertori}

Recall two descriptions of $S$-points of $\R^{2|1}$
\beq
\R^{2|1}(S)&\simeq& \{x,y \in C^\infty(S)^\ev,\ \theta \in C^\infty(S)^\odd\mid (x)_\red=\overline{(x)}_\red, (y)_\red=\overline{(y)}_\red\}\label{eq:r211}\\
&\simeq& \{z,w \in C^\infty(S)^{\ev}, \theta\in C^\infty(S)^\odd\mid (z)_\red=\overline{(w)}_\red\},\nonumber
\eeq
where reality conditions are imposed on the reduced manifold $S_\red\hookrightarrow S$. The isomorphism between descriptions in~\eqref{eq:r211} is $(x,y)\mapsto (x+iy,x-iy)=(z,w)$. We adopt the standard (though potentially misleading) notation~$\overline{z}:=w$ so that
\beq
\R^{2|1}(S)\simeq \{z,\bar z \in C^\infty(S)^{\ev}, \theta\in C^\infty(S)^\odd\mid (z)_\red=\overline{(\bar z)}_\red\} \label{eq:r212}
\eeq
but we emphasize that $z$ and $\bar z$ are only complex conjugates on their restriction to~$S_\red$; see Example~\ref{ex:SptRnm} for further discussion. 

Define the group of \emph{super Euclidean translations}, denoted~$\E^{2|1}$, as the (super) Lie group whose underlying supermanifold is $\R^{2|1}$ with multiplication
\beq
&&(z,\bar{z},\theta)\cdot (z',\bar{z}',\theta')=(z+z',\bar{z}+\bar{z}'+\theta\theta',\theta+\theta'),\quad (z,\bar{z},\theta),\ (z',\bar{z}',\theta')\in \R^{2|1}(S),\label{eq:defnE21}
\eeq
where the formula is given in terms of the functor of points description from~\eqref{eq:r212}. Since $\theta\theta'=-\theta'\theta$, we observe that~$\E^{2|1}$ is noncommutative. The reduced Lie group of~$\E^{2|1}$ is the usual group of Euclidean translations~$\E^2$, and the canonical inclusion~$\E^2\hookrightarrow \E^{2|1}$ is a homomorphism of Lie groups. The vector fields
\beq
D=\partial_\theta-\theta\partial_{\bar z}  \qquad Q=\partial_\theta+\theta\partial_{\bar z}\label{eq:defnQ}
\eeq
are left- and right-invariant on $\E^{2|1}$, respectively. These satisfy
\beq
\frac{1}{2}[D,D]=D^2=-\partial_{\bar z} \qquad \frac{1}{2}[Q,Q]=Q^2=\partial_{\bar z},\qquad [Q,D]=0,\label{eq:Qsq}
\eeq
where $[-,-]$ denotes the graded (or \emph{super}) commutator. Similarly to~\eqref{eq:r212}, we adopt the notation
$$
\C^\times(S)=\{\mu,\bar\mu\in (C^\infty(S)^\times)^\ev \mid(\mu)_\red=\overline{(\bar \mu)}_\red\}.
$$

\begin{defn}\label{defn:isom} 
Define the \emph{rigid conformal group} as
$$
\Conf(\R^{2|1}):=\E^{2|1}\rtimes \C^\times
$$
with the action of $\C^\times$ on $\E^{2|1}$ by 
\beq
&&(\mu,\bar{\mu})\cdot(z,\bar{z},\theta)=(\mu^{2}z,\bar{\mu}^{2}\bar{z},\bar{\mu}\theta),\quad (\mu,\bar{\mu})\in \C^\times(S),\ (z,\bar{z},\theta)\in \R^{2|1}(S).\label{eq:supertrans}
\eeq
A map $S\times \R^{2|1}\to S'\times \R^{2|1}$ over a base change $S\to S'$ is a \emph{fiberwise rigid conformal map} if it is equal to a composition
$$
S\times \R^{2|1}\to S\times \Conf(\R^{2|1})\times \R^{2|1}\to \R^{2|1}
$$
where the first arrow specifies an $S$-point of $\Conf(\R^{2|1})$, and the second arrow is the left action of $\Conf(\R^{2|1})$ on $\R^{2|1}$.
\end{defn}

\begin{defn} Define the supermanifold $\Lat$ of \emph{based lattices} as the open sub supermanifold $\Lat\subset \C^\times\times \C^\times$ whose $S$-points are $(\lambda_1,\bar\lambda_1,\lambda_2,\bar\lambda_2)\in \Lat(S)\subset (\C^\times\times\C^\times)(S)$ with the property that $(\lambda_1/\lambda_2,\bar\lambda_1/\bar\lambda_2)\in \HH(S)$ defines an $S$-point of the upper half plane, $\HH\subset \C$. We often use the shorthand $\Lambda=(\lambda_1,\bar\lambda_1,\lambda_2,\bar\lambda_2)$ for an $S$-point of~$\Lat$. 
\end{defn}

\begin{rmk}
We observe that the map on $S$-points $(\lambda_1,\bar\lambda_1,\lambda_2,\bar\lambda_2)\to (\lambda_1/\lambda_2,\bar\lambda_1/\bar\lambda_2,\lambda_2,\bar\lambda_2)$ gives an isomorphism $\Lat\simeq \HH\times \C^\times$. In particular, $\Lat$ is representable and is in the image of the faithful embedding of manifolds into supermanifolds. Furthermore, as an open submanifold $\Lat\subset \C^\times\times \C^\times$, it has a canonical complex structure. 
\end{rmk}

\begin{rmk}
The canonical $\Lat$-point $(\lambda_1,\bar\lambda_1,\lambda_2,\bar\lambda_2)\in \Lat(\Lat)\subset (\C^\times\times \C^\times)(\Lat)$ can be identified with the functions $\lambda_1,\bar\lambda_1,\lambda_2,\bar\lambda_2\in C^\infty(\Lat)$ that are the restrictions of the standard holomorphic coordinate functions along the inclusion $\Lat\hookrightarrow \C\times \C$. It will also be useful to identify these functions on $\Lat$ with natural transformations $\lambda_1,\bar\lambda_1,\lambda_2,\bar\lambda_2\colon \Lat\to C^\infty(-)^\ev$, that assign the functions $\lambda_1,\bar\lambda_1,\lambda_2,\bar\lambda_2 \in C^\infty(S)^\ev$ to an $S$-point $(\lambda_1,\bar\lambda_1,\lambda_2,\bar\lambda_2)\in \Lat(S)$. For this reason, we will sometimes abuse notation, letting $\lambda_1,\bar\lambda_1,\lambda_2,\bar\lambda_2$ denote both functions on $\Lat$ and part of the data of an $S$-point of $\Lat$. 
\end{rmk}

A family of based lattices can be identified with an $S$-family of monomorphisms $\Lambda\colon S\times \Z^2\to S\times \C\simeq S\times \R^2\subset S \times \R^{2|1}$ where the image of the generators of $\Z^2$ is given by $(\lambda_1,\bar\lambda_1)$ and $(\lambda_2,\bar\lambda_2)$. Let $T^{2|1}_\Lambda:=(S\times \R^{2|1})/\Z^2$ denote the quotient by the (free) $\Z^2$-action on $S\times \R^{2|1}$ determined by $\Lambda$. 

\begin{defn} \label{defn:21fibconf} A map $T^{2|1}_\Lambda\to T^{2|1}_{\Lambda'}$ over a base change $S\to S'$ is a \emph{fiberwise rigid conformal map} if there exists a commutative square
\beq
\begin{tikzpicture}[baseline=(basepoint)];
\node (A) at (0,0) {$S\times \R^{2|1}$};
\node (B) at (4,0) {$S'\times \R^{2|1}$};
\node (C) at (0,-1.5) {$T^{2|1}_\Lambda$};
\node (D) at (4,-1.5) {$T^{2|1}_{\Lambda'}$}; 
\draw[->,dashed] (A) to  (B);
\draw[->>] (A) to  (C);
\draw[->] (C) to (D);
\draw[->>] (B) to (D);
\path (0,-.75) coordinate (basepoint);
\end{tikzpicture}\label{eq:supertorusmap}
\eeq
where the vertical arrows are the $\Z^2$-quotient maps, and the dashed arrow is a fiberwise rigid conformal map. We further require that the dashed arrow is $\Z^2$-equivariant relative to a map $S \times \Z^2 \to S' \times \Z^2$ determined by the base change $S \to S'$ and an $S$-family of homomorphisms $S\times \Z^2\to S\times \Z^2$ specified by an $S$-point of~$\SL_2(\Z)$.
\end{defn}
\begin{rmk} The full super conformal group of $\R^{2|1}$ has as $S$-points maps $S\times \R^{2|1}\to S\times \R^{2|1}$ preserving the odd distribution generated by the vector field~$D$. The rigid conformal group is the subgroup of the super conformal group that preserves the flat structure on $\R^2\subset \R^{2|1}$ (but not the Euclidean metric itself). When applied to maps~\eqref{eq:supertorusmap}, this mimics uniformization of Riemann surfaces: one reduces from all conformal surfaces to those with a constant curvature metric. Uniformization can also be applied to the super case; e.g., see~\cite[\S5.2]{Wittenmoduli} for genus $g=1$ and~\cite{CraneRabin} for genus $g>1$. With this in mind, hereafter we drop the modifier ``rigid" for the maps $T^{2|1}_\Lambda\to T^{2|1}_{\Lambda'}$ in Definition~\ref{defn:21fibconf}. \end{rmk}

\begin{defn}\label{defn:supertori} The \emph{moduli stack of super tori} is the stackification of the prestack whose objects over $S$ are given by families $T^{2|1}_\Lambda\to S$ for $\Lambda\in \Lat(S)$, and whose morphisms over a base change $S\to S'$ are fiberwise conformal maps $T^{2|1}_\Lambda\to T^{2|1}_{\Lambda'}$. We refer to an $S$-point of the moduli stack of super tori as an \emph{$S$-family of super tori}.
\end{defn}

\begin{rmk}\label{rmk:trivtori} Explicitly, an $S$-family of super tori is a family of supermanifolds $T \to S$ for which there exists an open cover $\{S_\alpha\}$ of $S$ with isomorphisms $T|_{S_\alpha}\simeq T_{\Lambda_i}$ for $\Lambda_i\in \Lat(S_\alpha)$ and gluing data $\varphi_{ij}\colon T|_{S_\alpha\cap S_\beta}\simeq T^{2|1}_{\Lambda_i}|_{S_\alpha\cap S_\beta}\to T^{2|1}_{\Lambda_j}|_{S_\beta\cap S_\alpha}\simeq T|_{S_\beta\cap S_\alpha}$ where the middle arrow is a fiberwise conformal map. These gluing data are required to satisfy a cocycle condition. 
\end{rmk}

\begin{lem} \label{lem:55}The moduli stack of super tori is presented by the super Lie groupoid,
\beq
(\Conf(\R^{2|1})\times \SL_2(\Z)\times \Lat)/\Z^2\rightrightarrows \Lat,\label{eq:Mtorigrpd}
\eeq
where the $\Z^2$-quotient is by the action 
\beq
&&(m,n)\cdot (w,\bar w,\eta,\mu,\bar\mu,\gamma,\Lambda)\mapsto (w+n\lambda_1+m\lambda_2,\bar w+n\bar\lambda_1+m\bar\lambda_2,\eta,\mu,\bar\mu,\gamma,\Lambda), \label{eq:Lambdaact}
\eeq
$$
(w,\bar w,\eta)\in \E^{2|1}(S),  \ (\mu,\bar\mu) \in \C^\times(S), \ \gamma \in \SL_2(\Z)(S),
$$$$ 
\Lambda=(\lambda_1,\bar\lambda_1,\lambda_2,\bar\lambda_2)\in \Lat(S), \ (m,n)\in \Z^2(S).
$$
The source map in~\eqref{eq:Mtorigrpd} is the projection. The target map uses the isomorphism $\Conf(\R^{2|1})\simeq \E^{2|1}\rtimes \C^\times$, projects along 
$(\Conf(\R^{2|1})\times \SL_2(\Z)\times \Lat)/\Z^2\to \C^\times\times \SL_2(\Z)\times \Lat$, and composes with the $\C^\times\times \SL_2(\Z)$-action on $\Lat$ given by 
\beq
(\mu,\bar\mu,\left[\begin{array}{cc} a& b \\ c & d\end{array}\right])\cdot (\lambda_1,\bar\lambda_1,\lambda_2,\bar\lambda_2)&=&(\mu^{2}(a\lambda_1+b\lambda_2),\bar\mu^{2}(a\bar\lambda_1+b\bar\lambda_2),\label{eq:lataction}\\ 
&& \mu^{2}(c\lambda_1+d\lambda_2),\bar\mu^{2}(c\bar\lambda_1+d\bar\lambda_2))\in \Lat(S)\nonumber
\eeq
$$\quad (\mu,\bar\mu) \in \C^\times(S), \ \left[\begin{array}{cc} a& b \\ c & d\end{array}\right]\in \SL_2(\Z)(S),\quad (\lambda_1,\bar\lambda_1,\lambda_2,\bar\lambda_2)\in \Lat(S)\subset (\C\times \C)(S).
$$ 

\end{lem}

\bp We prove that $S$-points of the super Lie groupoid~\eqref{eq:Mtorigrpd} are equivalent to the prestack whose stackification is the stack of super tori in Definition~\ref{defn:supertori}. First we observe a pair of maps $S\to \Conf(\R^{2|1})\times \SL_2(\Z)$ determine the same map $T^{2|1}_\Lambda\to T^{2|1}_{\Lambda'}$ in the diagram~\eqref{eq:supertorusmap} if and only if they determine the same section of the bundle of Lie groups $(S\times \Conf(\R^{2|1})\times \SL_2(\Z))/\Z^2\to S$, for the quotient by the fiberwise subgroup $S\times \Z^2\stackrel{\Lambda}{\hookrightarrow}S\times \E^2\hookrightarrow S\times \E^{2|1}\hookrightarrow S\times \E^{2|1}\rtimes \C^\times\times \SL_2(\Z)$. 
Since objects of the prestack are determined by $S$-points of~$\Lat$, the case $S=\Lat$ is the universal one and the lemma follows. 
\ep

The following shows that families of super tori~$T^{2|1}_\Lambda$ are all isomorphic to a standard family. This allows one to think of $\Lat$ as the moduli space of super conformal structures on the standard super torus $\R^{2|1}/\Z^2$. 

\begin{lem} \label{lem:torusstnd}For $\Lambda\in \Lat(S)$ determining an $S$-family of super tori $T^{2|1}_\Lambda$, there exists an isomorphism of supermanifolds,
\beq
&&S\times (\R^{2|1}/\Z^2)\stackrel{\sim}{\to} T^{2|1}_\Lambda\label{eq:isotostd}
\eeq
for the $\Z^2$-action on $\R^{2|1}$ from the standard inclusion $\Z^2\subset \E^2\subset \E^{2|1}$. 
\end{lem}
\bp
For $\Lambda=(\lambda_1,\bar\lambda_1,\lambda_2,\bar\lambda_2)\in \Lat(S)$, consider the map 
$$
\R^{2|1}(S)\to \R^{2|1}(S),\qquad (x,y,\theta)\mapsto (\lambda_1x+\lambda_2y,\overline{\lambda}_1x+\overline{\lambda}_2y,\theta),
$$
where the source $S$-point is specified as in~\eqref{eq:r211} while the target is specified as in~\eqref{eq:r212}. It is easy to check that this map on $S$-points is a $\Z^2$-equivariant isomorphism for the standard $\Z^2$-action on the source and the $\Z^2$-action determined by~$\Lambda$ in the target. This gives the claimed isomorphism~\eqref{eq:isotostd}.
\ep

\subsection{Defining fields as a stack}\label{sec:maindef}

Let $M$ be a $G$-manifold. We recall that the stack~$[M\nsq G]$ classifies principal $G$-bundles with connection $(P,\nabla)$ equipped with a $G$-equivariant map~$P\to M$; see Example~\ref{defn:stackquoconn}.

\begin{defn} 
\label{defn:fields} For a $G$-manifold $M$ define the \emph{stack of fields}, denoted $\cL(M\nsq G)$, whose objects are maps~$\Phi\colon T^{2|1}\to [M\nsq G]$ where $T^{2|1}$ is an $S$-family of super tori, and whose morphisms are 2-commuting triangles
\beq
\begin{tikzpicture}[baseline=(basepoint)];
\node (A) at (0,0) {$T^{2|1}$};
\node (B) at (5,0) {$T'^{2|1}$};
\node (C) at (2.5,-1.5) {$[M\nsq G]$};
\node (D) at (2.5,-.6) {$\twocommute$};
\draw[->] (A) to node [above=1pt] {$f$} (B);
\draw[->] (A) to node [left=4pt] {$\Phi$} (C);
\draw[->] (B) to node [right=4pt] {$\Phi'$} (C);
\path (0,-.75) coordinate (basepoint);
\end{tikzpicture}\label{eq:commutetri}
\eeq
where $f$ is a fiberwise conformal map covering a base change $S\to S'$. 
\end{defn}

\begin{rmk}\label{rmk:unravel} Unwinding the definitions, an object of $\F(M\nsq G)$ is a $G$-bundle with connection over a family of a super tori,~$(P,\nabla)\to T^{2|1}$, and a $G$-equivariant map $\phi\colon P\to M$. Morphisms are fiberwise conformal maps of super tori $T^{2|1}\to T'^{2|1}$ covered by isomorphisms of $G$-bundles with connection $(P,\nabla)\to (P',\nabla')$ that are compatible with the maps to~$M$.
\end{rmk}

\begin{rmk}\label{rmk:nat}
We observe that a morphism of stacks $[M\nsq G]\to [N\nsq H]$ induces a morphism of stacks of fields, $\F(M\nsq G)\to \F(N\nsq H)$. 
\end{rmk}

\begin{rmk}\label{rmk:E2action}
The stack $\F(M\nsq G)$ admits a weak $\E^{2|1}$-action, 
\beq
\E^{2|1}\times \F(M\nsq G)\to \F(M\nsq G)\label{eq:weakE21action}
\eeq
determined by the $\E^{2|1}(S)$-action on $T^{2|1}_\Lambda$ by super translation as in Definition~\ref{defn:21fibconf}. 
In particular, the vector field~$Q$ on~$\E^{2|1}$ gives an infinitesimal action on the stack, which will eventually lead to the $Q$-cohomology sheaf that encodes $G$-equivariant elliptic cocycles for~$M$. Formally applying localization (as in~\S\ref{sec:local}) to this $Q$-cohomology sheaf leads us to consider the substack of zeros of the vector field~$Q^2=\partial_{\bar z}$. This is the substack whose objects are invariant under the action on~$\E^2<\E^{2|1}$. We analyze this substack (and give a precise definition) in~\S\ref{sec:inertia}. The groupoid presentation of~$\F(M\nsq G)$ we construct in the next subsection will allow us to relate this ``localized" substack with the mathematically better-known concept of an \emph{inertia stack}. 
\end{rmk}


\subsection{A groupoid presentation of fields}

Mapping stacks are reviewed in Example~\ref{eg:defnMap}. We start with the following observation relating the stack of fields with the mapping stack $\Map(\R^{2|1}/\Z^2,[M\nsq G])$. 
Using Lemma~\ref{lem:torusstnd} to view $\Lat$ as the space of super conformal structures on the standard super torus $\R^{2|1}/\Z^2$, define a map of stacks 
\beq
\Lat\times \Map(\R^{2|1}/\Z^2,[M\nsq G]) \to \F(M\nsq G)\label{eq:epistacksLat}
\eeq
that sends $\Lambda\in \Lat(S)$ and $\Phi\colon S\times \R^{2|1}/\Z^2\to [M\nsq G]$ to the composition 
\beq
T^{2|1}_\Lambda\stackrel{\sim}{\leftarrow} S\times (\R^{2|1}/\Z^2)\to [M\nsq G] \label{eq:epistacksLat2}
\eeq
where the isomorphism with $T^{2|1}_\Lambda$ is from Lemma~\ref{lem:torusstnd}. Define the map~\eqref{eq:epistacksLat} on morphisms by sending an isomorphism of the source to the isomorphism in the target determined by
\beq
&&
\begin{tikzpicture}[baseline=(basepoint)];
\node (A) at (1.5,-.75) {$T^{2|1}_\Lambda$};
\node (B) at (4,-.75) {$S\times \R^{2|1}/ \Z^2$};
\node (D) at (9,-.75) {$[(M\times \Omega^1(-;\fg))\sq G)].$}; 
\node (E) at (6.2,-.75) {$g\ \Downarrow$};
\draw[->] (B) to node [above] {$\sim$} (A);
\draw[->,bend left=20] (B) to (D);
\draw[->,bend right=20] (B) to (D);
\path (0,-.75) coordinate (basepoint);
\end{tikzpicture}\nonumber
\eeq
regarded as a 2-commuting triangle~\eqref{eq:commutetri} with $f=\id$.

\begin{lem}\label{lem:epi}
The map~\eqref{eq:epistacksLat} is an epimorphism of stacks with image consisting of objects of the form $T^{2|1}\simeq T^{2|1}_\Lambda$ and morphisms as in~\eqref{eq:commutetri} with the constraint that $f=\id$. 
\end{lem}

\bp
By construction, the image is as stated. To prove the epimorphism statement, note that for any $S$-family of super tori $T^{2|1}\to S$ there exists a cover $\{S_\alpha\}$ of $S$ such that $T^{2|1}|_{S_\alpha}\simeq T^{2|1}_{\Lambda_\alpha}$ for $\Lambda_\alpha\in \Lat(S_\alpha)$; see Remark~\ref{rmk:trivtori}. Hence for a map $\Phi\colon T^{2|1}\to [M\nsq G]$, there exists a cover $\{S_\alpha\}$ so that the restriction of $\Phi_\alpha$ to $S_\alpha$ is isomorphic to a map $\Phi_\alpha\colon T^{2|1}_{\Lambda_\alpha}\to [M\nsq G]$. Such a map $\Phi_\alpha$ can be uniquely expressed as the composition~\eqref{eq:epistacksLat2} using Lemma~\ref{lem:torusstnd}. The statement of the lemma then follows directly from Definition~\ref{defn:fields} and Definition~\ref{defn:epi} of an epimorphism of stacks. 
\ep

We recall from Example~\ref{defn:stackquoconn} that $M\nsq G:=(M\times \Omega^1(-;\fg))\sq G$ is the action groupoid for the diagonal action of $G$ on $M$ and on $\Omega^1(-;\fg)$ by gauge transformations~\eqref{eq:gaugedef}. Below and throughout, we use the terminology \emph{generalized supermanifold} for a sheaf on the site of supermanifolds. 

\begin{defn}
For super Lie groupoids $\mathcal{G}$ and $\mathcal{H}$, define the generalized supermanifold $\Fun(\mathcal{G},\mathcal{H})$ whose $S$-points are the set of functors $S\times\mathcal{G}\to \mathcal{H}$ and to a base change $S'\to S$ pulls back such functors.
\end{defn}
We will be particularly in interested in the sheaf $\Fun(\R^{2|1}\sq \Z^2,M\nsq G)$ whose $S$-points are functors
\beq
S\times (\R^{2|1}\sq \Z^2)\to M\nsq G.\label{eq:Liefunctor}
\eeq
We emphasize that $\Fun(\R^{2|1}\sq \Z^2,M\nsq G)$ is a generalized supermanifold, i.e., a sheaf of \emph{sets}. It can viewed as a subsheaf of the mapping sheaves
\beq
&&\label{rmk:itsasubsheaf}\resizebox{.9\textwidth}{!}{
$\Fun(\R^{2|1}\sq \Z^2,M\nsq G)\hookrightarrow \Map(\R^{2|1},M\times \Omega^1(-;\fg))\times \Map(\R^{2|1}\times \Z^2,M\times \Omega^1(-;\fg)\times G)$}
\eeq
where the inclusion sends a functor~\eqref{eq:Liefunctor} to its values on objects and morphisms, i.e., maps of sheaves $S\times \R^{2|1}\to M\times \Omega^1(-;\fg)$ and $S\times \R^{2|1}\times \Z^2\to M\times \Omega^1(-;\fg)\times G$. 

%
%


Regarding the generalized supermanifold $\Fun(\R^{2|1}\sq \Z^2,M\nsq G)$ as a stack, there is a map of stacks
\beq
u\colon \Fun(\R^{2|1}\sq \Z^2,M\nsq G)\to \Map(\R^{2|1}/\Z^2,[M\nsq G]) \label{eq:epistacksLat3}
\eeq
that on $S$-points sends a functor between generalized super Lie groupoids to a map on the underlying stacks
$$
S\times \R^{2|1}/\Z^2\simeq [S\times \R^{2|1}\sq \Z^2]\to [(M\times \Omega^1(-;\fg))\sq G)]\simeq [M\nsq G]. 
$$
The image of~\eqref{eq:epistacksLat3} consists of those $G$-bundles $P\to S\times \R^{2|1}/\Z^2$ with connection and $G$-equivariant map $P\to M$ with the property that $P$ is trivializable when pulled back along~$S\times \R^{2|1}\to S\times \R^{2|1}/\Z^2$. From this one sees that~\eqref{eq:epistacksLat3} is an epimorphism of stacks: by the contractibility of $\R^{2|1}$, for any principal bundle $P \to S\times \R^{2|1}/\Z^2$, there exists a cover $\{S_\alpha\}$ of $S$ such $P|_{S_\alpha}$ trivializes when pulled back along the quotient map $S_\alpha\times \R^{2|1}\to S_\alpha\times \R^{2|1}/\Z^2$.

See Definition~\ref{defn:genatlas} for the definition of a generalized atlas. 

\begin{lem}\label{lem:atlas1}
The map
\beq
&&\Lat\times \Fun(\R^{2|1}\sq \Z^2,M\nsq G)\stackrel{\id_\Lat\times u}{\longrightarrow} \Lat\times \Map(\R^{2|1}/\Z^2,[M\nsq G])\to  \F(M\nsq G)\label{eq:atlas1}
\eeq
determined by~\eqref{eq:epistacksLat} and~\eqref{eq:epistacksLat3} is a generalized atlas for the stack. 
\end{lem}

\bp 
By Lemma~\ref{lem:epi} and the epimorphism~\eqref{eq:epistacksLat3}, the map~\eqref{eq:atlas1} is a composition of epimorphisms of stacks and therefore an epimorphism. By Lemma~\ref{lem:genatlas} it remains to check that the 2-pullback 
\beq
\begin{tikzpicture}[baseline=(basepoint)];
\node (A) at (0,0) {$\mathcal{P}$};
\node (B) at (7,0) {$\Lat\times \Fun(\R^{2|1}\sq \Z^2,M\nsq G)$};
\node (C) at (0,-1.5) {$\Lat\times \Fun(\R^{2|1}\sq \Z^2,M\nsq G)$};
\node (D) at (7,-1.5) {$\F(M\nsq G)$}; 
\draw[->] (A) to node [above] {$t$} (B);
\draw[->] (A) to node [left] {$s$} (C);
\draw[->] (C) to (D);
\draw[->] (B) to (D);
\path (0,-.75) coordinate (basepoint);
\end{tikzpicture}\label{diag:2pullback}
\eeq
is equivalent to a generalized supermanifold and that $s$ and $t$ are submersions of sheaves.

We start by giving an explicit description of the 2-pullback. From~\eqref{eq:epistacksLat2} the image of an $S$-point under~\eqref{eq:atlas1} is the object of $\F(M\nsq G)$ given by
$$
T^{2|1}_\Lambda\stackrel{\sim}{\leftarrow} [S\times \R^{2|1}\sq \Z^2]\to [(M\times \Omega^1(-;\fg))\sq G)]\simeq [M\nsq G].
$$
An $S$-point of the 2-pullback therefore consists of
\beq
&&\begin{tikzpicture}[baseline=(basepoint)];
\node (A) at (0,0) {$T^{2|1}_\Lambda$};
\node (B) at (4,0) {$[S\times \R^{2|1}\sq \Z^2]$};
\node (BB) at (4,-1.5) {$[S\times \R^{2|1}\sq \Z^2]$};
\node (C) at (0,-1.5) {$T^{2|1}_{\Lambda'}$};
\node (D) at (9,-.75) {$[(M\times \Omega^1(-;\fg))\sq G]$}; 
\node (E) at (6.2,-.75) {$g\ \twocommute$};
\draw[->] (B) to node [above] {$\sim$} (A);
\draw[->] (A) to node [left] {$f$} (C);
\draw[->] (BB) to node [below] {$\sim$} (C);
\draw[->,bend left=10] (B) to (D);
\draw[->,bend right=10] (BB) to (D);
\draw[->,dashed ] (B) to node [left] {$[\tilde{f}]$} (BB);
\path (0,-.75) coordinate (basepoint);
\end{tikzpicture}\label{diag:2pullback2}
\eeq
where $f$ is a fiberwise conformal map and $g$ is an isomorphism between maps to the stack~$[M\nsq G]$. 

We claim that there is an arrow $[\tilde{f}]$ determined by a functor $\tilde{f}\colon S\times \R^{2|1}\sq \Z^2\to S\times \R^{2|1}\sq \Z^2$ making the square in~\eqref{diag:2pullback2} commute, and furthermore that this functor is determined up to unique isomorphism. Indeed, the value of $\tilde{f}$ on objects is determined by the isomorphisms~\eqref{eq:isotostd} together with a lift of $f$ to the universal cover~\eqref{eq:supertorusmap}. This map between universal covers is determined by an $S$-point of $\Conf(\R^{2|1})\times \SL_2(\Z)$. The lift is unique up to translations in the lattice. Hence, $f$ is a section of $\SL_2(\Z)\times (\Conf(\R^{2|1})\times S)/\Z^2\to S$ for the $\Z^2$-action on $\Conf(\R^{2|1})$ determined by~\eqref{eq:Lambdaact}.  

The remaining data in~\eqref{diag:2pullback2} is~$g$, the isomorphism between maps to $[M\nsq G]$. 
Pulling back to the cover
$$
S\times \R^{2|1}\to S\times \R^{2|1}/\Z^2\simeq [S\times \R^{2|1}\sq \Z^2]
$$
we find that $g$ is determined by $\tilde{g}\in \Map(\R^{2|1},G)(S)$, an automorphism of the trivial $G$-bundle over $S\times \R^{2|1}$. We claim that the data $\tilde{g}$ determines an isomorphism~$g$ between maps to $[M\nsq G]$ if and only if it satisfies the conditions to determine a natural transformation between functors $\phi,\phi'\colon S\times \R^{2|1}\sq \Z^2\rightrightarrows M\nsq G$. Indeed, the data of a natural transformation is a map $\eta\colon S\times \R^{2|1}\to M\times \Omega^1(-;\fg)\times G$ given by $\eta=\phi_0\times \tilde{g}$ where $\phi_0$ is the value of $\phi$ on objects. Then $\phi_0\times \tilde{g}$ determines a natural isomorphism if the usual diagrams commute
\beq
\begin{tikzpicture}[baseline=(basepoint)]
\node (A) at (2.5,0) {$S\times \R^{2|1}$};
\node (B) at (2.5,-1.5) {$M\times \Omega^1(-;\fg)\times G$};
\node (C) at (-1,-1.5) {$M\times \Omega^1(-;\fg)$};
\node (D) at (6,-1.5) {$M\times \Omega^1(-;\fg)$};
\draw[->] (A) to node[right] {$\eta$} (B);
\draw[->] (A) to node[above] {$\phi_0$} (C);
\draw[->] (A) to node[above] {$\phi_0'$} (D);
\draw[->] (B) to node[below] {$s$} (D);
\draw[->] (B) to node[below] {$t$} (C);
\path (0,-.75) coordinate (basepoint);
\end{tikzpicture}
\label{eq:nattransf1}\\
\begin{tikzpicture}[baseline=(basepoint)]
\node (A) at (0,0) {$S\times \R^{2|1}\times \Z^2$};
\node (B) at (5,0) {$M\times \Omega^1(-;\fg)\times G\times G $};
\node (C) at (0,-1.5) {$M\times \Omega^1(-;\fg)\times G\times G $};
\node (D) at (5,-1.5) {$M\times \Omega^1(-;\fg)\times G,$};
\draw[->] (A) to node[above] {$\phi_1\times \eta\circ s$} (B);
\draw[->] (A) to node[left] {$\eta\circ t\times \phi_1'$} (C);
\draw[->] (B) to node[right] {$c$} (D);
\draw[->] (C) to node[below] {$c$} (D);
\path (0,-.75) coordinate (basepoint);
\end{tikzpicture}\label{eq:nattransf2}
\eeq 
where $c$ is composition in the groupoid, which in this case is determined by multiplication $G\times G\to G$ using the description of the fibered product
$$
M\times \Omega^1(-;\fg)\times G\times G\simeq (M\times \Omega^1(-;\fg)\times G)\times_{M\times \Omega^1(-;\fg)} M\times \Omega^1(-;\fg)\times G\to M\times \Omega^1(-;\fg)\times G.
$$
The diagram~\eqref{eq:nattransf1} shows that the natural transformation provides an isomorphism between (trivial) $G$-bundles with connection compatible with an equivariant map to~$M$ with source and target of this isomorphism determined by $\phi_0$ and $\phi_0'$, respectively. In the second diagram, $\phi_1$ and $\phi_1'$ are the values of $\phi$ and $\phi'$ on morphisms. This diagram shows that $\phi'_1$ is uniquely determined by $\phi_1$ and $\tilde{g}$, since $\tilde{g}$ acts by conjugation to define $\phi_1'$. Finally, we note that a pair of natural transformations determined by the data $\tilde{g}$ give the \emph{same} isomorphism~$g$ between $G$-bundles with connection if they differ by an automorphism of the trivial $G$-bundle coming from the $\Z^2$-action determined by~$\phi_1$. In particular, natural transformations of the form
\beq
\tilde{g}\colon S\times \R^{2|1}\hookrightarrow S\times \R^{2|1}\times \Z^2\stackrel{\phi_1}\to M\times G\label{eq:gaugesub}
\eeq
determined by the inclusion at some $(n,m)\in \Z^2$ result in the identity isomorphism~$g=\id$. 

All together, this shows that an $S$-point of the 2-pullback is a pair of $S$-points of $\Lat\times \Fun(\R^{2|1}\sq \Z^2,M\nsq G)$ together with a section of the bundle of groups
\beq
(\SL_2(\Z)\times\Conf(\R^{2|1})\ltimes \Map(\R^{2|1},G)\times S)/\Z^2\to S\label{eq:morbundle}
\eeq
where $\tilde{g}\in \Map(\R^{2|1},G)(S)$, $\tilde{f}\in \SL_2(\Z)\times\Conf(\R^{2|1})(S)$, and the quotient in~\eqref{eq:morbundle} is by the $S$-family of $\Z^2$ subgroups of $\SL_2(\Z)\times\Conf(\R^{2|1})\ltimes \Map(\R^{2|1},G)$ given by~\eqref{eq:Lambdaact} and~\eqref{eq:gaugesub}. In this 2-pullback, the $S$-points of $\Lat$ are related by the action of $\tilde{f}\in (\SL_2(\Z)\times \Conf(\R^{2|1}))(S)$, whereas the $S$-points of $\Fun(\R^{2|1}\sq \Z^2,M\nsq G)$ are related by the action of the gauge transformation $\tilde{g}\in \Map(\R^{2|1},G)(S)$. From this description, we find that 
\beq
&&\mathcal{P}\simeq \big(\SL_2(\Z)\times \Conf(\R^{2|1})\ltimes \Map(\R^{2|1},G)\times \Lat\times \Fun(\R^{2|1}\sq \Z^2,M\nsq G)\big)/\Z^2\label{eq:heresthepullback}
\eeq
where the $\Z^2$-quotient is as in~\eqref{eq:morbundle} for $S=\Lat\times \Fun(\R^{2|1}\sq \Z^2,M\nsq G)$. Hence, $\mathcal{P}$ is indeed a sheaf of sets. Furthermore, the map $s$ in~\eqref{diag:2pullback} is induced by the evident projection, and the map $t$ comes from the action of $\SL_2(\Z)\times \Conf(\R^{2|1})$ on $\Lat$ and the action of~$\Map(\R^{2|1},G)$ on $\Fun(\R^{2|1}\sq \Z^2,M\nsq G)$. This shows that~\eqref{eq:morbundle} is a submersion for each~$S$, so we conclude that both $s$ and $t$ are submersions of sheaves. This proves the lemma. 
\ep

Proposition~\ref{prop:appenpresentation} shows that $\F(M\nsq G)$ is presented by the groupoid $\{\mathcal{P}\rightrightarrows \Lat\times \Fun(\R^{2|1}\sq \Z^2,M\nsq G)\}$:

\begin{cor} \label{cor:atlas1}
The generalized atlas from Lemma~\ref{lem:atlas1} has associated groupoid presentation 
\beq
\F(M\nsq G)\simeq \left[\begin{array}{c} \SL_2(\Z)\times \big(\Conf(\R^{2|1})\ltimes \Map(\R^{2|1},G)\times \Lat\times \Fun(\R^{2|1}\sq \Z^2,M\nsq G)\big)/\Z^2\\ s\downarrow\downarrow t \\ \Lat\times \Fun(\R^{2|1}\sq \Z^2,M\nsq G)\end{array}\right].\nonumber 
\eeq
\end{cor}

\subsection{Inertia fields}\label{sec:inertia}

Remark~\ref{rmk:E2action} explains why we should be interested in the substack of $\F(M\nsq G)$ whose objects are invariant under the $\E^2$-action. We recall in the case of maps from a 2-torus $T^2$ to a stack $\X$, the~$\E^2$-fixed points correspond to the (double) inertia stack,
\beq
&&\Map([\pt\sq \Z^2],\X)\simeq \Map([\R^2\sq \Z^2],\X)^{\E^2}
\subset \Map([\R^2\sq \Z^2],\X)\simeq \Map(T^2,\X),\nonumber
\eeq
using that the (free) $\E^2$-action on $T^2\simeq[\R^2\sq \Z^2]$ taken in stacks has quotient $[\pt\sq \Z^2]$; see~\cite[Theorem~3.6.4]{LupercioUribe} for the case of loop groupoids and compare~\cite{stringorbifolds} for a more physics-oriented point of view. The following definition takes $\F_0(M\nsq G):=\F(M\nsq G)^{\E^2}\subset \F(M\nsq G)$ as the $\E^2$-fixed substack in the same sense.

\begin{defn}\label{defn:inertia} 
Define the substack $\F_0(M\nsq G)\subset \F(M\nsq G)$ of \emph{inertia fields} as the stack underlying the full subgroupoid of Corollary~\ref{cor:atlas1} with objects the subsheaf 
\beq
\Lat\times \Fun(\R^{0|1}\sq \Z^2,M\nsq G)\hookrightarrow \Lat\times \Fun(\R^{2|1}\sq \Z^2,M\nsq G)\label{eq:subatlas}
\eeq
where the inclusion~\eqref{eq:subatlas} is induced by the quotient map $\R^{2|1}\to \R^{2|1}/\E^2\simeq \R^{0|1}$. 
\end{defn}

Taking the full subgroupoid associated to $\Lat\times \Fun(\R^{0|1}\sq \Z^2,M\nsq G)$ in the groupoid presentation from Corollary~\ref{cor:atlas1}, we obtain the following. 

\begin{lem} \label{lem:present}
The stack of inertia fields has the groupoid presentation 
\beq
\F_0(M\nsq G)\simeq \left[\begin{array}{c} \big(\SL_2(\Z)\times \Conf(\R^{2|1})\ltimes \Map(\R^{0|1},G)\times \Lat\times \Fun(\R^{0|1}\sq \Z^2,M\nsq G)\big)/\Z^2\\ s\downarrow\downarrow t \\ \Lat\times \Fun(\R^{0|1}\sq \Z^2,M\nsq G)\end{array}\right].\nonumber 
\eeq
\end{lem}
\bp
We need only restrict the morphisms in the groupoid presentation from Corollary~\ref{cor:atlas1} along the inclusion~\eqref{eq:subatlas}. The morphisms over an $S$-point of the target of this inclusion are determined by sections of the bundle of groups~\eqref{eq:morbundle} determining a diagram~\eqref{diag:2pullback2}. Morphisms over an $S$-point of $\Lat\times \Fun(\R^{0|1}\sq \Z^2,M\nsq G)$ therefore correspond to the subspace of sections of~\eqref{eq:morbundle} determining a diagram 
\beq
&&\begin{tikzpicture}[baseline=(basepoint)];
\node (A) at (1.5,0) {$T^{2|1}_\Lambda$};
\node (B) at (4,0) {$[S\times \R^{2|1}\sq \Z^2]$};
\node (BB) at (4,-1.5) {$[S\times \R^{2|1}\sq \Z^2]$};
\node (F) at (7,0) {$[S\times \R^{0|1}\sq \Z^2]$};
\node (FF) at (7,-1.5) {$[S\times \R^{0|1}\sq \Z^2]$};
\node (C) at (1.5,-1.5) {$T^{2|1}_{\Lambda'}$};
\node (D) at (11.5,-.75) {$[(M\times \Omega^1(-;\fg))\sq G]$}; 
\node (E) at (8.8,-.75) {$g\ \twocommute$};
\draw[->] (B) to node [above] {$\sim$} (A);
\draw[->] (A) to node [left] {$f$} (C);
\draw[->] (BB) to node [below] {$\sim$} (C);
\draw[->,bend left=10] (F) to (D);
\draw[->,bend right=10] (FF) to (D);
\draw[->,dashed]  (B) to node [left] {$[\tilde{f}]$} (BB);
\draw[->,dashed] (F) to node [left] {$[\tilde{f}_0]$} (FF);
\draw[->] (B) to (F);
\draw[->] (BB) to (FF);
\path (0,-.75) coordinate (basepoint);
\end{tikzpicture}\label{diag:2pullback3}
\eeq
where the map $[S\times \R^{2|1}\sq \Z^2]\to [S\times \R^{0|1}\sq \Z^2]$ is determined by the quotient map $\R^{2|1}\to \R^{2|1}/\E^2\simeq \R^{0|1}$. We observe that there is a uniquely determined functor $\tilde{f}_0\colon S\times \R^{0|1}\sq \Z^2\to S\times \R^{0|1}\sq \Z^2$ associated with a strictly commuting diagram of super Lie groupoids underlying the middle square in~\eqref{diag:2pullback3}. Indeed, $\tilde{f}_0$ is the $\E^2$-quotient of $\tilde{f}$, which is well-defined because $\Z^2\subset \E^2$. 
From this we see that the only restriction on morphisms in the subsheaf~\eqref{eq:subatlas} is on~$g$: to remain in the subsheaf $\Lat\times \Fun(\R^{0|1}\sq \Z^2,M\nsq G)$, $g$ must come from a natural transformation between functors $S\times \R^{0|1}\sq \Z^2\to M\nsq G$. This is the data of a map $\tilde{g}\colon S\times \R^{2|1}\to S\times \R^{0|1}\to G$. Hence $\tilde{g}$ must be determined by an $S$-point of the subgroup~$\Map(\R^{0|1},G)<\Map(\R^{2|1},G)$. This proves the lemma. 
\ep

\subsection{Functions on inertia fields as differential forms on inertia stacks}\label{sec:slogan}
We recall that for a supermanifold $N$, the algebra $C^\infty(\Map(\R^{0|1},N))$ is (a completion of the) algebra of differential forms on $N$. Using that the $\Z^2$-action on $\R^{0|1}$ is trivial in~\eqref{eq:subatlas}, we have the isomorphisms of sheaves
\beq
 \Fun(\R^{0|1}\sq \Z^2,M\nsq G)&\simeq&  \Fun(\R^{0|1}\sq \Z^2,M\nsq G)\nonumber \\
 &\simeq&  \Fun(\R^{0|1}\times \pt \sq \Z^2,M\nsq G)\nonumber \\
 &\simeq &\Map(\R^{0|1},\Fun(\pt\sq \Z^2,M\nsq G)).  \label{eq:piTinert2}
\eeq
So from this we have (up to a completion) 
$$
C^\infty( \Fun(\R^{0|1}\sq \Z^2,M\nsq G))\simeq \Omega^\bullet (\Fun(\pt\sq \Z^2,M\nsq G)). 
$$
Therefore we can view functions on the atlas for $\F_0(M\nsq G)$ as differential forms on the double inertia stack of $[M\nsq G]$ tensored with~ functions on $\Lat$. The structure of the stack $\F_0(M\nsq G)$ determines an action on these functions by the groupoid in Lemma~\ref{lem:present}. This leads to the informal Lemma~\ref{slogan1} stated in the introduction.

\section{Twisted sectors and equivariant de~Rham complexes}\label{sec:twisted}

This section is a bit technical, so we provide a brief overview. Broadly stated, our goal is to give a more explicit description of the generalized atlases from the previous section,
$$
\Lat\times \Fun(\R^{2|1}\sq \Z^2,M\nsq G)\to \F(M\nsq G),\qquad \Lat\times \Fun(\R^{0|1}\sq \Z^2,M\nsq G)\to \F_0(M\nsq G)
$$
that streamlines later constructions while also being amenable to computations. The description comes from realizing the generalized atlases as subsheaves of a somewhat simpler sheaf, refining ~\eqref{rmk:itsasubsheaf}. 

This subsheaf description leads to three relationships between functions on inertia fields and equivariant de~Rham complexes. First, in~\S\ref{sec:rephrase} we rephrase the groupoid presentation of inertia fields from Lemma~\ref{lem:present} in terms of a group action on objects coming from the super geometric interpretation of equivariant de~Rham cohomology in the Weil model (see Lemma~\ref{lem:deRhamactioninertia}). Second, in~\S\ref{sec:WZ} we construct a different groupoid presentation of $\F_0(M\nsq G)$ in terms of the \emph{Wess--Zumino gauge} that connects with equivariant de~Rham cohomology in the Cartan model. Third, in~\S\ref{subsec:twist} we define \emph{twisted sectors} $\F_0(M\nsq G)_h$ which sit in a span of stacks
\beq
\SM(\R^{0|1},[M^h\nsq G_0^h])\sq(\E^{0|1}\rtimes \C^\times) \twoheadleftarrow \F_0(M\nsq G)_h\hookrightarrow \F_0(M\nsq G).\label{eq:twistedsector}
\eeq 
The faithful functor on the right realizes a twisted sector as a substack of inertia fields depending on a choice of homomorphism~$h\colon \Z^2\to G$, or equivalently, a pair of commuting elements in~$G$. The arrow on the left has as its target the stack whose sheaf of functions determines the previously-known super geometric interpretation of the Weil model for $G_0^h$-equivariant de~Rham cohomology of $M^h$, where $G_0^h:=G^h\bigcap G_0$. Hence, the spans~\eqref{eq:twistedsector} for each~$h$ give a bridge between functions on~$\F_0(M\nsq G)$ and the data that enters into the de~Rham model for complex analytic equivariant elliptic cohomology for each pair of commuting elements in~$G$ (see Definition~\ref{defn:ellcocycle}). 
In order for the restriction to $\F_0(M\nsq G)_h$ to support a square zero odd operator, we must pass to gauge-invariant functions on $\F_0(M\nsq G)$ as defined in~\S\ref{sec:gaugeinvariant}. In terms of the Weil model, this is the passage to horizontal basic forms; in the Cartan model, this is the passage to $G$-invariant forms. In~\S\ref{sec:revisit} we show that the restriction of gauge-invariant functions along~\eqref{eq:twistedsector} does indeed determine a map to a chain complex that computes the $G_0^h$-equivariant de~Rham cohomology of $M^h$.


\subsection{Atlases as subsheaves}\label{sec:subsheaf}

We start with some notation: for super Lie groups $H$ and~$K$, let $\Hom(H,K)$ denote the generalized supermanifold that assigns the set of $S$-families of homomorphisms $S\times H\to S\times K$ over a supermanifold~$S$, and pulls these homomorphisms back along base changes. We recall from Definition~\ref{defn:subsheaf} that a \emph{subsheaf} of a sheaf $\Y$ is a sheaf $\X$ where $\X(S)\subset \Y(S)$ is a subset for each $S$. 

\begin{lem}\label{rmk:subsheaf}
The generalized supermanifold $\Fun(\R^{2|1}\sq \Z^2,M\nsq G)$ is a subsheaf, 
\beq
&&\Fun(\R^{2|1}\sq \Z^2,M\nsq G)\hookrightarrow\Map(\R^{2|1},M)\times \Omega^1(-\times \R^{2|1};\fg)\times\Hom(\Z^2,\Map(\R^{2|1},G)). \label{eq:subsheaf2} 
\eeq
\end{lem} 
\bp
As observed in ~\eqref{rmk:itsasubsheaf}, a functor $S\times \R^{2|1}\sq \Z^2\to M\nsq G$ is determined by the data of a map between morphisms and a map between objects of the respective groupoids. This gives an inclusion as a subsheaf
$$
\Fun(\R^{2|1}\sq \Z^2,M\nsq G)\hookrightarrow \Map(\R^{2|1},M\times \Omega^1(-;\fg))\times \Map(\R^{2|1}\times \Z^2,M\times \Omega^1(-;\fg)\times G). 
$$
However, this description has redundancies coming from properties these data satisfy to define a functor. One property is compatibility with source and target maps. Because $M\nsq G$ is an action groupoid, this implies that the map on morphisms is determined by the map on objects together with a map $S\times \R^{2|1}\times \Z^2\to G$. So in fact we have 
\beq
\phantom{BBB}\Fun(\R^{2|1}\sq \Z^2,M\nsq G)&\hookrightarrow& \Map(\R^{2|1},M\times \Omega^1(-;\fg))\times \Map(\R^{2|1}\times \Z^2,G)\nonumber \\
&\simeq &\Map(\R^{2|1},M)\times \Omega^1(-\times \R^{2|1};\fg)\times \Map(\R^{2|1}\times \Z^2,G).\label{eq:subsheaf} 
\eeq
Compatibility with composition refines this subsheaf further. Observe the inclusion
$$
\Hom(\Z^2,\Map(\R^{2|1},G))\hookrightarrow \Map(\Z^2,\Map(\R^{2|1},G))\simeq \Map(\R^{2|1}\times \Z^2,G)
$$ 
using the group structure on $\Map(\R^{2|1},G)$ coming from pointwise multiplication of maps to~$G$. For the value of a functor on morphisms to be compatible with composition, we obtain the subsheaf description~\eqref{eq:subsheaf}. This proves the lemma.
\ep

Replacing the source groupoid with $\R^{0|1}\sq \Z^2$ leads to the following. 

\begin{lem}\label{cor:subsheaf}
The generalized supermanifold $\Fun(\R^{0|1}\sq \Z^2,M\nsq G)$ is a subsheaf, 
\beq
&&\Fun(\R^{0|1}\sq \Z^2,M\nsq G)\hookrightarrow \Map(\R^{0|1},M)\times \Omega^1(-\times \R^{0|1};\fg)\times \Hom(\Z^2,\Map(\R^{0|1}, G)).\label{eq:objsubsh1}
\eeq
\end{lem}
\bp
Consider the square,
\beq
&&\begin{tikzpicture}[baseline=(basepoint)];
\node (A) at (0,0) {$\Fun(\R^{0|1}\sq \Z^2,M\nsq G)$};
\node (B) at (7,0) {$\Map(\R^{0|1},M)\times \Omega^1(-\times \R^{0|1};\fg)\times \Hom(\Z^2,\Map(\R^{0|1}, G))$};
\node (C) at (0,-1.5) {$\Fun(\R^{2|1}\sq \Z^2,M\nsq G)$};
\node (D) at (7,-1.5) {$\Map(\R^{2|1},M)\times \Omega^1(-\times \R^{2|1};\fg)\times\Hom(\Z^2,\Map(\R^{2|1},G))$}; 
\draw[->,dashed] (A) to (B);
\draw[->,right hook-latex] (A) to (C);
\draw[->,right hook-latex] (C) to (D);
\draw[->,right hook-latex] (B) to (D);
\path (0,-.75) coordinate (basepoint);
\end{tikzpicture}\label{diag:itsasquare}
\eeq
where the left vertical arrow is~\eqref{eq:subatlas}, the lower horizontal inclusion is~\eqref{eq:subsheaf2}, and the right vertical arrow is determined by maps
$$
\Omega^1(-\times \R^{0|1};\fg)\hookrightarrow \Omega^1(-\times \R^{2|1};\fg),\qquad \Map(\R^{0|1},M)\hookrightarrow \Map(\R^{2|1},M),$$$$\Hom(\Z^2,\Map(\R^{0|1},G))\hookrightarrow \Hom(\Z^2,\Map(\R^{2|1},G))
$$
induced by the quotient $\R^{2|1}\to \R^{2|1}/\E^2\simeq \R^{0|1}$. We observe that~\eqref{diag:itsasquare} is a pullback square: an $S$-point of the upper right gives data for a putative functor as in the upper left, but need not satisfy the required properties; it satisfies these properties if and only if it is comes from an $S$-point in the lower left. Hence, the square is a pullback square. Since pullbacks preserve monomorphisms, the dashed arrow is an inclusion. 
\ep

\begin{rmk}\label{rmk:invt}
We observe the isomorphisms of sheaves 
\beq
\Map(\R^{2|1},M)\times \Omega^1(-\times \R^{2|1};\fg)&\simeq& \Fun(\R^{2|1},M\nsq G)\label{eq:descent1} \\
\Map(\R^{0|1},M)\times \Omega^1(-\times \R^{0|1};\fg)&\simeq& \Fun(\R^{0|1},M\nsq G).\label{eq:descent2}
\eeq
The image of the subsheaves~\eqref{eq:subsheaf2} and~\eqref{eq:objsubsh1} are precisely those $S$-points that descend in the sense of the additional data and property of the dashed arrow
\beq
\begin{tikzpicture}[baseline=(basepoint)];
\node (A) at (0,0) {$S\times \R^{2|1}$};
\node (B) at (0,-1.5) {$S\times \R^{2|1}\sq \Z^2$};
\node (C) at (3,-.75) {$M\nsq G$};
\draw[->] (A) to (B);
\draw[->] (A) to (C);
\draw[->,dashed] (B) to (C);
\path (0,-.75) coordinate (basepoint);
\end{tikzpicture}\qquad 
\begin{tikzpicture}[baseline=(basepoint)];
\node (A) at (0,0) {$S\times \R^{0|1}$};
\node (B) at (0,-1.5) {$S\times \R^{0|1}\sq \Z^2$};
\node (C) at (3,-.75) {$M\nsq G.$};
\draw[->] (A) to (B);
\draw[->] (A) to (C);
\draw[->,dashed] (B) to (C);
\path (0,-.75) coordinate (basepoint);
\end{tikzpicture}\label{eq:descentcondition}
\eeq
The data of the dashed arrow is supplied by $S\to \Hom(\Z^2,\Map(\R^{2|1},G))$, respectively, $S\to \Hom(\Z^2,\Map(\R^{0|1},G))$. This allows one to view an $S$-point of the target of the inclusion~\eqref{eq:subsheaf2} (respectively,~\eqref{eq:objsubsh1}) as a triple of data: (1) a trivial $G$-bundle with connection over $S\times \R^{2|1}$ (respectively, $S\times \R^{0|1}$); (2) a $G$-equivariant map from the trivial $G$-bundle to~$M$ from~\eqref{eq:descent1} (respectively,~\eqref{eq:descent2}), and (3) a $\Z^2$-action on the trivial $G$-bundle specified by an $S$-point of $\Hom(\Z^2,\Map(\R^{2|1},G))$ (respectively, $\Hom(\Z^2,\Map(\R^{0|1},G))$). Relative to this triple, the condition~\eqref{eq:descentcondition} is precisely that the $G$-connection and map to $M$ be invariant under the $\Z^2$-action. In other words, the images~\eqref{eq:subsheaf2} and~\eqref{eq:objsubsh1} are $S$-points satisfying a descent condition.
\end{rmk}

\subsection{Rephrasing the groupoid presentation of inertia fields}\label{sec:rephrase}
The following lemma gives the first point of contact between inertia fields and equivariant cohomology, by way of the super geometric interpretation of the equivariant de~Rham complex reviewed in~\S\ref{sec:appende}. 


\begin{lem}\label{lem:deRhamactioninertia}
The target map in the groupoid presentation of inertia fields from Lemma~\ref{lem:present} is the composition of the projection $p$ and an action map $\alpha$
\beq
&&\big(\SL_2(\Z)\times (\E^{2|1}\rtimes \C^\times)\ltimes \Map(\R^{0|1},G)\times \Lat\times \Fun(\R^{0|1}\sq \Z^2,M\nsq G)\big)/\Z^2\nonumber \\
\phantom{BB} && \stackrel{p}{\to} \big(\SL_2(\Z)\times (\E^{0|1}\rtimes \C^\times)\ltimes \Map(\R^{0|1},G)\big) \times \Lat\times \Fun(\R^{0|1}\sq \Z^2,M\nsq G)\nonumber\\
\phantom{BB} &&\stackrel{\alpha}{\to} \Lat\times \Fun(\R^{0|1}\sq \Z^2,M\nsq G).\label{eq:itsanaction}
\eeq
The projection $p$ is determined by $\E^{2|1}\to \E^{0|1}$, and the action $\alpha$ is the restriction of the $\SL_2(\Z)\times (\E^{0|1}\rtimes \C^\times)\ltimes \Map(\R^{0|1},G)$-action on
$$
\Lat\times \Map(\R^{0|1},M)\times\Omega^1(-\times \R^{0|1};\fg)\times \Hom(\Z^2,\Map(\R^{0|1},G))\supset \Lat\times \Fun(\R^{0|1}\sq \Z^2,M\nsq G)
$$
determined by:
\begin{enumerate}
\item the action of $\SL_2(\Z)\times \C^\times$ on $\Lat$ from~\eqref{eq:lataction};
\item the action of $(\E^{0|1}\rtimes \C^\times) \ltimes \Map(\R^{0|1},G)$ on $\Map(\R^{0|1},M)$ coming from the precomposition action of $\E^{0|1}\rtimes \C^\times$ on $\R^{0|1}$ and the postcomposition action by $G$ on $M$, as computed in~\eqref{eq:PiTaction1}, \eqref{eq:PiTaction2}, \eqref{eq:act5}, and~\eqref{eq:act6};
\item the action of $(\E^{0|1}\rtimes \C^\times)\ltimes \Map(\R^{0|1},G)$ on $\Omega^1(-\times \R^{0|1};\fg)$ by extended gauge transformations, i.e., gauge transformations defined in~\eqref{eq:gaugedef} and the pullback of 1-forms along the $\E^{0|1}\rtimes \C^\times$-action on~$\R^{0|1}$ computed in~\eqref{eq:act1}, \eqref{eq:act2}, \eqref{eq:act3} and~\eqref{eq:act4}; 
\item the action of $\SL_2(\Z)\times (\E^{0|1}\rtimes \C^\times)\ltimes \Map(\R^{0|1},G)$ on $\Hom(\Z^2,\Map(\R^{0|1},G))$ where $\SL_2(\Z)$ acts by precomposition on $\Z^2$ and $(\E^{0|1}\rtimes \C^\times)\ltimes \Map(\R^{0|1},G)$ acts on $\Map(\R^{0|1},G)$ as in (2) with $M=G$ using the conjugation action of $G$ on itself. 
\end{enumerate} 
\end{lem}

\begin{rmk}\label{rmk:actiononsubsheaf}
Following Lemma~\ref{lem:01action}, the formulas~\eqref{eq:PiTaction1}-\eqref{eq:PiTaction2} and~\eqref{eq:act1}-\eqref{eq:act6} only compute the actions described above on the subsheaf
$$
\resizebox{\textwidth}{!}{%
$\Lat\times \Map(\R^{0|1},M)\times\underline{\fg}\times\Pi \fg \times \Hom(\Z^2,\Map(\R^{0|1},G))\subset \Lat\times \Map(\R^{0|1},M)\times \Omega^1(-\times \R^{0|1};\fg) \times \Hom(\Z^2,\Map(\R^{0|1},G))$
}
$$
associated with the inclusion $\underline{\fg}\times\Pi \fg\subset \Omega^1(-\times \R^{0|1};\fg)$. We refer to~\S\ref{appen:vb} for more on the sheaf $\underline{\fg}:=C^\infty(-;\fg)$ whose $S$-points are sections of the trivial bundle with fiber the vector space~$\fg$. In view of Lemma~\ref{lem:stalk}, this restriction completely determines the action on functions, and therefore is the relevant computation when comparing with the super geometric interpretation of the Weil model; see Corollary~\ref{cor:Weil}. For this reason, below we shall ignore the distinction between formulas for the actions on $\underline{\fg}\times \Pi \fg$ and formulas for the action on~$\Omega^1(-\times \R^{0|1};\fg)$. 
\end{rmk}

\begin{proof}[Proof of Lemma~\ref{lem:deRhamactioninertia}]
We will trace through the action of isomorphisms on objects using the description from~\eqref{diag:2pullback3}. In that diagram, the conformal isometry $f$ between families of super tori is determined by an $S$-point $\tilde{f}\in (\Conf(\R^{2|1})\times \SL_2(\Z))(S)$ and the isomorphism of $G$-bundles~$g$ is determined by an $S$-point $\tilde{g}\in \Map(\R^{0|1},G)(S)$. 

Since $f$ is a map between super tori, the action (1) is clear from the groupoid presentation of super tori in Lemma~\ref{lem:55}. Similarly, the $\SL_2(\Z)$-action in (4) comes from the change of basis action on $\Z^2$ acting on maps $S\times \R^{0|1}\times \Z^2\to G$. 

In the notation of~\eqref{diag:2pullback3}, the $\Conf(\R^{2|1})$-action on~$\tilde{f}$ determines an action of $\tilde{f}_0$ that factors through the quotient
$$
\Conf(\R^{2|1})/\E^2\simeq \E^{0|1}\rtimes \C^\times. 
$$
This gives rise to an $\E^{0|1}\rtimes \C^\times$-action by precomposition on maps $S\times \R^{0|1}\to M\times \Omega^1(-;\fg)$ and $S\times \R^{0|1}\times \Z^2\to M\times \Omega^1(-;\fg)\times G$, taking the trivial action of $\E^{0|1}\rtimes \C^\times$ on $\Z^2$ in the latter case. This restricts to an $\E^{0|1}\rtimes \C^\times$-action on $\Fun(\R^{0|1}\sq \Z^2,M\nsq G)$ along the inclusion~\eqref{eq:objsubsh1}, and determines the $\E^{0|1}\rtimes \C^\times$-actions in (2), (3) and (4) in the statement of the lemma. From the computations in the appendix, the action on $S\times \R^{0|1}\to M$ is given by the formulas~\eqref{eq:PiTaction1}, \eqref{eq:PiTaction2}, while the action on $S\times \R^{0|1}\to \Omega^1(-;\fg)$ is~\eqref{eq:act1}, \eqref{eq:act2}. 

The action of the an $S$-point of the super Lie group $\Map(\R^{0|1},G)$ on a map $S\times \R^{0|1}\to M\times\Omega^1(-;\fg)$ is determined by~\eqref{eq:nattransf1}, and comes from actions on the maps $S\times \R^{0|1}\to M$ and $S\times \R^{0|1}\to \Omega^1(-;\fg)$. The latter is the post-composition action of $G\ltimes\Pi\fg\simeq \Map(\R^{0|1},G)$ on $\Omega^1(-\times \R^{0|1};\fg)$ by gauge transformations, leading to the claimed formulas~\eqref{eq:act3}, \eqref{eq:act4}. The action of $g\colon S\times \R^{0|1}\to G$ on a map $\phi\colon S\times \R^{0|1}\to M$ is by the composition
$$
S\times \R^{0|1}\stackrel{g\times \phi}{\longrightarrow} G\times M\stackrel{{\rm act}}{\to} M
$$
where ${\rm act}$ is the $G$-action on $M$. A formula for this composition is computed in \eqref{eq:act5}, and~\eqref{eq:act6}. The action of $\Map(\R^{0|1},G)$ on $\Hom(\Z^2,\Map(\R^{0|1},G))$ is determined by ~\eqref{eq:nattransf2}, given by postcomposition with the conjugation action. \ep

\subsection{The Wess--Zumino gauge} \label{sec:WZ} Using Lemma~\ref{cor:subsheaf}, we describe another groupoid presentation of inertia fields that will turn out to be related to the Cartan model for equivariant cohomology. Define the subsheaf (see~\S\ref{sec:appenCartan})
\beq
\Omega^1(-\times \R^{0|1};\fg)_\wz \subset \Omega^1(-\times \R^{0|1};\fg)\label{eq:wz1}
\eeq
whose $S$-points are $\fg$-valued 1-forms on $S\times \R^{0|1}$ with $\chi=0$ in the description  
\beq
A=d\theta\otimes \chi+\theta d\theta \otimes X+A_S\in \Omega^1(S\times \R^{0|1})\otimes \fg\simeq \Omega^1(S\times \R^{0|1};\fg)\label{eq:1formexpand}
\eeq
where $\theta\in C^\infty(\R^{0|1})$ is the standard coordinate, $\chi\in \Pi\fg(S)\simeq C^\infty(S;\fg)^\odd$, $X\in \underline{\fg}(S):= C^\infty(S;\fg)^\ev$ and $A_S\in \Omega^1(S;\fg)$. The subscript~$\wz$ in~\eqref{eq:wz1} stands for ``Wess--Zumino," and the $\chi=0$ condition is a choice of gauge fixing; see Lemma~\ref{lem:WZ01}. 

\begin{defn}\label{defn:WZsubsheaf} Define the subsheaf $\Fun(\R^{0|1}\sq \Z^2,M\nsq G)_{\wz}\subset \Fun(\R^{0|1}\sq \Z^2,M\nsq G)$ as the pullback
\beq
\begin{tikzpicture}[baseline=(basepoint)];
\node (A) at (0,0) {$\Fun(\R^{0|1}\sq \Z^2,M\nsq G)_{\wz}$};
\node (B) at (7.5,0) {$\Map(\R^{0|1},M)\times \Omega^1(-\times \R^{0|1};\fg)_{\wz} \times \Hom(\Z^2,\Map(\R^{0|1}, G))$};
\node (D) at (7.5,-1.25) {$\Map(\R^{0|1},M)\times \Omega^1(-\times \R^{0|1};\fg)\times \Hom(\Z^2,\Map(\R^{0|1}, G))$};
\node (C) at (0,-1.25) {$\Fun(\R^{0|1}\sq \Z^2,M\nsq G)$};
\draw[->] (A) to (B);
\draw[->] (A) to (C);
\draw[->,right hook-latex] (B) to (D);
\draw[->,right hook-latex] (C) to (D);
\path (0,-.75) coordinate (basepoint);
\end{tikzpicture}\nonumber
\eeq
where the lower horizontal arrow is the inclusion~\eqref{eq:objsubsh1} and the right vertical inclusion is determined by~\eqref{eq:wz1}. 
\end{defn}

Equivalently, $\Fun(\R^{0|1}\sq \Z^2,M\nsq G)_{\wz}$ is the subsheaf of functors where the datum $A\in \Omega^1(S\times \R^{0|1};\fg)$ is in the image of the subsheaf~\eqref{eq:wz1}. Replacing the groupoid $\R^{0|1}\sq \Z^2$ with $\R^{0|1}$, we also adopt the notation $\Fun(\R^{0|1},M\nsq G)_{\wz}\subset \Fun(\R^{0|1},M\nsq G)$ for the subsheaf defined as the pullback
\beq
\begin{tikzpicture}[baseline=(basepoint)];
\node (A) at (0,0) {$\Fun(\R^{0|1},M\nsq G)_{\wz}$};
\node (B) at (7.5,0) {$\Map(\R^{0|1},M)\times \Omega^1(-\times \R^{0|1};\fg)_{\wz} $};
\node (D) at (7.5,-1.25) {$\Map(\R^{0|1},M)\times \Omega^1(-\times \R^{0|1};\fg).$};
\node (C) at (0,-1.25) {$\Fun(\R^{0|1},M\nsq G)$};
\draw[->] (A) to node [above] {$\simeq$} (B);
\draw[->,right hook-latex] (A) to (C);
\draw[->,right hook-latex] (B) to (D);
\draw[->] (C) to node [above] {$\simeq$} (D);
\path (0,-.75) coordinate (basepoint);
\end{tikzpicture}\nonumber
\eeq

\begin{lem}\label{lem:WZsubsheaf} There is a canonical inclusion 
\beq
&&\Fun(\R^{0|1}\sq \Z^2,M\nsq G)_\wz\hookrightarrow \Omega^1(-\times \R^{0|1};\fg)_\wz\times \Hom(\Z^2,G)\times \Map(\R^{0|1},M)\label{eq:lem:WZsubsheaf}
\eeq
where in addition to the restriction on the connection~\eqref{eq:wz1}, we have the further restriction on~$\Hom(\Z^2,G)\subset \Hom(\Z^2,\Map(\R^{0|1},G))$ induced by the canonical inclusion $G<\Map(\R^{0|1},G)$ of the reduced subgroup. 
\end{lem} 
\bp
Suppose we are given a functor $S\times \R^{0|1}\sq \Z^2\to M\nsq G$, part of whose data is an element $A\in \Omega^1(S\times \R^{0|1};\fg)_\wz$ and~$h\colon S\times \R^{0|1}\times \Z^2\to G$. From~\eqref{eq:act3} and~\eqref{eq:act4}, compatibility of these data with target maps in the respective groupoids demands that~$h$ be of the form
$$
h\colon S\times \R^{0|1}\times \Z^2\stackrel{p}{\to} S\times \Z^2\to G,
$$
where $p$ is the projection. 
Hence, we obtain the subsheaf description~\eqref{eq:lem:WZsubsheaf}
refining~\eqref{eq:objsubsh1}, and the statement of the lemma follows. 
\ep

\begin{lem} \label{lem:WZ}
The composition 
\beq
&&\Lat\times \Fun(\R^{0|1}\sq \Z^2,M\nsq G)_{\wz}\subset \Lat\times \Fun(\R^{0|1}\sq \Z^2,M\nsq G)\to \F_0(M\nsq G)\label{eq:WZatlas}
\eeq
is a generalized atlas for the stack, whose associated groupoid presentation is the full subgroupoid of Lemma~\ref{lem:present} whose objects are the subsheaf 
\beq
\Lat\times \Fun(\R^{0|1}\sq \Z^2,M\nsq G)_{\wz}\subset \Lat\times \Fun(\R^{0|1}\sq \Z^2,M\nsq G).\label{eq:WZsubsheaf}
\eeq
\end{lem}

\bp The claim follows if the inclusion of the full subgroupoid with objects~\eqref{eq:WZsubsheaf} into the groupoid from Lemma~\ref{lem:present} is an equivalence. Hence it suffices to show that this inclusion of groupoids is essentially surjective. This is a consequence of the computation~\eqref{eq:act4}, where a gauge transformation $-\chi \in \Pi \fg(S)< (G\ltimes\Pi \fg)(S)\simeq \Map(\R^{0|1},G)(S)$ (in the notation of~\eqref{eq:1formexpand}) puts an arbitrary connection $A\in \Omega^1(S\times \R^{0|1};\fg)$ in the Wess--Zumino gauge. 
\ep

\subsection{Twisted sectors}\label{subsec:twist}

Let $h_1,h_2\in G$ be a pair of commuting elements defining a homomorphism $h=\langle h_1,h_2\rangle \colon \Z^2\to G$. Consider the subsheaves
\beq
\Omega^1(-\times \R^{0|1};\fg^h)\subset \Omega^1(-\times \R^{0|1};\fg),\quad \Map(\R^{0|1},M^h)\subset \Map(\R^{0|1},M)\label{eq:someinclusion}
\eeq
where $\fg^h\subset \fg$ is the subalgebra fixed by the adjoint action of $h_1,h_2$, and $M^h\subset M$ is the submanifold fixed by $h_1,h_2$. We will also use $G_0^h \subset G$ as the subgroup fixed by $h_1, h_2$, i.e., the common centralizer of both elements. Finally, for $G_0<G$ the connected component of the identity, define
\beq
G_0^h:=(G_0)^h=G^h\bigcap G_0=G^{h_1,h_2}\bigcap G_0=\subset G\label{defn:G0h}
\eeq
Observe that $h$ determines a map $\pt \to \Hom(\Z^2,\Map(\R^{0|1},G))$ from the composition
$$
\Z^2\stackrel{h}{\to} G\hookrightarrow \Map(\R^{0|1},G) 
$$ 
with the canonical inclusion of the reduced manifold of $\Map(\R^{0|1},G)$. 
Together with~\eqref{eq:someinclusion}, we obtain 
a subsheaf 
\beq
&&\resizebox{.95\textwidth}{!}{%
$\Map(\R^{0|1},M^h)\times \Omega^1(-\times \R^{0|1};\fg^h){\hookrightarrow} \Map(\R^{0|1},M)\times \Omega^1(-\times \R^{0|1};\fg)\times \Hom(\Z^2,\Map(\R^{0|1}, G)).$}\label{eq:above}
\eeq

\begin{defn} \label{defn:twisted}
Define the \emph{$h$-twisted sector} as the stack presented by the groupoid, 
\beq
\resizebox{\textwidth}{!}{$
\F_0(M\nsq G)_h:= \left[\begin{array}{c} \big(\Conf(\R^{2|1})\ltimes \Map(\R^{0|1},G_0^h)\times \Lat\times \Omega^1(-\times \R^{0|1};\fg^h)\times \Map(\R^{0|1},M^h)\big)/\Z^2\\ s\downarrow\downarrow t \\ \Lat\times \Omega^1(-\times \R^{0|1};\fg^h)\times \Map(\R^{0|1},M^h)\end{array}\right]\nonumber $}
\eeq
where the $\Z^2$-quotient is for the restriction of the action defining the morphisms in Lemma~\ref{lem:present}.
The source map above is the projection and target map is the restriction of the action from Lemma~\ref{lem:deRhamactioninertia} to the subsheaf~\eqref{eq:above}.
\end{defn}

\begin{rmk} That the action restricts (so that the above definition is well-defined) follows from the fact that the action by the subgroup of gauge transformations $\Map(\R^{0|1},G_0^h)<\Map(\R^{0|1},G)$ preserves the subsheaf $\Omega^1(-\times \R^{0|1};\fg^h)\subset \Omega^1(-\times\R^{0|1};\fg)$. 
\end{rmk}

We will construct a faithful functor $\F_0(M\nsq G)_h\hookrightarrow \F_0(M\nsq G)$ induced by a functor between groupoids whose value on objects is determined by the following. 

 \begin{lem} \label{cor:jh}
The map~\eqref{eq:above} factors through the inclusion~\eqref{eq:objsubsh1}:
\beq
\resizebox{\textwidth}{!}{%
\begin{tikzpicture}[baseline=(basepoint)];
\node (A) at (0,0) {$\Map(\R^{0|1},M^h)\times \Omega^1(-\times \R^{0|1};\fg^h)$};
\node (B) at (8.5,0) {$\Map(\R^{0|1},M)\times \Omega^1(-\times \R^{0|1};\fg)\times \Hom(\Z^2,\Map(\R^{0|1}, G)).$};
\node (C) at (0,-1.5) {$\Fun(\R^{0|1}\sq \Z^2,M\nsq G)$};
\draw[->,right hook-latex] (A) to (B);
\draw[->,dashed] (A) to node [left] {$j_h$} (C);
\draw[->,right hook-latex] (C) to (B);
\path (0,-.75) coordinate (basepoint);
\end{tikzpicture}\nonumber
}
\eeq
\end{lem}

\bp
An $S$-point of the sheaf $\Map(\R^{0|1},M^h)\times \Omega^1(-\times \R^{0|1};\fg^h)$ determines a map 
\beq
\phi_0 \colon S\times \R^{0|1}\to M^h\times \Omega^1(-\times\R^{0|1};\fg^h)\hookrightarrow M\times \Omega^1(-\times\R^{0|1};\fg). \label{eq:objectsmap}
\eeq
Next we define the map $\phi_1:=\phi_0\times h$
\beq
S\times \R^{0|1}\times \Z^2\stackrel{\phi_1}{\longrightarrow} M\times \Omega^1(-\times \R^{0|1};\fg)\times G,\label{eq:morphismsmap}
\eeq
for $h\colon \Z^2\to G$ the data defining the twisted sector. We claim that $\phi_0$ and $\phi_1$ are the data of a functor $\phi\colon S\times \R^{0|1}\sq \Z^2\to M\nsq G$, and hence an $S$-point of $\Fun(\R^{0|1}\sq \Z^2,M\nsq G)$. To see this, first observe that $\phi_1$ is automatically compatible with the source map. It is compatible with the target map precisely because $\phi_0$ has image in the invariant subsheaf $M^h\times \Omega^1(-\times\R^{0|1};\fg^h)\hookrightarrow M\times \Omega^1(-\times \R^{0|1};\fg)$ for the $\Z^2$-action determined by $h$. Compatibility with unit and composition follows from $h$ being a homomorphism. Hence, the the map factors as claimed.
\ep

The above proof in fact provides a more refined factorization of $j_h$ as follows. 

\begin{cor}\label{cor:morefactor}
The map $j_h$ factors as 
\beq
&&\Map(\R^{0|1},M^h)\times \Omega^1(-\times \R^{0|1};\fg^h)\to \Fun(\R^{0|1}\sq \Z^2,M^h\nsq G_0^h)\to \Fun(\R^{0|1}\sq \Z^2,M\nsq G)\label{eq:factorjh}
\eeq
\end{cor}

\begin{lem}\label{lem:faithfulfunctor}
There is a faithful functor from the groupoid presentation of $\F_0(M\nsq G)_h$ to the groupoid presentation of $\F_0(M\nsq G)$ in Lemma~\ref{lem:present}. The value of this functor on objects is the map $j_h$ from Lemma~\ref{cor:jh}, and its value on morphisms is determined by the injective homomorphism
\beq
\Conf(\R^{2|1})\ltimes \Map(\R^{0|1},G_0^h)\hookrightarrow \Conf(\R^{2|1})\ltimes \Map(\R^{0|1},G)\times \SL_2(\Z).\label{eq:itsahomomorphism}
\eeq
This determines the map of stacks 
\beq
\F_0(M\nsq G)_h\hookrightarrow \F_0(M\nsq G). \label{eq:twistedinclude}
\eeq
\end{lem}

\bp
By  Definition~\ref{defn:twisted} and Lemma~\ref{lem:deRhamactioninertia}, the inclusion $j_h$ is equivariant relative to the homomorphism~\eqref{eq:itsahomomorphism}. Next we observe that the $\Z^2$-quotient defining the morphisms in Definition~\ref{defn:twisted} is (by definition) the restriction of the $\Z^2$-quotient defining the morphisms in the groupoid presentation from Lemma~\ref{lem:present}. The lemma immediately follows. 
\ep

\begin{cor}  
We have the factorization
$$
\F_0(M\nsq G)_h\to \F_0(M^h\nsq G_0^h)\to \F_0(M\nsq G) 
$$
associated with the inclusions of groupoids presenting the stacks. 
\end{cor}
\bp
This follows immediately from naturality of the groupoid presentation of Lemma~\ref{lem:present} and Corollary~\ref{cor:morefactor}.
\ep

\begin{lem}\label{lem:twisttoWeil}
There is a full and essentially surjective functor 
\beq
\resizebox{\textwidth}{!}{$
\left\{\begin{array}{c} \big(\Conf(\R^{2|1})\ltimes \Map(\R^{0|1},G_0^h)\times \Lat\times \Fun(\R^{0|1},M^h\nsq G_0^h)\big)/\Z^2\\ s\downarrow\downarrow t \\ \Lat\times \Fun(\R^{0|1},M^h\nsq G_0^h)\end{array}\right\}\to \Lat\times \Fun(\R^{0|1},M^h\nsq G_0^h)\sq ((\E^{0|1}\rtimes \C^\times)\ltimes \Map(\R^{0|1},G_0^h)) \nonumber $}
\eeq
from the groupoid presentation of $\F_0(M\nsq G)_h$ to an action groupoid defined by the usual $\C^\times$-action on $\Lat$ and the $(\E^{0|1}\rtimes \C^\times)\ltimes \Map(\R^{0|1},G_0^h)$-action on $\Fun(\R^{0|1},M^h\nsq G_0^h)$ from Lemma~\ref{lem:01present}. 
\end{lem}
\bp
The data of a functor is given by the identity map on objects and the map on morphisms determined by the surjective homomorphism 
$$
\Conf(\R^{2|1})\ltimes \Map(\R^{0|1},G_0^h)\simeq (\E^{2|1}\rtimes \C^\times)\ltimes \Map(\R^{0|1},G_0^h)\twoheadrightarrow (\E^{0|1}\rtimes \C^\times)\ltimes \Map(\R^{0|1},G_0^h)
$$
coming from the projection $\E^{2|1}\to \E^{2|1}/\E^2\simeq \E^{0|1}$. By restricting the action in Lemma~\ref{lem:deRhamactioninertia} along~\eqref{eq:above}, we find that these data satisfy the requisite properties to determine a functor. 
\ep

From Lemmas~\ref{lem:twisttoWeil} and~\ref{lem:01present}, we obtain a map of stacks
\beq
\F_0(M\nsq G)_h\twoheadrightarrow (\Lat \times \SM(\R^{0|1},[M\nsq G]))\sq \E^{0|1}\rtimes \C^\times.\label{eq:Weilcompare}
\eeq
From Lemma~\ref{lem:superWeil}, the sheaf of functions on $\SM(\R^{0|1},[M\nsq G])\sq\E^{0|1}\rtimes \C^\times$ is a completion of the Weil complex of the $G_0^h$-manifold $M^h$. As a consequence of~\eqref{eq:Weilcompare}, the sheaf of functions on $\F_0(M\nsq G)_h$ encodes a similar type of Weil complex with coefficients in $C^\infty(\Lat)$, as we now describe. 

From Lemmas~\ref{lem:PiT} and~\ref{lem:stalk} along with properties of the projective tensor product, we observe that functions on the objects in the groupoid presentation of $\F_0(M\nsq G)_h$ are 
\beq
 C^\infty(\Lat\times \Fun(\R^{0|1},M^h\nsq G_0^h))&\simeq &C^\infty(\Lat\times \Omega^1(-\times \R^{0|1};\fg^h)\times \Map(\R^{0|1},M^h))\nonumber\\
 &\simeq& \Omega(M^h;C^\infty(\Lat)))\otimes \mathcal{O}(\fg_\C^h)\otimes \Lambda \fg_\C^\vee.\label{eq:itsaWeilcomplex}
\eeq
This algebra is (a completion of) the Weil complex of $M$ with coefficients in $C^\infty(\Lat)$. We compute the action of morphisms in Definition~\ref{defn:twisted} as follows.

\begin{lem}\label{lem:sectordeRham}
The groupoid action on functions~\eqref{eq:itsaWeilcomplex} coming from the groupoid presentation of $\F_0(M\nsq G)_h$ in Definition~\ref{defn:twisted} is determined by an action of the group 
$$
(\Conf(\R^{2|1})\ltimes \Map(\R^{0|1},G_0^h))/\E^2\simeq (\E^{0|1}\rtimes \C^\times)\ltimes \Map(\R^{0|1},G_0^h)
$$
characterized by:
\begin{enumerate}
\item The $\E^{0|1}\rtimes \C^\times$-action on $C^\infty(\Lat)$ through the $\C^\times$-action on $\Lat$ from~\eqref{eq:lataction};
\item The $\E^{0|1}\rtimes \C^\times$-action on $\Omega(M^h)\otimes \mathcal{O}(\fg_\C^h)\otimes \Lambda(\fg_\C^h)^\vee$ generated by the (negative) grading operator and (negative) Weil differential,
$$
\alpha\otimes f\otimes \zeta \mapsto \bar\mu^{-k-2l-j}(\alpha\otimes f\otimes \zeta-\eta d_W(\alpha\otimes f\otimes \zeta)) $$$$ \alpha\in \Omega^k(M^h),\ f\in \Sym^l((\fg_\C^h)^\vee)\subset \mathcal{O}(\fg_\C^h), \ \zeta\in \Lambda^j(\fg_\C^h)^\vee
$$
for $(\mu,\bar\mu)$ complex coordinates on $\C^\times$ and $\eta\in C^\infty(\E^{0|1})$ the standard coordinate. 
\item The $G_0^h\ltimes \Pi \fg^h \simeq \Map(\R^{0|1},G_0^h)$-action on $\Omega(M^h)\otimes \mathcal{O}(\fg_\C^h)\otimes \Lambda(\fg_\C^h)^\vee$ from the Weil complex: $G_0^h$ acts through its action on $M^h$ and the adjoint action on $\fg_\C^h$, while $\Pi \fg$ acts by contraction with vector fields generating the infinitesimal $G$-action. 
\end{enumerate}

\end{lem}
\bp
By Lemma~\ref{lem:faithfulfunctor}, restricting to the subgroupoid associated with a twisted sector leads to a groupoid action on functions determined by restricting the actions (1)-(4) in the statement of Lemma~\ref{lem:deRhamactioninertia}. The statement (1) in the present lemma that specifies the action as the pullback of functions for the action in Lemma~\ref{lem:deRhamactioninertia} is therefore clear. 

The statements (2) and (3) follow from the formulas~\eqref{eq:PiTaction1}, \eqref{eq:PiTaction2}, \eqref{eq:act1}, \eqref{eq:act2}, \eqref{eq:act3}, \eqref{eq:act4}, \eqref{eq:act5} \eqref{eq:act6} for the $(\E^{0|1}\rtimes \C^\times)\ltimes \Map(\R^{0|1},G)$-action on $C^\infty(\Omega^1(-\times\R^{0|1};\fg)\times \Map(\R^{0|1},M))$ as computed in Lemma~\ref{lem:superWeil}. 
\ep

We will need to understand how these equivariant de~Rham complexes for each~$h$ fit together to define functions on $\F_0(M\nsq G)$. To start, we can ask for conditions under which $h$ and $h'$ determine isomorphic substacks $\F_0(M\nsq G)_h,\F_0(M\nsq G)_{h'}\hookrightarrow \F_0(M\nsq G)$. To that end, define a left action of $G\times \SL_2(\Z)$ on $\Hom(\Z^2,G)$ by
\beq
(g,\gamma)\cdot h=g(\gamma\cdot h)g^{-1}=(gh_1^dh_2^{-b}g^{-1},gh_1^{-c}h_2^ag^{-1})\qquad \gamma=\left[\begin{array}{cc} a& b \\ c & d\end{array}\right]\label{eq:HomZGact}
\eeq
for $h=\langle h_1,h_2\rangle$, i.e., the precomposition action by $\SL_2(\Z)$ and postcomposition action by~$G$. We observe that the action~\eqref{eq:HomZGact} is covered by morphisms of groupoids 
$$
M^h\nsq G_0^h\to M^{g(\gamma\cdot h)g^{-1}}\nsq G^{g(\gamma\cdot h)g^{-1}}_0
$$ 
induced by the diffeomorphisms $M^h\to M^{g(\gamma\cdot h)g^{-1}}$ and $G_0^h\to G^{g(\gamma\cdot h)g^{-1}}_0$ from the $g$-action on~$M$ and the conjugation action of~$g$ on~$G$. Using  naturality of the sheaf 
$$
\Fun(\R^{0|1},M^h\nsq G_0^h)\simeq \Omega^1(-\times \R^{0|1};\fg^h)\times \Map(\R^{0|1},M^h)
$$
in the stack $M^h\nsq G_0^h$ and Definition~\ref{defn:twisted}, we obtain the following. 

\begin{cor}\label{cor:mapbetweentwist}
An element $(g,\gamma)\in G\times \SL_2(\Z)$ induces a map of sheaves
\beq
&&\Map(\R^{0|1},M^h)\times \Omega^1(-\times \R^{0|1};\fg^h) \to \Map(\R^{0|1},M^{g(\gamma\cdot h)g^{-1}})\times \Omega^1(-\times \R^{0|1};\fg^{g(\gamma\cdot h)g^{-1}}) 
\label{eq:twistmap}
\eeq
that determines a morphism of stacks
\beq
\F_0(M\nsq G)_h\to \F_0(M\nsq G)_{g(\gamma\cdot h)g^{-1}} \label{eq:twistmap2}
\eeq
compatible with the maps~\eqref{eq:twistedinclude}. 
\end{cor}

We will also require a description of twisted sectors in the Wess--Zumino gauge. From Definition~\ref{defn:WZsubsheaf}, we have the pullback square
\beq
\begin{tikzpicture}[baseline=(basepoint)];
\node (A) at (0,0) {$ \Lat\times \Fun(\R^{0|1},M^h\nsq G_0^h)_\wz$};
\node (B) at (6,0) {$\Lat\times \Fun(\R^{0|1}\sq \Z^2,M\nsq G)_\wz $};
\node (D) at (6,-1.25) {$\Lat\times \Fun(\R^{0|1}\sq \Z^2,M\nsq G).$};
\node (C) at (0,-1.25) {$\Lat\times \Fun(\R^{0|1},M^h\nsq G_0^h)$};
\draw[->,right hook-latex] (A) to (B);
\draw[->,right hook-latex] (A) to (C);
\draw[->,right hook-latex] (B) to (D);
\draw[->,right hook-latex] (C) to (D);
\path (0,-.75) coordinate (basepoint);
\end{tikzpicture}\label{eq:WZsubgroupoid}
\eeq
The following is proved using the same argument as in Lemma~\ref{lem:WZ}. 

\begin{cor}\label{cor:sectordeRham}
The restriction of the groupoid presentation of the twisted sector $\F_0(M\nsq G)_h$ from Definition~\ref{defn:twisted} to the full subgroupoid with objects~$\Lat\times \Fun(\R^{0|1},M^h\nsq G_0^h)_\wz$ gives an equivalent groupoid. In particular, 
$$
\Lat\times \Fun(\R^{0|1},M^h\nsq G_0^h)_\wz\hookrightarrow \Lat\times \Fun(\R^{0|1},M^h\nsq G_0^h)\to \F_0(M\nsq G)_h
$$
is a generalized atlas for the stack. 
\end{cor}

\subsection{Inertia fields modulo gauge}\label{sec:gaugeinvariant}
Let $G_0\triangleleft G$ be the connected component of the identity. Consider the normal subgroup 
\beq
\E^2\times\Map(\R^{0|1},G_0)\triangleleft \Conf(\R^{2|1})\ltimes \Map(\R^{0|1},G)\times \SL_2(\Z)\label{eq:quotientgrp}
\eeq
with the associated quotient
$$
\E^{0|1}\rtimes \C^\times\times \pi_0(G)\times \SL_2(\Z)\simeq (\Conf(\R^{2|1})\ltimes \Map(\R^{0|1},G)\times \SL_2(\Z))/(\E^2\times\Map(\R^{0|1},G_0)).
$$
\begin{defn} 
 Define the coarse quotient in sheaves
\beq
\Ftil_0(M\ncq G)&:=& \Lat\times \Fun(\R^{0|1}\sq \Z^2,M\nsq G)/(\E^2\times \Map(\R^{0|1},G_0))\nonumber\\
&\simeq& \Lat\times \Fun(\R^{0|1}\sq \Z^2,M\nsq G)/\Map(\R^{0|1},G_0)\label{eq:coarsequotient}
\eeq
where the isomorphism uses that the $\E^2$-action is trivial. 
\end{defn}

\begin{defn}\label{defn:gaugeinvt} Define the stack of \emph{gauge-invariant inertia fields} $\F_0(M\ncq G)$ as the stack underlying the action groupoid, 
\beq
\F_0(M\ncq G)\simeq [\Ftil_0(M\ncq G)\sq \E^{0|1}\rtimes \C^\times\times \SL_2(\Z)\times \pi_0(G)].\label{eq:F0Gaction2}\label{eq:F0Gaction}
\eeq
gotten from the taking the coarse quotient of the objects by the action of~$\E^2\times \Map(\R^{0|1},G_0)$ in the description of the groupoid presentation of $\F_0(M\nsq G)$ in Lemma~\ref{lem:deRhamactioninertia}, then using the residual action by the quotient by the normal subgroup~\eqref{eq:quotientgrp}.
\end{defn}


By Lemma~\ref{lem:present} and Definition~\ref{defn:gaugeinvt}, there is a canonical quotient map 
\beq
q\colon \F_0(M\nsq G)\to \F_0(M\ncq G)\label{eq:F0MG}
\eeq
coming from the map on groupoid presentations induced by taking coarse quotients. Functions that descend along $q$ are precisely those functions that are invariant under the action of gauge transformations $\Map(\R^{0|1},G_0)<\Map(\R^{0|1},G)$ in the connected component of the identity. Below we will often refer to these as \emph{gauge-invariant functions}, though we note that this terminology does not require invariance under $\pi_0(G)=\Map(\R^{0|1},G)/\Map(\R^{0|1},G_0)$.

\begin{rmk}
The notation $\F_0(M\ncq G)$ might be somewhat misleading: this stack is not gotten by applying $\F_0(-)$ to a coarse quotient~$M\ncq G$, but rather is a coarse quotient of $\F_0(M\nsq G)$. Alas, this was the least cumbersome notation we were able to find. 
\end{rmk}

\begin{rmk} Incorporating the group $\E^2$ in~\eqref{eq:quotientgrp} and the coarse quotient~\eqref{eq:coarsequotient} is a matter of convenience: it results in a stack $\F_0(M\ncq G)$ underlying an action groupoid~\eqref{eq:F0Gaction}. A version of~\eqref{eq:F0Gaction} with a quotient by only the gauge transformations $\Map(\R^{0|1},G_0)$ would result in a somewhat more complicated stack than~$\F_0(M\ncq G)$, but its sheaf of functions would contain the same information since $\E^2$ acts trivially on objects. 
\end{rmk}

It will be useful to have a version of the stack $\F_0(M\ncq G)$ phrased as a quotient of the inertia fields in the Wess--Zumino gauge. Consider the map
\beq
( \Lat\times \Fun(\R^{0|1}\sq \Z^2,M\nsq G)_{\rm wz})/G_0&\to&\Ftil_0(M\ncq G)\label{eq:WZinvt}
\eeq
induced by the inclusion $\Fun(\R^{0|1}\sq \Z^2,M\nsq G)_\wz \hookrightarrow \Fun(\R^{0|1}\sq \Z^2,M\nsq G)$ and the homomorphism $G_0\hookrightarrow \Map(\R^{0|1},G_0$). 

\begin{lem} \label{lem:WZcoarse}
The map~\eqref{eq:WZinvt} is an isomorphism of generalized supermanifolds. Hence, postcomposition with $\Ftil_0(M\ncq G)\to \F_0(M\ncq G)$ determines a generalized atlas for the stack,
\beq
( \Lat\times\Fun(\R^{0|1}\sq \Z^2,M\nsq G)_{\rm wz})/G_0 \to \F_0(M\ncq G).
\eeq
\end{lem}
\bp
We note that there is an isomorphism of generalized supermanifolds
\beq
\Omega^1(-\times \R^{0|1};\fg)_\wz=\Omega^1(-;\fg)\times\underline{\fg}\simeq  \Omega^1(-\times \R^{0|1};\fg)/\Pi\fg\label{eq:WZquotient}
\eeq
for the free $\Pi\fg$-action on $\Omega^1(-\times \R^{0|1};\fg)$ by gauge transformations; see~\eqref{eq:act4}. Furthermore, $\Pi\fg\triangleleft G_0\ltimes \Pi\fg \simeq \Map(\R^{0|1},G_0)$ is normal, so that $G_0$ acts on the quotient $\Omega^1(-\times \R^{0|1};\fg)_\wz$ and we have the isomorphism
$$
(\Omega^1(-\times \R^{0|1};\fg)_\wz)/G_0\simeq \Omega^1(-\times \R^{0|1};\fg)/\Map(\R^{0|1},G_0).
$$ 
This shows that~\eqref{eq:WZinvt} is in fact an isomorphism of generalized supermanifolds, and the result follows. 
\ep

\subsection{Sheaves of chain complexes on gauge-invariant inertia fields}\label{sec:Qcoh}

Recall the canonical line bundle over $[\pt\sq \C^\times]$ given by the homomorphism of generalized Lie groups~\eqref{eq:GL1rep}. Consider the pullback of the parity reversal of this line bundle along 
\beq
&&\F_0(\pt\ncq G)\simeq [\Ftil_0(\pt\ncq G)\sq (\E^{0|1}\rtimes \C^\times\times\SL_2(\Z)\times\pi_0(G))]\to [\pt\sq \C^\times]\label{eq:sHodge}
\eeq
where the second arrow is the map of stacks determined by the projection homomorphism $\E^{0|1}\rtimes \C^\times\times\SL_2(\Z)\times\pi_0(G)\to \C^\times$. Let $\omegas^{-1/2}$ denote the sheaf of sections of the associated equivariant line bundle on $\Ftil_0(\pt\ncq G)$. For $k\in \Z$ let $\omegas^{-k/2}:=(\omegas^{-1/2})^{\otimes k}$ denote sections of the $k$th tensor power. Explicitly, sections of $\omega^{k/2}$ satisfy 
\beq
&&(\mu,\bar\mu)^*f= \mu^{-k}f \qquad f\in C^\infty(\Ftil_0(\pt\ncq G)),\ \ (\mu,\bar\mu)^*f\in C^\infty(\C^\times\times \Ftil_0(\pt\ncq G)),\label{eq:sectionsofHodge}
\eeq
where $(\mu,\bar\mu)\in C^\infty(\C^\times)$ are the standard complex coordinates and $(\mu,\bar\mu)^*f$ denotes the pullback of a function along the $\C^\times$-action. 
Multiplication of functions on $\Ftil_0(\pt\ncq G)$ gives maps of line bundles 
\beq
\omega^{k/2}\otimes \omega^{l/2}\to \omega^{(k+l)/2}. \label{eq:multiplicationonomega}
\eeq

\begin{rmk} The notation $\omega^{k/2}$ is intended to remind one of the Hodge bundle over the moduli stack of elliptic curves. This accounts for various minus signs; indeed the transformation law for weight~$k$ modular forms is of the same form as~\eqref{eq:sectionsofHodge}. 
\end{rmk}

Similarly, let $\omegasbar^{-1/2}$ denote the line bundle on $\F_0(\pt\ncq G)$ coming from the homomorphism of generalized Lie groups in~\eqref{eq:GL1rep2}, and $\omegasbar^{-k/2}:=(\omegasbar^{-1/2})^{\otimes k}$ denote the $k$th tensor power. Define the function 
\beq
\vol=\frac{1}{2i}(\lambda_1\bar\lambda_2-\lambda_2\bar\lambda_1)\in C^\infty(\Lat)^{\SL_2(\Z)}\subset C^\infty(\Lat) \label{eq:vol}
\eeq
whose value on a based lattice $\Lambda$ is the volume of $\R^2/\Lambda$ using the flat metric. Since this function takes values in the positive reals, it has a positive square root~$\vol^{1/2}\in C^\infty(\Lat)^{\SL_2(\Z)}$. We observe that $\vol^{1/2}$ determines a nonvanishing section of $\omegas^{-1/2}\otimes \omegasbar^{-1/2}$ over $\Ftil_0(\pt\ncq G)$, and hence an isomorphism 
\beq
\omegas^{k/2}\simeq \omegasbar^{-k/2}\label{eq:conjiso}
\eeq
for each $k\in \Z$. 

Define the map $\pi$ 
\beq
&&\resizebox{.95\textwidth}{!}{%
$\Ftil_0(M\ncq G)=\big(\Lat\times \Fun(\R^{0|1}\sq \Z^2,M\nsq G)_\wz\big)/G_0\stackrel{\pi}{\to} \big(\Lat\times \Fun(\R^{0|1}\sq \Z^2,\pt\nsq G)_\wz\big)/G_0=\Ftil_0(\pt\ncq G)$
}
\eeq
induced by the terminal $G$-equivariant map $M\to \pt$. Since this map is $\E^{0|1}\rtimes \C^\times\times \pi_0(G)\times \SL_2(\Z)$-equivariant, it determines a map of stacks we also denote by $\pi$, which is part of the commutative square
\beq
\begin{tikzpicture}[baseline=(basepoint)];
\node (A) at (0,0) {$\F_0(M\nsq G)$};
\node (B) at (6,0) {$\F_0(\pt\nsq G)$};
\node (D) at (6,-1.25) {$\F_0(\pt\ncq G)$};
\node (C) at (0,-1.25) {$\F_0(M\ncq G)$};
\draw[->] (A) to (B);
\draw[->] (A) to node [left] {$q$} (C);
\draw[->] (B) to node [right] {$q$} (D);
\draw[->] (C) to node [above] {$\pi$} (D);
\path (0,-.75) coordinate (basepoint);
\end{tikzpicture}\label{defn:pifirst}
\eeq
where the vertical arrows are the coarse quotient maps~\eqref{eq:F0MG}, and the upper horizontal map comes from naturality in the the $G$-equivariant map $M\to \pt$, following Remark~\ref{rmk:nat}. We consider the direct image $\pi_*C^\infty_{\F_0(M\ncq G)}$ of the sheaf of smooth functions, regarded as an $\E^{0|1}\rtimes \C^\times\times \pi_0(G)\times \SL_2(\Z)$-equivariant sheaf on $\Ftil_0(\pt\ncq G)$. 
Define the coarse quotient
\beq
\Mstil_G:=\Ftil_0(\pt\ncq G)/\C^\times,\qquad \Mst_G:=[\Mstil_G\sq \SL_2(\Z)\times \pi_0(G)]. \label{eq:Mtildef}
\eeq

\begin{lem} \label{lem:smoothchaincomplex}
The vector space of $\C^\times$-invariant sections of~$\bigoplus_{k\in \Z} (\pi_*C^\infty_{\F_0(M\ncq G)}\otimes \omega^{-k/2})$ on $\Ftil_0(\pt\ncq G)$ determines a $\pi_0(G)\times \SL_2(\Z)$-equivariant sheaf of commutative differential graded algebras (cdgas) on $\Mstil_G$ with graded multiplication determined by multiplication in $C^\infty_{\F_0(M\ncq G)}$ and the maps of line bundles~\eqref{eq:multiplicationonomega}. The differential is determined by the vector field generating the $\E^{0|1}$-action on $\Ftil_0(\pt\ncq G)$. Finally, this sheaf of cdgas is 2-periodic, with an invertible degree $-2$ (Bott) element determined by $\lambda_2^{-1}\in C^\infty(\Lat)$, viewed as a nonvanishing section of $\omega^{2/2}=\omega$. The Bott element is not $\SL_2(\Z)$-invariant. 
\end{lem}
\bp
Using Proposition~\ref{prop:opensofquot}, the first part of the argument is largely formal: sections on the $\C^\times$-quotient are precisely the $\C^\times$-invariant sections over $\C^\times$-invariant open subsheaves of~$\Ftil_0(M\ncq G)$. Because $-1\in \C^\times$ acts as the grading involution on functions, the $\Z$-grading from these $\C^\times$-weight spaces refines the $\Z/2$-grading on functions. Hence, we obtain a sheaf of graded commutative algebras on $\Ftil_0(M\ncq G)\cq \C^\times=\Mstil_G$ whose $k$th graded piece consists of sections satisfying~\eqref{eq:sectionsofHodge} over a $\C^\times$-invariant open subsheaf $U\subset \Ftil_0(M\ncq G)$. The existence of the $\pi_0(G)\times \SL_2(\Z)$-action is also formal, since this action commutes with the $\C^\times$-action defining the grading. Thus far we have verified that $\C^\times$-invariant sections of~$\bigoplus_{k\in \Z} (\pi_*C^\infty_{\F_0(M\ncq G)}\otimes \omega^{k/2})$ determine a sheaf of commutative graded algebras on $\Mst_G$. 


We endow this sheaf with a differential as follows. First, using~\eqref{eq:supertrans} we observe that there is a homomorphism $\E^{2|1}\rtimes \C^\times\to \E^{0|1}\rtimes \C^\times$, where the target semidirect product is defined by the action
$$
\C^\times\times \E^{0|1}\to \E^{0|1}\qquad (\mu,\bar\mu)\cdot \theta= \overline{\mu}\theta\qquad (\mu,\bar\mu)\in \C^\times(S), \theta\in \E^{0|1}(S). 
$$ 
Let $\partial_\theta$ be a generator for the Lie algebra of $\E^{0|1}$, and $\partial_\mu,\partial_{\bar \mu}$ be the standard generators for the complexified Lie algebra of~$\C^\times$. We use the same notation to denote the associated derivations on functions on $\Ftil_0(M\ncq G)$ coming from the infinitesimal $\E^{0|1}\rtimes \C^\times$-action. Since the action on functions is a right action, $\partial_\theta$ defines a map $\overline{\omega}^{k/2}\to \overline{\omega}^{(k+1)/2}$. Using the isomorphism of line bundles~\eqref{eq:conjiso}, the differential 
\beq
Q:=\vol^{1/2}\partial_\theta\label{notation:Q}
\eeq 
gives a map $\omega^{k/2}\to \omega^{(k-1)/2}$. We claim the operator $Q$ is a $C^\infty_{\Ftil_0(\pt\ncq G)}$-module map. To see this consider the $\E^{0|1}\rtimes \C^\times$-action as the restrictions along the inclusions
\beq
\resizebox{\textwidth}{!}{%
\begin{tikzpicture}[baseline=(basepoint)];
\node (A) at (0,0) {$\big(\Lat\times \Fun(\R^{0|1}\sq \Z^2,M\nsq G)_\wz\big)/G_0$};
\node (B) at (8,0) {$\big(\Lat\times \Omega^1(-\times \R^{0|1};\fg)_\wz\times \Map(\R^{0|1},M)\times \Hom(\Z^2,G)\big)/G_0$};
\node (C) at (0,-1.25) {$\big(\Lat\times \Fun(\R^{0|1}\sq \Z^2,\pt\nsq G)_\wz\big)/G_0$};
\node (D) at (8,-1.25) {$\big(\Lat\times \Omega^1(-\times \R^{0|1};\fg)_\wz\times \Hom(\Z^2,G)\big)/G_0.$};
\draw[->,right hook-latex] (A) to (B);
\draw[->] (A) to node [left] {$\pi$} (C);
\draw[->] (B) to node [right] {$p$} (D);
\draw[->,right hook-latex] (C) to  (D);
\path (0,-.75) coordinate (basepoint);
\end{tikzpicture}\nonumber
}
\eeq
By Lemmas~\ref{lem:deRhamactioninertia} and~\ref{lem:superCartan}, the odd derivation $Q$ restricts from the Cartan differential for the objects in the right hand column, up to the normalization factor of $\vol^{1/2}$. We recall that the Cartan differential acts trivially on 
$$
C^\infty(\Omega^1(-\times \R^{0|1};\fg)_\wz)\simeq \mathcal{O}(\fg_\C),
$$
and the $\E^{0|1}$-action on $\Lat\times \Hom(\Z^2,G)$ is also trivial. Therefore, $Q$ acts by the zero derivation on functions on~$\Ftil_0(\pt\ncq G)$. Since $Q$ acts by degree~$+1$ derivations on $\bigoplus_{k\in \Z} (\pi_*C^\infty_{\F_0(M\ncq G)}\otimes \omega^{-k/2})$, it therefore is linear over $\C^\times$-invariant functions on ${\Ftil_0(\pt\ncq G)}$, i.e., functions on $\Mstil_G$. Furthermore, because the $\E^{0|1}$-action commutes with the $\pi_0(G)\times \SL_2(\Z)$-action, the differential is compatible with the $\pi_0(G)\times \SL_2(\Z)$-equivariant structure. 

Finally, we observe that $\lambda_2^{-1}\in C^\infty(\Lat)\hookrightarrow C^\infty(\Ftil_0(M\ncq G))$ determines a (nonvanishing) section of $\omega^{2/2}$, and so descends to an invertible degree~$-2$ element in the associated cdga. This defines the claimed Bott class whose $\SL_2(\Z)$-transformation properties are determined by
$$
\lambda_2\mapsto c\lambda_1+d\lambda_2=\lambda_2(c\tau+d)\qquad \tau=\lambda_1/\lambda_2,\ \left[\begin{array}{cc} a & b \\ c & d\end{array}\right]\in \SL_2(\Z).
$$
\ep


%

\subsection{Gauge-invariant twisted sectors and equivariant de Rham complexes}\label{sec:revisit}
Now we turn attention to gauge-invariant twisted sectors. 

\begin{defn}\label{defn:twistedgauge} Define the sheaf 
$$
\Ftil_0(M\ncq G)_h:=\big(\Lat\times \Fun(\R^{0|1},M^h\nsq G_0^h)\big)/\E^2\times \Map(\R^{0|1},G_0^h)
$$
as the coarse quotient of objects in Definition~\ref{defn:twisted}, where $G_0^h$ is defined in~\eqref{defn:G0h}. Define the \emph{gauge-invariant $h$-twisted sector} as the stack
\beq
\F_0(M\ncq G)_h:= [\Ftil_0(M\ncq G)_h\sq \E^{0|1}\rtimes\C^\times] 
\eeq
\end{defn}
From~\eqref{eq:twistedinclude} and~\eqref{eq:F0Gaction}, we have the 2-commuting square
\beq
\begin{tikzpicture}[baseline=(basepoint)];
\node (A) at (0,0) {$\F_0(M\nsq G)_h$};
\node (B) at (6,0) {$\F_0(M\nsq G)$};
\node (D) at (6,-1.25) {$\F_0(M\ncq G)$};
\node (C) at (0,-1.25) {$\F_0(M\ncq G)_h$};
\draw[->] (A) to (B);
\draw[->] (A) to node [left] {$q$} (C);
\draw[->] (B) to node [right] {$q$} (D);
\draw[->] (C) to (D);
\path (0,-.75) coordinate (basepoint);
\end{tikzpicture}
\eeq
where the vertical arrows are the coarse quotient maps. By Lemma~\ref{lem:smoothchaincomplex}, the restriction along the inclusion $\Ftil_0(M\ncq G)_h\hookrightarrow \Ftil_0(M\ncq G)$ of generalized atlases gives a map of cdgas
\beq
\Gamma(\Ftil_0(M\ncq G)/\C^\times;\omega^{-\bullet/2})\to \Gamma(\Ftil_0(M\ncq G)_h/\C^\times;\omega^{-\bullet/2}). \label{eq:twistedsectorchain}
\eeq


\begin{prop}\label{prop:twisted}
There is an isomorphism of cdgas
\beq
\Big(\Omega^\bullet(M^h;C^\infty(\HH)[\beta,\beta^{-1}])\otimes \mathcal{O}(\fg_\C^h)\Big)^{G_0^h}\stackrel{\sim}{\to} \Gamma(\Ftil_0(M\ncq G)_h/\C^\times;\omega^{-\bullet/2})\label{eq:isovs}
\eeq
where the grading on the left is for differential forms with values in the graded ring $C^\infty(\HH)[\beta,\beta^{-1}]$ with $|\beta|=-2$ and $\mathcal{O}(\fg_\C^h)$ is in degree zero. The differential on the left is the Cartan differential $d-\beta^{-1}\iota$. For $\gamma\times [g]\in \SL_2(\Z)\times G/G_0\simeq \SL_2(\Z)\times \pi_0(G)$, a choice of representative $g\in G$ determines a commutative square
\beq
\begin{tikzpicture}[baseline=(basepoint)];
\node (A) at (0,0) {$(\Omega^\bullet(M^{h'};C^\infty(\HH)[\beta,\beta^{-1}])\otimes \mathcal{O}(\fg_\C^{h'}))^{G^{h'}_0}$};
\node (B) at (6,0) {$\Gamma(\Ftil_0(M\ncq G)_{h'}/\C^\times;\omega^{-\bullet/2})$};
\node (D) at (6,-1.4) {$\Gamma(\Ftil_0(M\ncq G)_h/\C^\times;\omega^{-\bullet/2})$};
\node (C) at (0,-1.4) {$(\Omega^\bullet(M^h;C^\infty(\HH)[\beta,\beta^{-1}])\otimes \mathcal{O}(\fg_\C^h))^{G_0^h}$};
\draw[->] (A) to (B);
\draw[->] (A) to node [left] {$(\gamma\times g)^*$} (C);
\draw[->] (B) to (D);
\draw[->] (C) to (D);
\path (0,-.75) coordinate (basepoint);
\end{tikzpicture}\label{eq:isovs2}
\eeq
where $h'=g(\gamma\cdot h)g^{-1}$ is the image of $h$ under the action~\eqref{eq:HomZGact}, the map on the right is the pullback along~\eqref{eq:twistmap2}, and the map on the left is determined by the pullback of differential forms and
$$
z\mapsto \Ad_g^*(z)/(\tau c+d), \ \ \beta\mapsto \beta/(c\tau+d) \qquad \tau\in \HH, \ \left[\begin{array}{cc} a & b \\ c & d\end{array}\right]\in \SL_2(\Z),\ z\in (\fg_\C^\vee)^{G_0}\subset  \mathcal{O}(\fg_\C)^{G_0}.
$$ 
\end{prop}

\bp
Consider the isomorphisms of algebras
\beq
C^\infty(\Ftil_0(M\ncq G)_h)&\simeq& \left(\Omega(M^h;C^\infty(\Lat))\otimes \mathcal{O}(\fg_\C^h)\otimes \Lambda \fg_\C^\vee\right)^{\Map(\R^{0|1},G_0^h)} \label{eq:Weiltwist}\\
&\simeq& \left(\Omega(M^h;C^\infty(\Lat))\otimes \mathcal{O}(\fg_\C^h)\right)^{G_0^h}\label{eq:twistedCartan}
\eeq
where the first isomorphism uses Lemmas~\ref{lem:PiT} and~\ref{lem:stalk}, while the second isomorphism uses the description of $\Omega^1(-\times \R^{0|1};\fg^h)_\wz$ as the quotient~\eqref{eq:WZquotient}, or equivalently, Corollary~\ref{cor:Cartanvs}. We then use Lemma~\ref{lem:sectordeRham} to compute the action of $\E^{0|1}\rtimes \C^\times\times \pi_0(G)\times \SL_2(\Z)$ on this graded algebra. We observe that the subalgebra generated by the elements
$$
\alpha \in \Omega^k(M^h;C^\infty(\Lat))\qquad z\in \fg_\C^\vee\subset\mathcal{O}(\fg_\C)
$$
is dense in~\eqref{eq:twistedCartan}, and therefore it suffices to compute the action on these functions. We will define the map~\eqref{eq:isovs} in terms of the isomorphism $\varphi$ defined by
\beq
\Omega(M^h;C^\infty(\Lat))\otimes \mathcal{O}(\fg_\C^h)&\stackrel{\varphi}{\to}& \Omega(M^h;C^\infty(\Lat))\otimes \mathcal{O}(\fg_\C^h)\supset C^\infty(\Ftil_0(M\ncq G)_h)\nonumber \\
\alpha&\mapsto&\vol^{k/2}\alpha\label{eq:varphi}\\
z&\mapsto& (\lambda_2^{-1}\vol)  z.\nonumber
\eeq
Namely, we restrict to the subalgebra of the source of $\varphi$ whose image is the direct sum over $\C^\times$-weight spaces in $C^\infty(\Ftil_0(M\ncq G)_h)$. We remark that the purpose of~\eqref{eq:varphi} is to implement the isomorphism between line bundles $\bar\omega^{k/2}$ and $\omega^{-k/2}$ in the proof of Lemma~\ref{lem:smoothchaincomplex}.

To show that $\varphi$ determines the isomorphism~\eqref{eq:isovs}, first we compute the $\C^\times$-weight spaces in $C^\infty(\Ftil_0(M\ncq G)_h)$ in terms of the actions 
\beq
(\mu,\bar\mu)\cdot\varphi(\alpha)&=&(\mu^k\bar\mu^k\vol^{k/2} )(\bar\mu^{-k} L_{\mu,\bar\mu}^*\alpha)=\mu^k(\vol^{k/2}L_{\mu,\bar\mu}^*\alpha)=\mu^k\varphi(L_{\mu,\bar\mu}^*\alpha)\label{eq:formact}\\
(\mu,\bar\mu)\cdot \varphi(z)&=&(\mu^{-2}\lambda^{-1})(\mu^2\bar\mu^2\vol) \bar\mu^{-2}z=(\lambda_2^{-1}\vol) z=\varphi(z)\label{eq:Lieact}
\eeq
where $L_{\mu,\bar\mu}\colon \Lat\to \Lat$ is the $\C^\times$-action on $\Lat$, and $L_{\mu,\bar\mu}^*\colon \Omega(M^h;C^\infty(\Lat))\to \Omega(M^h;C^\infty(\Lat))$ is the action on coefficients by pullback. Hence, a $k$-form $\alpha$ has $\C^\times$-weight $k$ and the coordinate functions $\varphi(z)$ have $\C^\times$-weight zero. Next observe that $\Lat/\C^\times\simeq \HH$ and pulling back along the map
$$
\Lat\to \HH,\qquad \Lat(S)\ni (\lambda_1,\bar\lambda_1,\lambda_2,\bar\lambda_2)\mapsto (\lambda_1/\lambda_2,\bar\lambda_1/\bar\lambda_2)=:(\tau,\bar\tau)\in \HH(S)
$$
gives an isomorphism $C^\infty(\HH)\stackrel{\sim}{\to} C^\infty(\Lat)^{\C^\times}$ with the $\C^\times$-invariant functions, i.e., the zeroth weight space. Functions in all other weight spaces can be written as $\lambda_2^{k}F(\tau,\bar\tau)$ for $k\in \Z$ and $F\in C^\infty(\HH)\stackrel{\sim}{\to} C^\infty(\Lat)^{\C^\times}$. This follows by noting that the odd weight spaces are trivial and $\lambda_2^{-1}\mapsto \mu^{-2}\lambda_2^{-1}$ under the $\C^\times$-action. Hence, the direct sum over all weight spaces for the $\C^\times$-action is $\Omega(M;C^\infty(\HH)[\lambda_2,\lambda_2^{-1}])\otimes\mathcal{O}(\fg_\C^h)\subset \Omega(M;C^\infty(\Lat))\otimes\mathcal{O}(\fg_\C^h)$, where the $k$th weight space is spanned by degree $j$ differential forms with values in the subspace of $C^\infty(\HH)[\lambda_2,\lambda^{-1}_2]$ of homogeneous Laurent degree $k-j$. This shows that the map~$\varphi$ restricted to weight spaces yields the isomorphism~\eqref{eq:isovs} of graded vector spaces, where we identify $\beta\in C^\infty(\HH)[\beta,\beta^{-1}]$ with $\lambda_2^{-1}\in C^\infty(\Lat)$ and take $G_0^h$-invariants. 

We verify commutativity of~\eqref{eq:isovs2} for the $\SL_2(\Z)$-action via
$$
\lambda_2\mapsto c\lambda_1+d\lambda_2=\lambda_2(c\tau+d),\qquad \implies \qquad \beta\mapsto \beta/(c\tau+d)
$$
$$
\varphi(z)= \frac{\vol}{\lambda_2} z\mapsto \frac{\vol}{(c\lambda_1+d\lambda_2)} z=  \frac{\vol}{\lambda_2}\frac{z}{(c\tau+d)} =\frac{\varphi(z)}{c\tau+d}
$$
where we used that $\vol\in C^\infty(\Lat)^{\SL_2(\Z)}$ is $\SL_2(\Z)$-invariant. Finally, we observe that $M^h=M^{\gamma\cdot h}$ as subsets of $M$ for $\gamma \in \SL_2(\Z)$. Hence the action map $\Omega(M^h)\simeq C^\infty(\Map(\R^{0|1},M^h))\to \Omega(M^{\gamma\cdot h})\simeq C^\infty(\Map(\R^{0|1},M^{\gamma \cdot h}))$ is the identity map of algebras.

Next we verify commutativity of~\eqref{eq:isovs2} for the $\pi_0(G)$-action using the isomorphism $\pi_0(G)\simeq G/G_0$. Since $\Lat$ (and in particular, $\lambda_2$ and $\vol$) are $G$-invariant, the $\pi_0(G)$-action on coefficients $C^\infty(\HH)[\beta,\beta^{-1}]\simeq C^\infty(\HH)[\lambda_2,\lambda_2^{-1}]\subset C^\infty(\Lat)$ is trivial. The action on 
$$
C^\infty((\Map(\R^{0|1},M^h)\times \Omega^1(-\times \R^{0|1};\fg^h)_\wz)/G_0^h)\simeq (\Omega(M^h)\otimes  \mathcal{O}(\fg^h_\C))^{G_0^h}
$$ 
is determined by the $G$-action on $\Map(\R^{0|1},M^h)\subset \Map(\R^{0|1},M)$ and $\Omega^1(-\times \R^{0|1};\fg^h)_\wz\subset \Omega^1(-\times \R^{0|1};\fg)_\wz$. By Lemma~\ref{lem:sectordeRham} part (3) and using the identification $\Omega^1(-\times \R^{0|1};\fg^h)_\wz\simeq \Omega^1(-\times \R^{0|1};\fg^h)/\Pi\fg$, the action on $C^\infty(\Omega^1(-\times \R^{0|1};\fg^h)_\wz)\simeq \mathcal{\O}(\fg_\C)$ is determined by the adjoint action of $G$ on $\fg_\C$, and the action on $C^\infty(\Map(\R^{0|1},M^h))\simeq \Omega(M^h)$ is given by the left action of $G$ on $M$. This determines the action of $[g]\in G/G_0\simeq \pi_0(G)$ by choosing a representative $g\in G$ and observing
$$
\Omega(M^h;C^\infty(\Lat))\otimes \mathcal{O}(\fg_\C^h)\to \Omega(M^{ghg^{-1}};C^\infty(\Lat))\otimes\mathcal{O}(\fg_\C^{ghg^{-1}}). 
$$ 
This is compatible with the maps~\eqref{eq:twistmap} by their definition. 

Lastly, we compute the infinitesimal $\E^{0|1}$-action. By Lemma~\ref{lem:Cartandiff}, we recall that this action comes from descending the $\E^{0|1}$-action on $\Map(\R^{0|1},M^h)\times \Omega^1(-\times \R^{0|1};\fg^h)$ to the subquotient $(\Map(\R^{0|1},M^h)\times \Omega^1(-\times \R^{0|1};\fg^h)_\wz)/G_0^h$. The resulting odd derivation on $(\Omega(M^h)\otimes\mathcal{O}(\fg_\C^h))^{G_0^h} \simeq C^\infty((\Map(\R^{0|1},M^h)\times \Omega^1(-\times \R^{0|1};\fg^h)_\wz)/G_0^h)$ is the Cartan differential. 
The normalization on $Q$ from~\eqref{notation:Q} and intertwining with $\varphi$ modifies the Cartan differential $d_C$ as
$$
\varphi^{-1}\circ Q\circ \varphi=\varphi^{-1}\circ (\vol^{1/2}d_C)\circ \varphi=d-\lambda_2\iota. 
$$
This agrees with the Cartan differential $d-\beta^{-1}\iota$ under the isomorphism~\eqref{eq:isovs}. 
\ep

\section{Deforming twisted sectors and analytic functions}\label{sec:Qcoh1}

Our eventual goal is to understand functions on $\F_0(M\ncq G)$ in terms of restrictions to each gauge-invariant twisted sector~$\F_0(M\ncq G)_h\subset \F_0(M\ncq G)$ together with compatibility requirements as~$h\colon \Z^2\to G$ varies. Proposition~\ref{prop:twisted} both characterizes the data of each restriction and gives a compatibility requirement when~$h$ and~$h'$ are related by the action of the discrete group $\pi_0(G)\times \SL_2(\Z)$ on conjugacy classes $[h]$ and $[h']$. The focus of this section is to understand compatibility coming from \emph{continuous} deformations of twisted sectors, i.e., when~$h$ and~$h'$ are close to each other in the space of maps $\Z^2\to G$. Motivated by ideas in physics, we define a sheaf of \emph{analytic functions} on $\F_0(M\ncq G)$ that roughly requires an analytic dependence on these continuous deformations of $h\colon \Z^2\to G$. In~\S\ref{sec:compare}, we will show that sections of this analytic sheaf determine cocycles for equivariant elliptic cohomology. 
We briefly state the idea of analytic functions defined below. A deformation of a map $h=\langle h_1,h_2\rangle \colon \Z^2\to G$ is given by $h_\epsilon=\langle h_1e^{\epsilon_1 X_1},h_2e^{\epsilon_2 X_2}\rangle \colon \Z^2\to G$ where the Lie algebra elements~$X_1,X_2\in \mf{g}$ are such that $h_1e^{\epsilon_1 X_1}$ and $h_2e^{\epsilon_2 X_2}$ commute. The elements $X_1,X_2$ determine a 1-dimensional complex subspace ${\rm Span}((X_1,0),(0,X_2))\subset \fg\times \fg\simeq \fg_\C$. Then a function $f$ on pairs of commuting elements has \emph{analytic dependence} if for every family~$h_\epsilon$ of pairs of commuting elements, the restriction of~$f$ determines a holomorphic function on an $\epsilon$-ball in~$\C\simeq {\rm Span}((X_1,0),(0,X_2))$. When applying this definition to inertia fields, the particular identification~$\fg\times \fg\simeq \fg_\C$ will depend on a modular parameter $\tau\in \HH$ or~$\Lambda\in \Lat$. Roughly, a function on inertia fields is analytic if it (1) has analytic dependence on pairs of commuting elements and (2) depends holomorphically on the modular parameter. 

The difficulty in making the above sketch precise stems from the fact that (especially for nonabelian~$G$) the space of commuting elements and attendant fixed point sets $M^{h_1,h_2}$ can be very complicated. To simplify this situation, we will restrict to deformations $h_\epsilon$ as above in which $X_1,X_2$ commute and the abelian subalgebra they generate $\langle X_1,X_2\rangle=:\mathfrak{a}$ is a subalgebra of~$\fg^h\supset \mathfrak{a}$. It turns out that the fixed point sets in this case satisfy~$M^{h_\epsilon}\subset M^h$ and $G^{h_\epsilon}_0\subset G_0^h$ for $\epsilon$ sufficiently small; see~\eqref{eq:BGdeform} below. These inclusions will give compatibility conditions on functions between restrictions to $h$-twisted sectors and $h'$-twisted sectors; see Lemma~\ref{lem:compatibledeform}.

\subsection{Deforming twisted sectors}
For any commuting subalgebra $\mathfrak{a}\subset \fg^h$, define
\beq
M^{h,\mf{a}}\subset M^h,\quad G^{h,\mf{a}}_0\subset G_0^h,\quad \fg^{h,\mf{a}}\subset \fg^h\label{eq:theseareinclusions}
\eeq
as follows. Let $M^{h,\mf{a}}$ be the submanifold of $M^h$ on which the vector fields associated with the infinitesimal $\mf{a}$-action are all zero, and let $G^{h,\mf{a}}_0$ be defined the same way using the conjugation action of $G$ on itself. Finally, let $\fg^{h,\mf{a}}$ denote the subalgebra of $\fg^h$ on which the adjoint $\mf{a}$-action is trivial. Using the inclusions~\eqref{eq:theseareinclusions}, for each commuting subalgebra $\mf{a}\subset \fg^h$ we obtain a subsheaf of the atlas for the $h$-twisted sector
\beq
&&\Lat\times \Fun(\R^{0|1},M^{h,\mf{a}}\nsq G^{h,\mf{a}}_0)_\wz \to \Lat\times \Fun(\R^{0|1},M^h\nsq G_0^h)_\wz\to  \F_0(M\nsq G)_h,\label{eq:substacksector}
\eeq
where the second arrow is the atlas for the $h$-twisted sector from Corollary~\ref{cor:sectordeRham}. 

For $\epsilon>0$, let $B_\epsilon\subset\mf{a}\times\mf{a}$ denote an $\epsilon$-ball centered at the origin $(0,0)\in \mf{a}\times \mf{a}$. We will define a map $\tilde{\delta}_\epsilon$ that sits in the diagram 
\beq
\resizebox{\textwidth}{!}{%
\begin{tikzpicture}[baseline=(basepoint)];
\node (BB) at (6,1.5) {$\Lat\times \Omega^1(-\times \R^{0|1};\fg^h)_\wz\times \Map(\R^{0|1},M^h)\times G_0^h\times G_0^h$};
\node (B) at (6,0) {$\Lat\times \Fun(\R^{0|1}\sq \Z^2,M^h\nsq G_0^h)_\wz$};
\node (A) at (-2,0) {$B_\epsilon\times \Lat\times \Fun(\R^{0|1},M^{h,\mf{a}}\nsq G^{h,\mf{a}}_0)_\wz$};
\node (D) at (6,-1.5) {$ \Lat\times \Fun(\R^{0|1},M^h\nsq G_0^h)_\wz$};
\node (C) at (-2,-1.5) {$\Lat\times \Fun(\R^{0|1},M^{h,\mf{a}}\nsq G^{h,\mf{a}}_0)_\wz$};
\draw[->,dashed] (A) to node [below] {$\delta_\epsilon$} (B);
\draw[->] (A) to node [left] {$p$} (C);
\draw[->,right hook-latex] (D) to (B);
\draw[->,right hook-latex] (B) to (BB);
\draw[->] (C) to (D);
\draw[->] (A) to node [above] {$\tilde{\delta}_\epsilon$} (BB);
\path (0,1) coordinate (basepoint);
\end{tikzpicture}\label{eq:opensubanaly}}
\eeq
and show that there is a uniquely determined map~$\delta_\epsilon$ making the diagram commute. The lower inclusion on the right is determined by the restriction of~\eqref{eq:factorjh} to the Wess--Zumino gauge, while the upper inclusion uses the Lemma~\ref{lem:WZsubsheaf}. The arrow $p$ is the projection. The lower horizontal arrow is from~\eqref{eq:substacksector}.

\begin{defn}\label{eq:gCaction}
For $h=(h_1,h_2)\in \Hom(\Z^2,G)$ and $\mf{a}\subset \mf{g}^h$ a commuting subalgebra, define a map
$$
 \Lat\times \mf{a}\times \mf{a}\times\Fun(\R^{0|1},M^{h,\mf{a}}\nsq G^{h,\mf{a}}_0)_\wz \stackrel{\tilde{\delta}_\epsilon}{\to}\Lat\times G_0^h\times G_0^h\times \Omega^1(-\times \R^{0|1};\fg^h)_\wz  \times\Map(\R^{0|1},M^h)
$$
that on $S$-points is 
\beq
(\Lambda,(X_1, X_2),A,\phi)\mapsto (\Lambda,(h_1e^{X_1},h_2e^{X_2}),A+\theta d\theta (\lambda_2X_1-\lambda_1 X_2),\phi)\label{eq:tdeform}
\eeq
where 
$$
(X_1,X_2)\in (\mf{a}\times \mf{a})(S), \quad A\in \Omega^1(S\times \R^{0|1};\fg^{h,\mf{a}})_\wz,\quad \phi\in \Map(\R^{0|1},M^{h,\mf{a}})(S),
$$
and therefore $(h_1e^{X_1},h_2e^{X_2})\in (G_0^h\times G_0^h)(S)$. 
\end{defn}


\begin{lem}\label{lem:defo} The map $\tilde{\delta}_\epsilon$ uniquely determines a map $\delta_\epsilon$ in~\eqref{eq:opensubanaly} making the diagram commute. 
\end{lem}
\bp
The claim amounts to showing that the data in the target of~\eqref{eq:tdeform} define a functor
$$
\big(B_\epsilon\times \Lat\times \Fun(\R^{0|1},M^{h,\mf{a}}\nsq G^{h,\mf{a}}_0)_\wz\big)\times \R^{0|1}\sq \Z^2\to M^h\nsq G_0^h.
$$
First we observe that by virtue of $X_1,X_2\in \mf{a}(S)\subset \fg^h(S)$ being in an abelian subalgebra, the $S$-points $h_1e^{X_1},h_2e^{X_2}\in G_0^h(S)$ commute and so determine an $S$-family of homomorphisms
$$
(h_1e^{X_1},h_2e^{X_2})\in \Hom(\Z^2,G_0^h)(S)\hookrightarrow (G_0^h\times G_0^h)(S).
$$
From the subsheaf description of Lemma~\ref{lem:WZsubsheaf}, this reduces the claim in the present lemma to checking compatibility with the source and target maps. The former is automatic. Compatibility with the target map demands that $A+\theta d\theta(\lambda_2 X_1-\lambda_1 X_2)$ and $\phi$ be invariant under the $\Z^2$-action generated by $(h_1e^{X_1},h_2e^{X_2})$. From the assumption that $A\in \Omega^1(S\times \R^{0|1};\fg^{h,\mf{a}})$ and $\phi\in \Map(\R^{0|1},M^{h,\mf{a}})(S)$, we obtain the stronger statement of invariance under the separate $\Z^2$-actions of $(h_1,h_2)$ and $(e^{X_1},e^{X_2})$. 
This proves the lemma.\ep

For $\mf{a}\subset \fg^h$ a commuting subalgebra and $\epsilon>0$ satisfying the condition in Lemma~\ref{lem:defo}, there are a pair of maps
$$
\delta_\epsilon,\delta_0\colon B_\epsilon\times \Lat\times \Fun(\R^{0|1},M^{h,\mf{a}}\nsq G^{h,\mf{a}}_0)_\wz\rightrightarrows \Lat\times \Fun(\R^{0|1}\sq \Z^2,M^h\nsq G_0^h)_\wz
$$
where $\delta_0$ is the composition of the projection with the first arrow in~\eqref{eq:substacksector} and the inclusion determined by~\eqref{eq:factorjh}. 
Given an open subsheaf $U\subset \Lat\times \Fun(\R^{0|1}\sq \Z^2,M\nsq G)_\wz$, we shall consider the diagram 
 \beq
 \resizebox{\textwidth}{!}{%
\begin{tikzpicture}[baseline=(basepoint)];
\node (A) at (0,0) {$U_0,U_\epsilon$};
\node (B) at (6,0) {$U\bigcap \Lat\times \Fun(\R^{0|1}\sq \Z^2,M^h\nsq G_0^h)_\wz$};
\node (D) at (6,-1.25) {$\Lat\times \Fun(\R^{0|1}\sq \Z^2,M^h\nsq G_0^h)_\wz$};
\node (E) at (11.5,-1.25){$\Lat\times \Fun(\R^{0|1}\sq \Z^2,M\nsq G)_\wz,$};
\node (EE) at (11.5,0) {$U$};
\node (C) at (0,-1.25) {$B_\epsilon\times \Lat\times \Fun(\R^{0|1},M^{h,\mf{a}}\nsq G^{h,\mf{a}}_0)_\wz$};
\draw[->] (A) to (B);
\draw[->,right hook-latex] (A) to (C);
\draw[->,right hook-latex] (B) to node [right] {$i_U$} (D);
\draw[->] (C) to node [above] {$\delta_\epsilon,\delta_0$} (D);
\draw[->] (D) to (E);
\draw[->,right hook-latex] (EE) to (E);
\draw[->] (B) to (EE);
\path (0,-.5) coordinate (basepoint);
\end{tikzpicture}}\label{eq:analyticdiagram}
\eeq
where the inclusion $i_U$ defines an open subsheaf of $\Lat\times \Fun(\R^{0|1}\sq \Z^2,M^h\nsq G_0^h)_\wz$, and the further pullback to $\Lat\times \Fun(\R^{0|1}\sq \Z^2,M^h\nsq G_0^h)_\wz$ along $\delta_0$ and $\delta_\epsilon$, respectively, determine open subsheaves 
$$
U_0,U_\epsilon \subset B_\epsilon\times \Lat\times \Fun(\R^{0|1},M^{h,\mf{a}}\nsq G^{h,\mf{a}}_0)_\wz,
$$
In the case that $U=\Lat\times \Fun(\R^{0|1}\sq \Z^2,M^h\nsq G_0^h)_\wz$, we observe that $U_\epsilon$ parameterizes a deformation of the map
\beq
\Lat\times \Fun(\R^{0|1},M^{h,\mf{a}}\nsq G^{h,\mf{a}}_0)_\wz\to \F_0(M\nsq G)\label{eq:maptobedeformed}
\eeq
with deformation parameter in~$B_\epsilon$, whereas $U_0$ parameterizes the constant deformation of~\eqref{eq:maptobedeformed}. In the general case, $U_0$ and $U_\epsilon$ denote the restrictions of these deformations to an open subsheaf. We can compare them on $U_0\bigcap U_\epsilon\subset B_\epsilon\times \Lat\times \Fun(\R^{0|1},M^{h,\mf{a}}\nsq G^{h,\mf{a}}_0)_\wz$, which is also an open subsheaf.

\begin{defn}\label{defn:deform} A function $f\in C^\infty(U)$ on an open subsheaf $U\subset \Lat\times \Fun(\R^{0|1}\sq \Z^2,M\nsq G)_\wz$ is \emph{invariant under deformations} if for all $h\colon \Z^2\to G$ and $\mf{a}\subset \fg^h$ there exists an $\epsilon$ such $\delta_0^*f=\delta_\epsilon^*f\in C^\infty(U_\epsilon\bigcap U_0)$ where $\delta_0^*f$ and $\delta_\epsilon^*f$ denote the pullback of $f$ along the lower row in the diagram~\eqref{eq:analyticdiagram}. Define the \emph{sheaf of deformation invariant functions} as the sheafification of the subsheaf of deformation invariant smooth functions.
\end{defn}

Given $h\colon\Z^2\to G$ suppose $h'=(h_1',h_2')=(h_1e^{X_1},h_2e^{X_2})$ for $X_1,X_2\in B_\epsilon\subset \mf{g}^h$ where $B_\epsilon$ an $\epsilon$-ball centered at~0, and where $[X_1,X_2]=0$. A result of Block and Getzler~\cite[Lemma~1.3]{BlockGetzler} (and its mild generalization~\cite[Lemma~3.2]{BET0}) shows that there exists an $\epsilon>0$ such that for all such $h'$ coming from $X_1,X_2\in B_\epsilon$ we have equalities of subspaces, 
\beq
M^{h'}=M^{h,\mf{a}}\subset M^h,\quad G^{h'}_0=G^{h,\mf{a}}_0\subset G_0^h,\quad \fg^{h'}=\fg^{h,\mf{a}}\subset \fg^h.\label{eq:BGdeform}
\eeq
where $\mf{a}=\langle X_1,X_2\rangle \subset \mf{g}^h$ is the abelian subalgebra generated by $X_1,X_2$. 
%
%
This determines a restriction map 
\beq
{\rm res}\colon \mathcal{O}(\fg_\C^h;\Omega^\bullet(M^h;C^\infty(\Lat)))\to \mathcal{O}(\fg_\C^{h'};\Omega^\bullet(M^{h'};C^\infty(\Lat))).\label{eq:analyticrestrict}
\eeq
Under the isomorphism of algebras~\eqref{eq:twistedCartan}, invariance under~\eqref{eq:tdeform} demands the following condition. 

\begin{lem}\label{lem:compatibledeform} Suppose $U\subset \Lat\times \Fun(\R^{0|1}\sq \Z^2,M\nsq G)_\wz$ is an open subsheaf containing the objects of the $h$-twisted sector $\Lat\times \Fun(\R^{0|1},M^h\nsq G_0^h)_\wz\subset U$ and $h'\colon \Z^2\to G$ is determined by $X_1,X_2\in B_\epsilon \subset \fg^h$ satisfying the conditions for~\eqref{eq:BGdeform}. Then for a function $f\in C^\infty(U)$ invariant under deformations, for $\epsilon$ sufficiently small the images of $f$ under the restrictions
 \beq
 \resizebox{\textwidth}{!}{%
\begin{tikzpicture}[baseline=(basepoint)];
\node (AA) at (-1,0) {$f$};
\node (A) at (0,0) {$\in C^\infty(U)$};
\node (BB) at (1,1) {$f_h$};
\node (B) at (4,1) {$\in C^\infty(\Lat\times \Fun(\R^{0|1},M^h\nsq G_0^h)_\wz)$};
\node (C) at (10,1) {$\mathcal{O}(\fg_\C^h;\Omega^\bullet(M^h;C^\infty(\Lat)))$};
\node (DD) at (.7,-1) {$f_{h'} $};
\node (D) at (4,-1) {$\in C^\infty(\Lat\times \Fun(\R^{0|1},M^{h'}\nsq G^{h'}_0)_\wz)$};
\node (E) at (10,-1) {$\mathcal{O}(\fg_\C^{h'};\Omega^\bullet(M^{h'};C^\infty(\Lat)))$};
\draw[->] (A) to (B);
\draw[->] (B) to node [above] {$\simeq$} (C);
\draw[->] (A) to (D);
\draw[->] (D) to node [above] {$\simeq$} (E);
\draw[->] (C) to node [right] {${\rm res}$} (E);
\draw[|->,bend left] (AA) to (BB);
\draw[|->,bend right] (AA) to (DD);
\path (0,-.5) coordinate (basepoint);
\end{tikzpicture}}\nonumber 
\eeq
 satisfy
$$
{\rm res}(f_h(X))=f_{h'}(X+(\lambda_2 X_1-\lambda_1 X_2))\in \mathcal{O}(\fg_\C^{h'};\Omega^\bullet(M^{h'};C^\infty(\Lat)))
$$
for $h'=(h_1e^{X_1},h_2e^{X_2})$, where the isomorphisms in the diagram follow from Corollary~\ref{cor:Cartanvs} and ${\rm res}$ is the restriction~\eqref{eq:analyticrestrict}. 
\end{lem}
\bp 
Under the isomorphisms in the diagram from Corollary~\ref{cor:Cartanvs}, we recall that the dependence on $\fg_\C^{h'}$ coincides with the dependence on the $\theta d\theta$-component of $A\in \Omega^1(-\times \R^{0|1};\fg^h)_\wz$. Invariance under~\eqref{eq:tdeform} requires that $f_{h'}$ be the restriction of~$f_h$ followed by the pullback along translation by $\lambda_2 X_1-\lambda_1 X_2$ acting on $\Omega^1(-\times \R^{0|1};\fg)_\wz$. This reproduces the claimed formula.
\ep

\subsection{Analytic functions on $\F_0(M\ncq G)$}\label{sec:analytic}

The definition of analyticity has two parts: (1) holomorphic dependence on the complex structure on a super torus and (2) invariance under deformations of $G$-bundles. The former requires an appropriate adjustment for the normalization of this complex structure, using an extension of the automorphism $\varphi$ from~\eqref{eq:varphi} defined as follows. Consider the isomorphism
\beq
\varphi\colon \Lat\times \Fun(\R^{0|1}\sq \Z^2,M\nsq G)_\wz \to \Lat\times \Fun(\R^{0|1}\sq \Z^2,M\nsq G)_\wz\label{eq:varphimap}
\eeq
that (using the subsheaf description from Lemma~\ref{lem:WZsubsheaf}) is determined by the map on $S$-points
$$
(\Lambda,A,h,(x,\psi))\stackrel{\varphi}{\mapsto} (\Lambda,(\Lambda^*(\vol/\lambda_2)) A,h,(x,(\Lambda^*\vol^{1/2})\psi)
$$
$$
\Lambda\colon S\to \Lat,\quad A\in \Omega^1(S\times \R^{0|1};\fg)_\wz,\quad h\colon S\times \Z^2\to G,\quad (x,\psi)\in \Pi TM(S)
$$
for $\vol\in C^\infty(\Lat)$ defined in~\eqref{eq:vol}, and $\lambda_2\in C^\infty(\Lat)$ the restriction along the inclusion~$\Lat\subset \C\times \C$ of the second holomorphic coordinate function on~$\C\times \C$. 
\begin{defn}\label{defn:hollat} For $U\subset \Lat\times \Fun(\R^{0|1}\sq \Z^2,M\nsq G)_\wz$ an open subsheaf, $f\in C^\infty(U)$ has \emph{holomorphic dependence on $\Lat$} if 
$$
\partial_{\bar\lambda_1}((\varphi^{-1})^*f)=0\qquad \partial_{\bar\lambda_2}((\varphi^{-1})^*f)=0
$$
where $\partial_{\bar\lambda_1},\partial_{\bar\lambda_2}$ are the restrictions of the complex vector fields on $\Lat$. 
\end{defn}

The definition of analytic functions imposes the conditions from Definitions~\ref{defn:deform} and~\ref{defn:hollat} simultaneously on open subsheaves of $\F_0(M\ncq G)$. By Lemma~\ref{prop:opensofquot}, open subsheaves
$$
U\hookrightarrow \Lat\times\Fun(\R^{0|1}\sq \Z^2,M\nsq G)_\wz/G_0=\Ftil_0(M\ncq G)
$$
are equivalent to $G_0$-invariant open subsheaves $\widetilde{U}\hookrightarrow \Lat\times\Fun(\R^{0|1}\sq \Z^2,M\nsq G)_\wz$, with the identifications
\beq
\widetilde{U}/G_0\simeq U,\qquad C^\infty(\widetilde{U})^{G_0}\simeq C^\infty(U). \label{eq:weidentify}
\eeq
We use these identifications and notation in the following. 

\begin{defn}\label{defn:generalanalytic}
Let $U\to \Ftil_0(M\ncq G)$ be an open subsheaf. A function $f\in C^\infty(U)$ is \emph{analytic} if (1) the associated $G_0$-invariant function $\tilde{f}\in C^\infty(\widetilde{U})$ is invariant under deformations as in Definition~\ref{defn:deform} and (2) $\tilde{f}\in C^\infty(\widetilde{U})$ has holomorphic dependence on~$\Lat$ as in Definition~\ref{defn:hollat}. Define the \emph{sheaf of analytic functions} on $\Ftil_0(M\ncq G)$ as the sheafification of the presheaf of analytic functions.
\end{defn}
\begin{rmk} Since sheafification preserves monomorphisms, the sheaf of analytic functions is a subsheaf of smooth functions on $(\Lat\times \Fun(\R^{0|1}\sq \Z^2,M\nsq G)_\wz)\cq G_0=\Ftil_0(M\ncq G)$. \end{rmk}

\begin{lem} The sheaf of analytic functions on $\Ftil_0(M\ncq G)$ descends canonically to a sheaf on the stack $\cL_0(M\ncq G)$ and is natural for maps of stacks $\F_0(M\ncq G)\to \F_0(N\ncq H)$ associated with smooth maps $M\to N$ that are equivariant with respect to group homomorphisms~$G\to H$. 
 \end{lem}
 
 \bp 
From the quotient groupoid presentation~\eqref{eq:F0Gaction} of $\F_0(M\ncq G)$, descent is equivalent to checking that (as a subsheaf of smooth functions) the sheaf of analytic functions on $\Ftil_0(M\ncq G)$ is preserved by the action of $\E^{0|1}\rtimes \C^\times\times \pi_0(G)\times \SL_2(\Z)$.  Because $\E^{0|1}\times \pi_0(G)$ acts trivially on $\Lat$, holomorphic dependence on $\Lat$ is purely a question of the $\C^\times\times \SL_2(\Z)$-action. This action preserves the kernel of the derivations in Definition~\ref{defn:hollat}. 
Hence, functions with holomorphic dependence on $\Lat$ descend to the stack.

To verify invariance under deformations of $G$-bundles descends, consider open subsheaves $U,V$ related by the action of~$\alpha\in (\E^{0|1}\rtimes \C^\times\times \SL_2(\Z)\times \pi_0(G))(U)$ as in the triangle on the right,
\beq
\resizebox{\textwidth}{!}{%
\begin{tikzpicture}[baseline=(basepoint)];
\node (A) at (1,0) {$U$};
\node (B) at (6,-.75) {$\Ftil_0(M\ncq G)$};
\node (C) at (1,-1.5) {$V$};
\node (AA) at (-4,0) {$U\bigcap \Lat\times \Fun(\R^{0|1}\sq \Z^2,M^h\nsq G_0^h)_\wz/G_0^h$};
\node (BB) at (-4,-1.5) {$V\bigcap \Lat\times \Fun(\R^{0|1}\sq \Z^2,M^{k}\nsq G_0^{k})_\wz/G^{k}_0$};
\draw[->,right hook-latex] (A) to (B);
\draw[->] (A) to node [left] {$\alpha$} (C);
\draw[->,right hook-latex] (C) to (B);
\draw[->] (AA) to (A);
\draw[->] (BB) to (C);
\draw[->,dashed] (AA) to (BB);
\path (0,-.75) coordinate (basepoint);
\end{tikzpicture}\nonumber}
\eeq
Invariance under deformations from Definition~\ref{defn:deform} involves the rectangle on the left. We will show that for each $h$ and $\alpha$ there exists a $k\colon \Z^2\to G$ making the square commute. Furthermore, for any abelian subalgebra $\mf{a}\subset \fg^h$ and $\epsilon>0$, there is an abelian subalgebra $\mf{b}\subset \fg^{k}$ and a $\delta>0$ so that the condition for being invariant under deformations relative to $(h,\mf{a},\epsilon)$ are sent to conditions for $(k,\mf{b},\delta)$ under $\alpha$. 

Because $\pi_0(G)$ and $\SL_2(\Z)$ are discrete, by restricting to a subsheaf of $U$ we can assume that the composition $\alpha\colon U\to \E^{0|1}\rtimes \C^\times\times  \pi_0(G)\times \SL_2(\Z)\to \pi_0(G)\times \SL_2(\Z)$ is constant, determined by an element $\gamma\in \SL_2(\Z)$ and $[g]\in \pi_0(G)\simeq G/G_0$. For the latter group element we fix a choice of representative $g\in G$. Then we observe that the choice of $k=(g,\gamma)\cdot h$ from~\eqref{eq:HomZGact} makes the square on the left commute where we choose the map given by Corollary~\ref{cor:mapbetweentwist}. Furthermore, the invariance condition for $\mf{a}\subset \fg^h$ corresponds to the invariance condition for $\mf{b}=\Ad_{g}\mf{a}\subset \fg^{k}=\fg^{ghg^{-1}}$. 
This shows that the analytic conditions are compatible with the isomorphism $\alpha$, and hence that the sheaf of analytic functions descends to the stack.
 
It remains to check naturality. However, this follows from the fact that the diagram~\eqref{eq:opensubanaly} is natural for equivariant maps $M\to N$ relative to group homomorphisms $\varphi\colon G\to H$, where restricting the map~$d\varphi \colon \fg\to \mf{h}$ to an abelian subalgebra intertwines the deformations. 
 \ep

\begin{notation} Let $\mc{O}_{\cL_0(M\ncq G)}$ denote the sheaf of analytic functions on $\cL_0(M\ncq G)$.
 \end{notation}
 
\subsection{The $Q$-cohomology sheaf}
In this subsection, let $\omega_{C^\infty}^{1/2}$ denote the sheaf of (smooth) sections of the line bundle on $\Ftil(\pt\ncq G)$ previously denoted by $\omega^{1/2}$ and defined by~\eqref{eq:sHodge}. Since the $\C^\times$-action on $\Ftil_0(\pt\ncq G)$ preserves the subsheaf $\mc{O}_{\Ftil_0(\pt\ncq G)}\subset C^\infty_{\Ftil_0(\pt\ncq G)}$, the sheaf $\omega_{C^\infty}^{1/2}$ has an analytic structure in the sense that there is a sheaf of $\mathcal{O}_{\F(\pt\ncq G)}$-modules $\omega^{1/2}$ and an isomorphism of sheaves
$$
\omega^{1/2}\otimes_{\mc{O}_{\Ftil_0(\pt\ncq G)}} C^\infty_{\Ftil_0(\pt\ncq G)} \simeq \omega^{1/2}_{C^\infty}. 
$$
Similarly define $\omega^{k/2}:=(\omega^{1/2})^{\otimes k}$ as tensor powers of $\omega^{1/2}$, which equivalently are subsheaves of tensor powers $(\omega^{1/2}_{C^\infty})^{\otimes k}$ with an $\mc{O}_{\Ftil_0(\pt\ncq G)}$-module structure. Explicitly, $\omega^{-k/2}$ has as its global sections analytic functions on $\Ftil_0(\pt\ncq G)$ satisfing the condition~\eqref{eq:sectionsofHodge}. 

As in~\eqref{defn:pifirst}, let $\pi$ be the map
\beq\label{defn:pi}
\pi\colon \F_0(M\ncq G)\to \F_0(\pt\ncq G)
\eeq
induced by the $G$-equivariant map $M\to \pt$. Consider the direct image sheaf $\pi_*\mc{O}_{\cL_0(M\ncq G)}$. We now give a variant of the construction in Lemma~\ref{lem:smoothchaincomplex} showing that $\pi_*\mc{O}_{\cL_0(M\ncq G)}$ determines a sheaf of chain complexes on the stack~$\Mst_G$ defined in~\eqref{eq:Mtildef}. 
For future convenience, we phrase this construction in a somewhat more general setting. Let $\mathcal{O}_{\Mstil_G}$ denote the sheaf whose sections are the $\C^\times$-invariant analytic functions on~$\Ftil_0(\pt\ncq G)$, using the correspondence between $\C^\times$-invariant open subsheaves of $\Ftil_0(\pt\ncq G)$ and open subsheaves of the coarse quotient, $\Mstil_G=\Ftil_0(\pt\ncq G)/\C^\times$; see Lemma~\ref{prop:opensofquot}. 

\begin{lem} \label{lem:chaincomplex}
Let $\mathcal{A}$ be a $\Z/2$-graded, $\E^{0|1}\rtimes \C^\times$-equivariant sheaf of $\mathcal{O}_{\Ftil_0(\pt\ncq G)}$-modules on~$\Ftil_0(\pt\ncq G)$ such that the action of $\{\pm 1\}\subset \C^\times$ is equal to the grading involution on~$\mathcal{A}$. The $\C^\times$-invariant sections, 
\beq
\Ch^\bullet(\mathcal{A}):=((\mathcal{A}\otimes \omega^{-\bullet/2})^{\C^\times},Q)\label{eq:CHdef}
\eeq
determine a sheaf of 2-periodic chain complexes in $\mathcal{O}_{\Mstil_G}$-modules on the coarse quotient~$\Mstil_G:=\Ftil_0(\pt\ncq G)\cq \C^\times$. If~$\mathcal{A}$ has the additional structure of a sheaf of super algebras, then $\Ch^\bullet(\mathcal{A})$ carries the structure of a sheaf of cdgas. If $\mathcal{A}$ is $\E^{0|1}\rtimes \C^\times\times\SL_2(\Z)\times\pi_0(G)$-equivariant, then $\Ch^\bullet(\mathcal{A})$ is $\SL_2(\Z)\times \pi_0(G)$-equivariant and defines a sheaf on the stack
$$
\Mst_G=[\Mstil_G\sq \SL_2(\Z)\times \pi_0(G)].
$$
\end{lem}
\bp 
The argument that $(\mathcal{A}\otimes \omega^{-\bullet/2})^{\C^\times}$ determines a graded vector space with odd differential $Q$ (defined by~\eqref{notation:Q}) is the same as in the proof of Lemma~\ref{lem:smoothchaincomplex}, using that the grading by $\C^\times$-weight spaces is compatible with the $\Z/2$-grading on $\mathcal{A}$ by the assumption that $\{\pm 1\}\subset \C^\times$ acts by the grading involution. Then the previously defined map $Q$ on smooth sections of $\omega^{k/2}$ restricts to analytic sections because the sheaf $\mathcal{A}$ of $\mathcal{O}_{\Ftil_0(\pt\ncq G)}$-modules is assumed to be $\E^{0|1}$-equivariant. If $\mathcal{A}$ has an algebra structure, then $Q$ acts by odd derivations on this algebra and we obtain a sheaf of cdgas. 
\ep

\begin{defn} \label{defn:Qcoh}The \emph{$Q$-cohomology sheaf} of a $G$-manifold $M$ is the $\SL_2(\Z)\times\pi_0(G)$-equivariant sheaf of commutative differential graded algebras on $\Mstil_G$ given by $\Ch^\bullet(\pi_*\mathcal{O}_{\F_0(M\ncq G)})$ from Lemma~\ref{lem:chaincomplex} using $\pi$ from~\eqref{defn:pi}. 
\end{defn}

\begin{notation}\label{eq:Qcohomnotation}
To simplify notation we write $\Ch^\bullet(\F_0(M\ncq G))=\Ch^\bullet(\pi_*\mathcal{O}_{\F_0(M\ncq G)})$. We continue to denote the differential on this complex by $Q$. 
\end{notation}

To end the section, we prove a key property satisfied by the restriction of analytic functions to twisted sectors. Suppose that $h$ and $h'$ satisfy the hypothesis of Lemma~\ref{lem:compatibledeform}. Consider the diagram  
 \beq
&&\begin{tikzpicture}[baseline=(basepoint)];
\node (A) at (0,0) {$ \Gamma(\Ftil_0(M\ncq G)_h/\C^\times;\omega^{-\bullet/2})$};
\node (B) at (7,0) {$\Big(\Omega^\bullet(M^h;C^\infty(\HH)[\beta,\beta^{-1}])\otimes \mathcal{O}(\fg_\C^h)\Big)^{G_0^h}$};
\node (C) at (0,-1.5) {$ \Gamma(\Ftil_0(M\ncq G)_{h'}/\C^\times;\omega^{-\bullet/2})$};
\node (D) at (7,-1.5) {$\Big(\Omega^\bullet(M^{h'};C^\infty(\HH)[\beta,\beta^{-1}])\otimes \mathcal{O}(\fg_\C^{h'})\Big)^{G^{h'}_0}$};
\draw[->] (A) to node [above] {$\simeq$} (B);
\draw[->] (C) to node [above] {$\simeq$} (D);
\draw[->] (A) to (C);
\draw[->] (B) to node [right] {${\rm res}$} (D);
\path (0,-.75) coordinate (basepoint);
\end{tikzpicture}\label{diag:analyticfibers}
\eeq
where the horizontal isomorphisms are from Proposition~\ref{prop:twisted}, the arrow ${\rm res}$ is determined by~\eqref{eq:analyticrestrict}, and the remaining arrow is uniquely determined by ${\rm res}$. Restricting to analytic sections in~\eqref{diag:analyticfibers} imposes the following compatibility condition. 

\begin{lem}\label{lem:deRhamanalytic} Let $U\subset \F_0(M\ncq G)$ be a $\C^\times$-invariant open subsheaf containing $\Ftil_0(M\ncq G)_h$ and $\Ftil_0(M\ncq G)_{h'}$ for $h=(h_1,h_2)$ and $h'=(h_1e^{X_1},h_2e^{X_2})$ for $X_1,X_2\in B_\epsilon\subset \fg^h_\C$ as in Lemma~\ref{lem:compatibledeform}. Let $f\in \Ch^k(\F_0(M\ncq G))(U)$ be an analytic section of $\omega^{-k/2}$ over $U$. 
Then for $\epsilon$ sufficiently small, the restriction of $f$ to twisted sectors determines elements $f_h$ and $f_{h'}$ under the isomorphisms~\eqref{eq:isovs}
 \beq
 \resizebox{\textwidth}{!}{%
\begin{tikzpicture}[baseline=(basepoint)];
\node (A) at (0,0) {$f_h\in \mathcal{O}(\Ftil_0(M\ncq G)_h/\C^\times;\omega^{-k/2})$};
\node (B) at (7,0) {$\Big(\Omega^k(M^h;\mathcal{O}(\HH)[\beta,\beta^{-1}])\otimes \mathcal{O}(\fg_\C^h)\Big)^{G_0^h}$};
\node (C) at (0,-1.5) {$f_{h'}\in \mathcal{O}(\Ftil_0(M\ncq G)_{h'}/\C^\times;\omega^{-k/2})$};
\node (D) at (7,-1.5) {$\Big(\Omega^k(M^{h'};\mathcal{O}(\HH)[\beta,\beta^{-1}])\otimes \mathcal{O}(\fg_\C^{h'})\Big)^{G^{h'}_0}$};
\draw[->] (A) to node [above] {$\simeq$} (B);
\draw[->] (C) to node [above] {$\simeq$} (D);
\draw[->] (A) to (C);
\draw[->] (B) to node [right] {${\rm res}$} (D);
\path (0,-.75) coordinate (basepoint);
\end{tikzpicture}}\nonumber
\eeq
satisfying the compatibility condition
\beq
&&{\rm res}(f_h(X))=f_{h'}(X+(X_1-\tau X_2))\in \Big(\Omega^\bullet(M^{h'};\mathcal{O}(\HH)[\beta,\beta^{-1}])\otimes \mathcal{O}(\fg_\C^{h'})\Big)^{G^{h'}_0}  \label{eq:formulaforanalytic}
\eeq
where $\tau=\lambda_1/\lambda_2$. 
\end{lem}

\bp
The fact that the isomorphism~\eqref{eq:isovs} when restricted to analytic functions has image in differential forms valued in $\mathcal{O}(\HH)[\beta,\beta^{-1}]$ follows from the compatibility of this isomorphism with holomorphic dependence on lattices in Definition~\ref{defn:hollat}; this comes from the isomorphism~\eqref{eq:varphi} being compatible (by definition) with~\eqref{eq:varphimap}. The remaining part of the statement follows from Lemma~\ref{lem:compatibledeform}, using that $\lambda_2 X_1-\lambda_1 X_2=\lambda_2(X_1-\tau X_2)$ and taking care with the normalization of coordinate functions on $\fg^h_\C$ from~\eqref{eq:varphi}.
\ep

\section{A cocycle model for equivariant elliptic cohomology}\label{sec:compare}

In this section we prove Theorem~\ref{thm:comparison}, giving an isomorphism of sheaves of cdgas on~$\Bun_G(\EE)$,
\beq
\iota^{-1}i^*\Ch^\bullet(\F_0(M\ncq G))\stackrel{\sim}{\to} \dEll_G^\bullet(M). \label{eq:mainiso}
\eeq
The functor $i^*$ restricts a sheaf on supermanifolds to a sheaf on manifolds; it was omitted from the notation in the statement of Theorem~\ref{thm:comparison}, but we include it in this section to keep track of the ambient category. With the sheaf $\Ch^\bullet(\F_0(M\ncq G))$ on $\Mst_G$ defined in the previous section, we define the functor~$i^*$ in~\S\ref{sec:sheafsection} and the functor~$\iota^{-1}$ in~\S\ref{sec:comparisonstack}. The composition of these gives a functor from sheaves on $\Mst_G$ to sheaves on $\Bun_G(\EE)$. We prove a few technical results about open subsheaves of~$\Bun_G(\EE)$ and~$\Mst_G$ in~\S\ref{sec:topology} before proving Theorem~\ref{thm:comparison} in~\S\ref{sec:thm:comparison}.

\subsection{From sheaves on supermanifolds to sheaves on manifolds}\label{sec:sheafsection}
Consider the restriction of presheaves on ${\sf SMfld}$ to presheaves on ${\sf Mfld}$ along the fully faithful embedding~$i$,
\beq
i\colon  {\sf Mfld}\hookrightarrow  {\sf SMfld},\qquad i^*\colon {\sf PSh}_{\sf SMfld}\to {\sf PSh}_{\sf Mfld}\label{eq:embed}
\eeq
that regards an ordinary manifold as a supermanifold using its sheaf of complex-valued smooth functions. The functor $i^*$ has a fully faithful left adjoint,
$$
i_!^{\rm pre}\colon {\sf PSh}_{\sf Mfld}\to {\sf PSh}_{\sf SMfld}\quad i_!^{\pre}(\X)=\colim_{N\to \X}i(N),
$$
i.e., $i_!^{\pre}\X$ is defined by expressing a presheaf $\X$ on manifolds as a colimit of representables $N\to \X$, and then computing the colimit in presheaves on supermanifolds.

\begin{rmk} The functor $i_!^{\rm pre}$ is the left Kan extension 
\beq
\begin{tikzpicture}[baseline=(basepoint)];
\node (A) at (0,0) {${\sf Mfld}$};
\node (B) at (3,0) {${\sf SMfld}$};
\node (BB) at (6,0) {${\sf PSh}_{\sf SMfld}$};
\node (C) at (3,-1.2) {${\sf PSh}_{\sf Mfld}$};
\draw[->] (A) to node [above] {$i$} (B);
\draw[->] (B) to node [above] {$Y$} (BB);
\draw[->] (A) to node [below] {$Y$} (C);
\draw[->,dashed] (C) to node [right=8pt] {${\rm Lan}=i_!^\pre$} (BB);
\path (0,-.75) coordinate (basepoint);
\end{tikzpicture}\nonumber
\eeq
where $Y$ denotes the Yoneda embedding in the two respective cases. 
\end{rmk} 

Define $i_!$ as the composition of $i_!^{\pre}$ and the sheafification functor. We observe the adjunction in sheaves
\beq
{\sf Sh}_{\sf Mfld}(\X,i^*\Y)\simeq {\sf Sh}_{\sf SMfld}(i_!\X,\Y)\label{eq:adjunction}
\eeq
where above $i^*$ denotes the restriction of $i^*$ to sheaves. 

The endofunctor $i_!i^*\colon {\sf Sh}_{\sf SMfld}\to {\sf Sh}_{\sf SMfld}$ extends the reduced manifold functor to sheaves, i.e., given a sheaf $\Y\in {\sf Sh}_{\sf SMfld}$ it produces a sheaf on supermanifolds that is in the image of $i^*$ together with a map
$$
i_!i^*\Y\to \Y. 
$$
On representables, the above becomes the injection of supermanifolds $S_\red \hookrightarrow S$. This need not be an injection for an arbitrary sheaf $\Y$, as the following example shows. 

\begin{ex}\label{ex:istartV} Let $V$ be a real vector space. We recall from~\S\ref{appen:vb} that $\underline{V}$ is the sheaf whose $S$-points are sections of the trivial bundle with fiber~$V$,
$$
\underline{V}(S)=C^\infty(S)\otimes_\R V.
$$ 
We observe that $i^*\underline{V}\simeq V_\C$, since 
$$
(i^*\underline{V})(S)=C^\infty(S)\otimes_\R V \simeq C^\infty(S)\otimes_\C V_\C\simeq C^\infty(S;V_\C)={\sf Mfld}(S,V_\C).
$$ 
Then $V_\C\simeq i_!V_\C\simeq i_!i^*\underline{V}\to \underline{V}$ is the surjection $V_\C\to\underline{V}$ discussed in~\S\ref{appen:vb}. 
\end{ex}

\begin{lem} \label{lem:warmupopen}Given a supermanifold $T$, its category of open supermanifolds $\mathrm{Open}(T)$ is equivalent to the category of open subsets of its reduced manifold~$\mathrm{Open}(T_{\mathrm{red}})$ with the equivalence $\mathrm{Open}(T)\to \mathrm{Open}(T_{\mathrm{red}})$ implemented by the reduced manifold functor~$U \mapsto U_{\mathrm{red}}$. \end{lem}

\begin{proof} First we observe that the functor is well-defined: if $U\subset T$ is open, $U_{\rm red}\subset T_{\rm red}$ is open. We define an inverse functor as follows. For an open subset $V\subset T_{\mathrm{red}}$, consider the presheaf~$U$ on supermanifolds whose $S$-points are the subset $U(S)\subset T(S)$ given by the condition that $S_{\mathrm{red}} \to T_{\mathrm{red}}$ factors through $V$. This defines a subsheaf by definition. Checking that~$U$ is representable and defines an open subsheaf can be done locally in~$T$. This allows us to assume~$T=\R^{n|m}$. Then we observe that~$U$ is represented by the supermanifold gotten by restricting the structure sheaf of $\R^{n|m}$ to $V\subset \R^n$. This is an open sub supermanifold $U\subset \R^{n|m}$ whose reduced manifold is~$V$. This construction is natural in the open subset $V\subset T_{\rm red}$ and determines the required inverse functor.\end{proof}

For a sheaf $\X\in {\sf Sh}_{\sf SMfld}$ and $\Y\in {\sf Sh}_{\sf Mfld}$, we recall the categories $\mathrm{Open}(\X)$ and $\mathrm{Open}(\Y)$ of open subsheaves; see Definition~\ref{defn:opensub}. 

\begin{lem} \label{lem:reducedtopology} Given $\X \in {\sf Sh}_{\sf SMfld}$, $\mathrm{Open}(\X) \stackrel{i^*}{\to} \mathrm{Open}(i^*\X)$ is a natural isomorphism. \end{lem}

\begin{proof} We first note that the functor $i^*$ is well-defined. 
This follows from the fact that~$i^*$ sends subsheaves to subsheaves, preserves pullbacks (being a right adjoint), and the observation in Lemma~\ref{lem:warmupopen} that the reduced manifold functor sends open sub supermanifolds to open submanifolds. 
We construct an inverse to $i^*$ analogously to the inverse functor in the proof of Lemma~\ref{lem:warmupopen}: given an open subsheaf $V \subset i^* \X$, define a subsheaf $U \subset \X$ whose $S$-points are maps $S\to \X$ such that $i^*S \to i^*\X$ factors through $V$. This condition is natural and so defines a subsheaf $U\subset \X$ for which $i^* U = V$ as subsheaves of $i^*\X$. To check that $U\subset \X$ is an open subsheaf, pull it back along a map $S\to \X$ and apply the argument from the proof of Lemma~\ref{lem:warmupopen} to show that the pullback is an open sub supermanifold of $S$. This determines the claimed inverse functor. \end{proof}

\begin{rmk}
We recall the terminology \emph{generalized manifold} and \emph{generalized supermanifold} for a sheaf on the site of manifolds and a sheaf on the site of supermanifolds, respectively. We use this terminology below especially when considering categories of sheaves on generalized (super) manifolds. 
\end{rmk}

\begin{defn}\label{defn:sheafrestrict} Let $\Y$ be a generalized supermanifold and $i^*\Y$ its restriction to manifolds. Given a sheaf $\mathcal{A}$ on $\Y$, define a sheaf $i^*\mathcal{A}$ on $i^*\Y$ whose value on an open subsheaf $U\subset i^*\Y$ is
\beq
(i^*\mathcal{A})(U):= \mathcal{A}((i^*)^{-1}U)\label{eq:istarcolim}
\eeq
using the inverse to the equivalence from Lemma~\ref{lem:reducedtopology}. 
\end{defn}

\begin{ex} Consider the sheaf of functions $C^\infty_N$ on a supermanifold $N$. View $N$ as a representable sheaf. Then $i^*N \simeq N_\red$ is the reduced manifold of~$N$, and~$i^*C^\infty_N$ is the sheaf on $N_\red$ defining the supermanifold $N$, i.e., the pair $(N_\red,i^*C^\infty_N)$ is the smooth manifold and sheaf of rings defining the supermanifold $N$ as a locally ringed space. 
\end{ex}

\subsection{Comparing moduli stacks}\label{sec:comparisonstack}

Our next goal is to construct a map 
$$
\iota\colon \Bun_G(\EE)\dashrightarrow i^*\Mst_G 
$$
so that the $Q$-cohomology sheaf from Definition~\ref{defn:Qcoh} can be compared with the de~Rham model for equivariant elliptic cohomology. The moduli stack of $G$-bundles on elliptic curves is defined as
$$
\Bun_G(\EE) \simeq [\HH \times \mc{C}^2[G]\sq \SL_2(\Z) \times \pi_0(G)],\qquad \mathcal{C}^2[G]=\mathcal{C}^2(G)/G_0.
$$ 
See~\eqref{eq:BunG} for details.

\begin{lem}\label{lem:iota}
There is a functor of generalized super Lie groupoids,
$$
(\Lat\times\Hom(\Z^2,G))\sq (\C^\times\times G)\to \big(\Lat\times \Fun(\R^{0|1}\sq \Z^2,\pt\nsq G)_\wz \big)\sq (\C^\times \times G).
$$
After taking coarse quotients by $\C^\times\times G_0$ and $\C^\times\times G_0$ and restricting along~\eqref{eq:embed}, this yields a $\pi_0(G)\times \SL_2(\Z)$-equivariant morphism of generalized manifolds $\HH\times\mathcal{C}^2[G]\to i^*\Mstil_G$ determining a morphism of stacks, 
\beq
\iota\colon \Bun_G(\EE)\to [i^*\Mstil_G\sq \pi_0(G)\times \SL_2(\Z)]\simeq i^*\Mst_G.\label{eq:iota}
\eeq
\end{lem}
\bp
Consider the map 
\beq
\Lat\times\Hom(\Z^2,G)&\to& \Lat \times \Fun(\R^{0|1}\sq \Z^2,\pt\nsq G)_\wz \label{eq:objmap} \\
&&\subset \Lat\times \Hom(\Z^2,G)\times \Omega^1(-\times \R^{0|1};\mf{g})_\wz\nonumber
\eeq
determined by the inclusion along $0\colon \pt \to \Omega^1(-\times \R^{0|1};\mf{g})_\wz$ for the 1-form $0\in \Omega^1(\pt\times \R^{0|1};\fg)_\wz$. This inclusion is $\C^\times\times G$-equivariant, and so determines a map of quotient super Lie groupoids. 

Using Lemma~\ref{lem:quotientcom} and that $i^*$ preserves products, we observe the isomorphism of generalized manifolds
$$
\HH\times\mathcal{C}^2[G]\simeq \Lat\cq \C^\times\times i^*\Hom(\Z^2,G)\cq G_0\simeq i^*(\Lat\times\Hom(\Z^2,G)\cq (\C^\times\times G_0)).
$$ 
Hence by the definition~\eqref{eq:coarsequotient} of $\Ftil_0(\pt\ncq G)$, the isomorphism~\eqref{eq:WZinvt}, and the definition of $\Mstil_G$ in~\eqref{eq:Mtildef}, applying coarse quotients to the map of generalized super Lie groupoids determined by the $G\times \C^\times$-equivariant map~\eqref{eq:objmap} and applying $i^*$ produces a $\pi_0(G)\times \SL_2(\Z)$-equivariant morphism of generalized manifolds
$$
\HH\times\mathcal{C}^2[G]\to i^*\Mstil_G.
$$
This completes the proof. 
\ep

\subsection{The topology of open subsheaves in~$\Mst_G$}\label{sec:topology}

By the definition of inverse image sheaf,
\beq
(\iota^{-1}i^*\Ch^\bullet(\F_0(M\ncq G)))(V\times U_h^\epsilon)&\simeq &\colim_{U\supset \iota(V\times U_h^\epsilon)} (i^*\Ch^\bullet(\F_0(M\ncq G)))(U)\label{eq:colim}
\eeq
where the colimit is indexed by open subsheaves $U \subset  i^*\Mstil_G$ such that $U$ contains the image of $V\times U_h^\epsilon$ under $\iota\colon \HH\times \mathcal{C}^2[G]\to i^*\Mstil_G$. We remind that the colimits are taken in complexes of sheaves, i.e., are strict colimits rather than homotopy colimits of cdgas. 

To compute~\eqref{eq:colim} we need a more explicit description of the open subsheaves $U\subset i^*\Mstil_G$ in the colimit. We have
$$
\Mstil_G=(\Lat\times \Fun(\R^{0|1}\sq \Z^2,\pt\nsq G)_\wz)/(G_0\times \C^\times)\simeq \HH\times (\Fun(\R^{0|1}\sq \Z^2,\pt\nsq G)_\wz)/G_0.
$$

\begin{lem} There is an equivalence of categories between open subsheaves of $\Mstil_G$ and $G_0\times\C^\times$-invariant open subsheaves of $\Lat\times(\mathcal{C}^2(G)\times \fg_\C)^{\Z^2}$. 
\end{lem}

\bp By Lemma~\ref{lem:reducedtopology}, Lemma~\ref{lem:orbitopen} and Proposition~\ref{prop:opensofquot}, open subsheaves of $\Mstil_G$ are equivalent to $G_0\times \C^\times $-invariant open subsheaves of 
\beq
i^*(\Lat\times \Fun(\R^{0|1}\sq \Z^2,\pt\nsq G)_\wz) &\simeq& \Lat\times i^*\big((\Hom(\Z^2,G)\times \Omega^1(-\times \R^{0|1};\fg)_\wz)^{\Z^2}\big)\nonumber\\
&\simeq& \Lat\times \big(i^*(\Hom(\Z^2,G)\times \underline{\fg}\times \Omega^1(-;\fg))\big)^{\Z^2}\nonumber\\
&\simeq&\Lat\times (\mathcal{C}^2(G)\times \fg_\C\times \Omega^1(-;\fg))^{\Z^2},\label{eq:istarcomp}
\eeq
where the description of functors as a $\Z^2$-fixed subsheaf is in Remark~\ref{rmk:invt}. 
The above isomorphisms use that $i^*$ preserves limits (being a right adjoint) and Example~\ref{ex:istartV}. Consider the $G_0\times \C^\times$-equivariant map
\beq
p\colon \Lat\times (\mathcal{C}^2(G)\times \fg_\C\times \Omega^1(-;\fg)\times \mathcal{C}^2(G))^{\Z^2}\to \Lat\times (\fg_\C\times\mathcal{C}^2(G))^{\Z^2}\label{eq:mapcalledp}
\eeq
determined by the obvious projection. Hence, a $G_0\times \C^\times$-invariant open subsheaf $U'\subset \Lat\times (\fg_\C\times\mathcal{C}^2(G))^{\Z^2}$ determines the $G_0\times \C^\times$-invariant open subsheaf $p^*U'\subset \Lat\times (\mathcal{C}^2(G)\times \fg_\C\times \Omega^1(-;\fg)\times \mathcal{C}^2(G))^{\Z^2}$. Lemma~\ref{lem:technicalomega} below shows that $p^*$ induces an equivalence of categories of open subsheaves. \ep

\begin{lem}\label{lem:technicalomega} Let $Z$ be a concrete sheaf on the site of manifolds and $p\colon \tilde{Z} \to Z$ a morphism of sheaves such that for all $z\colon \pt\to Z$ we have $z^* \tilde{Z} \simeq \Omega^1(-; V_z)$ for a finite-dimensional real vector space $V_z$. Suppose moreover that we have a section $0\colon Z \to \tilde{Z}$ of $p$. Then for any open subsheaf $U \subset \tilde{Z}$, we have $U = p^* 0^* U$. In particular, all open subsheaves of $\tilde{Z}$ pull back from open subsheaves of $Z$. \end{lem}

\begin{rmk}
We observe that the map $p$ in~\eqref{eq:mapcalledp} has a section determined by inclusion along $0\colon \pt\to \Omega^1(-;\fg)$, the zero differential form; this is why we use the notation $0$ for the section in this lemma.
\end{rmk}

\begin{proof} 
We observe that the open subsheaf $0^*U\subset Z$ of the concrete sheaf is determined by the subset
$$0^*U(\pt) := \{ z \colon \pt \to Z \mid z^*U\neq \emptyset\}\subset Z(\pt)$$ 
where $\emptyset$ is the empty sheaf. Indeed, the pullback $z^* U$ is an open subsheaf of $\Omega^1(-; V_z)$. By Proposition~\ref{prop:opensofOmega}, this open subsheaf is either empty or the entire sheaf. It remains to show that $U \simeq p^* 0^*U$.

This is equivalent to showing that for all manifolds $N$ and $N$-points $N \stackrel{f}{\to} \tilde{Z}$, we have $f^* U \simeq (p \circ f)^* 0^*U$. Consider the commutative square 
\beq
\begin{tikzpicture}[baseline=(basepoint)];
\node (A) at (-1.5,0) {$\tilde{N}$};
\node (B) at (1,0) {$\tilde{Z}$};
\node (D) at (1,-1.25) {$Z$};
\node (C) at (-1.5,-1.25) {$N$};
\draw[->] (A) to node [above] {$\tilde{f}$} (B);
\draw[->,bend left] (D) to node [left] {$0$} (B);
\draw[->,bend left] (C) to node [left] {$0_N$} (A);
\draw[->] (A) to node [right] {$p_N$} (C);
\draw[->] (B) to node [right] {$p$} (D);
\draw[->] (C) to node [below] {$p \circ f$} (D);
\path (0,-.75) coordinate (basepoint);
\end{tikzpicture}\nonumber 
\eeq
where $\tilde{N}\simeq N \times_Z \tilde{Z}$ is the pullback and $0_N\colon N \to \tilde{N}$ is the pullback of the section~$0\colon Z\to \tilde{Z}$. Define $U_{\tilde{N}} := \tilde{f}^*U\subset \tilde{N}$ and observe
$$
U_N := 0_N^* U_{\tilde{N}} \simeq (\tilde{f} \circ 0)^* U \simeq (0 \circ (p \circ f))^* U \simeq (p \circ f)^* 0^*U.
$$ 
An alternative characterization of $U_N$ is 
\beq
U_N \simeq \{ n\colon \pt\to N \mid n^*U_{\tilde{N}}\neq \emptyset \}.\label{anothercharacterization}
\eeq
As before, this equivalence follows from Proposition~\ref{prop:opensofOmega}. Next, consider the section $s\colon  N \to \tilde{N}$ determined by $N \stackrel{\text{id}}{\to} N, N \stackrel{f}{\to} \tilde{Z}$. The characterization~\eqref{anothercharacterization} of $U_N$ makes it equally clear that $U_N \simeq s^* U_{\tilde{N}}$, but $\tilde{f} \circ s \simeq f$, so $s^* U_{\tilde{N}} \simeq s^* \tilde{f}^* U \simeq f^* U$ and hence we have shown $f^* U \simeq (p \circ f)^* 0^*U$, as desired.
 \end{proof}

\subsection{The proof of Theorem~\ref{thm:comparison}}\label{sec:thm:comparison}

As reviewed in~\S\ref{appen:ell}, the sheaf $\dEll_G^\bullet(M)$ is defined relative to a basis of open subsheaves of $\HH\times\mathcal{C}^2[G]$ given by $V\times U_h^\epsilon$ where $V\subset \HH$ is an open subset and $U_h^\epsilon\subset \mathcal{C}^2[G]$ is an open subsheaf characterized by~\eqref{eq:Utildefn}. From Definition~\ref{defn:ellcocycle}, the value of $\dEll_G^\bullet(M)$ on $V\times U_h^\epsilon$ is a collection of compatible differential forms in $\Omega_{G^{h'}_0}(M^{h'};\mathcal{O}(V)[\beta, \beta^{-1}])$ for each $h'\in U_h^\epsilon$. This gives maps ${\rm res}_{h'}$ in the diagram,
\beq
\begin{tikzpicture}[baseline=(basepoint)];
\node (A) at (0,0) {$\iota^{-1}i^*\Ch^\bullet(\F_0(M\ncq G))(V\times U_h^\epsilon)$};
\node (B) at (6,0) {$\dEll^\bullet(M\nsq G)(V\times U_h^\epsilon) $};
\node (C) at (6,-1.5) {$\Omega^{\bullet}_{G^{h'}_0}(M^{h'};\mathcal{O}(V)[\beta, \beta^{-1}]).$};
\draw[->,dashed] (A) to node [above] {$\kappa_{V\times U_h^\epsilon}$}  (B);
\draw[->] (B) to node [right] {${\rm res}_{h'}$} (C);
\draw[->,dashed] (A) to node [below] {$\kappa_{h'}$} (C);
\path (0,-.75) coordinate (basepoint);
\end{tikzpicture}\label{eq:plan}
\eeq
We will construct a morphism of cdgas $\kappa_{V\times U_h^\epsilon}$ by constructing a compatible family of maps~$\kappa_{h'}$.

\begin{lem}
For each $h'\in U_h^\epsilon$, and $V\subset \HH$, there is a map of cdgas
\beq
\kappa_{h'}\colon \iota^{-1}i^*\Ch^\bullet(\cL_0(M\ncq G))(V\times U_h^\epsilon)\to  \Omega^{\bullet}_{G_0^{h'}} (M^{h'};\mathcal{O}(V)[\beta, \beta^{-1}])\label{eq:kappah}
\eeq
determined by restriction to the $h'$-twisted sector and the isomorphism~\eqref{eq:isovs}. The target of~\eqref{eq:kappah} is the version of the Cartan complex from Definition~\ref{defn:valuesinasheaf}.
\end{lem} 

\bp
The source of~\eqref{eq:kappah} is a colimit over open subsheaves~$U\subset \Mstil_G$ where $i^*U$ contains $\iota(V\times U_h^\epsilon)$. Consider the pullback squares depending on an open subsheaf~$U\subset \Mstil_G$ and $h'$,
\beq
\begin{tikzpicture}[baseline=(basepoint)];
\node (A) at (0,0) {$U$};
\node (B) at (6,0) {$\Mstil_G\simeq \HH\times (\Fun(\R^{0|1}\sq \Z^2,\pt\nsq G)_\wz)/G_0$};
\node (BB) at (9.2,0) {$\null$};
\node (C) at (6,1.5) {$\Ftil_0(M\ncq G)\simeq \Lat\times (\Fun(\R^{0|1}\sq \Z^2,M\nsq G)_\wz)/G_0$};
\node (D) at (6,3) {$\Lat\times (\Fun(\R^{0|1},M^{h'}\nsq G^{h'}_0)_\wz)/G^{h'}_0$};
\node (DD) at (9,3) {$\null$};
\node (E) at (0,1.5) {$\pi'^{-1}U$};
\node (F) at (0,3) {$\pi_{h'}^{-1}U$};
\draw[->] (A) to  (B);
\draw[->] (C) to node [right] {$\pi'$} (B);
\draw[->,right hook-latex] (D) to (C);
\draw[->] (E) to (A);
\draw[->] (E) to (C);
\draw[->] (F) to (D);
\draw[->] (F) to (E);
\draw[->,bend left=60] (DD) to node [right] {$\pi_{h'}$} (BB);
\path (0,1.5) coordinate (basepoint);
\end{tikzpicture}\label{eq:plan2}
\eeq
where $\pi'$ is the composition of $\pi\colon \Ftil_0(M\ncq G)\to \Ftil_0(\pt\ncq G)$ from~\eqref{defn:pi} with the $\C^\times$-quotient map $\Ftil_0(\pt\ncq G)\to \Ftil_0(\pt\ncq G)/\C^\times=\Mstil_G$, and $\pi_{h'}$ is the composition of the inclusion determined by Lemma~\ref{cor:jh} and $\pi'$. Then there is a restriction map with values in the open subsheaf $\pi^{-1}_{h'}U$ of the gauge-invariant $h'$-twisted sector
\beq
&&\Ch^\bullet(\F_0(M\ncq G))(U)\subset \mathcal{O}_{\F_0(M\ncq G)}(\pi'^{-1}U)\subset C^\infty(\pi'^{-1}U)\to C^\infty(\pi^{-1}_{h'}U)\label{eq:colimitrestrict}
\eeq
where the first inclusion is as the direct sum of $\C^\times$-weight spaces, the second is the inclusion of analytic functions into smooth functions, and the final arrow is restriction along the map $\pi_{h'}^{-1}U\to \pi'^{-1}U$ in diagram~\eqref{eq:plan2}. 

Next we specialize~\eqref{eq:colimitrestrict} to the case of an open subsheaf $U\subset \Mstil_G$ where $i^*U$ contains $\iota(V\times U_h^\epsilon)$. We have
$$
\Lat\times \Fun(\R^{0|1},M^{h'}\nsq G^{h'}_0)_\wz\simeq \Lat\times \underline{\fg}^{h'}\times \Omega^1(-;\fg^{h'})\times \Map(\R^{0|1},M^{h'}).
$$
Using Proposition~\ref{prop:opensofquot}, we identify $\pi_{h'}^{-1}U$ in~\eqref{eq:plan2} with a $G_0\times \C^\times$-invariant open subsheaf of the right hand side above. Furthermore, by inspection of~\eqref{eq:plan2} we observe that this open subsheaf has the form $W\times \Map(\R^{0|1},M^{h'})$ for 
$$
W\subset \Lat\times \underline{\fg}^{h'}\times \Omega^1(-;\fg^{h'}).
$$
By Lemmas~\ref{lem:reducedtopology} and~\ref{lem:technicalomega}, $W$ is determined by a $\C^\times$-invariant open subset of the representable sheaf~$\Lat\times \fg_\C^h$, where this open subset must contain $q^{-1}V\times \{0\} \subset \Lat\times \fg_\C^{h'}$ where $q\colon \Lat\to \Lat/\C^\times\simeq \HH$ is the $\C^\times$-quotient map. Setting $W'=W/\C^\times$, we may assume (by cofinality) that $W'\subset V\times \fg_\C^{h'}\subset \HH\times \fg_\C^{h'}$ is an open subset containing $V\times\{0\}$. Restriction along~\eqref{eq:colimitrestrict} and using the isomorphism~\eqref{eq:isovs} to identify the $\C^\times$-weight spaces with differential form data, we obtain a map
$$
\iota^{-1}i^*\Ch^\bullet(\cL_0(M\ncq G))(V\times U_h^\epsilon)\to  \colim_{W'\supset V\times \{0\}} \mathcal{O}(W';\Omega^{\bullet}_{G_0^{h'}} (M^{h'})[\beta, \beta^{-1}]).
$$
The statement of the lemma follows immediately by identifying this colimit with the definition of the right hand side of~\eqref{eq:kappah}, per Definition~\ref{defn:valuesinasheaf}. \ep

\begin{lem} 
The maps~\eqref{eq:kappah} assemble into a map of cdgas
\beq
\kappa_{V\times U_h^\epsilon} \colon \iota^{-1}i^*\Ch^\bullet(\cL_0(M\ncq G))(V\times U_h^\epsilon) \to \dEll_G(M)(V\times U_h^\epsilon).\label{eq:kappa1}
\eeq
\end{lem}
\bp
It remains to show that the maps~\eqref{eq:kappah} for each $h'\in U_h^\epsilon$ are compatible with the invariance and analytic conditions in Definition~\ref{defn:ellcocycle}. The invariance condition follows from the fact that the sheaf $\Ch^\bullet(\cL_0(M\ncq G))$ on $\Mstil_G$ consists of $G_0\times \C^\times$-invariant sections, which requires the values at conjugate twisted sectors to be compatible as computed in Proposition~\ref{prop:twisted}. This is the same as the invariance condition in Definition~\ref{defn:ellcocycle}. Finally, the compatibility condition for functions restricted to twisted sectors in Lemma~\ref{lem:deRhamanalytic} matches the analytic condition in Definition~\ref{defn:ellcocycle}. Hence, we obtain a map of cdgas as claimed. 
\ep

\begin{lem} 
The maps~\eqref{eq:kappa1} assemble into a morphism of sheaves of cdgas on $\Bun_G(\EE)$
\beq
\kappa \colon \iota^{-1}i^*\Ch^\bullet(\cL_0(M\ncq G)) \to \dEll_G(M).\label{eq:kappa}
\eeq
\end{lem}

\bp
By Corollary~\ref{cor:topofCG},  the open subsheaves $V\times U_h^\epsilon$ form a basis for the topology of $\HH\times\mathcal{C}^2[G]$. Hence, to show that the maps $\kappa_{V\times U_h^\epsilon}$ determine a morphism of sheaves on $\Bun_G(\EE)\simeq [\HH\times\mathcal{C}^2[G]\sq \SL_2(\Z)\times \pi_0(G)]$, we need to check: (1) compatibility for restrictions along inclusions $V\times U_h^\epsilon\subset V'\times (U_h^\epsilon)'$ and (2) equivariance for $\SL_2(\Z)\times \pi_0(G)$-action. 

Part (1) follows from how the map was constructed as a colimit of restrictions, and so is automatically compatible with a further restriction~$V\times U_h^\epsilon\subset V'\times (U_h^\epsilon)'$. 

For (2), we need to compute the action of $\SL_2(\Z)$ and $\pi_0(G)$ on twisted sectors. These are computed in Proposition~\ref{prop:twisted}, and agree with the actions in Definition~\ref{defn:ellcocycle}. 
\ep

\begin{thm}\label{lem:proofofmain}
The map $\kappa$ is an isomorphism of sheaves of cdgas. 
\end{thm}

\bp
For $h\in \mathcal{C}^2(G)$, we obtain a map 
\beq
j_h\colon [\HH\sq \Gamma] \to \Bun_G(\EE)\label{eq:jh}
\eeq
that on objects includes at $[h]\in \mathcal{C}^2[G]$, where $\Gamma<\SL_2(\Z)\times \pi_0(G)$ is the stabilizer of $[h]$ for the $\SL_2(\Z)\times \pi_0(G)$-action on $\mathcal{C}^2[G]$. We will show that the morphism of sheaves when pulled back to $j_h$ is an isomorphism for all $h$. This implies the claimed isomorphism of sheaves, e.g., because it implies the map induces an isomorphism on stalks.

In~\cite[Proposition~3.7]{BET0}, an isomorphism was constructed 
\beq
j_h^*\dEll_G^\bullet(M)(V)\simeq \Omega^{\bullet}_{G_0^h} (M^{h};\mathcal{O}(V)[\beta, \beta^{-1}]), \qquad V\subset \HH \label{eq:stalkcompute}
\eeq
for the $\Gamma$-equivariant structure on the right hand side inherited from Definition~\ref{defn:ellcocycle}. We will compute 
\beq
j_h^*\iota^{-1}i^*\Ch^\bullet(\cL_0(M\ncq G))=\colim_{i^*U\supset \iota(V\times \{[h]\}) } \Ch^\bullet(\cL_0(M\ncq G))(U)\label{eq:thecolimit}
\eeq
where the colimit is indexed by open subsheaves $U\subset \Mstil_G$ such that $i^*U\supset \iota (V\times\{ [h]\})$. By cofinality of diagrams computing the colimit, we may assume that for some $\epsilon>0$, 
$$
U=U\bigcap (\HH\times B_\epsilon^{\fg} \times \Omega^1(-\times\R^{0|1};\fg^h)_\wz)/G_0
$$
where $B_\epsilon^{\fg}\subset \fg\times \fg$ is an $\epsilon$-ball at $(0,0)$ and the intersection is taken in
\beq
&&\resizebox{.9\textwidth}{!}{$
(\HH\times\Omega^1(-\times \R^{0|1};\fg)_{\wz} \times G\times G)/G_0\supset (\HH\times\Fun(\R^{0|1}\sq \Z^2,\pt\nsq G)_\wz)/G_0 \simeq \Mstil_G$} \label{eq:itsabigsheaf}
\eeq
with $(\HH\times \Omega^1(-\times\R^{0|1};\fg^h)_\wz\times B_\epsilon^{\fg} )/G_0$ regarded as an open subsheaf using the map 
\beq
&& \fg \times \fg \stackrel{i_h}{\to} G\times G\supset \Hom(\Z^2,G) \qquad (X_1,X_2)\mapsto (h_1e^{X_1},h_2e^{X_2})\quad X_1,X_2\in \fg(S).\label{eq:wanttobeinj}
\eeq

We define subsheaves of~$U$ that will be of use in computing the right hand side of~\eqref{eq:thecolimit}. Let $T_h\subset G_0^h$ denote a choice of maximal torus for the identity component, let $\mf{t}_{\fg^h}$ denote the Lie algebra of $T_h$, and let $N_h$ denote the normalizer of $T_h\subset G_0^h$. Let $B_\epsilon^{\mf{t}_{\fg^h}}\subset \mf{t}_{\fg^h}\times \mf{t}_{\fg^h}$ and $B_\epsilon^{\fg^h}\subset \fg^h\times \fg^h$ denote open balls containing the origin, $(0,0)\in \mf{t}_{\fg^h}\times\mf{t}_{\fg^h}\subset \fg^h\times\fg^h$. Define (non-open) subsheaves
\beq
U_h&:=&U\bigcap (\HH\times \{h\}\times \Omega^1(-\times\R^{0|1};\mf{t}_{\fg^h})_\wz)/N_h\subset \Mstil_G\label{eq:Udefns1}\\
U_{\mf{t}_{\fg^h}}&:=&U\bigcap (\HH\times B_\epsilon^{\mf{t}_{\fg^h}}\times \Omega^1(-\times\R^{0|1};\mf{t}_{\fg^h})_\wz)/N_h\subset \Mstil_G\label{eq:Udefns2}\\
U_{\fg^h}&:=&U\bigcap (\HH\times B_\epsilon^{\fg^h} \times \Omega^1(-\times\R^{0|1};\fg^h)_\wz)/G_0^h\subset \Mstil_G\label{eq:Udefns3}
\eeq
where the intersections~\eqref{eq:Udefns1} and~\eqref{eq:Udefns2} are taken in the sheaf
$$
(\HH\times\Omega^1(-\times \R^{0|1};\fg)_{\wz} \times \Hom(\Z^2,G))/G_0\supset (\HH\times \Fun(\R^{0|1}\sq \Z^2,\pt\nsq G))/G_0
$$
while the intersection~\eqref{eq:Udefns3} is taken inside of the larger sheaf~\eqref{eq:itsabigsheaf}.
The intersections~\eqref{eq:Udefns2} and~\eqref{eq:Udefns3} use the restriction of $i_h$ to $\mf{t}_{\fg^h}\times \mf{t}_{\fg^h}$ and $\fg^h\times \fg^h$. This map is an injection in a neighborhood of $(0,0)$, and so for sufficiently small $\epsilon$, $i_h$ does indeed define a subsheaf. 
This completes the construction of the subsheaves~\eqref{eq:Udefns1}-\eqref{eq:Udefns3}. 

There are the evident inclusions 
$$
U_h\hookrightarrow U_{\mf{t}_{\fg^h}}\hookrightarrow U_{\fg^h}\hookrightarrow U\subset \Mstil_G,
$$
and we shall consider the restriction of analytic functions to these subsheaves. By Proposition~\ref{prop:twisted}, the restriction to $U_h$ can be identified with equivariant differential forms,
$$
\Ch^\bullet(\cL_0(M\ncq G))(U)\to \Ch^\bullet(\cL_0(M\ncq G))(U_h)\simeq \mathcal{O}(\widetilde{U}_h;\Omega^\bullet(M)[\beta,\beta^{-1}])^{N_h}
$$
where $\widetilde{U}_h$ is the open subset of $\HH\times(\mf{t}_{\fg^h})_\C$ containing $V\times \{0\}$ gotten by taking the preimage of $U_h$ along $\HH\times(\mf{t}_{\fg^h})_\C\to \HH\times \Omega^1(-;\mf{t}_{\fg^h})/N_h\supset U_h$. By Lemma~\ref{lem:deRhamanalytic}, the further restriction 
\beq
\Ch^\bullet(\cL_0(M\ncq G))(U_{\mf{t}_{\fg^h}})\to \Ch^\bullet(\cL_0(M\ncq G))(U_h)\label{eq:analyticisomrestriction}
\eeq
is an isomorphism for $\epsilon$ sufficiently small in~\eqref{eq:Udefns2}; equivalently, this follows from applying Definition~\ref{defn:deform} to the diagram~\eqref{eq:analyticdiagram} for $\mf{a}= \mf{t}_{\fg^h}$. From Proposition~\ref{prop:torus}, the restriction 
$$
\Ch^\bullet(\cL_0(M\ncq G))(U_{\fg^h})\to \Ch^\bullet(\cL_0(M\ncq G))(U_{\mf{t}_{\fg^h}})
$$
is an isomorphism. Proposition~\ref{prop:ellgrp} shows that the restriction 
$$
\Ch^\bullet(\cL_0(M\ncq G))(U)\to \Ch^\bullet(\cL_0(M\ncq G))(U_{\fg^h})
$$
is also an isomorphism. Altogether, we've shown that the restriction map
$$
\Ch^\bullet(\cL_0(M\ncq G))(U)\to \Ch^\bullet(\cL_0(M\ncq G))(U_h)
$$
is an isomorphism. Hence, the relevant colimit can be computed as 
\beq
j_h^*\iota^{-1}i^*\Ch^\bullet(\cL_0(M\ncq G))&=&\colim_{U\subset \Mstil_G} \Ch^\bullet(\cL_0(M\ncq G))(U)\nonumber\\
&\simeq&\colim_{U_h} \Ch^\bullet(\cL_0(M\ncq G))(U_h)\nonumber\\
&\simeq& \colim_{\widetilde{U}_h\subset \HH\times(\mf{t}_{\fg^h})_\C} \mathcal{O}(\widetilde{U}_h;\Omega^\bullet(M)[\beta,\beta^{-1}])^{N_h}\nonumber\\
&\simeq & (\Omega_{T^h}(M;\mathcal{O}(V)[\beta,\beta^{-1}]))^{N_h}\nonumber\\
&\simeq & \Omega_{G_0^h}(M;\mathcal{O}(V)[\beta,\beta^{-1}])\nonumber
\eeq
where the first colimit is over $U\subset \Mstil_G$ so that $i^*U$ contains $\iota(V\times \{[h]\})$, the second colimit is over $U_h\subset (\HH\times \{h\}\times \Omega^1(-\times\R^{0|1};\mf{t}_{\fg^h})_\wz)/N_h$ containing $V\times \{0\}$ 
and the third colimit is over $\widetilde{U}_h\subset \HH\times (\mf{t}_{\fg^h})_\C$ containing $V\times \{0\}$. Together with~\eqref{eq:stalkcompute}, we find that the map~$\kappa$ from~\eqref{eq:kappa} induces an isomorphism. This proves the theorem. 
\ep

\section{Free fermions and the equivariant elliptic Euler class}\label{sec:freefer}

Building on the work of~\cite{Quillendet,Wittenanomaly,BismutFreed1,BismutFreed2}, Freed~\cite{Freed_Det_Torsion,Freed_Det} explained how fermionic path integrals in quantum field theory can be given mathematically rigorous meaning in terms of determinant line bundles. In~\S\ref{sec:detreview}-\S\ref{sec:anomalies} we review this construction and compare it with the Looijenga line bundles over $\Bun_G(\EE)$. Sections~\S\ref{sec:superfamofops}-\S\ref{sec:sdetline} extend the determinant line construction, associating to any complex representation $\rho\colon G\to U(n)$ an analytic line bundle with section over the supermoduli space $\F_0(\pt\ncq G)$. We show that this constructs equivariant elliptic Euler classes as partition functions of free fermion theories. We show that the obstruction to trivializing the determinant line bundle (up to a power of the Hodge bundle) agrees with the string obstruction from equivariant elliptic cohomology. Hence, anomalies for the free fermions correspond to the obstruction to being oriented in elliptic cohomology. 

\subsection{Review of determinant line bundles on $\Bun_{U(1)}(\EE)$}\label{sec:detreview}

We start with the determinant line associated to a Dirac operator on an elliptic curve twisted by a flat line bundle; we refer to~\cite[\S4]{Freed_Det} for details. The constructions in this subsection take place in the category of complex manifolds and stacks on the site of complex manifolds. Consider the $\HH$-family~$\widetilde{\EE}$ given by the quotient,
$$
\widetilde{\EE}:=(\HH\times \C)/\Z^2\to \HH,\qquad (m,n)\cdot (\tau,z)=(\tau,z+m\tau+n)
$$
for the $\Z^2$-action on $(\tau,z)\in \HH\times \C$ as indicated. The fiber of this family at~$\tau\in \HH$ is the complex elliptic curve~$\widetilde{\EE}_\tau=\C/\tau\Z\oplus \Z$. We endow this family with the fiberwise periodic-periodic (or non-bounding) spin structure, wherein the spinor bundle $\bS$ at a fiber~$\tau\in \HH$ is a trivial holomorphic square root of the canonical bundle, $\bS\otimes \bS\simeq \Omega^{0,1}_{E_\tau}$. Given a degree zero holomorphic line bundle~$L$ on $E_\tau$, the Dolbeault operator~$\bar\partial_L$ of~$L$ can be identified with the $L$-twisted (chiral) Dirac operator on~$E_\tau$. Varying $\tau$ and $L$ gives a family of operators we also denote by $\bar\partial_L$ over~$\Pic^0(\widetilde{\EE})$, the moduli of degree zero holomorphic line bundles over~$\widetilde{\EE}$. This moduli space is canonically isomorphic to the dual of $\widetilde{\EE}$
$$
\Pic^0(\widetilde{\EE})\simeq \widetilde{\EE}^\vee\simeq (\HH\times \C)/\Z^2,\qquad (m,n)\cdot (\tau,z)=(\tau,z+m-\tau n),
$$
where $\widetilde{\EE}^\vee$ is defined as the $\Z^2$-quotient for the action indicated above. The isomorphism $\widetilde{\EE}\stackrel{\sim}{\to}\Pic^0(\widetilde{\EE})$ is determined by sending~$(\tau,z)\in \HH\times \C$ to the flat complex line bundle corresponding to the character of $\pi_1(\widetilde{\EE}_\tau)$
$$
m\tau +n\mapsto e^{2\pi i (mu+nv)},\qquad z=u-\tau v,\qquad m\tau+n\in \pi_1(\EE_\tau).
$$
Let $\MP_2(\Z)$ be the metaplectic double cover of $\SL_2(\Z)$. The $\SL_2(\Z)$-action on $\HH$ by fractional linear transformations lifts to an $\MP_2(\Z)$-action on elliptic curves with periodic-periodic spin structure, giving moduli stacks over~$\Mell$,
$$
[\HH\sq \MP_2(\Z)],\qquad [\widetilde{\EE}\sq \MP_2(\Z)], \qquad [\widetilde{\EE}^\vee\sq \MP_2(\Z)], 
$$
of elliptic curves with periodic-periodic spin structure, the universal elliptic curve with periodic-periodic spin structure, and the universal dual elliptic curve with spin structure, respectively. Following the discussion above, we also view $[\widetilde{\EE}^\vee\sq \MP_2(\Z)]$ as the moduli stack of elliptic curves with periodic-periodic spin structure and a degree zero holomorphic line bundle. 

From~\cite{Quillen} and~\cite[\S4]{Freed_Det}, there is a determinant line $\Det(\bar\partial_L)\to \widetilde{\EE}^\vee$ with metric and connection associated with the family of operators $\bar\partial_L$ defined above. Since the $\MP_2(\Z)$-action preserves the geometric data, $\Det(\bar\partial_L)$ has an $\MP_2(\Z)$-equivariant structure for which the metric and connection are preserved. Furthermore, there is an $\MP_2(\Z)$-invariant holomorphic section~$\det(\bar\partial_L)$. Pulling back along $p\colon \HH\times \C\to \widetilde{\EE}^\vee$, there is a (non-holomorphic) trivialization of $\Det(\bar\partial_L)$ where~$\det(\bar\partial_L)$ can be expressed as the function on $\HH\times \C$ given by the formula~\cite[Proposition~4.10]{Freed_Det},
\beq
&&\resizebox{.9\textwidth}{!}{$p^*\det(\bar\partial_L)=q^{(6v^2+1)/12}e^{-\pi i u}(e^{\pi i z}-e^{-\pi i z})\prod_{n=1}^\infty(1-q^ne^{2\pi i z})(1-q^ne^{-2\pi i z})$},\label{eq:Freedtriv}
\eeq
for $q=e^{2\pi i \tau}$, and $z=u-\tau v$ for $u,v\in \R$. 
The right hand side of~\eqref{eq:Freedtriv} closely resembles the Weierstrass $\sigma$-function $\sigma(\tau,z)\in \mathcal{O}(\HH\times \C)$ defined in~\eqref{eq:Weierstrass}. We recall that $\sigma(\tau,z)$ transforms under the action of $\MP_2(\Z)\ltimes \Z^2$ by nonvanishing holomorphic functions on~$\HH\times \C$. The action factors through $\SL_2(\Z)\ltimes \Z^2$, and is given by
\begin{eqnarray} 
\sigma\left(\frac{a\tau+b}{c\tau+d}, \frac{z}{c\tau+d}\right) &=&(c\tau+d)^{-1}e^{\pi i \frac{cz^2}{c\tau+d}} \sigma(\tau, z) \label{eq:sigmasign1} \\ 
\sigma(\tau, z + m\tau+n) &=&(-1)^{m+n} e^{-\pi i (m^2\tau+2mz)} \sigma(\tau, z). \label{eq:sigmasign} 
\end{eqnarray} 
Hence, $\sigma(\tau,z)$ a Jacobi form of index $1/2$ and weight~$-1$.
 

\begin{defn}\label{defn:A}
Let $\mathcal{A}$ be the holomorphic line bundle on $[\widetilde{\EE}^\vee \sq \MP_2(\Z)]$ determined by the transformation properties~\eqref{eq:sigmasign1} and \eqref{eq:sigmasign}, i.e., $\mathcal{A}$ is the holomorphic line on $[\widetilde{\EE}^\vee\sq \MP_2(\Z)]\simeq [(\HH\times \C)\sq (\MP_2(\Z)\ltimes \Z^2)]$ for which $\sigma(\tau,z)$ determines a section. 
\end{defn}
In an abuse of notation, we will use $\sigma(\tau,z)$ to denote the function~\eqref{eq:Weierstrass} and $\sigma$ to denote the section of $\mathcal{A}$. The following can be extracted from~\cite[\S3]{Theta}. 
\begin{lem} \label{lem:detLoo}
There is an isomorphism of holomorphic line bundles on $[\widetilde{\EE}^\vee\sq\MP_2(\Z)]$ 
\beq
\Det(\bar\partial_L)\simeq \mathcal{A}\label{eq:detlooiso}
\eeq
that sends the section $\det(\bar\partial_L)$ to the section $\sigma$ of $\mathcal{A}$ determined by the Weierstrass $\sigma$-function~\eqref{eq:Weierstrass}.
\end{lem}

\bp
Consider the trivialization of $\Det(\bar\partial_L)$ on $\HH\times \C$ where the determinant section is given by~\eqref{eq:Freedtriv}. Then take the trivialization of~$\mathcal{A}$ in which the section~$\sigma$ is the function $\sigma(\tau,z)\in \mathcal{O}(\HH\times \C)$. The ratio of sections in these chosen trivializations is 
$$
p^*\det(\bar\partial_L)/\sigma=q^{(6v^2+1)/12}e^{-\pi i u}\eta(q)^2=q^{(3v^2+1)/6} e^{-\pi i u}\prod_{n>0} (1-q^n)^2\in C^\infty(\HH\times \C)
$$ 
where $\eta(q)\in \mathcal{O}(\HH)$ is the Dedekind $\eta$-function. By inspection, this ratio is a nowhere vanishing function on $\HH \times \C$. Therefore it defines a nowhere vanishing section of $\Det(\bar\partial_L) \otimes \mc{A}^{-1}$ on $[\widetilde{\EE}^{\vee}\sq\MP_2(\Z)]$. We conclude that $\Det(\bar\partial_L) \simeq \mc{A}$ on this stack. Moreover, under the chosen isomorphism $\Det(\bar\partial_L) \stackrel{\sim}{\to} \mc{A}$, we have by construction that $\det(\bar\partial_L)$ maps to $\sigma$. 
\ep

\begin{rmk}
Since the transformation properties of $\sigma(\tau,z)$ are determined by the $\SL_2(\Z)\ltimes \Z^2$-action on $\HH\times \C$, we observe that $\mathcal{A}$ pulls back along $[\widetilde{\EE}^\vee\sq \MP_2(\Z)]\to [\widetilde{\EE}^\vee\sq \SL_2(\Z)]\simeq \EE^\vee$. From~\eqref{eq:detlooiso} we conclude that $\Det(\bar\partial_L)$ also pulls back from $\EE^\vee$. This observation can also be understood from the construction of $\Det(\bar\partial_L)$ in~\cite{Quillen} that does not depend on a spin structure; see also~\cite[\S2]{Freed_Det}. 
\end{rmk}

There is a completely analogous story when we consider the Dirac operator twisted by $n$-tuples of flat line bundles. This yields a determinant line bundle, 
$$
\Det(\bar\partial_{L_1\oplus\cdots \oplus L_n})\to \widetilde{\EE}^\vee\times_\HH\cdots \times_\HH \widetilde{\EE}^\vee\simeq (\HH\times \C^n)/\Z^{2n}
$$
over the $n$-fold iterated fibered product that again has a holomorphic section $\det(\bar\partial_{L_1\oplus\cdots \oplus L_n})$. Replaying the above analysis gives the following: 

\begin{lem} \label{lem:detLoon}
There is an isomorphism of  holomorphic line bundles on $[(\HH\times \C^n)/\Z^{2n}\sq\MP_2(\Z)]$ 
\beq
\Det(\bar\partial_{L_1\oplus\cdots \oplus L_n})\simeq \mathcal{A}^{\boxtimes n}.\label{eq:boxiso}
\eeq
This isomorphism sends the section $\det(\bar\partial_{L_1\oplus\cdots \oplus L_n})$ to the section of $\mathcal{A}^{\boxtimes n}$ determined by the product of Weierstrass $\sigma$-functions,  $\prod_{i=1}^n \sigma(\tau,z_i)\in \mathcal{O}(\mb{H}\times \C^n)$. 
\end{lem}

We introduce the abbreviated notation 
\beq
\sigma(\tau,  \mathbf{z}) := \prod_{i=1}^n \sigma(\tau,  z_i)\in \mathcal{O}(\HH\times \C^n),\qquad {\bf z}\in \C^n.\label{eq:sigmatau}
\eeq

\subsection{Determinant and Looijenga line bundles on $\Bun_G(\EE)$}

For the remainder of the section we restrict to the case that $G$ is connected with torsion-free fundamental group. This implies that any pair of commuting elements can be simultaneously conjugated into a maximal torus~\cite{Borel}. In particular, we have 
$$
\widetilde{\Bun}_G(\EE)\simeq \HH\times (T\times T)/W,\qquad \Bun_G(\EE)\simeq [\widetilde{\Bun}_G(\EE)\sq \SL_2(\Z)],
$$
where $W=N(T)/T$ is the Weyl group. Let $X_*(T)={\rm ker}(\exp \colon \mathfrak{t}\to T)$ denote the cocharacter lattice. The epimorphism in sheaves
\beq
&&\HH\times \mf{t}_\C \simeq \HH\times \mf{t}\times \mf{t} \to (\HH\times \mf{t}\times \mf{t})/X_*(T)\oplus X_*(T)\simeq \HH\times T\times T\to  \HH\times (T\times T)/W\label{eq:bigepi}
\eeq
determines an epimorphism $\HH\times \mf{t}_\C\to \Bun_G(\EE)$ of stacks. 


Consider now the special case $G = U(n)$. The standard basis makes the identification $\mf{t}_{\C} \simeq \C^n$ between the complexified Lie algebra of the maximal torus and $\C^n$. Using this identification, the function $\sigma(\tau, \mathbf{z)}$ from~\eqref{eq:sigmatau} defines a function on $\HH \times \mf{t}_{\C}\simeq \HH\times \C^n$. This function is $W$-invariant and has transformation properties under $X_*(T)^{\oplus 2} \rtimes \SL_2(\Z)\simeq \Z^{2n}\rtimes \SL_2(\Z)$ generalizing~\eqref{eq:sigmasign1} and \eqref{eq:sigmasign}: 
\begin{eqnarray} 
\sigma\left(\frac{a\tau+b}{c\tau+d}, \frac{{\bf z}}{c\tau+d}\right) &=&(c\tau+d)^{-n}e^{\pi i \frac{c({\bf z}\cdot {\bf z})}{c\tau+d}} \sigma(\tau, {\bf z}) \label{eq:bigsigmasign1} \\ 
\sigma(\tau, {\bf z} + {\bf m}\tau+{\bf n}) &=&(-1)^{|{\bf m}|+|{\bf n}|} e^{-\pi i (({\bf m}\cdot {\bf m})\tau+2({\bf m}\cdot {\bf z}))} \sigma(\tau, {\bf z}). \label{eq:bigsigmasign} 
\end{eqnarray} 
where $\mathbf{z} \in \C^n\simeq \mf{t}_{\C}$, $\mathbf{m}, \mathbf{n} \in \Z^n\simeq X_*(T)$, the dot product above is the standard dot product on $\Z^n$ extended $\C$-bilinearly to $\C^n\simeq \mf{t}_{\C}$, and $|\mathbf{m}|=\sum_{i=1}^n m_i$ for $\mathbf{m}=(m_1,\dots,m_n) \in \Z^n$. We recall that the Weyl group of $U(n)$ is $W=\Sigma_n$, the symmetric group. 

\begin{defn}\label{defn:Aboxtime}
Let $\mc{A}_n$ denote the line bundle on 
$$
\Bun_{U(n)}(\EE)\simeq [((\HH\times\mf{t}_\C)\cq \Z^{2n}\rtimes \Sigma_n) \sq \SL_2(\Z)]
$$ 
determined by the $\Sigma_n$-invariant function~\eqref{eq:sigmatau} with the above transformation properties for the action of $X_*(T)^{\oplus 2} \rtimes \SL_2(\Z)\simeq \Z^{2n}\rtimes \SL_2(\Z)$. \end{defn}

\begin{rmk} 
Let $\bar\partial_V$ denote the family of operators over $\Bun_{U(n)}(\EE)$ that at a point $(\tau,[h_1,h_2])\in \HH\times \mathcal{C}^2[G]$ is the $\bar\partial$-operator on the flat (and hence holomorphic) rank~$n$ bundle over $\widetilde{\mathcal{E}}_\tau$ associated to $\pi_1(\widetilde{\mathcal{E}}_\tau)\stackrel{\langle h_1,h_2\rangle}{\to} U(n)\to \Aut(\C^n)$ for the standard representation of $U(n)$. By reducing to $(h_1,h_2)$ in a maximal torus and applying Lemma~\ref{lem:detLoon}, there is an isomorphism of line bundles $\Det(\bar\partial_V)\simeq \mathcal{A}_n$ sending $\det(\bar\partial_V)$ to the ($W$-invariant) function~$\sigma(\tau,  \mathbf{z})$. Hence, $\mathcal{A}_n$ is also a determinant line bundle with determinant section determined by~$\sigma(\tau,\mathbf{z})$. 
\end{rmk}



We recall that a pairing $(- \cdot -)$ on $X_*(T)$ is \emph{even} if ${\bf n}\cdot {\bf n} \in 2\Z$ for ${\bf n} \in X_*(T)$. 

\begin{defn} Let $G$ be a connected compact Lie group with torsion free fundamental group. Given a $W$-invariant bilinear form $(- \cdot -)$ on $\mf{t}_\C$ whose restriction to $X_*(T)$ is even, define the \emph{Looijenga line bundle} for the given pairing as the holomorphic line bundle on $\Bun_G(\EE)$ whose sections are functions $f \in \mc{O}(\HH \times \mf{t}_{\C})^W$ satisfying 
\beq
f \Big( \frac{a \tau + b}{c \tau + d}, \frac{\mathbf{z}}{c \tau + d} \Big) &=& e^{\pi i \frac{c( \mathbf{z} \cdot \mathbf{z})}{c\tau + d}} f(\tau, \mathbf{z}) \label{eq:Loo1}\\
f(\tau, \mathbf{z} + \mathbf{m} + \mathbf{n} \tau) &=& e^{-\pi i ((\mathbf{n} \cdot \mathbf{n}) \tau + 2(\mathbf{n} \cdot \mathbf{z})} f(\tau, \mathbf{z}) \label{eq:Loo2}
\eeq
for the action by $X_*(T)^{\oplus 2}\rtimes \SL_2(\Z)$, where $\mathbf{m}, \mathbf{n} \in X_*(T)$. 
\end{defn}

\begin{notation}
A $W$-invariant quadratic form on $X_*(T)$ determines a class 
$$
[c]\in \Big( \mathrm{Sym}^2(\H^2(BT; \Z)) \Big)^W \simeq (\H^4(BT; \Z))^W
$$ 
Conversely, given a degree~4-cocycle $c$ with $[c]\in (\H^4(BT; \Z))^W$, we use the notation 
$$
\mc{L}(c) \in \mathrm{Pic}(\Bun_G(\EE))
$$ 
for the Looijenga line associated with the bilinear form $2c$, with the factor of~2 to guarantee evenness. The isomorphism class of the holomorphic line bundle $\mc{L}(c)$ depends only on the cohomology class $[c]$. 
\end{notation}

\begin{rmk} The transformation properties~\eqref{eq:sigmasign1} and \eqref{eq:sigmasign} of $\sigma(\tau,z)$ are nearly the same as those for the Looijenga line~\eqref{eq:Loo1} and~\eqref{eq:Loo2} on $\Bun_{U(1)}(\EE)$ associated with the bilinear form $\ell(n, m) = nm$ on $X_*(U(1)) \simeq \Z$. However, this bilinear form is not even, and moreover there are additional signs in~\eqref{eq:sigmasign}. Instead, $\mc{A}$ is a \emph{square-root} of the Looijenga line in that $\mathcal{A}\otimes\mathcal{A}\simeq \mc{L}(c)$ for $c$ the (positive) generator of $H^4(BU(1);\Z)$, i.e., the Looijenga line built from bilinear form $\ell(n, m) = 2nm$. Hence, morally $\mc{A}$ is a Looijenga line of ``half-level.'' Relatedly, $\sigma$ is a Jacobi form of index $1/2$. \end{rmk}

\subsection{Representations and anomalies}\label{sec:anomalies}

Recall the line bundle $\mc{A}_n$ on $\Bun_{U(n)}(\EE)$ from Definition~\ref{defn:Aboxtime} of which $\sigma(\tau, \mathbf{z})$ is a section. For a representation $\rho\colon G\to U(n)$, define the \emph{anomaly line} $\rho^*\mathcal{A}_n$ as the pullback of $\mathcal{A}_n$ under the map of stacks $\Bun_G(\EE) \to \Bun_{U(n)}(\EE)$ determined by~$\rho$. Vanishing of certain characteristic classes will allow us to compare~$\rho^*\mathcal{A}_n$ with a tensor product of Hodge bundles and a Looijenga line bundle. To set up these classes, define $[c_k(\rho)]\in \H^{2k}(BG;\Z)$ as the pullback of the universal Chern class $c_k\in \H^{2k}(BU(n))$ along~$B\rho$,
\beq
&&\begin{tikzpicture}[baseline=(basepoint)];
\node (A) at (0,0) {$BG$};
\node (B) at (2.5,0) {$BU(n)$};
\node (C) at (5,0) {$BSO(2n)$};
\node (D) at (5,1.2) {$B\Spin(2n)$};
\draw[->] (A) to node [below] {$B\rho$} (B);
\draw[->] (B) to  (C);
\draw[->] (D) to (C);
\draw[->,dashed] (A) to node [above] {$B\tilde\rho$} (D);
\path (0,-.75) coordinate (basepoint);
\end{tikzpicture}\nonumber
\eeq
and $[w_k(\rho)]\in \H^k(BG;\Z/2)$ as the pullback of the universal Stiefel--Whitney class $w_k\in \H^k(BSO(2n);\Z/2)$ along the lower composition. The mod~2 reduction of $[c_1(\rho)]$ is $[w_2(\rho)]$. If $[w_2(\rho)]=0$, then $\rho$ is spin and there exists a lift $B\tilde{\rho}$. The pullback of the fractional Pontryagin class $\frac{p_1}{2}\in \H^2(B\Spin(2n);\Z)$ gives a class $[\frac{p_1}{2}(\rho)]\in \H^2(BG;\Z)$, and we have the relation 
\beq
2\cdot [\frac{p_1}{2}(\rho)]=[c_1(\rho)^2-2c_2(\rho)]\in \H^4(BG;\Z).\label{eq:p12}
\eeq
For $G$ connected, we hence find $[c_1(\rho)^2-2c_2(\rho)]$ is (canonically) divisible by~2 when $[c_1(\rho)]\equiv 0 \mod 2$. In this case, we use the notation $[\frac{1}{2}c_1(\rho)^2-c_2(\rho)]$ for $[\frac{p_1}{2}(\rho)]$.

\begin{prop} \label{prop:sigma} Suppose $G$ is connected with torsion-free fundamental group, and let $\rho\colon G \to U(n)$ be a representation with $[w_2(\rho)]=0$. Then there is an isomorphism of holomorphic lines $\rho^*\mc{A} \simeq \mc{L}(\frac{1}{2} c_1(\rho)^2 - c_2(\rho)) \otimes \omega^{-n}$ on $\Bun_G(\EE)$. \end{prop}

\begin{proof}
Given $T<U(n)$ and $T_G<G$ as above, let $\mf{t}_G:={\rm Lie}(T_G) \subset \fg$ denote the Lie algebra of $T_G$. Then $\rho$ gives a map $\HH \times (\mf{t}_{G})_\C\to \HH \times \mf{t}_\C\simeq \HH \times \C^n$. Let $\sigma_{\rho}(\tau, \mathbf{z})\in \mathcal{O}(\Lat\times (\mf{t}_{G})_\C)$ denote the pullback of $\sigma(\tau, \mathbf{z})$ along this map.

We have the formulas for cocycle representatives of characteristic classes in terms of Chern roots:
\beq
&&c_1(\rho)=\sum n_i,\qquad \frac{1}{2}c_1(\rho)^2-c_2(\rho)=\sum n_i^2. \label{eq:Cherncocycle}
\eeq
More precisely, the right-hand side of the equations above represent homogeneous forms on $X_*(T)$ that are then pulled back to $X_*(T_G)$ via $X_*(\rho)$, and we recall that degree $k$ homogeneous forms on $X_*(T_G)$ represent degree $2k$ integral cohomology classes on $BT_G$.

The transformation properties of $\sigma_{\rho}$ can be read off from~\eqref{eq:bigsigmasign1} and \eqref{eq:bigsigmasign}, with $\mathbf{m}, \mathbf{n} \in X_*(T_G), \mathbf{z} \in (\mf{t}_{G})_\C$. We certainly still have 
$$
\sigma_{\rho}(\tau+n, \mathbf{z}) = \sigma_{\rho}(\tau, \mathbf{z}).
$$
Next, the condition $[w_2(\rho)]=0$ and~\eqref{eq:Cherncocycle} implies $|{\bf m}|=\sum_i m_i$ takes values in $2\Z$, and so
$$
\sigma_{\rho}(\tau, \mathbf{z}+\mathbf{m}) = \sigma_{\rho}(\tau,\mathbf{z}).$$ This same condition also simplifies~\eqref{eq:bigsigmasign} to the following: 
$$
\sigma_{\rho}(\tau, \mathbf{z} + \mathbf{n} \tau) = e^{-2\pi i \tau (\mathbf{n} \cdot \mathbf{n})/2} e^{-2\pi i (\mathbf{n} \cdot \mathbf{z})} \sigma_{\rho}(\tau, \mathbf{z})
$$ 
$$
\sigma_{\rho} \Big(-\frac{1}{\tau}, \frac{\mathbf{z}}{\tau} \Big) = \tau^{-n} e^{i \pi (\mathbf{z} \cdot \mathbf{z}) / \tau} \sigma_{\rho}(\tau, \mathbf{z})
$$ 
for $\mathbf{n} \in X_*(T_G)$. The second formula shows that $\sigma_\rho(\tau,{\bf z})$ transforms with weight~$-n$. Hence, $\sigma_{\rho}$ transforms as a section of $\omega^{-n}$ tensor with the Looijenga line for the bilinear form that restricts from the standard bilinear form on $\Z^n$ along the representation $X_*(\rho)\colon X_*(T_G) \to X_*(T) \simeq \Z^n$. This form is even when $\sum n_i$ is even (i.e., $[w_2(\rho)]=0$) and so does indeed define a Looijenga line associated with the cocycle~$\frac{1}{2} c_1^2(\rho) - c_2(\rho)$. By~\eqref{eq:p12}, this represents an integral class in $\H^4(BG; \Z)$ under the hypothesis that $[w_2(\rho)]=0$. \end{proof}

It is natural to ask for which representations $[\frac{1}{2} c_1(\rho)^2 - c_2(\rho)]=0$, since then $\mc{L}(\frac{1}{2} c_1(\rho)^2 - c_2(\rho))$ trivializes. This essentially never happens for representations of connected groups:

\begin{cor} Given a representation $\rho\colon G \to U(n)$ for $G$ connected, if $[\frac{1}{2} c_1(\rho)^2 - c_2(\rho)]=0$ then 
 $\rho$ is the trivial representation. \end{cor}

\bp Vanishing requires the quadratic form $\mathbf{n} \cdot \mathbf{n} = \sum n_i^2$ on $X_*(T_G)$ to identically vanish. As a pullback of a positive definite quadratic form this can only happen if the map $X_*(T_G) \to X_*(T)$ is trivial, implying that the representation on maximal tori is trivial. Hence, the entire representation is trivial, proving the claim. \ep

\begin{rmk} \label{rmk:p1/2} One way to obtain an interesting class of representations $\rho\colon G \to U(n)$ with $[w_2(\rho)] = 0$ and $[\frac{1}{2} c_1^2(\rho) - c_2(\rho)] = 0$ is to pass to virtual representations, $\rho=\rho_1\ominus \rho_2$. Then the elliptic Euler class is defined as the quotient of Euler classes, viewed as a (meromorphic) section of holomorphic lines on $\Bun_G(\EE)$. Such virtual representation arise naturally when studying $\theta$-functions of lattices, e.g., see~\cite[\S5.3]{HopkinsICM2002}. They also arise physically as part of the Green--Schwartz anomaly cancelation procedure~\cite{GreenSchwarz}, where a gauge anomaly and gravitational anomaly cancel one another. The line bundle $\mathcal{A}_n$ plays the role of the gauge anomaly, whereas the gravitational anomaly arises from the $\mathcal{N}=(0,1)$ super symmetric sigma model with target a Riemannian manifold~$M$. The anomaly cancellation condition is then $(p_1)_G(M)/2=(p_1)_G(V)/2$ in the Borel equivariant cohomology of $M$ where $V$ is a vector bundle classified by a map~$M\to BG$ and the representation~$\rho$, and $p_1(V)/2$ is the pullback of $p_1(\rho)/2$; e.g., see~\cite[pg.~511]{strings1}.
\end{rmk} 

\subsection{A family of operators over $\F_0(\pt\nsq G)$}\label{sec:superfamofops}

We construct a family of operators over $\F_0(\pt\nsq G)$ that for $G=U(1)$ is analogous to the $\bar\partial_L$ from~\S\ref{sec:detreview}. 

Throughout this subsection,~$V=\C^n$. Given a representation $\rho\colon G\to U(V)=U(n)$, there is a vector bundle with connection denoted~$(V,\rho(\nabla))$ on the stack~$[\pt\nsq G]$ gotten by taking associated vector bundles and associated connections for a $G$-bundle using the representation~$V$ (see~Example~\ref{ex:assoc}). Let $\mathcal{V}$ denote the sheaf of sections of~$V$ on~$[\pt\nsq G]$. We also recall the line bundle~$\omega^{1/2}$ over $\F_0(\pt\nsq G)$ as defined by~\eqref{eq:sHodge} and the line bundle $\overline{\omega}^{1/2}$ defined similarly. 

\begin{defn}\label{defn:fermions} Given a representation $\rho\colon G\to U(V)$, define the vector bundle $\Fer(V)\to \F_0(\pt\nsq G)$ as follows. To an $S$-point $\Phi\colon T^{2|1}\to [\pt\nsq G]$ of $\F_0(\pt\nsq G)$, $\Fer(V)$ assigns the sheaf of $C^\infty(S)$-modules given by $\pi_*\Phi^*\mathcal{V}\otimes \omega^{1/2}$, where $\pi\colon T^{2|1}\to S$ is the projection. 
To an isomorphism in $\F_0(\pt\nsq G)$ associated with a fiberwise conformal map $f\colon T^{2|1}\to T'^{2|1}$ covered by an isomorphism $g\colon P\to P'$ of $G$-bundles, $\Fer(V)$ assigns the isomorphism $\rho(g)\circ f^*$ of sheaves of $C^\infty(S)$-modules by pulling back along~$f$ and applying the isomorphism $\rho(g)$ of $C^\infty(S)$-modules associated to~$g$ via~$G\stackrel{\rho}{\to} \Aut(V)$. 
\end{defn}

\begin{defn}
Let $\Hom(\Fer(V),\Fer(V)\otimes \overline{\omega}^{1/2})$ denote the $\Hom$-sheaf on $\F_0(\pt\nsq G)$ whose sections over $S$ are maps between sections of $\Fer(V)$ and $\Fer(V)\otimes \overline{\omega}^{1/2}$ over $S$ (not necessarily $C^\infty(S)$-linear maps). Define a section of this sheaf~$D_\rho:=\rho(\nabla)_D\otimes \id_{\omega^{1/2}}$, where $\rho(\nabla)$ is the associated connection on $\Phi^*\mathcal{V}$ over~$T^{2|1}$, which we evaluate on the vector field $D=\partial_\theta-\theta\partial_{\bar z}$.
\end{defn}

\begin{rmk} 
We recall that a differential operator is determined by its associated morphism on sheaves of sections of the corresponding vector bundles. We view $D_\rho$ above as family of differential operators over the stack $\F_0(\pt\nsq G)$. 
\end{rmk}

\begin{rmk} The operator $D_\rho$ is a section of the asserted sheaf because the vector field $D=\partial_\theta-\theta\partial_{\bar z}$ on $S\times \R^{2|1}$ transforms as $D\mapsto \overline{\mu}^{-1}D$ under the action of $(\mu,\bar\mu)\in \C^\times(S)$. This change in the $\C^\times$-weight of a section leads to the factor of~$\overline{\omega}^{1/2}$. 
\end{rmk}

We obtain the following explicit description of $\Fer(V)$ and $D_\rho$, using that the pullback of $V$ along the composition
$$
S\times \R^{2|1}\twoheadrightarrow T^{2|1}_\Lambda\stackrel{\Phi}{\to} [\pt\nsq G]
$$
trivializes.

\begin{lem} \label{lem:FerWZ}In the groupoid presentation from Lemma~\ref{lem:present},~$\Fer(V)\to \cL_0(\pt\nsq G)$ is determined by the equivariant vector bundle over the groupoid that to an $S$-point of objects specified by $\Lambda\in \Lat(S)$, $h\in \Hom(\Z^2,G)$ and $A\in \Omega^1(S\times \R^{0|1};\mf{g})$ assigns the $C^\infty(S)$-module
\beq
&&\Gamma(T^{2|1}_\Lambda,\Phi^*\Pi V)=C^\infty\left(S\times \R^{2|1}; \Pi V\right)^{\Z^2},\label{eq:sections}
\eeq
where the $\Z^2$-action on $S\times \R^{2|1}$ is through $\Lambda\in \Lat(S)$ and on $\Pi V$ is through $S\times \Z^2\stackrel{h}{\to} G\stackrel{\rho}{\to} \End(V)\simeq \End(\Pi V)$. For an isomorphism in this groupoid presentation over a base change $S\to S'$, we obtain a map of modules as the composition
\beq
C^\infty(S\times \R^{2|1},\Pi V)^{\Z^2}&\stackrel{f^*}{\longrightarrow}& C^\infty(S\times \R^{2|1},\Pi V)^{\Z^2}\stackrel{\mu^{-1}}{\longrightarrow} C^\infty(S\times \R^{2|1},\Pi V)^{\Z^2}\nonumber\\
&\stackrel{\rho(g)}{\to} & C^\infty(S\times \R^{2|1},\Pi V)^{\Z^2}{\to} C^\infty(S'\times \R^{2|1},\Pi V)^{\Z^2},\label{eq:sectionmorphisms}
\eeq
where $f^*$ is the pullback along the action of $f\in \E^{2|1}(S)$ on $S\times \R^{2|1}$, $\mu^{-1}$ is the pullback by the action of $(\mu,\bar\mu)\in \C^\times$ on $S\times \R^{0|1}$ and the action of $\mu^{-1}$ on~$C^\infty(S\times \R^{2|1};\Pi V)$, $\rho(g)$ is the $C^\infty(S)$-module isomorphism associated to the gauge transformation~$g$ from the 2-commuting data in~\eqref{eq:commutetri}, and the final arrow is the pullback along the base change $S\to S'$. The family of operators $D_\rho$ acting on sections in the form~\eqref{eq:sections} is given by
\beq
D_\rho=D+d\rho(A)(D)\label{eq:Drho}
\eeq
where $D=\partial_\theta-\theta\partial_{\bar z}$ and $d\rho(A)$ is the image of $A$ under the map 
$$
A\in \Omega^1_S(\R^{0|1};\mf{g})\to \Omega^1_S(\R^{0|1};\End(V))
$$
determined by $d\rho\colon\mf{g}\to \End(V)$. 
\end{lem}

\bp
 We recall that~$\omega^{1/2}$ trivializes on $S$-points of objects in this groupoid presentation. So at the level of objects, tensoring a vector bundle with $\omega^{1/2}$ is the same as tensoring with a trivial odd line, meaning it is the same as the parity reversal functor on vector bundles. Incorporating this change of parity, the claimed value of $\Fer(V)$ on objects follows from the construction of an associated bundle over $[\pt\nsq G]$ from a representation in Example~\ref{ex:assoc}, along with definitions of pullback and direct image sheaf. The value on morphisms also follows from Example~\ref{ex:assoc}, together with the transformation properties of sections of $\omega^{1/2}$ (whence the dilation by $\mu^{-1}$ in~\eqref{eq:sectionmorphisms}). The family of operators takes the claimed form because the $G$-connection at an $S$-point is determined by the trivial connection and a $\fg$-valued 1-form~$A$. The associated connection is therefore determined by the operator~
 $$
 d+\rho(A)\colon C^\infty(S\times \R^{2|1}; \Pi V)\to \Omega^1(S\times \R^{2|1}; \Pi V),\qquad \rho(A)\in \Omega^1(S\times \R^{2|1}; \End(\Pi V))
 $$ 
 Contracting with the vector field $D$ gives~\eqref{eq:Drho}.
\ep

\subsection{The free fermion theory}
This subsection explains how the family of operators $D_\rho$ constructs a theory of free fermions over $\F_0(\pt\nsq G)$. This discussion is not logically essential and so can be skipped by readers less interested in the physical motivation.
 
 The standard hermitian pairing on $V=\C^n$ gives a $\C$-bilinear pairing $\langle-,-\rangle\colon \overline{V}\oplus V\to \C$. This determines a functional on sections of $\Fer(\overline{V}\oplus V)$ as follows. 

\begin{defn}\label{defn:sFF}
The \emph{free fermion classical action} with background gauge fields is the functional on sections of~$\Fer(\overline{V}\oplus V)$ specified by
\beq
\mathcal{S}(\psi)=\frac{1}{2}\int \langle \psi,D_{\rho^\dagger\oplus \rho}\psi\rangle d\theta d^2z \label{eq:classaction}
\eeq
where the integral is over the fibers of the bundle of super tori 
\beq
T^{2|1}=(S\times \R^{2|1})/\Z^2\to S \label{eq:21integral}
\eeq
using the fiberwise Berezinian measure $d\theta d^2z=d\theta \frac{1}{2i}d\bar z dz$.
\end{defn}

\begin{rmk} Sections of $\Fer(\overline{V}\oplus V)$ endowed with the classical action~\eqref{eq:classaction} is a superspace extension of the chiral free fermion classical field theory in Example~\ref{ex:fermions}. 
\end{rmk}
\begin{lem} \label{lem:Sinvt}The functional~\eqref{eq:classaction} defined on objects over each $S$-point descends to a functional defined on the stack. 
\end{lem}

\bp
From the definitions, we observe that $\langle \psi,D_{\rho^\dagger\oplus \rho}\psi\rangle$ is a section of $\omega^{1/2}\otimes\omega^{1/2}\otimes \overline{\omega}^{1/2}\simeq \omega\otimes\bar\omega^{1/2}$. On the other hand, the section of the Berezinian line transforms as $d\theta  d\bar z d z\mapsto \frac{1}{\bar \mu} d\theta \mu^2 dz \bar \mu^2 d\bar z =\bar \mu \mu^2 (d\theta d^2z)$, and so is a section of $\omega^{-1}\otimes \overline{\omega}^{-1/2}$. The integral~\eqref{eq:classaction} therefore defines a section of the trivial line, giving a function on the stack. 
\ep

\begin{rmk}
The \emph{partition function} of the classical field theory from Definition~\ref{defn:sFF} is the regularized Pfaffian of the family of operators~$D_{\rho^\dagger\oplus \rho}$. These operators have the form
$$
D_{\rho^\dagger\oplus \rho}=\left(\begin{array}{cc} 0 & D_{\rho} \\ D_{\rho^\dagger} & 0\end{array}\right)
$$
with respect to the block diagonal decomposition $\Fer(\overline{V}\oplus V)\simeq \Fer(\overline{V})\oplus \Fer(V)$. For a Pfaffian of a skew matrix on a $2d$-dimensional vector space, we have the formula 
\beq
\pf\left(\begin{array}{cc} 0 & A \\ -A^T & 0\end{array}\right)=(-1)^{d(d-1)/2}\det(A).\label{eq:Pfaffinite}
\eeq
This explains our focus on the determinant of~$D_\rho$ below: it determines (up to a phase) the partition function of the free fermion theory. 
\end{rmk}

\subsection{Eigenvalues of $D_\rho$ and comparison with $\bar\partial_L$}

This subsection serves to justify Definition~\ref{defn:calD} of $\Det(D_\rho)$ as the unique analytic extension of the Bismut--Freed--Quillen line bundle; however the results that follow are not logically dependent on this subsection, and so can be skipped by readers willing to accept Definition~\ref{defn:calD}.

Let $\theta\in C^\infty(\R^{2|1})^\odd$ be the standard odd coordinate function. Since $\theta\mapsto \bar\mu\theta$ for the $\C^\times$-action on $\R^{2|1}$, we observe that multiplying a section of $\Fer(V)$ by $\theta$ determines a map of vector bundles $\Fer(V)\to \Fer(V)\otimes \bar\omega^{-1/2}$ over $\F_0(\pt\nsq G)$. 

\begin{defn} Let $\psi$ be a local section of $\Fer(V)$ on $\F_0(\pt\nsq G)$. Then $\psi$ is an \emph{eigenvector with eigenvalue $\eigen$} if there exists a function $\eigen$ such that 
\beq
D_{\rho}\psi=\theta \eigen \psi.\label{eq:eigenvalue}
\eeq
\end{defn}

\begin{lem} \label{lem:eigen} The eigenvalues of $D_{\rho}$ are analytic functions on $\F_0(\pt\ncq G)$.
\end{lem}
\bp 
We observe~\eqref{eq:eigenvalue} is gauge-invariant, and in particular the function $\eigen$ is a gauge-invariant function. By Definition~\ref{defn:generalanalytic}, analyticity therefore amounts to showing that $\epsilon$ depends holomorphically on $\Lat$ and is invariant under deformations. Using naturality in~$G$, we may assume $G=U(n)$ and $D_\rho=\mathcal{D}$. Since any pair of commuting elements in $U(n)$ can be diagonalized, conjugation invariance allows us to restrict attention to eigenvalues for~$\mathcal{D}$ restricted to $\F_0(\pt\ncq T)\hookrightarrow \F_0(\pt\ncq U(n))$. 
The map
\beq
&&\pi\colon \Lat\times\mf{t}\times \mf{t}\times\Omega^1(-\times \R^{0|1};\mf{t})_\wz \to \Lat\times T\times T\times\Omega^1(-\times \R^{0|1};\mf{t})_\wz \to \F_0(\pt\nsq T) \label{eq:pi}
\eeq
is an epimorphism, where the first arrow is determined by the exponential map $\exp\colon \mf{t}\to T$ and the second arrow is from Lemma~\ref{lem:WZ}. Therefore, for any $S$-point of~$\F_0(\pt\ncq T)$ there exists a cover of $S$ that factors through~$\pi$. Since analyticity is a local condition, it therefore suffices to consider eigenvalues of $\pi^*\mathcal{D}$ acting on sections of~$\pi^*\Fer(V)$. 


By Lemma~\ref{lem:FerWZ}, $\pi^*\Fer(V)$ has as sections
$$
C^\infty(\Lat\times\mf{t}\times \mf{t}\times\Omega^1(-\times \R^{0|1};\mf{t})_\wz\times \R^{2|1};\Pi V )^{\Z^2}\simeq C^\infty(\Lat\times\mf{t}\times \mf{t}\times \underline{\mf{t}}\times \R^{2|1};\Pi V)^{\Z^2}.
$$
Taylor expanding in the odd variable $\theta\in C^\infty(\R^{2|1})$, a section can be written as $\psi=\alpha+\theta a$ for
\beq
&&(\alpha,a)\in C^\infty(\Lat\times\mf{t}\times \mf{t}\times \underline{\mf{t}}\times \R^{2};\Pi V)^{\Z^2}\oplus C^\infty(\Lat\times\mf{t}\times \mf{t}\times \underline{\mf{t}}\times \R^{2};V)^{\Z^2}.\label{eq:components}
\eeq
The $\Z^2$-action with respect to which invariants are taken is the action on $\Lat\times\mf{t}\times \mf{t}\times \underline{\mf{t}}\times \R^{2}$, 
$$
(n,m)\cdot (\lambda_1,\bar\lambda_1,\lambda_2,\bar\lambda_2, x_1,x_2,X,z) =(\lambda_1,\bar\lambda_1,\lambda_2,\bar\lambda_2, x_1+n,x_2+m,X,z+n\lambda_1+m\lambda_2,\bar z+n\bar\lambda_1+m\bar \lambda_2),
$$
for $(\lambda_1,\bar\lambda_1,\lambda_2,\bar\lambda_2)\in \Lat(S), (x_1,x_2)\in \mf{t}^{\times 2}(S), X\in \underline{\mf{t}}(S), (z,\bar z)\in \R^2(S)$, and the action on $V$ through $h\colon S\times \Z^2\to T\stackrel{\rho}{\to}\Aut(V)$. For a connection in the Wess--Zumino gauge of the form $A=\theta d\theta \otimes X$, the operator $\pi^*\mathcal{D}$ takes the form
$$
\pi^*\mathcal{D}=(\partial_\theta-\theta\partial_{\bar z})-\theta \otimes \rho(X).
$$
Hence the eigenvalue equation~\eqref{eq:eigenvalue} for $\psi=\alpha+\theta a$ becomes the more standard eigenvalue equation the $\partial_{\bar z}$-operator deformed by~$X$,
\beq
a=0,\qquad (\partial_{\bar z}+\rho(X))\alpha=-\eigen\alpha. \label{eq:eigenvector2}
\eeq
Let $\{v_i\}$ denote the standard basis of $\Pi V=\Pi \C^n$. We observe that the following $\Pi V$-valued functions on $\Lat\times\mf{t}\times \mf{t}\times \underline{\mf{t}}\times \R^{2}$ are $\Z^2$-invariant, 
\beq
&&\alpha^i_{k,l}=\exp\left(\frac{\pi}{\vol}(-z((k-x_2^i)\bar\lambda_1+(l+x_1^i)\bar\lambda_2)+\bar z((k-x_2^i)\lambda_1+(l+x_1^i)\lambda_2)\right)\otimes v_i\label{eq:basis1}
\eeq
for $(j,k)\in \Z^2$ and $i=1,\dots,n$. Hence the above define elements of the odd summand of~\eqref{eq:components}, and we observe their span is dense. Furthermore, since 
\beq
(\partial_{\bar z}+\rho(X))\alpha^i_{k,l}&=&\frac{\lambda_2}{\vol} \pi \left((k-x_2^i)\lambda_1/\lambda_2+(l+x_1^i)+\frac{\vol}{\lambda_2}\rho(X)^i\right)\alpha^i_{k,l}\label{eq:eigen}\\
&=&\pi (k\tau+l+z_i)\alpha^i_{k,l}\nonumber
\eeq
they determine eigenvectors via~\eqref{eq:eigenvector2}, where we have defined 
\beq
z_i:=x_1^i-\tau x_2^i+\frac{\vol}{\lambda_2}\rho(X)^i\qquad \tau=\lambda_1/\lambda_2\label{eq:holocoord}
\eeq 
for $\rho(X)^i$ the value of the $i$th coordinate function on~$\mf{t}_\C$ on~$X$. These eigenvalues are analytic in the sense of Definition~\ref{defn:generalanalytic}: the dependence on $\Lat$ is holomorphic (taking care with the normalizations from~\eqref{eq:varphimap}), and the behavior under deformations of $G$-bundles agrees with~\eqref{eq:formulaforanalytic}. Since the analytic condition is a closed condition on the space of smooth functions and the span of~\eqref{eq:basis1} is dense in the solution set~\eqref{eq:eigenvector2}, we conclude that all eigenvalues are analytic. 
\ep

\begin{rmk}\label{rmk:analytic}
The eigenvalues of the previous lemma are observables in the associated field theory with action functional from Definition~\ref{defn:sFF}. Our definition of analyticity was motivated by the flavor of holomorphy these observables possess. 
\end{rmk}

We observe that the family of operators~$D_\rho$ is natural with respect to group homomorphisms, and in fact pulls back from a universal family along the functor induced by the representation~$\rho\colon G\to U(n)$
$$
\cL_0(\pt\nsq G)\to \cL_0(\pt\nsq U(n)),
$$
where~$D_\rho$ over $\cL_0(\pt\nsq U(n))$ is defined using the representation~$\rho=\id\colon U(n)\to U(n)$. One way to ensure that a putative determinant line bundle $\Det(D_\rho)$ is natural with respect to homomorphisms of compact Lie groups is to construct it universally over $ \cL_0(\pt\nsq U(n))$. 

\begin{notation}
We use the notation $\mathcal{D}:=D_\id$ for the universal family $D_\rho$ associated with $\id\colon U(n)\to U(n)$. 
\end{notation}

We have the following comparison between the families of operators $\mathcal{D}$ and $\bar\partial_{L_1\oplus \cdots L_n}$. 

\begin{lem}\label{lem:dbarcompare} There exists a map 
\beq
j\colon (\Lat\times \C^n)\cq \Z^{2n}\to \F_0(\pt\nsq U(n))\label{eq:jmap}
\eeq
such that the odd summand $j^*\Fer(V)\simeq (j^*\Fer(V))^\ev\oplus (j^*\Fer(V))^\odd$ is isomorphic to the vector bundle on which $\bar\partial_{L_1\oplus \cdots \oplus L_n}$ acts. Furthermore, the pullback of~$\mathcal{D}$ restricted to this odd summand is
$$
j^*\mathcal{D}|_{\Fer(V)^{\odd}}=\bar\partial_{L_1\oplus \cdots \oplus L_n}.$$ 
\end{lem}
\bp 
We will construct $j$ as the lower line of the diagram
\beq
&&\begin{tikzpicture}[baseline=(basepoint)];
\node (A) at (0,0) {$\Lat\times \C^n$};
\node (B) at (4,0) {$\Lat\times \mf{t}\times \mf{t}\times \underline{\mf{t}}\times \Omega^1(-;\mf{t})$};
\node (C) at (0,-1.2) {$(\Lat\times \C^n)\cq \Z^{2n}$};
\node (D) at (4,-1.2) {$\Lat \times T\times T\times \underline{\mf{t}}\times \Omega^1(-;\mf{t})$};
\node (E) at (8,-1.2) {$\F_0(\pt\nsq T)$};
\node (F) at (10.5,-1.2) {$\F_0(\pt\nsq U(n))$};
\draw[->] (A) to node [above] {$\tilde{j}$} (B);
\draw[->] (A) to  (C);
\draw[->] (C) to (D);
\draw[->] (B) to (D);
\draw[->] (D) to (E);
\draw[->] (E) to (F);
\draw[->] (B) to node [above] {$\pi$} (E);
\path (0,-.75) coordinate (basepoint);
\end{tikzpicture}\label{eq:tildej}
\eeq
Here, $\tilde{j}$ is the inverse to the isomorphism $\Lat\times \R^{n}\times \R^n\stackrel{\sim}{\to} \Lat\times \C^n$ given by
\beq
(\lambda_1,\bar\lambda_1,\lambda_2,\bar\lambda_2,x,y)\mapsto (\lambda_1,\bar\lambda_1,\lambda_2,\bar\lambda_2,\lambda_2 x-\lambda_1 y,\bar\lambda_2x-\bar\lambda_1 y)\label{eq:squareiso}\\
(\lambda_1,\bar\lambda_1,\lambda_2,\bar\lambda_2)\in \Lat(S),\ (x,y)\in \R^n(S), \nonumber
\eeq
composed with the inclusion $\Lat \times \R^{n}\times \R^n\simeq \Lat \times\mf{t}\times \mf{t}\hookrightarrow \Lat\times \mf{t}\times \mf{t}\times \underline{\mf{t}}\times \Omega^1(-;\mf{t})$  along $0\colon \pt\to \underline{\mf{t}}\times \Omega^1(0;\mf{t})$. The map $\tilde{j}$ is $\Z^{2n}$-equivariant, so determines the map between the $\Z^{2n}$-quotients. The map $\pi$ is from~\eqref{eq:pi} (using the generalized atlas from Lemma~\ref{lem:WZ}), and the remaining map comes from the inclusion $T\hookrightarrow U(n)$ of the standard maximal torus. Then by the description of $\pi^*\Fer(V)$ in~\eqref{eq:components}, we see that the odd summand of $j^*\Fer(V)$ is indeed  a direct sum of $n$ flat line bundles over $\widetilde{\EE}$, proving the first part of the lemma. The second part follows from the description of $\mathcal{D}$ from~\eqref{eq:eigenvector2} where on $j^*\Fer(V)^\odd$ it is the $\bar\partial_{\bar z}$-operator.
\ep

Consider the composition of~\eqref{eq:jmap} with the coarse quotient by gauge transformations~\eqref{eq:F0MG},
\beq
&&(\Lat\times \C^n)/ \Z^{2n}\stackrel{j}{\to} \F_0(\pt\nsq U(n))\to \F_0(\pt\ncq U(n)). \label{eq:extendj}
\eeq
Proposition~\ref{prop:Unext} below shows that there is a unique line bundle and section on $\F_0(\pt\nsq U(n))$ satisfying the following conditions: (1) the line bundle and section pull back from an analytic line bundle and analytic section over $\F_0(\pt\ncq U(n))$ and (2) the further pullback of the line bundle and section along $j$ is the Bismut--Freed--Quillen determinant line bundle of $\bar\partial_{L_1\oplus \cdots L_n}$ with its holomorphic determinant section. We define the determinant line of $\mathcal{D}$ as this uniquely specified line bundle on $\F_0(\pt\nsq U(n))$ and define its section to be the determinant section. This terminology is justified as follows. Condition (1) is justified because the eigenvalues of $\mathcal{D}$ are analytic (and in particular, gauge-invariant), and so we would expect the determinant of this operator to similarly be analytic, pulling back from $\F_0(\pt\ncq U(n))$. Condition (2) is justified by the requirement that determinant lines be natural with respect to restriction of their family of operators. 


\subsection{The determinant line of $D_\rho$ over $\Mst_G$}\label{sec:sdetline}

In view of the source of~\eqref{eq:jmap}, we shall require a description of the Bismut--Freed--Quillen line bundle~$\mathcal{A}$ and section $\sigma$ over the universal curve $\EE^\vee$ pulled back along the equivalence of stacks, 
$$
[\Lat\sq \SL_2(\Z)\times \C^\times]\stackrel{\sim}{\longrightarrow} [\HH\sq \SL_2(\Z)]\simeq \Mell. 
$$
This follows the standard translation between modular forms as functions on $\HH$ and functions on $\Lat$, e.g., see~\cite[pg.~4]{ZagierMF}. We first recall the description of the $\sigma$-function as an $\SL_2(\Z)$-invariant function on $\Lat\times \C$ with $\C^\times$-transformation properties, 
\beq
&&\sigma(\lambda_1,\lambda_2,z)= 2\pi i z\lambda_2 \exp\left(-\sum_{k>0} \frac{E_{2k}(\lambda_1,\lambda_2) (2\pi i z)^{2k}}{2k}\right)\in \mathcal{O}(\Lat \times \C)\label{eq:sigmaLat}\label{eq:sigmalattices}
\eeq
using the standard definition 
$$
E_{2k}(\lambda_1,\lambda_2):=\lambda_2^{-2k}E_{2k}(\lambda_1/\lambda_2)\in \mathcal{O}(\Lat)
$$
of Eisenstein series as functions on based lattices. Then the transformation properties of the $\sigma$-function~\eqref{eq:sigmalattices} give a cocycle for the line bundle on the quotient
$$
\mathcal{E}^\vee\simeq [\Lat\times \C\sq \SL_2(\Z)\ltimes \Z^2\times \C^\times]. 
$$
We similarly define the product of $\sigma$-functions (following~\eqref{eq:sigmatau})
\beq
\sigma(\lambda_1,\lambda_2,  \mathbf{z}) := \prod_{i=1}^n \sigma(\lambda_1,\lambda_2,  z_i)\in \mathcal{O}(\Lat\times \C^n) \label{eq:sigmaT}
\eeq
which we similarly view as defining a section of $\mathcal{A}_n$ over $\Bun_{U(n)}(\EE)$. 

\begin{prop} \label{prop:Unext}
The determinant line bundle $\Det(\bar\partial_{L_1\oplus \cdots L_n})$ has a unique extension along~\eqref{eq:extendj}, i.e., there is an analytic line bundle $\Det(\mathcal{D})$ on~$\F_0(\pt\ncq U(n))$ whose pullback along~\eqref{eq:extendj} is isomorphic to $\Det(\bar\partial_{L_1\oplus \cdots L_n})$, and (up to unique isomorphism) $\Det(\mathcal{D})$ is the unique line bundle on $\F_0(\pt\ncq U(n))$ with this property. Furthermore, there is a unique analytic section $\det(\mathcal{D})$ of $\Det(\mathcal{D})$ over $\F_0(\pt\ncq U(n))$ whose restriction along~\eqref{eq:extendj} is $\det(\bar\partial_L)$. 

\end{prop}

\bp
Consider the evident factorization of the composition~\eqref{eq:extendj} as
\beq
(\Lat\times \C^n)/ \Z^{2n}\to \F_0(\pt\ncq T)\to \F_0(\pt\ncq U(n))\label{eq:thefirstarrow}
\eeq
for $T=U(1)^{\times n}<U(n)$ the standard maximal torus. We start by extending the determinant line along the first arrow above. We have
$$
\Ftil_0(\pt\ncq T)\simeq (\Lat\times T\times T\times\Omega^1(-\times \R^{0|1};\mf{t})_\wz)/T\simeq \Lat\times T\times T\times\Omega^1(-\times\R^{0|1};\mf{t})_\wz.
$$
Consider the open cover determined by $\exp\colon\mf{t}\to T$,
\beq
\Lat\times\mf{t}\times\mf{t}\times \Omega^1(-\times\R^{0|1};\mf{t})_\wz\twoheadrightarrow \Ftil_0(\pt\ncq T).\label{eq:coverofT}
\eeq
We observe that any holomorphic function on $\Lat\times\mf{t}_\C$ uniquely extends along
\beq
&&\Lat\times\mf{t}_\C \simeq \Lat\times\mf{t}\times\mf{t} \hookrightarrow \Lat\times\mf{t}\times\mf{t}\times \Omega^1(-\times\R^{0|1};\mf{t})_\wz,\nonumber
\eeq
via the assignment 
\beq
f(w_1,\dots,w_n)\mapsto f(z_1,\dots,z_n)\label{eq:analyticext}
\eeq 
for $z_i$ as defined in~\eqref{eq:holocoord} with $w_j=x_1^j+ix_2^j\in \mf{t}_\C$. Furthermore, for any open subsheaf $U\subset \Ftil_0(\pt\ncq T)$ pulled back along the cover~\eqref{eq:coverofT}, the above assignment determines an analytic function on the pullback of~$U$. We will refer to~\eqref{eq:analyticext} as the \emph{unique analytic extension} of a holomorphic function on $\Lat\times\mf{t}\times\mf{t}$. Consider the unique analytic extension for the $\sigma$-function,~\eqref{eq:sigmaT}, using the standard identification $\mathfrak{t}_\C\simeq \C^n$ for the Lie algebra of the maximal torus of~$U(n)$.
Since the usual $\sigma$-function transforms under the action by $\Z^{2n}\rtimes \SL_2(\Z)\times \C^\times$ by a nonvanishing holomorphic function, it is a formal consequence that the image of~\eqref{eq:sigmaT} under~\eqref{eq:analyticext} transforms under this action by a nonvanishing analytic function, namely the unique analytic extension of the transformation law for~$\sigma$. This gives rise to a line bundle with section over the quotient, 
$$
[\Lat\times\mf{t}\times\mf{t}\times \Omega^1(-\times\R^{0|1};\mf{t})_\wz\sq \Z^{2n}\rtimes \SL_2(\Z)\times \C^\times]\simeq [\Ftil_0(\pt\ncq T)\sq \SL_2(\Z)\times \C^\times], 
$$
where the equivalence uses that the $\Z^{2n}$-action is free with quotient $\Ftil_0(\pt\ncq T)$. We observe that the unique analytic extension of $\sigma(\lambda_1,\lambda_2,{\bf z})$ is invariant under the $\E^{0|1}$-action. Hence, this unique analytic extension determines a line bundle with section over $\F_0(\pt\ncq T)$. By construction, this extends $\det(\bar\partial_{L_1\oplus \cdots L_n})$ and $\Det(\bar\partial_{L_1\oplus \cdots L_n})$ along the first arrow of~\eqref{eq:thefirstarrow}, and is the unique such analytic extension. 

Invariance of~\eqref{eq:sigmaT} under the action of the symmetric group (i.e., the Weyl group of $U(1)^{\times n}<U(n)$) together with Propositions~\ref{prop:ellgrp} and~\ref{prop:torus} yields a further gauge-invariant extension of the line bundle with section along the second arrow of~\eqref{eq:thefirstarrow}.
By construction, the $\E^{0|1}\rtimes \C^\times \times \SL_2(\Z)$-equivariant properties of this extension are completely determined by restriction to $\Ftil_0(\pt\ncq T)$. Hence, the equivariant structure is unique. This completes the proof. 
\ep

\begin{defn}\label{defn:calD} Define $\Det(\mathcal{D})$ and $\det(\mathcal{D})$ as the line bundle and section over $\F_0(\pt\ncq U(n))$ constructed by Lemma~\ref{prop:Unext}. \end{defn}

\begin{defn}\label{defn:detDrho} For $\rho\colon G\to U(n)$ a representation, let $\Det(D_\rho)$ denote the pullback of $\Det(\mathcal{D})$ along the induced functor $\F_0(\pt\ncq G)\to\F_0(\pt\ncq U(n))$, and $\det(D_\rho)$ denote the pull back of the section  $\det(\mathcal{D})$. 
\end{defn}

\begin{prop} \label{prop:chaincomparison}
By Lemma~\ref{lem:chaincomplex}, the sheaf $\Det(\mathcal{D})$ determines a sheaf of chain complexes on $\Mst_{U(n)}$. This sheaf is concentrated in degree~$2n$ with trivial differential and $\det(\mathcal{D})$ descends to a $Q$-closed global section over $\Mst_{U(n)}$, which we continue to denote by $\det(\mathcal{D})$.  Furthermore, $\Ch^{2n}(\Det(\mathcal{D}))$ pulls back along $\iota\colon \Bun_{U(n)}(\EE)\to i^*\Mst_{U(n)}$ to~$\mathcal{A}_n$ from Definition~\ref{defn:Aboxtime}, and the section $\det(\mathcal{D})$ pulls back to~$\sigma(\tau,{\bf z})$. 
\end{prop}

\bp
The first statement follows immediately from Lemma~\ref{lem:chaincomplex}. To prove the second statement, we must compute the transformation properties of sections of $\Det(\mathcal{D})$ under the $\C^\times$-action. By construction, these transformation properties are determined by the $\C^\times$-action on the $\sigma$-function, which can be computed from~\eqref{eq:sigmalattices}. We see that $\sigma(\lambda_1,\lambda_2,{\bf z})\mapsto \mu^{2n} \sigma(\lambda_1,\lambda_2,{\bf z})$ for $(\mu,\bar\mu)\in \C^\times$. Hence, sections are in the $2n$th weight space of the $\C^\times$-action, so that the resulting sheaf of chain complexes is indeed concentrated in degree~$2n$. The only possible differential is the trivial one, and indeed, the section $\det(\mathcal{D})$ is invariant under the $\E^{0|1}$-action by construction. 

Finally we observe that there is a 2-commutative square
\beq
\begin{tikzpicture}[baseline=(basepoint)];
\node (A) at (0,0) {$(\Lat\times \C^n)\cq \Z^{2n}$};
\node (B) at (4,0) {$i^*\F_0(\pt\ncq U(n))$};
\node (C) at (0,-1.2) {$\Bun_{U(n)}(\EE)$};
\node (D) at (4,-1.2) {$i^*\Mst_{U(n)}$};
\draw[->] (A) to (B);
\draw[->] (A) to  (C);
\draw[->] (C) to node [above] {$\iota$} (D);
\draw[->] (B) to (D);
\path (0,-.75) coordinate (basepoint);
\end{tikzpicture}\nonumber
\eeq
where the upper horizontal arrow is the functor $i^*$ applied to~\eqref{eq:extendj}. The vertical arrows take $\C^\times$-quotients, with the vertical arrow on the left then composing with the map $(\HH\times \C^n)/\Z^2\to \Bun_{U(n)}(\EE)$ induced by the inclusion $T\hookrightarrow U(n)$ of the standard maximal torus. Now we observe that the pullback of $\Ch^{2n}(\Det(\mathcal{D}))$ along the quotient map $\F_0(\pt\ncq U(n))\to \Mst_{U(n)}$ is $\Det(\mathcal{D})$: this pullback simply tensors the $2n$th weight space with analytic functions, which recovers the sections of the analytic line bundle $\Det(\mathcal{D})$. By Proposition~\ref{prop:Unext} this line bundle and section restrict to the line bundle $\mathcal{A}_n$ with section $\sigma$ over $(\Lat\times \C^n)\cq \Z^{2n}$. Then by~\eqref{eq:sigmalattices} and Lemma~\ref{lem:detLoon}, this determines the claimed line and section over $\Bun_{U(n)}(\EE)$ by descent. This proves the last statement in the proposition. 
\ep

From naturality and Definition~\ref{defn:detDrho}, the following is immediate:

\begin{cor}
For $\rho\colon G\to U(n)$ a representation, the sheaf $\Det(D_\rho)$ determines a sheaf of chain complexes on $\Mst_{G}$ concentrated in degree~$2n$ with trivial differential, and $\det(D_\rho)$ descends to a global section over $\Mst_{G}$, which we continue to denote by $\det(D_\rho)$.  
\end{cor}

For any representation $\rho\colon G\to U(V)$, there is a twisted $G$-equivariant elliptic Euler class of $\rho$; see~\cite[Proposition 6.3]{BET0}. 

\begin{thm}\label{thm:Eulercompare}
The section~$\det(D_\rho)$ constructs the twisted elliptic Euler class associated to $\rho$ by restriction to~$\Bun_G(\EE)$. 
\end{thm}

\bp By Proposition~\ref{prop:chaincomparison}, the pullback of $\det(D_{\rho})$ to $\Bun_G(\EE)$ agrees with the pullback of $\sigma(\tau, \mathbf{z})$ along $\Bun_G(\EE) \to \Bun_{U(n)}(\EE)$ determined by $\rho$, and this agrees with the definition of the twisted elliptic Euler class of the representation $\rho$ from~\cite{BET0}. \ep

\section{Elliptic group laws from $U(1)$-gauge fields}\label{sec:FGL}

\subsection{A group structure on fields} 
\begin{lem}\label{lem:GKunneth}
For a $G$-manifold $X$ and $H$-manifold $Y$, there are equivalences of stacks,
\beq
\F_0(X\times Y \nsq G\times H)\stackrel{\sim}{\to} \F_0(X\nsq G)\times_{\F_0(\pt)} \F_0(Y\nsq H).\label{eq:fibprod}
\eeq
\end{lem}
\bp Consider the 2-commutative diagram of stacks 
\beq
\begin{tikzpicture}[baseline=(basepoint)];
\node (A) at (0,0) {$[(X\times Y)\nsq (G\times H)]$};
\node (B) at (4,0) {$[X\nsq G]$};
\node (C) at (0,-1.2) {$[Y\nsq H]$};
\node (D) at (4,-1.2) {$\pt$}; 
\draw[->] (A) to  (B);
\draw[->] (A) to  (C);
\draw[->] (C) to (D);
\draw[->] (B) to (D);
\path (0,-.75) coordinate (basepoint);
\end{tikzpicture}\nonumber
\eeq
By naturality (see Remark~\ref{rmk:nat}) and the universal property of the fibered product, we obtain a canonical map~\eqref{eq:fibprod}. To show that this map is an equivalence, we first observe that in this case the 2-fibered product in stacks coincides with the strict fibered product. This is because the maps $\F_0(X\nsq G)\to \F_0(\pt)$ and $\F_0(Y\nsq H)\to \F_0(\pt)$ are \emph{isofibrations} for each $S$-point: any isomorphism~$\F_0(\pt)$ over~$S$ together with a lift of its source object to $\F_0(X\nsq G)$ can be lifted to an isomorphism in $\F_0(X\nsq G)$ over $S$. The isofibration property gives a weak equivalence between the 2-fibered product and the strict fibered product at each $S$-point (e.g.,~\cite[Lemma~8.2]{CoyneNoohi}) and hence an equivalence of stacks. 

Finally, we verify that the map from $\F_0(X\times Y \nsq G\times H)$ to the strict fibered product is an equivalence of stacks using the the groupoid presentations of $\F_0((X\times Y) \nsq (G\times H))$, $\F_0(X\nsq G)$ and $\F_0(Y\nsq H)$ offered by Lemma~\ref{lem:present}. In particular, the pair of isomorphisms 
$$
\Omega^1(-\times \R^{0|1};\fg\oplus\mathfrak{h})\simeq \Omega^1(-\times \R^{0|1};\fg)\oplus \Omega^1(-\times \R^{0|1};\mathfrak{h}),
$$$$ \Hom(\Z^2,\Map(\R^{0|1},G\times H))\simeq \Hom(\Z^2,\Map(\R^{0|1},G))\times \Hom(\Z^2,\Map(\R^{0|1},H))
$$
yield the desired equivalence. 
\ep

We now specialize to $T$ a torus. Naturality for the multiplication homomorphism $T\times T\to T$ and Lemma~\ref{lem:GKunneth} gives a map 
\beq
&&\F_0(\pt\nsq T)\times_{\F_0(\pt)} \F_0(\pt\nsq T) \simeq \F_0(\pt\nsq T\times T)\to \F_0(\pt\nsq T).\label{eq:torusnat}
\eeq
This map corresponds to the tensor product operation operation on $T$-bundles with connection over super tori. The multiplication operation descends along the composition of coarse quotient maps
\beq
\F_0(\pt\nsq T)\to \F_0(\pt\ncq T)\to \Mst_T\label{eq:itsanotherqupotoient}
\eeq
in the following sense.

\begin{prop} \label{prop:fieldgrouplaw} The quotient map~\eqref{eq:itsanotherqupotoient} determines a multiplication operation on $\Mst_T$ compatible with the one on $\F_0(\pt\nsq T)$ in the form of a 2-commutative diagrams of stacks:
\beq
\begin{tikzpicture}[baseline=(basepoint)];
\node (A) at (0,0) {$ \F_0(\pt\nsq T)\times_{\F_0(\pt)} \F_0(\pt\nsq T)$};
\node (C) at (0,-1.2) {$\F_0(\pt\nsq T)$};
\node (B) at (6,0) {$\Mst_T\times_\Mst \Mst_T$};
\node (D) at (6,-1.2) {$\Mst_T$};
\draw[->] (A) to  (B);
\draw[->] (A) to node [left] {$p_1,p_2,m$} (C);
\draw[->] (C) to (D);
\draw[->,dashed] (B) to node [right] {$p_1,p_2,m$} (D);
\path (0,-.75) coordinate (basepoint);
\end{tikzpicture}\nonumber
\eeq
where $p_1,p_2$ are the projections and $m$ is determined by~\eqref{eq:torusnat}.
\end{prop}
\bp
By the isofibration property discussed in the proof of Lemma~\ref{lem:GKunneth}, the 2-fibered products of stacks in the diagram above are presented by strict fiber products of generalized super Lie groupoids, using the groupoid presentation of $\F_0(\pt\nsq T)$ from Lemma~\ref{lem:present} and the definition of $\Mst_T$ as the stackification of a groupoid~\eqref{eq:Mtildef}. This allows us to construct the diagram above as the stackification of a strictly commuting diagram of generalized super Lie groupoids. But then the dashed arrow is uniquely determined by taking a quotient of the maps $p_1,p_2,m$ defined at the level of groupoids. The stackification of the strictly commuting diagram of generalized Lie groupoids is then the claimed 2-commuting diagram of stacks. 
\ep

\subsection{Compatibility with the elliptic group law}

\begin{lem}\label{lem:Mell} The map
$$
\iota\colon \Mell=\Bun_{\{e\}}(\EE)\to i^*\Mst_{\{e\}}
$$
for $\iota$ from~\eqref{eq:iota} is an isomorphism of stacks. 
\end{lem}
\bp
We compute the coarse quotient,
$$
\Mstil_{\{e\}}:=\Ftil_0(\pt)\cq \C^\times=\Lat\cq \C^\times\simeq \HH.
$$
By its construction in Lemma~\ref{lem:iota}, the map~$\iota$ is induced by a map of groupoids which in this case is the identity functor $\HH\sq\SL_2(\Z)\to \HH\sq\SL_2(\Z)$. The claimed result follows. 
\ep

We recall that
$$
\Ftil_0(\pt\ncq T) \simeq \Lat\times T\times T\times\Omega^1(-\times \R^{0|1};\mf{t})_\wz ,
$$
since the action of $T$ by gauge transformations is trivial as $T$ is abelian. The coarse quotient by the $\C^\times$-action is
\beq
\Mstil_T=\Ftil_0(\pt\ncq T)\cq \C^\times\simeq \HH\times T\times T\times\Omega^1(-\times \R^{0|1};\mf{t})_\wz.\label{eq:Mstformula}
\eeq
Restricting to manifolds, observing that $i^*$ preserves products (being a right adjoint), and using Example~\ref{ex:istartV} we find
$$
i^*\Mstil_T\simeq \HH\times T\times T\times \mf{t}_\C \times \Omega^1(-;\mf{t}). 
$$ 
By the construction of $\iota$ in Lemma~\ref{lem:iota}, the map
\beq
&&[\HH\times T\times T\sq\SL_2(\Z)]\simeq \Bun_{T}(\EE) \stackrel{\iota}{\to} i^*\Mst_T=[\HH\times T\times T\times\Omega^1(-;\mf{t})\times \mf{t}_{\C}\sq\SL_2(\Z)]\label{eq:iotaT}
\eeq
is inclusion along the map $0\colon \pt\to \Omega^1(-;\mf{t})\times \mf{t}_{\C}$ determined by the $0$-form in $\Omega^1(\pt;\mf{t})$ and the origin of $\mf{t}_\C$. 

Below we adopt the simplified notation
$$
\mathcal{O}_{\Mst_G}:=\Ch^0(\F_0(\pt\ncq G)),\qquad \mathcal{O}_{\Mstil_G}:=\Ch^0(\Ftil_0(\pt\ncq G)). 
$$
We recall that there is an isomorphism of complex analytic stacks
$$
\Bun_{T}(\EE)\simeq (\EE^\vee)^{\times {\rm rk}(T)}\simeq \EE^\vee\times_{\Mell}\cdots \times_{\Mell} \EE^\vee,
$$
between $\Bun_T(\EE)$ and the ${\rm rk}(T)$th iterated fibered product of $\EE^\vee$ over $\Mell$. In particular, their sheaves of holomorphic functions are isomorphic. The following is then a direct corollary of Theorem~\ref{thm:comparison}; see Remark~\ref{rmk:Hodge}. 

\begin{cor}\label{cor:coefficients}
When $G=T$, there is a canonical isomorphism of sheaves
\beq
\iota^{-1}i^*\mathcal{O}_{\Mst_{T}}\simeq \mathcal{O}_{(\EE^\vee)^{\times {\rm rk}(T)}}. \label{eq:isoofsheavesEvee}
\eeq
\end{cor}

\bp
By Proposition 3.8 and Example 2.3 of~\cite{BET0}, $\dEll^0_T(\pt) \simeq \mc{O}_{(\EE^\vee)^{\times {\rm rk}(T)}}$. The result follows from composing with the isomorphism of Theorem~\ref{thm:comparison}. 
\ep

\begin{prop}\label{prop:85}
For $G=U(1)$, the map $\iota$ is an injective homomorphism of commutative group objects over~$\Mell$, meaning there are commutative squares of stacks
\beq
\begin{tikzpicture}[baseline=(basepoint)];
\node (A) at (0,0) {$\EE^\vee\times_{\Mell} \EE^\vee$};
\node (C) at (0,-1.2) {$\EE^\vee$};
\node (B) at (5,0) {$i^*\Mst_{U(1)}\times_{\Mell} i^*\Mst_{U(1)}$};
\node (D) at (5,-1.2) {$i^*\Mst_{U(1)}$};
\draw[->] (A) to  (B);
\draw[->] (A) to node [left] {$p_1,p_2,m$} (C);
\draw[->] (C) to (D);
\draw[->] (B) to node [right] {$p_1,p_2,m$} (D);
\path (0,-.75) coordinate (basepoint);
\end{tikzpicture}\nonumber
\eeq
where $p_1,p_2$ are the projections and $m$ is the multiplication on $\EE^\vee$, respectively, the map induced by naturality in the multiplication $U(1)\times U(1)\to U(1)$. Moreover, this homomorphism is compatible with the isomorphism of sheaves~\eqref{eq:isoofsheavesEvee} in the sense that if $f$ denotes any of the maps $p_1, p_2$, or $m$, then we have the following commutative diagram of sheaves on $\EE^{\vee} \times_{\Mell} \EE^{\vee}$:
\beq
\begin{tikzpicture}[baseline=(basepoint)];
\node (A) at (0,0) {$f^{-1} \iota^{-1} i^* \mc{O}_{\mc{M}_{U(1)}}$};
\node (C) at (0,-1.2) {$\iota^{-1} i^* \mc{O}_{\mc{M}_{U(1) \times U(1)}}$};
\node (B) at (4,0) {$f^{-1} \mc{O}_{\EE^{\vee}}$};
\node (D) at (4,-1.2) {$\mc{O}_{(\EE^{\vee})^{\times 2}}.$};
\draw[->] (A) to node [above] {$\sim$} (B);
\draw[->] (A) to (C);
\draw[->] (C) to node [above] {$\sim$} (D);
\draw[->] (B) to (D);
\path (0,-.75) coordinate (basepoint);
\end{tikzpicture}\nonumber
\eeq
\end{prop}

\bp Using the isomorphism of (smooth) Lie groups over $\HH$,
\beq
&&\widetilde{\EE}^\vee\simeq \HH\times U(1)\times U(1)\hookrightarrow \HH\times U(1)\times U(1)\times \C\times \Omega^1(-;\R)\simeq i^*\Mstil_{U(1)}\label{eq:homomorphismthing}
\eeq
and the description of $\iota$ as inclusion at zero~\eqref{eq:iotaT}, compatibility with projections and multiplications is evident at the level of the groupoid presentations. Note that the first map above uses the isomorphism $\mb{H} \times U(1) \times U(1) \to \widetilde{\EE}^{\vee}$ given by $(\tau, e^{2\pi i x}, e^{2\pi i y}) \mapsto (\tau, z = x - \tau y)$. The second statement regarding compatibility with analytic functions follows from Corollary~\ref{cor:coefficients}: analytic functions pull back along~\eqref{eq:homomorphismthing} to holomorphic functions on $\widetilde{\EE}^\vee$. \ep

\subsection{Elliptic formal group laws from trivializations of $\Det(\mathcal{D})$}

We briefly recall some general facts about formal groups and formal group laws over~$\C$. Let $A$ be a one-dimensional complex analytic group with identity element $0\in A$ and multiplication $m\colon A\times A\to A$. A \emph{coordinate} on $A$ is a function $t \in \mc{O}_A(U)$ defined on a neighborhood $U$ of $0\in A$ which vanishes to (exactly) first order at $0$. We may expand functions in power series of~$t$ at $0\in A$, giving an injection $\mc{O}_A(U) \hookrightarrow \C[[x]]$ that sends $t$ to $x$. Similarly, we obtain an injection $\mc{O}_{A \times A}(U \times U) \hookrightarrow \C[[x_1, x_2]]$. This allows one to express the image of $m^*t$ in $\C[[x_1, x_2]]$ as $F \in \C[[x_1, x_2]]$. Commutativity, associativity, and unitality of~$m$ imply that the element $F$ defines a \emph{formal group law} over $\C$. More generally, given a complex analytic family $A\to B$ of 1-dimensional complex analytic groups over~$B$ and a coordinate $t\in \mathcal{O}(U)$ defined on a neighborhood of the identity section $0\colon B\to A$ that vanishes to first order along~$0$, we obtain a formal group law $F\in \mathcal{O}(B)[[x_1,x_2]]$. In the example of interest, take $B=\HH$ and $A=\widetilde{\EE}^\vee$. Then a choice of coordinate determines a formal group law~$F \in \mc{O}(\HH)[[x_1, x_2]]$.

Let $\Det(\mathcal{D})$ and $\det(\mathcal{D})$ denote the determinant line bundle and section on $\F_0(\pt\ncq U(1))$. Recall from Proposition~\ref{prop:chaincomparison} that under $\EE^{\vee} \stackrel{\iota}{\to} i^* \mc{M}_{U(1)}$, the analytic line $\Ch^2(\Det(\mc{D}))$ with section $\det(\mc{D})$ pulls back to $\mc{A}$ with section~$\sigma$. A trivialization of $\Ch^2(\Det(\mathcal{D}))$ on an open subsheaf $U\subset \Mstil_{U(1)}$ of $\Ch^2(\Det(\mathcal{D}))$ is \emph{analytic} if it is a trivialization in the category of $\mathcal{O}_{\Mstil_{U(1)}}$-modules, i.e., it is an isomorphism between the sheaf of sections of $\Ch^2(\Det(\mathcal{D}))$ on $U$ with $\mathcal{O}_{\Mstil_{U(1)}}$ on $U$. 

\begin{thm}\label{thm:GL} 
A choice of analytic trivialization of $C^2(\Det(\mc{D}))$ on $\Mstil_{U(1)}$ over an open subsheaf $U\subset \Mstil_{U(1)}$ containing the image of $\Mstil_{\{e\}} \hookrightarrow \Mstil_{U(1)}$ uniquely determines a trivialization of $\mathcal{A}$ in a neighborhood of $\HH \hookrightarrow \widetilde{\EE}^\vee$, and the section $\det(\mc{D})$ expressed as a function in the trivialization uniquely determines a coordinate on $\widetilde{\EE}^\vee$. Furthermore, there exists a two-variable power series $F \in \mc{O}(\HH)[[x_1, x_2]]$ such that (possibly after shrinking $U$) we have an equality of analytic functions on $U$
$$
m^* \det(\mc{D}) = F(p_1^*\det(\mc{D}), p_2^*\det(\mc{D})).
$$  
This $F$ is equal to the formal group law of $\widetilde{\EE}^\vee$ with the choice of coordinate determined by~$\sigma$ in the trivialization of $\mc{A}$. \end{thm}

\begin{proof}[Proof of Theorem~\ref{thm:GL}] 
The order of vanishing of a section of a holomorphic line bundle is independent of the choice of holomorphic trivialization. Hence, the first sentence follows from Proposition~\ref{prop:chaincomparison} and the fact that $\sigma$ determines a coordinate in the standard trivialization of $\mathcal{A}$, together with the fact that $\det(\mathcal{D})$ is the unique analytic extension of $\sigma$~\eqref{eq:sigmalattices} for $z_i$ as defined in~\eqref{eq:holocoord}. Indeed, $\sigma$ has a first-order zero at $z = 0$ from~\eqref{eq:sigmalattices}. Consequently,~$\sigma$ yields a formal group law~$F$ associated to complex analytic group $\widetilde{\EE}^\vee$,
 \beq \label{eq:fglEvee} 
 m^* \sigma = F(p_1^*\sigma, p_2^*\sigma).
 \eeq 
By Proposition~\ref{prop:85} and the fact that holomorphic functions on $\widetilde{\EE}^\vee\simeq \Bun_{U(1)}(\EE)$ uniquely extend to sections of $\mathcal{O}_{\Mstil_{U(1)}}$, we must have the same equality of functions on $U\subset \Mstil_{U(1)}$, \beq \label{eq:fglfields} m^*\det(\mc{D}) = F(p_1^*\det(\mc{D}), p_2^* \det(\mc{D})),\eeq 
proving the proposition.
\ep

\appendix

\section{Supermanifolds and superstacks}\label{appen:equivdeRham}


%

\subsection{Supermanifolds}\label{sec:smfld}

We recall that a \emph{locally ringed space} $(X,\mathcal{O})$ is a topological space~$X$ equipped with a sheaf of rings~$\mathcal{O}$ (called the \emph{structure sheaf}) whose stalks are local rings. A \emph{$k|l$-dimensional supermanifold} is a second countable, Hausdorff locally ringed space that is locally isomorphic to $(U,C^\infty_U\otimes_\C\Lambda^\bullet(\C^l))$ where $U\subset \R^k$ an open submanifold and $C^\infty_U\otimes_\C\Lambda^\bullet(\C^l)$ defines a sheaf of $\Z/2$-graded $\C$-algebras on~$U$. Here, $C^\infty_U$ is the sheaf of $\C$-valued smooth functions on $U$ and $\Lambda^\bullet(\C^l)$ is the exterior algebra. Morphisms between supermanifolds are defined as maps between the locally ringed spaces that respect the $\Z/2$-grading; these objects and morphisms comprise the category of supermanifolds, denoted~${\sf SMfld}$. Standard references include~\cite{Leites,DM}. 

For a supermanifold~$N$, we follow the common notation wherein $C^\infty(N)$ denotes the $\Z/2$-graded algebra of global sections of the structure sheaf of~$N$. We refer to elements of $C^\infty(N)$ as \emph{functions} on~$N$. The existence of partitions of unity implies that a map $N\to N'$ of supermanifolds is determined by a map $C^\infty(N')\to C^\infty(N)$ of $\Z/2$-graded algebras. 

\begin{rmk} There is a similarly defined category of supermanifolds with structure sheaves defined over~$\R$. We call these \emph{real supermanifolds} below, though we only need them when explaining differences between real supermanifolds and the category ${\sf SMfld}$. This category ${\sf SMfld}$---often called the category of cs-supermanifolds---is somewhat more subtle than the category of real supermanifolds; e.g., see Example~\ref{ex:GL1} below. However, ${\sf SMfld}$ is important both in our context of Wick-rotated quantum field theory (e.g., see~\cite[pg.~95, Example~4.9.3]{strings1}) and in connection with complex analytic elliptic curves. 
\end{rmk}

There is a faithful functor 
\beq
i\colon {\sf Mfld}\hookrightarrow {\sf SMfld},\qquad M\mapsto (M,C^\infty(-;\C))\label{eq:itsfaithful}
\eeq 
that regards a smooth manifold with its sheaf of complex-valued functions as a supermanifold. With this in mind, we will often implicitly identify a manifold with its image in supermanifolds. The functor~$i$ has a right adjoint: for a supermanifold~$N$, let $N_\red$ denote the \emph{reduced supermanifold} obtained by taking the quotient of the structure sheaf of~$N$ by its nilpotent ideal. Then~$N_\red$ has the structure of an ordinary manifold regarded as a supermanifold. Furthermore, there is an evident map of supermanifolds $N_\red\hookrightarrow N$ induced by the map of $\Z/2$-graded algebras $C^\infty(N)\twoheadrightarrow C^\infty(N)/{\rm nil}=:C^\infty(N_\red)$. Note that the sheaf of functions on $N_\red$ has a real structure (i.e., a complex-conjugation map) but this structure does not extend to the structure sheaf of~$N$. In particular, the complex conjugate of a function on a supermanifold has no meaning in general. 

\begin{ex} 
The \emph{superpoint} is the supermanifold~$\R^{0|1}$ given by the 1-point manifold equipped with the structure sheaf $C^\infty(\R^{0|1})\simeq \Lambda^\bullet(\C)$. Often we write this algebra as $C^\infty(\R^{0|1})\simeq \C[\theta]$ where $\theta$ is odd (so in particular, $\theta^2=0$). The reduced manifold of $\R^{0|1}$ is the 1-point manifold, and $\R^{0|1}$ can be regarded as a nilpotent thickening of the point. 
\end{ex}

\begin{rmk} The superpoint plays a role in the theory of supermanifolds loosely analogous to the scheme ${\rm Spec}(k[\epsilon]/\epsilon^2)$ in algebraic geometry. \end{rmk}

\begin{ex}\label{ex:Rnm}
Generalizing the previous example are the supermanifolds~$\R^{n|m}$ for natural numbers~$n$ and $m$. These are defined as the ordinary manifold $\R^n$ equipped with the structure sheaf $U\mapsto C^\infty(U)\otimes_\C\Lambda^\bullet(\C^m)$ for $U\subset \R^n$. We observe that~$(\R^{n|m})_\red=\R^n$. A choice of identification
$$
C^\infty(\R^{n|m})\simeq C^\infty(\R^n)[\theta_1,\dots, \theta_m]
$$
for odd (and hence anticommuting, square zero) elements $\theta_i$ allows us to write functions as a finite sum
\beq
f=\sum_I f_I\theta_I \qquad f\in C^\infty(\R^{n|m})\label{eq:Taylor}
\eeq
where $I=\{i_1,i_2,\dots, i_k\}$ is a multi-index, $\theta_I=\theta_{i_1}\theta_{i_2}\cdots \theta_{i_k}$. We observe that
$$
f_I=(\partial_{\theta_{i_k}}\cdots \partial_{\theta_{i_1}}f)|_{\theta_1=\cdots=\theta_m=0},
$$
and so~\eqref{eq:Taylor} is often referred to as \emph{Taylor expansion in the odd coordinates.}
\end{ex}

\begin{ex}\label{ex:Batchelor}
Batchelor's Theorem~\cite{batchelor} states that any supermanifold~$N$ is isomorphic to $(N_\red,\Gamma(\Lambda^\bullet E^\vee))$ for $E\to N_\red$ a complex vector bundle and $\Gamma(\Lambda^\bullet E^\vee)$ the associated sheaf of $\Z/2$-graded algebras on $N_\red$. This supermanifold is denoted $\Pi E$. In Example~\ref{ex:Rnm}, $\R^{n|m}$ is $\Pi\underline{\C}^m$ for $\underline{\C}^m\to \R^n$ the trivial rank $m$ vector bundle. 
\end{ex}

\subsection{Presheaves and the functor of points}
It is often convenient to enlarge the category of supermanifolds to the category of presheaves of sets on supermanifolds, 
$$
{\sf SMfld}\stackrel{Y}{\to} {\sf Fun}({\sf SMfld}^\op,{\sf Set})\qquad N\mapsto (S\mapsto {\sf SMfld}(S,N))
$$
where $Y$ is the Yoneda embedding. Maps between representable presheaves are in bijection with maps between the associated supermanifolds. Here and throughout~$S$ denotes a test supermanifold and the set $N(S):={\sf SMfld}(S,N)$ is the (set of) \emph{$S$-points} of~$N$. 

\begin{ex}\label{ex:SptRnm}
The $S$-points of~$\R^{n|m}$ are given by
\beq
&&\R^{n|m}(S)\simeq \{x_1,\dots,x_n\in C^\infty(S)^\ev,\ \theta_1,\dots,\theta_m\in C^\infty(S)^\odd\mid (x_i)_\red=\overline{(x_i)}_\red\}\label{eq:Rnmfunofpts}
\eeq
where $(x_i)_\red$ denotes the restriction to the reduced manifold, $S_\red\hookrightarrow S$: on this reduced manifold, complex-valued functions have a well-defined conjugation. 
\end{ex}

The functor of points is often just a convenient device for organizing computations. At other times, there are presheaves on supermanifolds that fail to be representable but are still geometrically interesting, as illustrated in the following three examples. 

\begin{ex} \label{defn:mappresh}
For a pair of supermanifolds $N$ and $N'$, let $\SM(N',N)$ denote the presheaf on the site of supermanifolds whose value on a test supermanifold~$S$ is the set
$$
\Map(N',N)(S):={\sf SMfld}(S\times N',N). 
$$
The failure of representability is for exactly the same reason that maps between ordinary manifolds cannot usually be represented by a (finite-dimensional) manifold. 
\end{ex}

\begin{ex} \label{ex:Cinftypre} Let $C^\infty(-)$ denote the presheaf that assigns to a supermanifold its ring of functions. This presheaf is not representable; this is in contrast to the situation for real supermanifolds, where it is represented by $\R^{1|1}$. 
\end{ex} 

\begin{ex} Let $V$ be a real vector space. Define the presheaf~$\Omega^k(-;V)$ by
$$
\Omega^k(S;V):=\Omega^k(S)\otimes V.
$$
This presheaf is not representable for any~$k$. 
\end{ex}

\begin{defn}\label{defn:funs} Define the \emph{functions} on a presheaf as the set of natural transformations from the presheaf to the sheaf~$C^\infty(-)$. 
\end{defn}
Because $C^\infty(-)$ is a presheaf of $\Z/2$-graded algebras (rather than just sets), the functions on a presheaf carry this additional structure. 


\subsection{Vector bundles on supermanifolds}\label{appen:vb}

A \emph{vector bundle} on a supermanifold is a locally free sheaf of modules over the structure sheaf; by a partition of unity argument, this sheaf is determined by its global sections as a module over the global sections of the structure sheaf. For a vector bundle $V$ on $N$, we use the notation $\Gamma(V)$ to denote the module over~$C^\infty(N)$. 

\begin{ex} The tangent space $TN$ of a supermanifold is the sheaf of derivations of its structure sheaf, $C^\infty(N)$. The sheaf of 1-forms, $\Omega^1(N)$, is the dual sheaf of~$C^\infty(N)$-modules. \end{ex}

For a vector bundle $V$ on $N$, let $\Pi V$ denote the vector bundle associated with the opposite parity when viewed as a $C^\infty(N)$-module~$\Gamma(V)$. We note that this potentially conflicts with the notation $\Pi E$ from Example~\ref{ex:Batchelor}: above $\Pi E$ denoted a supermanifold, whereas here $\Pi V$ is a vector bundle over a supermanifold. However, the presheaf on supermanifolds from Example~\ref{ex:Batchelor} is the same as the presheaf defined by the vector bundle on supermanifolds, 
defined as 
\beq
\Pi V(S):=\{x\in N(S), v\in \Gamma(S,x^*V)^\odd\}.\label{eq:oddbundle}
\eeq
More generally, define the $S$-points of a vector bundle~$V$ over an arbitrary supermanifold 
$$
V(S):=\{x\in N(S), v\in \Gamma(S,x^*V)^\ev\}.
$$
This definition of $V(S)$  introduces yet another possibility for confusion when~$V$ is a vector bundle over an ordinary manifold: an {\it a priori} similar construction would be to take the total space of the vector bundle~$V$ (as a manifold) and consider the presheaf that this total space represents; these presheaves are not the same. This is a feature that is not present for real supermanifolds. To spell out what is going on, consider three presheaves on ${\sf SMfld}$ defined by a (purely even) finite-dimensional real vector space $V$:
\begin{enumerate}
\item The presheaf $\underline{V}$ that assigns to $S$ the even sections of the trivial bundle over~$S$ with fiber~$V$, $(C^\infty(S)\otimes_\R V)^\ev\simeq (C^\infty(S)\otimes_\C V_\C)^\ev$. For ${\rm dim}(V)>0$, this presheaf is not representable. 
\item The presheaf $V$ whose value on $S$ is $\{ x\in (C^\infty(S)\otimes_\R V)^\ev \,\, | \,\, x_\red=\overline{x}_\red \}$. Equivalently, consider $V$ as an ordinary manifold and then promote it to a supermanifold; this object represents the presheaf $V$. 
\item The presheaf $V_\C$ whose value on $S$ is $\{ (z,w)\in (C^\infty(S)\otimes_\R V)^\ev\simeq (C^\infty(S)\otimes_\C V_\C)^\ev \,\, | \,\, z_\red=\overline{w}_\red \}$. Equivalently, consider $V_\C\simeq V\otimes_\R \C$ as an ordinary manifold and promote it to a supermanifold; this object represents the presheaf $V_{\C}$. As explained after~\eqref{eq:Rnmfunofpts} above, we adopt the customary notation of $(z, \bar z)$ for the pair~$(z, w)$.
\end{enumerate}
While $\underline{V}$ is typically not representable, it does admit a map from the representable presheaf~$V_{\C}$,
\beq
V_\C\twoheadrightarrow \underline{V}\qquad (z, \bar z) \in V_{\C}(S) \mapsto z \in \underline{V}(S).\label{eq:mapvb}
\eeq 
This morphism is a surjection (i.e., surjective on sets of $S$-points). This allows us to compute functions on $\underline{V}$ via descent:

\begin{lem} \label{lem:holodesc} The pullback of the sheaf of $C^\infty$-functions on $\underline{V}$ along~\eqref{eq:mapvb} is isomorphic to the sheaf of holomorphic functions on $V_\C$. In particular, 
\beq
 C^\infty(\underline{V}) \simeq \mathcal{O}(V_\C). \label{eq:descentC}
\eeq
\end{lem}

\begin{proof} By choosing a basis of $V$, it suffices to prove the statement for $V=\R$ and $V_\C=\R_\C\simeq \C$. We identify an $S$-point $S\to \C$ with $(z,\bar z)\in C^\infty(S)^{\ev}$ as in Example~\ref{ex:SptRnm}; we further identify $(z,\bar z)\in C^\infty(S)^{\ev}$ with the pullback of the standard complex coordinate functions on $\C$ along a map~$S\to \C$. In terms of the functor of points, a smooth function on $\C$ is determined by a natural transformation $\C\to C^\infty(-)$. Given a smooth function $f\in C^\infty(\C)$, this natural transformation sends $(z,\bar z)\in \C(S)$ to $f(z,\bar z)\in C^\infty(S)$. Viewing a function on $\underline{\R}$ as a natural transformation $\underline{\R}\to C^\infty$, descent along~\eqref{eq:mapvb} requires that $f(z,\bar z)\in C^\infty(S)$ is independent of $\bar z$ for every $S$-point of~$\C$. Hence, $f$ comes from a holomorphic function.
\end{proof}

We defined $\underline{V}$ above for $V$ an even vector space, but we may define the notation more generally.

\begin{defn} For $V \simeq V^0 \oplus \Pi V^1$ a $\mb{Z}/2$-graded real vector space written in terms of its decomposition into even and odd parts define, $$\underline{V} := \underline{V^0} \times \Pi V^1.$$ \end{defn}

The above definition is essentially convenient shorthand; note that per this more general definition, we have $C^{\infty}(\underline{V}) \simeq \mc{O}(V^0_{\mb{C}}) \otimes_\C \Lambda V^1_\C$.

\begin{lem} \label{lem:Cinfty}
There is an isomorphism of sheaves $C^\infty(-)\simeq \underline{\R^{1|1}}$. 
\end{lem}

\bp By definition, an $S$-point of $\underline{\R^{1|1}} \simeq \underline{\R} \times \mb{R}^{0|1}$ is a pair $(z, \theta)$ for $z \in (C^{\infty}(S))^{\text{ev}}$ and $\theta \in (C^{\infty}(S))^{\text{odd}}$. This is the same as an element of $C^{\infty}(S)$, yielding the desired lemma. \ep

\begin{prop} \label{prop:underlinesplaywell} For $S \in {\sf SMfld}$ and $V \in {\sf Vect}$ a finite-dimensional real vector space, we have that $C^{\infty}(S \times \underline{V}) \simeq C^{\infty}(S) \otimes \mc{O}(V_{\C})$. \end{prop}

\begin{proof} Recall that for $S, T \in {\sf SMfld}$, we have $C^{\infty}(S \times T) \simeq C^{\infty}(S) \otimes C^{\infty}(T)$ using the projective tensor product of Fr\'echet spaces, e.g., see~\cite[Example~49]{HST}. It remains to show that the subalgebras of $C^{\infty}(S \times V_{\C})$ are the same
$$
C^{\infty}(S \times \underline{V})\hookrightarrow C^{\infty}(S \times V_{\C})\simeq C^{\infty}(S) \otimes C^\infty(V_{\C}) \hookleftarrow C^{\infty}(S) \otimes \mc{O}(V_{\C})
$$
where the inclusion on the left is the pullback along the surjection $S\times V_{\C} \twoheadrightarrow S\times \underline{V}$. Both of these subspaces of $C^{\infty}(S \times V_{\C})$ are closed. By Lemma~\ref{lem:holodesc}, the algebraic tensor product $C^\infty(S)\otimes_{\rm alg}\mathcal{O}(V_\C)\subset C^{\infty}(S \times \underline{V})$ is a subalgebra. Hence $C^\infty(S)\otimes\mathcal{O}(V_\C)\subset C^{\infty}(S \times \underline{V})$ and it remains to prove the opposite inclusion. 

Use Grothendieck's representation of elements in a projective tensor product~\cite[page~94]{Schaefer} to write an element of $C^{\infty}(S) \otimes \mc{O}(V_{\C})$ as $\sum \lambda_i f_i g_i$ for an absolutely convergent sum with $\sum |\lambda_i| < \infty$ and $f_i \in C^{\infty}(S), g_i \in C^{\infty}(V_{\C})$ null sequences. The same argument from Lemma~\ref{lem:holodesc} shows $0 = \sum \lambda_i f_i \partial_{\overline{z}} g_i$. We may assume that the $\partial_{\overline{z}} g_i$ are either zero or linearly independent: appending each $i$ inductively, if $\partial_{\overline{z}} g_i$ is linearly dependent with the $\partial_{\overline{z}} g_j$ for $j < i$, then $g_i$ may be rewritten as $g_i' + \sum c_j g_j$ for $g_i'$ with $\partial_{\overline{z}} g_i = 0$, and one may modify the earlier coefficients $f_j \mapsto f_j + c_j f_i$ to replace $g_i$ by $g_i'$. Using the Fr\'echet topology on $C^{\infty}(\C)$ guarantees that this procedure produces another absolutely convergent sequence of the same sort. The linear relation hence yields that we must have $\partial_{\bar z} g_i=0$ for all $i$, so $C^{\infty}(S \times \underline{V}) \subset C^{\infty}(S) {\otimes} \mc{O}(V_{\C})$ as well, as desired. \end{proof}

We require some technical results computing functions on nonrepresentable presheaves involving $\Omega^1(-)$. 

\begin{prop} \label{coolprop} For any $S \in {\sf SMfld}$, the natural pullback morphism induces an isomorphism $C^{\infty}(S \times \Omega^1) \simeq C^{\infty}(S)$.\end{prop}

Before we give the proof of this general statement, we prove a version of this statement for sheaves on the site of manifolds. Doing so is logically unnecessary, but gives useful intuition for the result.

\begin{prop} The set of natural transformations $\Omega^1 \to C^{\infty}$ as sheaves on ${\sf Mfld}$ is isomorphic to the set~$\C$, i.e., the only functions on $\Omega^1$ are constant functions. \end{prop}

\bp Let a natural transformation $F\colon \Omega^1\to C^\infty$ be given, and let $c\in \C\simeq C^\infty(\pt)$ denote the function on $\pt$ associated with the composition 
$$
\pt\stackrel{0}{\to}\Omega^1\stackrel{F}{\to}C^\infty
$$
for the unique map $\pt\to \Omega^1$ associated with the 1-form $0\in \Omega^1(\pt)$. Now let $M$ be a manifold and $M\to \Omega^1$ a map associated to a differential form $\omega\in \Omega^1(M)$. Consider the composition 
$$
\pt\stackrel{x}{\to} M\stackrel{\omega}{\to} \Omega^1\stackrel{F}{\to}C^\infty. 
$$
Naturality shows that the function $F(\omega)\in C^\infty(M)$ takes the value $c$ at each point $x\in M$. Hence $F(\omega)$ is constant, and the value it takes is the same for any manifold~$M$. This identifies the natural transformation $F$ with the complex number $c$, giving the claimed bijection. \ep


\begin{proof}[Proof of Proposition~\ref{coolprop}] Consider a natural transformation $S \times \Omega^1 \stackrel{F}{\to} C^{\infty}$. We need to show that this natural transformation is equivalent data to a function on $S$. Analogous to the reduction to the 1-point manifold in the previous proof, the goal below is to leverage naturality to reduce the claim to computations involving the supermanifolds $\R^{0|n}$. This is combinatorially more demanding, but the general idea is the same. 

Let $\sigma_n: (C^{\infty})^{2n} \to \Omega^1$ be the morphism that on $T$-points sends $a_1, \cdots, a_n, b_1, \cdots, b_n \in C^{\infty}(T)$ to $\sum a_i db_i \in \Omega^1(T)$. From Lemma~\ref{lem:Cinfty}, the composition $F \circ(\id_S\times \sigma_n) \colon S\times (C^\infty)^{2n}\to C^\infty$ can be identified with an element 
$$
\sum \lambda_i f_i g_i\in C^{\infty}(S) \otimes C^{\infty}(\underline{\mb{R}^{2n|2n}})\simeq C^{\infty}(S \times (\underline{\mb{R}^{1|1}})^{2n})
$$$$
\lambda_i\in \C, \ \sum |\lambda_i| < \infty, \  f_i\in C^\infty(S), \ g_i\in C^{\infty}(\underline{\mb{R}^{2n|2n}})\simeq \mc{O}(\C^{2n})[\eta_1, \cdots, \eta_n, \theta_1, \cdots, \theta_n]
$$ 
where we have used Grothendieck's representation of functions in the projective tensor product. The result will follow from showing $\sum \lambda_i f_i g_i$ is equal to $f\in C^\infty(S)$. 

The fact that 
$$
\sum a_i db_i=\sum a_i d(b_i+w), 
$$
for complex numbers $w_1,\dots, w_n$ translates into $\sum \lambda_i f_i g_i$ being invariant under translations in the coordinates~$z_i$. We conclude that the $\sum \lambda_i f_i \partial_{z_j} g_i = 0$ vanishes. We further claim that this implies $\partial_{z_j}g_i=0$ for all $i,j$. This follows from the fact that we may assume the $\{g_i\}$ are linearly independent and the $\{\partial_{z_j} g_i\}$ are linearly independent or identically zero: appending each $g_i$ in turn, if the latest $g_i$ has $z_j$ derivative linearly dependent with the prior terms, simply resum to absorb some of the $g_i$ dependency in prior terms and make the newest $g_i$ constant in $z_j$. So we conclude that $\partial_{z_j}g_i=0$ for all $i,j$. 

We observe that $\sigma_n$ is invariant under the action by $(\C^\times)^{n}$ given on $T$-points by 
$$
(a_1,\dots, a_n,b_1,\dots, b_n)\mapsto (\mu a_1,\dots ,\mu a_n,\mu^{-1} b_1,\dots ,\mu^{-1} b_n)\qquad a_1, \cdots, a_n, b_1, \cdots, b_n \in C^{\infty}(T)
$$
as $\sum a_i db_i \in \Omega^1(T)$ is invariant under this action. From this we learn that 
$$
\sum \lambda_i f_i g_i(y_j, \eta_j, \theta_j) \equiv \sum \lambda_i f_i g_i(\mu_j y_j, \mu_j \eta_j, \mu_j^{-1} \theta_j).
$$ 
We conclude that $g_i$ must be a polynomial in the combinations $\{y_j \theta_j\}$ and $\{\eta_j \theta_j\}$. 

Next we observe that $\sigma_n$ is invariant under the precomposition action by the symmetric group~$\Sigma_n$ permuting the factors $((C^\infty)^{\times n})^{\times n}$. Hence, $\sum \lambda_i f_i g_i$ must also be invariant under this action.  Since the $\lambda_i$ and $f_i$ are automatically invariant, this requires the $g_i$ to be invariant under the action. Resumming the Grothendieck representation if necessary, this allows us to write 
$$
F \circ \sigma_n = \sum_{\ell_1,\ell_2\ge 0} f_{\ell_1, \ell_2} \Big( \sum y_j \theta_j \Big)^{\ell_1} \Big( \sum \eta_j \theta_j \Big)^{\ell_2}
$$ 
where the $f_{\ell_1,\ell_2}\in C^\infty(S)$. 

Note that upon comparing different $n$, the functions $f_{\ell_1, \ell_2} \in C^{\infty}(S)$ must agree, i.e., they are well-defined independently of $n$. More explicitly, given a map $T \stackrel{\phi}{\to} S$ and $a_i, b_i \in C^{\infty}(T)$, we have that 
\beq \label{eq:wonderful} 
F\Big(\phi, \sum a_i db_i \Big) = \sum (f_{\ell_1, \ell_2} \circ \phi) \Big( \sum a_i^{\text{even}} b_i^{\text{odd}} \Big)^{\ell_1} \Big( \sum a_i^{\text{odd}} b_i^{\text{odd}} \Big)^{\ell_2}.
\eeq 
Now, take $T = \mb{R}^{0|4n}$ coordinatized by $\eta_i, \theta_{i1}, \theta_{i2}, \theta_{i3}$ for $1 \le i \le n$; for convenience, we take these indices to be valued mod $n$, so $\eta_{n+1}$ denotes $\eta_1$. Now, take $\phi$ an arbitrary map, and consider the following choices of $a_i, b_i$: first, for $1 \le i \le n$, take $a_i = 1$ and $b_i = \theta_{i1} \theta_{i2} \theta_{i3}$, while for $n  + 1 \le i \le 2n$, take $a_i = \eta_{i}$ and $b_i = \theta_{i1} \theta_{i2} \theta_{i3}$. Alternatively, for $1 \le i \le n$, take $a_i' = \theta_{i1} \theta_{i2}, b_i' = \theta_{i3}$, and similarly for $n+1 \le i \le 2n, 2n+1 \le i \le 3n$, take $a_i' = \theta_{i1} \theta_{i3}, b_i' = \theta_{i2}$ and $a_i' = \theta_{i2} \theta_{i3}, b_i' = \theta_{i1}$, respectively. And then for $3n+1 \le i \le 6n$, instead set $a_i' = \eta_i a_{i-3n}', b_i' = b_{i-3n}'$. Then $\sum a_i db_i = \sum a_i' db_i'$, so by~\eqref{eq:wonderful}, we have the equality in $C^{\infty}(S) \otimes C^{\infty}(\underline{\mb{R}^{2n|2n}})$,
$$
\resizebox{\textwidth}{!}{$
\sum (f_{\ell_1, \ell_2} \circ \phi) \Big( \sum \theta_{i1} \theta_{i2} \theta_{i3} \Big)^{\ell_1} \Big( \sum \eta_i \theta_{i1} \theta_{i2} \theta_{i3} \Big)^{\ell_2} = \sum (f_{\ell_1, \ell_2} \circ \phi) \Big(3 \sum \theta_{i1} \theta_{i2} \theta_{i3} \Big)^{\ell_1} \Big(3 \sum \eta_i \theta_{i1} \theta_{i2} \theta_{i3} \Big)^{\ell_2}.$}
$$ 
As we vary over all $n$, we therefore find $f_{\ell_1, \ell_2} \circ \phi = 0$ unless $(\ell_1, \ell_2) = (0, 0)$. And then varying over all $\phi$, we find that all $f_{\ell_1, \ell_2}$ must identically vanish, and so $F \circ \sigma$ is just $f_{0, 0}$, a function on $C^{\infty}(S)$. As all one-forms can be written locally as in the essential image of $(C^{\infty})^{2n} \to \Omega^1$ for some $n$, we have indeed learned that $F$ is always just given by the function $f_{0, 0}$, as desired.
\end{proof}

\begin{cor} For any $\X \in {\sf Shv(SMfld)}$, the natural pullback morphism induces an isomorphism $C^{\infty}(\X \times \Omega^1) \simeq C^{\infty}(\X)$. \end{cor}

\begin{proof} For any $f \in C^{\infty}(\X \times \Omega^1)$ and $S$-point of $\X$ for $S \in {\sf SMfld}$, we obtain a pulled-back function in $C^{\infty}(S \times \Omega^1)$, which, by the prior theorem, must arise from pulling a function on $C^{\infty}(S)$. The assignment of a function in $C^{\infty}(S)$ for every $S$-point of $\X$ is exactly the datum of a function on $\X$; this returns the original function under pullback. \end{proof}

\begin{cor} \label{cor:functionsonOmega} For $V$ a real vector space and $\X \in {\sf Shv(SMfld)}$, we have $C^{\infty}(\X \times \Omega^1(-; V)) \simeq C^{\infty}(\X)$. \end{cor}

\begin{proof} Choose a basis of $V$ and apply the prior corollary iteratively. \end{proof}

\subsection{Open subsheaves and topologies} \label{sec:topologies}

\begin{defn}\label{defn:open}
An \emph{open cover} of a supermanifold $N$ is a collection of local isomorphisms $\{ U_\alpha \to N\}$ with the property that the induced map on reduced manifolds $\{(U_\alpha)_\red\to N_\red\}$ is an ordinary open cover of the manifold~$N_\red$. An \emph{open sub-supermanifold} is a local isomorphism $U\to N$ with the property that the map on reduced manifolds $U_\red\to N_\red$ is the inclusion of an open submanifold. 
\end{defn}

\begin{defn} 
A presheaf of sets on a supermanifold~$N$ is a \emph{sheaf} when it satisfies the sheaf condition for all open covers of~$N$. A presheaf of sets on the site of supermanifolds is a \emph{sheaf} when it satisfies the sheaf condition for all open covers of all supermanifolds. 
\end{defn}

\begin{defn} \label{defn:subsheaf} Given a sheaf of sets $\X$ on ${\sf SMfld}$ or ${\sf Mfld}$, a \emph{subsheaf} $\Y \subset \X$ is a subset $\Y(S) \subset \X(S)$ for each $S$ that is natural in $S$ and satisfies the sheaf condition. For subsheaves $\Y_1,\Y_2\subset \X$, an \emph{inclusion of subsheaves} $\Y_1\to \Y_2$ is a collection of inclusions $\Y_1(S)\subset \Y_2(S)$ natural in $S$.  We often write $\Y_1\subset \Y_2$ for an inclusion of subsheaves.\end{defn}

The collection of subsheaves of a sheaf $\X$ together with inclusions of subsheaves form a category. We observe that the morphism sets between subsheaves $\Y_1,\Y_2\subset \X$ are either empty or contain a single point. 


\begin{defn} \label{defn:opensub} An \emph{open subsheaf} of a sheaf $\X$ is a subsheaf $\mc{U} \subset \X$ with the property that for any representable presheaf~$S$ and any $S$-point $S\to \X$, the pullback 
\beq \begin{tikzpicture}[baseline=(basepoint)]; 
\node (A) at (0,0) {$U$}; 
\node (B) at (4,0) {$\mc{U}$}; 
\node (C) at (0,-1) {$S$}; 
\node (D) at (4,-1) {$\X$};
\draw[->] (A) to (B); 
\draw[->] (A) to (C); 
\draw[->] (B) to  (D);
\draw[->] (C) to  (D);
\path (0,-.75) coordinate (basepoint); 
\end{tikzpicture}\nonumber
\eeq
is representable by a supermanifold $U$, and the map $U\to S$ is an open sub-supermanifold. Let ${\rm Open}(\X)$ denote the category whose objects are open subsheaves of $\X$ and morphisms are inclusions, which we write e.g., as $\mc{U} \subset \mc{V} \subset \X$. An \emph{open cover} of a sheaf $\X$ is a collection of open subsheaves $\{\mc{U}_\alpha \to \X\}$ with the property that for any representable presheaf~$S$ and any $S$-point $S\to \X$, the pullback as above defines an open cover $\{U_\alpha \to S\}$.
\end{defn}

\begin{rmk}
When a sheaf is representable, open subsheaves coincide with open sub-supermanifolds and open covers coincide with usual open covers. 
\end{rmk}

For an arbitrary sheaf~$\X$, the empty sheaf and $\X$ itself always define open subsheaves. We also observe that for open subsheaves $\mathcal{U},\mathcal{V}\subset \X$, the pullback $\mathcal{U}\bigcap \mathcal{V}$ is again an open subsheaf. The following property then verifies that open subsheaves behave the same way as open subsets of a topological space. 

\begin{lem} \label{lem:unionsofopen} Given a sheaf of sets $\X$ as above and an arbitrary collection $\mathcal{U}_{\alpha} \subset \X$ of open subsheaves, the union $\mathcal{U}=\bigcup_{\alpha} \mathcal{U}_{\alpha} \subset \X$ is an open subsheaf. \end{lem}

\begin{proof} 
We recall that the subsheaf $\bigcup_{\alpha} \mathcal{U}_{\alpha} \subset \X$ is the sheafification of the presheaf $S\mapsto \bigcup_{\alpha} \mathcal{U}_{\alpha}(S)$; this is again a subsheaf because sheafification preserves monomorphisms. It is then straightforward to see that the pullback $f^*\mathcal{U}$ along $f\colon S\to \X$ is given by the open sub-representable $\bigcup_{\alpha} f^*(\mathcal{U}_{\alpha})\subset S$. Hence $\mathcal{U}$ is open. \end{proof}

\begin{defn}\label{defn:sheafonsheaf}
A \emph{presheaf} on a sheaf~$\X$ is a functor ${\rm Open}(\X)^{\rm op}\to {\sf Set}$. A presheaf is a \emph{sheaf} if it satisfies the usual sheaf condition for open subsheaves. 
\end{defn}

\begin{terminology}
When a sheaf on the site of supermanifolds plays the role of a generalized space, we often refer to it as a \emph{generalized supermanifold}. The reason for this terminology is that we work with sheaves (e.g., of functions) on generalized supermanifolds, and wish to avoid confusing statements about ``sheaves on sheaves." 
\end{terminology}

We end this subsection with a characterization of open subsheaves of~$\Omega^k(-;V)$. 

\begin{prop}\label{prop:opensofOmega} For $V$ a finite-dimensional real vector space and $k>0$, the only open subsheaves of $\Omega^k(-;V)$ are the empty sheaf and the entire sheaf. \end{prop}

\bp
Assuming that $\mathcal{U}\subset \Omega^k(-;V)$ is a nonempty subsheaf, there exists a map $S\to \Omega^k(-;V)$ such that the pullback of $\mathcal{U}$ to $S$ is a nonempty open sub supermanifold. Then we can choose a map $\pt\to S$ with image in (the reduced manifold of) this sub supermanifold so that the pullback of $\mathcal{U}$ along the unique map $0\colon \pt\to \Omega^k(-;V)$ is~$\pt$, the unique nonempty open subset of the 1-point supermanifold. For $S=M$ an ordinary manifold, using naturality and letting $x\colon \pt\to M$ run over all possible points, the pullback of $\mathcal{U}$ along $M\to \Omega^k(-;V)$ must be the open subset~$M$ of $M$. Now, for a supermanifold~$S$ consider the composition $S_{\rm red}\hookrightarrow S\to \Omega^k(-;V)$. The pullback of $\mathcal{U}$ to $S_{\rm red}$ is $S_{\rm red}$, and this must pullback from an open sub-supermanifold of~$S$. The only such sub-supermanifold is $S$ itself. Since the pullback of $\mathcal{U}$ along any map $S\to \Omega^k(-;V)$ is $S$, this shows that $\mathcal{U}=\Omega^k(-;V)$. 
\ep

\subsection{Super Lie groups}\label{sec:appensLie}

A \emph{super Lie group} is a group object in supermanifolds, meaning a supermanifold~$G$, a multiplication map $\mu\colon G\times G\to G$, identity element $e\hookrightarrow G$, and inversion $(-)^{-1}\colon G\to G$ making all the expected diagrams commute. In terms of the functor of points, a super Lie group is a representable presheaf~$G$ on supermanifolds where~$G(S)$ is endowed with the structure of a group (in sets) for every supermanifold~$S$, and the property that $G(S')\to G(S)$ is a group homomorphism for each map of supermanifolds~$S\to S'$. 

The \emph{super Lie algebra} of a super Lie group consists of vector fields invariant under the left action of the group on itself.
The Lie bracket of a super Lie algebra is skew in the graded sense; in particular, the bracket of an odd element with itself need not be zero. We often drop the adjective ``super" when describing Lie algebras and Lie groups in supermanifolds. 

\begin{ex}
Let $\E^{0|1}$ denote the Lie group gotten from $\R^{0|1}$ with its additive structure. In terms of the functor of points, this multiplication is given by
\beq
\E^{0|1}(S)\times \E^{0|1}(S)&\to& \E^{0|1}(S)\nonumber\\
(\theta,\eta)&\mapsto& (\theta+\eta)\qquad \theta,\eta\in \E^{0|1}(S)\simeq C^\infty(S)^\odd.\nonumber
\eeq
We note the multiplication is commutative. The Lie algebra of $\E^{0|1}$ is generated by an element $\partial_\theta$ that satisfies $[\partial_\theta,\partial_\theta]=0$ (which is a nontrivial relation since $\partial_\theta$ is odd). 
\end{ex}


\begin{ex}
Let $\C^\times$ denote the super Lie group associated with the ordinary Lie group $\C^\times$ with complex multiplication. As usual, the Lie algebra of $\C^\times$ is generated by commuting even elements $\partial_{\mu}$ and $\partial_{\bar \mu}$ for the standard complex coordinate $(\mu,\bar\mu)$ on $\C^\times\subset \C$. When working with the functor of points of $\C^\times$, we use the notation
\beq
\C^\times(S)\simeq \{\mu,\bar\mu\in (C^\infty(S)^{\times})^{\ev} \mid (\mu)_\red=\overline{(\bar\mu)}_\red\},
\eeq
where we emphasize that $\mu,\bar\mu$ are not complex conjugates in $C^\infty(S)$ (as this does not make sense) but instead are only conjugate on restriction to the reduced manifold of $S$; compare~\eqref{eq:r211}. Under the inclusion $\C^\times\subset \C\simeq \R^2$, this translates into the functor of points description of Example~\ref{ex:SptRnm} by taking $\mu=x+iy$, $\bar\mu=x-iy$ and
$$
\C^\times(S)\simeq \{x,y\in (C^\infty(S)^\times)^\ev \mid (x)_\red=\overline{(x)}_\red, (y)_\red=\overline{(y)}_\red\}\subset \R^2(S).
$$

\end{ex}

\begin{ex}\label{ex:E01} Let $\E^{0|1}\rtimes \C^\times$ denote the semi-direct product for the left action of $\C^\times$ on~$\E^{0|1}$,
\beq
(\mu,\bar\mu,\eta)\mapsto \bar\mu\eta\qquad (\mu,\bar\mu)\in \C^\times(S), \ \eta\in \E^{0|1}(S). \label{eq:defnCstaract}
\eeq
In the semi-direct product $\E^{0|1}\rtimes \C^\times$, the nontrivial bracket is $[\partial_{\bar\mu},\partial_\theta]=\partial_\theta$. 
\end{ex}

\begin{rmk}
One of the main super Lie groups of interest in this paper is $\E^{2|1}$; see~\eqref{eq:defnE21}. The $\C^\times$-action on $\E^{0|1}$ as~\eqref{eq:defnCstaract} defining $\E^{0|1}\rtimes \C^\times$ is chosen so the projection $\E^{2|1}\to \E^{0|1}$ and the identity map $\C^\times\to \C^\times$ determine a homomorphism
$$
\E^{2|1}\rtimes \C^\times\to \E^{0|1}\rtimes \C^\times. 
$$
\end{rmk}



\begin{ex}\label{ex:GL1} Let $\GL_1$ denote the presheaf whose value on~$S$ consists of automorphisms of the trivial line bundle on~$S$, i.e., automorphisms of $C^\infty(S)$ as a $C^\infty(S)$-module. Such automorphisms are simply invertible even functions
$$
\GL_1(S)\simeq (C^\infty(S)^\ev)^\times.
$$
One strange feature of supermanifolds with structure sheaves defined over~$\C$ (when compared with the category of real supermanifolds, with structure sheaves over~$\R$) is that~$\GL_1$ fails to be representable, as we now explain. First, since $\GL_1(\pt)\simeq \C^\times$, if $\GL_1$ where representable its reduced manifold would be $\C^\times$. This lifts to a surjective homomorphism of group objects in sheaves 
\beq
\C^\times\to \GL_1\qquad (\mu,\bar\mu)\mapsto \mu \label{eq:GL1rep}
\eeq 
that assigns to an $S$-point $(\mu,\bar\mu)\in \C^\times(S)$ the invertible even function $\mu\in C^\infty(S)^\ev$. But this map is clearly not an isomorphism as there is data lost by forgetting $\bar\mu$. So since $\GL_1$ is not isomorphic to $\C^\times$, it fails to be representable. We learned this fact from Lars Borutzky, Stephan Stolz, and Peter Teichner. For convenience, we record the ``conjugate" homomorphism
\beq
\C^\times\to \GL_1\label{eq:GL1rep2}\qquad (\mu,\bar\mu)\mapsto \bar\mu.
\eeq 

\end{ex}


\begin{ex}\label{ex:spt}
Let $\Aut(\R^{0|1})$ denote the presheaf on supermanifolds whose $S$-points are isomorphisms $S\times \R^{0|1}\to S\times \R^{0|1}$ over $S$. 
Composing isomorphisms gives $\Aut(\R^{0|1})$ the structure of a group object in presheaves. This presheaf fails to be representable for basically the same reasons as $\GL_1$ defined above. There is a (surjective) homomorphism 
\beq
\E^{0|1}\rtimes \C^\times \to \Aut(\R^{0|1})\label{eq:AutR01}
\eeq
for the left action of $\E^{0|1}$ on $\R^{0|1}$ by translation and the $\C^\times$-action defined by $\theta\mapsto \bar\mu\theta$.

\end{ex}

\subsection{Group actions on sheaves}

We recall that a $G$-action on a sheaf $\X$ consists of maps~$G(S)\times \X(S)\to \X(S)$ that satisfy the properties to define an action in sets and are natural in $S$. Our goal below is to characterize open subsheaves of the quotient sheaf~$\X\cq G$. We begin with some basic definitions. 

\begin{defn}
For a sheaf $\X$ with $G$-action, a subsheaf $\Y\subset \X$ is \emph{$G$-invariant} if for any $y\in \Y(S)$ and $g\in G(S)$, the composition
\beq
(g,y)\in G(S)\times \Y(S)\subset G(S)\times \X(S)\to \X(S)\label{eq:invariantsub}
\eeq
has image in the subset $\Y(S)\subset \X(S)$. 
\end{defn}

For a subsheaf that isn't necessarily invariant, we denote the image under the composition~\eqref{eq:invariantsub} by $g\cdot y$. 

\begin{defn}
For a sheaf $\X$ with $G$-action and a subsheaf $\Y\subset \X$, the \emph{orbit} of $\Y$ is the subsheaf of $\X$ given by the sheafification of the presheaf with $S$-points
$$
(G\cdot \Y)(S)=\{x\in \X(S)\mid \exists (g,y)\in G(S)\times \Y(S), \ x=g\cdot y\}
$$
using the notation defined above. 
\end{defn}

\begin{rmk} Sheafification preserves monomorphisms, so that $G\cdot \Y$ is a subsheaf of $\X$. \end{rmk}

\begin{lem} \label{lem:orbitopen} Let $\X$ be a sheaf with the action of a Lie group $G$. Given an open subsheaf $U \subset \X$, its orbit $G \cdot U \subset \X$ is a $G$-invariant open subsheaf. \end{lem}

\begin{proof} By construction, the orbit of a subsheaf is a $G$-invariant subsheaf.  For an $S\to \X$, the pullback of $G\cdot U$ to~$S$ consists of the union over the set of maps $g\colon S\to G$ of the pullback of the image of~$U$ under~$g$. Being union of open subsets of $S$, the result is an open subset. \ep

\begin{lem} \label{lem:sheafquotient} The coarse quotient $\X\cq G$ has as $S$-points a set subject to an equivalence relation. An element of the set is given by the data of (1) an open cover $\{S_{\alpha}\}$ of $S$ (2) morphisms $x_\alpha\colon S_{\alpha} \to \X$ and (3) maps $g_{\alpha \beta}\colon S_\alpha\bigcap S_\beta\to G$. These data are required to satisfy the properties $x_\alpha=g_{\alpha\beta} \cdot x_\beta$ and $g_{\alpha\beta}g_{\beta\gamma}=g_{\alpha\gamma}$. The equivalence relation is $(x_\alpha,g_{\alpha\beta})\sim (x_{\gamma},g_{\gamma\delta})$ if there is a mutual refinement of covers $\{S_{\alpha'}\}$ and $g_{\alpha'}\colon S_{\alpha'}\to G$ such the action of $g_{\alpha'}$ sends the restrictions of the $x_\alpha$ to the $x_\gamma$ and this is compatible with the restrictions of the $g_{\alpha\beta}$ and $g_{\gamma\delta}$. 
 \end{lem}

\begin{proof} The quotient in presheaves defined by the universal property has the explicit description in terms of $S$-points $x\in \X(S)$ modulo an equivalence relation, 
$$
(\X/G)_{\rm pre}(S):=\X(S)/\sim \qquad  x\sim x' \ {\rm if} \ \exists g\in G(S) \ {\rm such\ that} \ g\cdot x=x'.
$$
Since sheafification is a left adjoint, the sheafification of this presheaf has the correct universal property in the category of sheaves. It is easy to see that the claimed data modulo the equivalence relation gives an $S$-point of the sheafification. \end{proof}

\begin{prop} \label{prop:opensofquot} Given a sheaf $\X$ with $G$-action as above, consider the canonical quotient map $\pi: \X \to \X \cq G$. There is an equivalence of categories $\mathrm{Open}(\X)^G \stackrel{\sim}{\to} \mathrm{Open}(\X\cq G)$ between the $G$-invariant open subsheaves of $\X$ and the open subsheaves of $\X\cq G$. 
\end{prop}

\begin{proof} Given an open subsheaf $U\subset \X/G$, $\pi^{-1}(U)\subset \X$ is clearly open and $G$-invariant. 
For $U \subset V \subset \X\cq G$, we see there is a unique map $\pi^{-1}(U) \subset \pi^{-1}(V) \subset \X$, and this determines a functor $\mathrm{Open}(\X\cq G)\to \mathrm{Open}(\X)^G$. 

For the candidate inverse, if $\widetilde{U} \subset \X$ is a $G$-invariant open, we claim that $\widetilde{U}\cq G \to \X\cq G$ determines an open subsheaf. First observe that 
\beq
\begin{tikzpicture}[baseline=(basepoint)];
\node (A) at (0,0) {$\widetilde{U}$};
\node (B) at (3,0) {$\X$};
\node (C) at (0,-1.25) {$\widetilde{U}\cq G$};
\node (D) at (3,-1.25) {$\X\cq G$};
\draw[->] (A) to (B);
\draw[->] (B) to (D);
\draw[->] (A) to (C);
\draw[->] (C) to (D);
\path (0,-.75) coordinate (basepoint);
\end{tikzpicture}\nonumber
\eeq
is a pullback square. Next, using the notation from Lemma~\ref{lem:sheafquotient}, we observe that a map $S\to \X/G$ can be locally lifted to maps $S_\alpha\to \X$. Then the pullback of $\tilde{U}/G\to \X/G$ agrees with the the pullback of $\tilde{U}\to \X$ on $S_\alpha$. The pullback of $\tilde{U}$ is open since $\tilde{U}$ is an open subsheaf. Hence, the pullback of $\tilde{U}/G$ to $S$ is locally open (subordinate to $\{S_\alpha\}$), and therefore is open. For $\tilde{V}\subset \tilde{U}\subset \X$ an inclusion of $G$-invariant open subsheaves, we get a morphism $\tilde{V}/G\subset \tilde{U}/G\subset \X/G$. All together, this gives a functor $\mathrm{Open}(\X)^G\to \mathrm{Open}(\X\cq G)$. It is easy to see that the compositions with the previously constructed functor $\mathrm{Open}(\X\cq G)\to \mathrm{Open}(\X)^G$ give identities. We conclude the claimed equivalence of categories. 
\ep

Recall the full and faithful functor $i\colon {\sf Mfld}\to {\sf SMfld}$ from~\eqref{eq:itsfaithful} and~\S\ref{sec:sheafsection}. 

\begin{lem}\label{lem:quotientcom} Let $G$ be an ordinary Lie group, regarded as a super Lie group. Then for any sheaf $\X\in {\sf Sh}_{\sf SMfld}$ with a $G$-action, there is an isomorphism of coarse quotients
$$
(i^*\X)/G\simeq i^*(\X/G). 
$$
\end{lem}

\bp
Let $S\in {\sf Mfld}$ be a test manifold, so that $(i^*\X)(S)=\X(S)$, where we regard $S$ as a supermanifold using~$i$. Using the description of the the quotient sheaf from Lemma~\ref{lem:sheafquotient}, an $S$-point of $(i^*\X)/G$ is a cover $\{S_\alpha\}$ of~$S$, elements $x_\alpha\in \X(S_\alpha)$ and $g_{\alpha\beta}\in G(S_\alpha\cap S_\beta)$ so that $g_{\alpha\beta}\cdot x_\alpha=x_\beta\in \X(S_\alpha\cap S_\beta)$. Such data determine the same $S$-point of $(i^*\X)/G$ if (after passing to a common refinement of open covers) there exist $g_\alpha \in \X(S_\alpha)$ such that $g_\alpha \cdot x_\alpha=x_\alpha'$ compatible with the transition data $g_{\alpha\beta}$ and $g_{\alpha\beta}'$. An $S$-point of~$i^*(\X/G)$ is a $i(S)$-point of $\X/G$, where we regard $S$ as a supermanifold using~$i$. We can identify such $S$-points with precisely the same data under the same equivalence relation. 
\ep

\subsection{Super families of homomorphisms $\Z^2\to G$}

We will require some results about $S$-families of homomorphisms $S\times \Z^2\to G$ for $G$ an ordinary compact Lie group and $S$ a supermanifold. The relevant results combine aspects of ordinary smooth geometry and super geometry. We begin with some standard results for smooth manifolds. 

For $A\subset B$ a submanifold, let $\mathsf{Nbhd}_A(B)$ be a tubular neighborhood of~$A$ in~$B$.

\begin{thm}[Slice theorem] Given an isometric action of a compact Lie group $G$ on a Riemannian manifold $M$ and a point $x \in M$ with stabilizer $G_x$, the inclusion of the orbit $G / G_x \hookrightarrow M$ extends to a local diffeomorphism $G \times_{G_x} \mathsf{Nbhd}_0(T_xM/T_x(G \cdot x)) \to \mathsf{Nbhd}_{G \cdot x}(M)$ given by $(g, X) \mapsto g \exp_x(X)$.
 \end{thm}

The conclusion of the theorem above yields a diffeomorphism of manifolds. Hence, {\it a fortiori}, we have an isomorphism of supermanifolds. This gives the following version of the slice theorem for super families.

\begin{thm}[Slice theorem for $S$-points]\label{thm:slice} Given an isometric action of a compact Lie group $G$ on a Riemannian manifold $M$ and a point $x \in M$, for any $\epsilon > 0$, let $N_{\epsilon}$ denote the ball of radius $\epsilon$ about $0$ in the vector space $T_xM / T_x(G \cdot x)$. Regard $M$ and $G$ as supermanifolds. Then there exists some $\delta > 0$ such that the open sub supermanifold $B_{\delta}(x) \subset M$ satisfies the following: for any $y\in U(S)$ for $S \in {\sf SMfld}$, there exists an open cover $\{S_{\alpha}\}$ of $S$ and $S_{\alpha}$-points of $g_\alpha \in G(S_{\alpha}), X_\alpha \in N_{\epsilon}(S_\alpha)$ such that $y|_{S_\alpha}=g_\alpha \exp_x(X_{\alpha})\in U(S_\alpha)$. \end{thm}

\begin{proof} The conclusion of the slice theorem yields a diffeomorphism $\mathsf{Nbhd}_{G \cdot x}(M) \simeq (G \times \mathsf{Nbhd}_0(T_xM/T_x(G \cdot x))) / G_x$. Shrinking the neighborhoods (and $\epsilon$) as necessary, we may suppose the neighborhood on the right hand side simply is $N_{\epsilon}$, i.e. $\mathsf{Nbhd}_{G \cdot x}(M) \simeq (G \times N_{\epsilon}) / G_x$. In other words, $\mathsf{Nbhd}_{G \cdot x}(M)$ is the base of a principal $G_x$-bundle with total space $G \times N_{\epsilon}$. Hence, it suffices to note that $S$-points of quotients of principal bundles may, locally on $S$, be lifted to the total space; this assertion in turn follows by using local trivializations of the principal bundle. Finally, choose some $\delta > 0$ such that $B_{\delta}(x) \subset \mathsf{Nbhd}_{G \cdot x}(M)$. \end{proof}

\begin{lem} \label{lem:centslices} Suppose $G$ a compact Lie group with given element $h \in G$. Then for any $S \in {\sf SMfld}$ and any $X \in \mf{g}(S)$ sufficiently small, locally on $S$ there exists some small $Y \in \mf{g}^h(S)$ such that $he^X$ is conjugate to $he^Y$ as $S$-points of $G$. More precisely, given $\epsilon > 0$, there exists $\delta > 0$ such that if $X \in (B_{\delta}(\mf{g}))(S)$, then there locally on $S$ exists $Y \in (B_{\epsilon}(\mf{g}^h))(S)$ with $he^X$ conjugate to $he^Y$. The above balls are taken centered at the origin.  \end{lem}

\begin{proof} This is an extension of~\cite[Lemma A.8]{BET0} to $S$-families of supermanifolds rather than just ordinary manifolds. We apply Theorem~\ref{thm:slice} with $M=G$ for the conjugation action by $G$ and set $x=h \in G$. We freely apply the identification $T_xM := T_hG \simeq \mf{g}$ via left-multiplication by $h$ (or $h^{-1}$ in the reverse direction), and use that the Riemannian exponential map for a bi-invariant metric on $G$ coincides with the exponential on its Lie algebra. 
We have the direct sum decomposition $T_hG \simeq \mf{g} \simeq T_h(G \cdot h) \oplus \mf{g}^h$. Hence, we may identify~$\mathrm{Nbhd}_0(T_hG / T_h(G \cdot h))$ with a neighborhood of $0$ in $\fg^h$. The desired conclusion then follows from Theorem~\ref{thm:slice}: given the $S$-point $he^X\in G(S)$, there locally exist group elements $g_\alpha \in G(S_\alpha)$ and $Y_{\alpha} \in \mf{g}^h(S_\alpha)$ such that $he^X = g_{\alpha} he^{Y_{\alpha}} g_{\alpha}^{-1}$ as $S_\alpha$-points of $G$. 
\end{proof}


We recall a useful result of Block and Getzler.

\begin{lem}[\cite{BlockGetzler}, Lemma 1.3] \label{lem:BG} For $G$ a compact Lie group and $h \in G$, there exists some $\delta > 0$ such that if $X \in B_{\delta}(\mf{g}^h)$, then $C(he^X) \subset C(h)$. \end{lem}

We generalize this pointwise statement to one for $S$-families.

\begin{lem} \label{lem:superBG} For $G$ a compact Lie group and $h \in G$, there exists some $\delta > 0$ such that if $X \in B_{\delta}(\mf{g}^h)(S)$ and $g \in G(S)$ such that $g$ and $he^X\in G(S)$ commute, then $g \in C(h)(S)$. \end{lem}

\begin{proof} Choose a faithful representation $G \hookrightarrow U(N)$. Then $C_G(h) = G \cap C_{U(N)}(h)$, so it suffices to show the statement for $G = U(N)$. Choose a basis so that $h$ is diagonal, with $N$ diagonal entries given by $k_1$ consecutive copies of $\lambda_1$, followed by $k_2$ consecutive copies of $\lambda_2$, and so forth, where all the $\lambda_i$ are distinct. We then claim we may choose $\delta$ sufficiently small such that if $|\mu_i|, |\mu_j| < \delta$, then $\lambda_i e^{\mu_i}, \lambda_j e^{\mu_j}$ are still distinct for all distinct $i, j$. Indeed, as $S$-points of $U(N)$ 
we may directly compute the condition for $g$ to commute with the $S$-family of diagonal matrices $he^X$; we find that $g$ must still be an $S$-family of block-diagonal matrices specified by an $S$-point of $C(h)$ inside $U(N)$. \end{proof}

\begin{lem} \label{lem:ellcentslices} For $G$ a compact Lie group and $h = (h_1, h_2)\in G^{\times 2}$ a pair of commuting elements of $G$, then for any $\epsilon > 0$, there exists $\delta > 0$ such that if $(X_1, X_2) \in B_{\delta}(\mf{g})^{\times 2}(S)$ such that the $h_ie^{X_i}$ continue to commute, there exist (locally on $S$) $Y_1, Y_2 \in B_{\epsilon}(\mf{g}^{h})(S)$ such that $h_ie^{X_i}$ are simultaneously conjugated to $h_ie^{Y_i}$. \end{lem}

\begin{proof} Given an $\epsilon$, we choose $\delta$ to be the minimum of the $\delta$ produced in Lemma~\ref{lem:centslices} and the values specified by the Lemma~\ref{lem:superBG} for $h_1$ and $h_2$. We now apply the above Lemma~\ref{lem:centslices} to assume, without loss of generality, that $X_1 \in B_{\delta}(\mf{g}^{h_1})(S)$ by simultaneous conjugation of $X_1, X_2$. (Note that as the metrics here are induced by the conjugation-invariant Killing form, conjugation preserves neighborhoods of the form $B_{\delta}$.) But now, as the $S$-points $h_2e^{X_2}$ and $h_1e^{X_1}$ commute, and $X_1$ is sufficiently small in the sense of Lemma~\ref{lem:superBG}, we conclude that the $S$-point $h_2e^{X_2}$ must in fact be an $S$-point of $C(h_1)$. As we already have $h_2 \in C(h_1)$, this statement is equivalent to $e^{X_2} \in C(h_1)(S)$, i.e., $X_2 \in \mf{g}^{h_1}(S)$. 

We again apply Lemma~\ref{lem:centslices}, now to the group $C(h_1)$, in order to simultaneously conjugate $X_1, X_2$ with the conclusion that $X_2 \in \mf{g}^{h_2}(S)$. As we were conjugating within $C(h_1)$, we still have that $X_1, X_2 \in \mf{g}^{h_1}(S)$. Applying Lemma~\ref{lem:superBG} to the commuting points $h_1e^{X_1}, h_2e^{X_2}$, where now we use that $X_2$ is sufficiently small, we conclude that $h_1e^{X_1} \in C(h_2)(S) \implies X_1 \in \mf{g}^{h_2}(S)$. Hence both $X_1, X_2$ are $S$-points of $\mf{g}^{h_1} \cap \mf{g}^{h_2} = \mf{g}^h$, as desired. \end{proof}

Informally, the previous lemma states that any deformation of $g\in G$ by a Lie algebra element $X$ is conjugate to a deformation through $Y$ in the Lie algebra of the centralizer of~$g$, and the same for commuting pairs of elements. 

We finally record a few further results on the structure of conjugacy classes and conjugacy-invariant functions. 

\begin{prop}\label{prop:grp} Let $G$ be a compact Lie group. Given $h \in G(S)$, there exists some $\delta > 0$ such that for $X, X' \in (B_{\delta}(\mf{g}^h))(S)$, the set of $S$-points which conjugates $he^X$ to $he^{X'}$ is contained in $C(h)(S)$.  \end{prop}

\begin{proof} As in Proposition A.15 of~\cite{BET0} and the above Lemma~\ref{lem:superBG}, it suffices to embed $G$ in $U(N)$ and then prove the statement for $U(N)$. We once again choose $\delta$ as in Lemma~\ref{lem:superBG} such that the eigenvalues stay distinct, so that the result follows by explicit matrix computation of $S$-families. \end{proof}

\begin{prop} \label{prop:ellgrp} For $G$ a compact Lie group with $h = (h_1, h_2)$ a commuting pair of $S$-points of $G$, there exists some $\delta > 0$ such that for $(X_1, X_2), (X_1', X_2')$ commuting pairs of elements in $(B_{\delta}(\mf{g}^h))(S)$, then the set of $S$-points which simultaneously conjugate $h_ie^{X_i}$ to $h_ie^{X_i'}$ is contained in $C(h)(S)$.\end{prop}

\begin{proof} This is a direct corollary of the previous Proposition~\ref{prop:grp}. \end{proof}

\begin{prop} \label{prop:torus} For $G$ a connected Lie group with maximal torus $T<G$ and Weyl group $W=N(T)/T$, the ring of $W$-invariant holomorphic functions on $\mf{t}_{\C}$ is equivalent to the ring of $G$-invariant holomorphic functions on $\mf{g}_{\C}$. Moreover, the same result holds for holomorphic functions defined on $\epsilon$-balls about the origin, i.e. $\mc{O}(B_{\epsilon}(\mf{t}_{\C}))^W \simeq \mc{O}(B_{\epsilon}(\mf{g}_{\C}))^G$, where balls are taken about the origin. Taking $\epsilon \to 0$, the same result holds for germs of holomorphic functions at the origin. \end{prop}

\begin{proof} This statement, known as Chevalley restriction, is well-known for algebraic functions. One elegant way to observe the above statements for holomorphic functions, holomorphic functions defined on balls, or germs of holomorphic functions is to use the geometry of the Grothendieck--Springer resolution. See, for example,~\cite[Theorem~3.1.38]{ChrissGinzburg} for details. \end{proof}

%

\subsection{Super Lie groupoids and stacks}\label{appen:groupoids}

We recommend the appendix of~\cite{HKST} for background on stacks on the site of supermanifolds; we review some basic aspects of the theory below. 

A \emph{super Lie groupoid} $\mathcal{G}=\{G_1\rightrightarrows G_0\}$ consists of a supermanifold $G_0$ of objects, a supermanifold $G_1$ of morphisms, source and target maps $s,t\colon G_1 \to G_0$, a unit map $G_0\to G_1$ and a composition map $c\colon G_1\times_{G_0}G_1\to G_1$. The source map is required to be a submersion so that the fibered product $G_1\times_{G_0}G_1$ exists in supermanifolds. These data are required to satisfy the axioms of a groupoid object. 

A \emph{functor} $\mathcal{G}\to \mathcal{H}$ is the data of maps of supermanifolds $G_i\to H_i$ for $i=0,1$ satisfying the axioms of a functor. A \emph{natural transformation} between functors $\mathcal{G}\rightrightarrows \mathcal{H}$ is the data of a map of supermanifolds $G_0\to H_1$ satisfying the axioms of a natural transformation. 

Let ${\sf Grpd}$ denote the strict 2-category whose objects are super Lie groupoids, 1-morphisms are functors between super Lie groupoids, and 2-morphisms are natural transformations between functors.

\begin{ex}
Let a super Lie group~$G$ act on a supermanifold~$M$. We can form the \emph{action groupoid} denoted~$M\sq G$, whose objects are~$M$ and morphisms are~$G\times M$. The source map $s\colon G\times M\to M$ is the projection, and the target map $t\colon G\times M\to M$ is the action map. The unit $M\to G\times M$ is the inclusion along the identity element $e\in G$. 
\end{ex}

The basic objects of differential (super) geometry can be lifted to Lie groupoids by asking for structures on $G_0$ that are equivariant for the action by $G_1$. We give a few definitions demonstrating this. 

\begin{defn} A \emph{function} on a Lie groupoid is the data of a function $f\in C^\infty(G_0)$ satisfying the property $s^*f=t^*f$ on $G_1$. \end{defn}

\begin{defn} A \emph{vector bundle} on a Lie groupoid is the data of a vector bundle $V\to G_0$ and isomorphism of vector bundles $s^*V\simeq t^*V$ on $G_1$. This isomorphism of vector bundles must be compatible with composition, i.e., it satisfies a cocycle condition on $G_1\times_{G_0}G_1$. \end{defn}

\begin{defn} Similar to the previous definition, a \emph{sheaf} on a Lie groupoid is a sheaf $\F$ on $G_0$ and an isomorphism of sheaves $s^*\F\simeq t^*\F$ on $G_1$ satisfying a condition on $G_1\times_{G_0}G_1$. \end{defn}

The data of a super Lie groupoid can be repackaged in terms of the functor of points, giving a functor from supermanifolds to groupoid (in sets). This perspective leads to a useful enlargement of the category of super Lie groupoids.

\begin{defn}
A \emph{generalized super Lie groupoid}, denoted $\{ G_1 \rightrightarrows G_0 \}$, is a pair of sheaves on the site of supermanifolds with the source, target, unit and composition maps as before, which together define a functor ${\sf SMfld}^\op \to {\sf Grpd}$ to groupoids, given by~$S \mapsto \{ G_1(S) \rightrightarrows G_0(S) \}$.
\end{defn}

\begin{rmk} 
Functions, vector bundles, and sheaves can be defined over generalized super Lie groupoids by rephrasing the former definitions in terms of the functor of points. 
\end{rmk}



\begin{defn}
A \emph{stack} on the site of supermanifolds is a category fibered in groupoids over supermanifolds satisfying descent with respect to open covers (see Definition~\ref{defn:open}). In particular, for each~$S$, a stack assigns a groupoid called the \emph{$S$-points} of the stack, and to each map $S\to S'$, a stack assigns a functor between the associated groupoids. 
\end{defn}

\begin{ex} Any sheaf of sets defines a stack, where we regard a set as a \emph{discrete} category, having only identity morphisms. In particular, supermanifolds (regarded as sheaves) define stacks. A stack that is isomorphic to a supermanifold (regarded as a stack) is called \emph{representable}. \end{ex}

\begin{ex}\label{ex:stackify} The $S$-points of a generalized super Lie groupoid $\{G_1\rightrightarrows G_0\}$ define a \emph{prestack}: the fibered category whose fiber at $S$ is the groupoid $\{G_1(S)\rightrightarrows G_0(S)\}$ need not satisfy descent for open covers. There is a 2-functor called {\it stackification} that is left adjoint to the forgetful functor from stacks to prestacks; see~\cite{HKST} for details.  \end{ex}

\begin{defn}
The stackification of the prestack associated to a super Lie groupoid is the \emph{underlying stack} of the super Lie groupoid. We use square brackets to denote this stack, e.g.,~$[G_1\rightrightarrows G_0]$ or $[M\sq G]$.
\end{defn}

\begin{rmk} Although we distinguish notationally between super Lie groupoids and stacks, we will often commit the abuse of notation of identifying a supermanifold with its super Lie groupoid and associated stack, meaning the supermanifold~$M$ denotes the super Lie groupoid~$\{M\rightrightarrows M\}$ and its underlying stack~$[M\rightrightarrows M]$. \end{rmk}

\begin{rmk}\label{rmk:grpdtostack}
We review a useful trick that allows one to understand much of the geometry of a stack $[G_1\rightrightarrows G_0]$ in terms of the geometry of the super Lie groupoid $\{G_1\rightrightarrows G_0\}$. Using the adjunction between the forgetful functor and stackification, the groupoid of maps $[G_1\rightrightarrows G_0]\to \X$ for $\X$ a stack is equivalent to the groupoid of maps from the prestack associated to $\{G_1\rightrightarrows G_0\}$ to the stack $\X$. For example, if $\X=C^\infty(-)$, this shows that functions on the stack $[G_1\rightrightarrows G_0]$ are the same as invariant functions on the super Lie groupoid $\{G_1\rightrightarrows G_0\}$. For $\X=\Vect$, the stack classifying vector bundles, this shows that a vector bundle on $[G_1\rightrightarrows G_0]$ is the same data as a groupoid equivariant vector bundle on~$\{G_1\rightrightarrows G_0\}$. 
\end{rmk}

\begin{ex} The stackification of the prestack defined by the action groupoid $M\sq G$ has objects over $S$ defined as pairs $(P, \phi)$ for $P$ a principal $G$-bundle over $S$ and $\phi: P \to M$ a $G$-equivariant map. Morphisms are given by isomorphisms $f\colon (P_1,  \phi_1) \to (P_2, \phi_2)$, i.e., an isomorphism of $G$-bundles $f\colon  P_1 \to P_2$ over a base change $S\to S'$ such that $\phi_2 \circ f = \phi_1$. We observe that the prestack defined by~$M\sq G$ can be viewed as the sub-prestack of $[M\sq G]$ consisting of \emph{trivial} $G$-bundles. 
\end{ex}


\begin{ex}\label{defn:stackquoconn} For a $G$-manifold $M$, define the (generalized) super Lie groupoid
$$
M\nsq G:=(\Omega^1(-;\mathfrak{g})\times M)\sq G
$$
for the $G$-action on $M$ and the $G$-action on the presheaf $\Omega^1(-;\mf{g})$ by gauge transformations. To spell this out, the value on~$S$ consists of a map $S\to M$, and $A\in \Omega^1(S;\mathfrak{g})$. A morphism over~$S$ is determined by an $S$-point of~$G$, which acts on the map $S\to M$ by postcomposition, and acts on the $\mf{g}$-valued 1-form as
\beq
A\mapsto \Ad_{g^{-1}}A+g^{-1}dg. \label{eq:gaugedef}
\eeq
There is a forgetful map $M\nsq G\to M\sq G$ that forgets the $\mf{g}$-valued 1-form. In fact, $M\nsq G$ can be realized as the pullback (which is also a 2-pullback)
\beq
\begin{tikzpicture}[baseline=(basepoint)];
\node (A) at (0,0) {$M\nsq G$};
\node (B) at (3,0) {$\pt\nsq G$};
\node (D) at (0,-1.25) {$M\sq G$};
\node (E) at (3,-1.25) {$\pt\sq G$};
\draw[->] (A) to  (B);
\draw[->] (A) to (D);
\draw[->] (B) to (E);
\draw[->] (D) to (E);
\path (0,-.75) coordinate (basepoint);
\end{tikzpicture}\label{eq:pullbackofMnsqG}
\eeq 
where the horizontal arrows come from the map $M\to \pt$ and the vertical arrows forget the data of the connection. 
The stackification of the prestack associated to $M\nsq G$ has objects over $S$ defined as triples $(P, \nabla, \phi)$ for $P$ a principal $G$-bundle over~$S$, $\nabla$ a connection on $P$, and $\phi\colon P \to M$ a $G$-equivariant map. Morphisms are given by isomorphisms $f\colon (P_1, \nabla_1, \phi_1) \to (P_2, \nabla_2, \phi_2)$, i.e., an isomorphism of principal bundles $f\colon P_1 \to P_2$ over a base change $S\to S'$ such that $\phi_2 \circ f = \phi_1$ and $\nabla_1$ is carried to~$\nabla_2$ under $f$. The prestack defined by~$M\nsq G$ can be viewed as the sub-prestack of $[M\nsq G]$ consisting of \emph{trivial} $G$-bundles with connection. By stackifying the diagram~\ref{eq:pullbackofMnsqG}, the stack $[M\nsq G]$ can also be described as a 2-pullback.\end{ex}


\begin{ex} \label{ex:assoc}
Given a representation $\rho\colon G\to U(V)$, there is an associated vector bundle~$V$ on $[\pt\nsq G]$ with connection~$\rho(V)$. Using the description of $[\pt\nsq G]$ as the stack underlyling the super Lie groupoid $\Omega^1(-;\mf{g})\sq G$ and Remark~\ref{rmk:grpdtostack}, the vector bundle with connection $(V,\rho(\nabla))$ is determined by a $G$-equivariant vector bundle with invariant connection on the quotient groupoid $\Omega^1(-;\mf{g})\sq G$. At an $S$-point $A\in \Omega^1(S;\mf{g})$, this vector bundle is the $C^\infty(S)$-module $C^\infty(S)\otimes V$ with connection $d+\rho(A)$ where $\rho(A)$ is the image of $A\in \Omega^1(S;\mf{g})$ under the map $\Omega^1(S;\mf{g})\to \Omega^1(S;\End(V))$ determined by $d\rho\colon \mf{g}\to \End(V)$. For an $S$-point of $G$, we obtain an isomorphism of vector bundles over~$S$ given by the automorphisms of the $C^\infty(S)$-module $C^\infty(S)\otimes V$ from $S\to G\stackrel{\rho}{\to} \Aut(V)$, and the connection is transformed according to the uniquely determined $G$-action on $\Omega^1(S;\End(V))$ for which the map $\Omega^1(S;\mf{g})\to \Omega^1(S;\End(V))$ is~$G$-equivariant. 
\end{ex}

Generalizing the assignment $\mathcal{G}\mapsto [\mathcal{G}]$ that sends a super Lie groupoid to a stack from Example~\ref{ex:stackify}, there is a 2-functor from the 2-category of super Lie groupoids, functors and natural transformations to the bicategory of stacks, 
$$
[-]\colon {\sf LieGrpd}\to {\sf Stacks}. 
$$
The image of this functor consists of \emph{differentiable stacks}, as we review presently. 

\begin{defn}\label{defn:epi}
 A map of stacks $\X\to \Y$ is an \emph{epimorphism} if for every $S\to \Y$ there exists an open cover $\{S_\alpha\}$ of $S$ and a 2-commuting diagram 
\beq
\begin{tikzpicture}[baseline=(basepoint)];
\node (A) at (0,0) {$\coprod_\alpha S_\alpha$};
\node (B) at (3,0) {$\X$};
\node (D) at (0,-1.25) {$S$};
\node (E) at (3,-1.25) {$\Y.$};
\draw[->] (A) to  (B);
\draw[->] (A) to (D);
\draw[->] (B) to (E);
\draw[->] (D) to (E);
\path (0,-.75) coordinate (basepoint);
\end{tikzpicture}\nonumber
\eeq 
More informally, an epimorphism $\X\to \Y$ has local sections at every $S$-point of $\Y$. 
\end{defn}

\begin{defn} A map of stacks $\X\to \Y$ is \emph{representable} if for every $S\to \X$ for $S \in {\sf SMfld}$ there is a 2-pullback
\beq
\begin{tikzpicture}[baseline=(basepoint)];
\node (A) at (0,0) {$T$};
\node (B) at (3,0) {$\X$};
\node (D) at (0,-1.25) {$S$};
\node (E) at (3,-1.25) {$\Y$};
\draw[->] (A) to  (B);
\draw[->] (A) to (D);
\draw[->] (B) to (E);
\draw[->] (D) to (E);
\path (0,-.75) coordinate (basepoint);
\end{tikzpicture}\nonumber
\eeq
for a representable stack $T\simeq S\times_\Y \X$. 
A representable map of stacks $\X\to \Y$ is a \emph{submersion} if in addition the maps $T\to S$ are submersions for every $S\to \X$. A representable map $\X\to \Y$ is \emph{open} if $T\to S$ is the inclusion of an open sub supermanifold for every $S\to \X$ (compare Definition~\ref{defn:opensub}). 
\end{defn}

\begin{defn} \label{defn:genatlas} A \emph{differentiable stack} is a stack $\X$ for which there exists a representable submersion $X\to \X$ for $X\in {\sf SMfld}$ a representable stack. In this case, such a map $X\to \X$ is called an \emph{atlas}. A \emph{generalized atlas} is a representable submersion $X\to \X$ where $X$ is a sheaf (of sets) on the site of supermanifolds.  \end{defn}


\begin{lem}\label{lem:genatlas} An epimorphism $U \to \X$ is a generalized atlas if $U$ is a sheaf of sets and the maps $U\times_\X U\rightrightarrows U$ are submersions. \end{lem}

\begin{proof} 
The proof is identical to~\cite[Lemma~2.2]{BehrendXu}. \end{proof}

%
%

\begin{defn} A \emph{groupoid presentation} of a stack $\X$ is a generalized super Lie groupoid $\{G_1\rightrightarrows G_0\}$ whose underlying stack is equivalent to~$\X$, i.e., $\X\simeq [G_1\rightrightarrows G_0]$. When $\{G_1\rightrightarrows G_0\}$ is a super Lie groupoid (i.e., $G_0$ and $G_1$ are supermanifolds) $\X\simeq [G_1\rightrightarrows G_0]$ is a \emph{super Lie groupoid} presentation.  \end{defn}

\begin{prop} \label{prop:appenpresentation} A generalized atlas $\mc{U} \to \mc{X}$ determines the groupoid presentation~$\mc{U} \times_{\mc{X}} \mc{U} \rightrightarrows \mc{U}$. \end{prop}

\begin{proof} The proof is identical to~\cite[Proposition~2.2]{BehrendXu}. \end{proof}

\begin{ex}\label{eg:defnMap} We generalize the mapping object from Example~\ref{defn:mappresh} to stacks. Given stacks $\X$ and $\Y$, let $\Map(\X,\Y)$ denote the stack that assigns to~$S$ the groupoid of maps of stacks $S\times \X\to \Y$, and pulls these groupoids back (using precomposition) along base changes $S'\to S$. 
\end{ex}

\section{Supergeometry of the equivariant de~Rham complex}\label{sec:appende}

The goal of this section is to explain an interpretation of equivariant de~Rham cohomology using the geometry of supermanifolds. 
We claim no originality. The supergeometric interpretation of the de~Rham complex as functions on the odd tangent bundle was likely first observed by Kontsevich~\cite[\S7.2]{Kontsevich}, \cite{AKSZ}; our approach is influenced by the more recent account by Hohnhold, Kreck, Stolz and Teichner~\cite{HKST}. The connection between moduli spaces of gauge fields on $0|1$-dimensional supermanifolds and equivariant de~Rham cohomology was inferred by Mathai--Quillen~\cite{MathaiQuillen}. It has been applied in various places in the physics literature, notably the work of Kalkman~\cite{Kalkman}, Dijkgraaf--Moore~\cite[\S3]{DijkgraafMoore} and Blau--Thompson~\cite[\S2]{BlauThompson}. Variations on similar themes have since appeared in the mathematics literature, sometimes independently. The presentation below draws from the textbook by Guillemin--Sternberg~\cite{GuilleminSternberg}, the concise treatment by Wu~\cite[\S4.1-4.2]{Wu}, and the announcement of results by Han--Schommer-Pries--Stolz--Teichner~\cite{ST11,Stolzclass}. We were introduced to this intersection of super geometry and equivariant cohomology through an inspiring course taught by Stolz and Teichner in Bonn in the fall of~2009. 


\begin{rmk} \label{rmk:semi}\label{rmk:leftright} Our conventions for group actions are as follows. We follow the usual convention in geometry (as in in field theory~\cite{Freed5,strings1}) where symmetries act on the left. For a group $G$ acting on $\Sigma$ we get a left action on mapping objects
$$
G\times \SM(\Sigma,M)\to \SM(\Sigma,M),\qquad (f,\phi)\mapsto \phi\circ g^{-1},\quad g\in G, f\in \SM(\Sigma,M)
$$ 
by inversion composed with the precomposition $G$-action. Furthermore, for a left $G$-action on a (super) manifold, we obtain a right $G$-action on functions (or differential forms)  by pullback. 
 Finally, we recall a useful fact: an action of a semidirect product $H\ltimes G$ on $M$ is equivalent to an $H$-action on $M$ and a $G$-action $G\times M\to M$ that is $H$-equivariant for the diagonal $H$-action on $G\times M$.
\end{rmk}

\subsection{Models for equivariant de~Rham cohomology}\label{eq:equivariantdeRham}

We give a brief overview of equivariant de~Rham cohomology; for more detailed introductions we recommend the survey~\cite{MeinrenkenCartan} and the textbook~\cite{GuilleminSternberg}. Let $G$ be a compact Lie group acting on a manifold~$M$ on the left. This induces a right $G$-action on $\Omega^\bullet(M)$ by pullback of forms. For $X\in \mathfrak{g}_\C$, let $L_X$ denote the corresponding derivation on $\Omega^\bullet(M)$ and $\iota_X$ the contraction with~$X$. Let $d$ be the de~Rham differential on $\Omega^\bullet(M)$. These satisfy the relations,
\beq
\begin{array}{cc}
[d,d]=0,\ \ [L_X,d]=0,\ \ [d,\iota_X]=L_X,\ \ [\iota_X,\iota_Y]=0 \\ 
\phantom{\null}[L_X,\iota_Y]=\iota_{[X,Y]}, \ \  [L_X,L_Y]=L_{[Y,X]}=L_{-[X,Y]}
\end{array}\label{eq:Gstar}
\eeq
where the brackets are graded commutators for $L_X$ even and $\iota_X, d$ odd so that, e.g., $[d,\iota_X]=d\iota_X+\iota_Xd$ and $[L_X,\iota_Y]=L_X\iota_Y-\iota_YL_X$. The sign $[L_X,L_Y]=L_{[Y,X]}=L_{-[X,Y]}$ comes from the fact that $G$ acts on differential forms on the right; see Remark~\ref{rmk:leftright}. Let
$$
W^\bullet(\mathfrak{g}):=\bigoplus_{\bullet=2k+l}\Sym^k(\mf{g}_{\C}^\vee) \otimes \Lambda^l(\mathfrak{g}^\vee_\C)
$$ 
denote the Weil algebra of $\fg$. It carries an action of a Lie super algebra with generators and relations specified as in~\eqref{eq:Gstar} as well as a $G$-action that lifts the $\fg$-action determined by~$L_X$, $X\in \fg$. The (graded) tensor product $\Omega(M)\otimes W(\mf{g})$ again carries an action by this Lie super algebra and has a compatible $G$-action. In particular, the de~Rham differential and differential on $W(\fg)$ combine to a differential denoted $d_W$ on $\Omega(M)\otimes W(\mf{g})$. Define the \emph{horizontal subalgebra} $(\Omega(M)\otimes W(\mf{g}))_{\rm hor}\subset \Omega(M)\otimes W(\mf{g})$ as the subalgebra on which the contraction operators $\iota_X$ act by zero for all $X\in \mathfrak{g}_\C$. Define the \emph{basic subalgebra}, $(\Omega(M)\otimes W(\mf{g}))_{\rm bas}= (\Omega(M)\otimes W(\mf{g}))_{\rm hor}^G$ as the $G$-invariant subspace. In particular, $L_X$ acts by zero on the basic subalgebra for all $X\in \mathfrak{g}_\C$.  The chain complex $((\Omega(M)\otimes W(\mf{g}))_{\rm bas},d_W)$ is the \emph{Weil model} for equivariant cohomology. The Cartan model starts with the graded algebra
\beq
\bigoplus_{i+2j=\bullet} \Omega^i(M)\otimes \Sym^j (\mf{g}_{\C}^\vee)\simeq \bigoplus_{i+2j=\bullet} \Sym^j (\mf{g}_{\C}^\vee;\Omega^i(M)) \label{eq:Cartanvs}
\eeq
that we identify with polynomial functions on $\mathfrak{g}_{\C}$ valued in~$\Omega^\bullet(M)$. Define a differential~$d_C$ whose value on an element~$\alpha$ is 
\beq
(d_C\alpha)(X)=d(\alpha(X))-\iota_X\alpha(X),\qquad X\in \mathfrak{g}_\C\label{eq:Cartandiff}
\eeq
where $d$ is the ordinary de~Rham differential on forms, and $\iota_X$ denotes contraction with the vector field on~$M$ associated to~$X$ under the infinitesimal action of~$G$ on~$M$. For often use the shorthand notation for this differential,
\beq
\label{eq:Cartandiff2}
d_C=d-\iota. 
\eeq
We observe that $d_C^2=0$ if we restrict to the subalgebra of~\eqref{eq:Cartanvs} invariant under the operators~$L_X$ acting on differential forms by Lie derivative and on $\Sym(\fg_\C^\vee)$ by the coadjoint representation (acting on the right). This gives the chain complex $((\Omega(M)\otimes \Sym(\mf{g}_{\C}^\vee))^{G},d_C)$, which is the \emph{Cartan model} for equivariant cohomology. It is quasi-isomorphic to $((\Omega(M)\otimes W(\mf{g}))_{\rm bas},d_W)$, and these chain complexes compute the $G$-equivariant cohomology of $M$ with complex coefficients,
$$
\H_G(M;\C):=\H(M\times_G EG;\C)\simeq \H((\Omega(M)\otimes W(\mf{g}))_{\rm bas},d)\simeq \H((\Omega^\bullet(M)\otimes \Sym(\mf{g}_{\C}^\vee))^G,d_C).
$$

\begin{rmk} \label{rmk:pi0G}Let $G_0\triangleleft G$ denote the connected component of the identity. Then $\pi_0(G)\simeq G/G_0$ acts on $\H_{G_0}(M;\C)$, and $\H_G(M)\simeq {\H_{G_0}(M)}^{\pi_0(G)}$. 
\end{rmk}
\subsection{The superspace $\SM(\R^{0|1},M)$ and the de~Rham complex}\label{appen:TM1}

Recall the presheaf $\SM(\R^{0|1},M)$ from Example~\ref{defn:mappresh}. See~\cite[Proposition~3.1]{HKST} for a detailed proof of the following. 

\begin{lem}\label{lemPIT} Let $M$ be a manifold. A choice of coordinate on $\R^{0|1}$ determines an isomorphism of presheaves $\SM(\R^{0|1},M)\simeq \Pi TM$ that is natural in~$M$. In particular, $\Map(\R^{0|1},M)$ is representable, and its algebra of functions is naturally isomorphic to the super algebra of differential forms on~$M$, $C^\infty(\SM(\R^{0|1},M))\simeq \Omega^\bullet(M)$. \end{lem}

\begin{proof}[Proof sketch]
Taylor expansion in a choice of odd coordinate $\theta\in C^\infty(\R^{0|1})$ yields isomorphisms of super vector spaces
$$
C^\infty(S\times \R^{0|1})\simeq C^\infty(S)[\theta]\simeq C^\infty(S)\oplus \theta \cdot C^\infty(S).
$$
Hence a map $S\times \R^{0|1}\to M$ is determined by the data of a linear map $x\oplus \psi\colon C^\infty(M)\to C^\infty(S)\oplus \theta \cdot C^\infty(S) \simeq C^\infty(S\times \R^{0|1})$. The homomorphism property for $x\oplus \psi$ implies that $x\colon C^\infty(M)\to C^\infty(S)^\ev$ is an algebra homomorphism and $\psi\colon C^\infty(M)\to C^\infty(S)^\odd$ is an odd derivation with respect to~$x$. Such data is precisely an $S$-point of the odd tangent bundle,~$\Pi TM$, in~\eqref{eq:oddbundle}. By definition, functions on $\Pi TM$ are differential forms on~$M$ (see Example~\ref{ex:Batchelor}). 
\ep


\begin{rmk} Naturality can be rephrased as a natural isomorphism $\Pi T\simeq\Map(\R^{0|1},-)$, regarding these as functors from manifolds to supermanifolds. 
\end{rmk}

\begin{rmk} 
One often writes $x+\theta \psi$ for the composition $C^\infty(M)\stackrel{x\oplus \psi}{\to} C^\infty(S)\oplus \theta\cdot C^\infty(S)\simeq C^\infty(S\times \R^{0|1})$, and specifies an $S$-point of $\Pi TM\simeq \Map(\R^{0|1},M)$ as $(x,\psi)\in \Pi TM(S)$. 
\end{rmk}

\begin{rmk}\label{rmk:useful}
It will be useful to have an explicit identification between differential forms on~$M$ and natural transformations of presheaves $\SM(\R^{0|1},M)\to C^\infty(-)$, following Definition~\ref{defn:funs}. Given a differential form $\alpha=fdx_1dx_2\cdots dx_k\in \Omega^k(M)$, we obtain a natural transformation that assigns to an $S$-point $(x,\psi)\in \Pi TM(S)$ the function on~$S$:
$$
\alpha(x,\psi)=x(f)\psi(x_1)\cdots \psi(x_k)\in C^\infty(S).
$$
It is easy to verify this formula determines a natural transformation from $S$-points of $\SM(\R^{0|1},M)$ to $C^\infty(-)$, and that all such natural transformations are generated by these (it suffices to check this locally in~$M$). 
\end{rmk}

Recall the definition of the Lie group $\E^{0|1}\rtimes \C^\times$ from Example~\ref{ex:E01} and its action on~$\R^{0|1}$ from~\eqref{eq:AutR01}. Consider the left action
\beq
\E^{0|1}\rtimes \C^\times\times \Map(\R^{0|1},M)\to \Map(\R^{0|1},M),\label{eq:PiTaction}
\eeq
 by precomposition. 

\begin{lem}\label{lem:PiT}
The pullback of $f\in \Omega^k(M)\hookrightarrow \Omega^\bullet(M)\simeq C^\infty(\Map(\R^{0|1},M))$ along~\eqref{eq:PiTaction} is 
\beq
&&f\mapsto \bar\mu^{-k} f-(\bar\mu\eta) (\bar\mu^{-(k+1)}df)=\bar\mu^{-k} (f-\eta df)\in C^\infty(\E^{0|1}\rtimes \C^\times)\otimes \Omega^\bullet(M)\label{eq:itsaformula}
\eeq
for the standard complex coordinates $(\mu,\bar\mu)$ on $\C^\times$ and the standard odd coordinate $\eta$ on~$\E^{0|1}$. We use the isomorphism $C^\infty(\E^{0|1}\rtimes \C^\times\times \Map(\R^{0|1},M))\simeq C^\infty(\E^{0|1}\rtimes \C^\times)\otimes \Omega^\bullet(M)$ in~\eqref{eq:itsaformula}. 
\end{lem}

\begin{proof}
Taking care with the signs for our action conventions (see Remark~\ref{rmk:leftright}), an $S$-point formula for the action~\eqref{eq:PiTaction} is
\beq
\E^{0|1}\times \Pi TM\to \Pi TM,&\quad& (\eta,x,\psi)\mapsto (x-\eta\psi,\psi) \label{eq:PiTaction1}\\
\C^\times \times \Pi TM\to \Pi TM,&\quad& (\mu,\bar\mu, x,\psi)\mapsto (x,\bar\mu^{-1}\psi),\label{eq:PiTaction2}
\eeq
for $(x,\psi)\in \Pi TM(S)\simeq \Map(\R^{0|1},M)(S)$, $\eta\in \E^{0|1}(S),$ and $(\mu,\bar\mu)\in \C^\times(S).$ 
Using Remark~\ref{rmk:useful}, one finds that the $\E^{0|1}$-action is $f\mapsto f-\eta df$ for $d$ the de~Rham differential, and the pullback along the~$\C^\times$-action is $f\mapsto \bar\mu^{-k} f$ for  $f\in \Omega^k(M)$. The result follows. 
\ep

\begin{cor} 
The derivative at zero of the $\E^{0|1}$-action in~\eqref{eq:PiTaction} is the odd derivation $-d$ on $\Omega^\bullet(M)\simeq C^\infty(\Map(\R^{0|1},M)$ where $d$ is the de~Rham differential. The derivative at zero of the $\C^\times$-action in~\eqref{eq:PiTaction} is the even derivation $-\deg$ on $\Omega^\bullet(M)$ where $\deg(f)=kf$ for~$f\in \Omega^k(M)$.
\end{cor}


\subsection{The superstack $\SM(\R^{0|1},[M\nsq G])$}\label{sec:equivdeRham}

We recall the generalized super Lie groupoid $M\nsq G$ from Example~\ref{defn:stackquoconn} and its underlying stack $[M\nsq G]$, as well as the mapping stack construction from Example~\ref{eg:defnMap}. We refer to~\cite[{\S}A.2]{Stoffel} for background on group actions on stacks and quotients by these actions. 

\begin{lem} \label{lem:01present}There are groupoid presentations
\beq
&&\SM(\R^{0|1},[M\nsq G])\simeq [\Map(\R^{0|1},M)\times \Omega^1(-\times \R^{0|1};\mf{g})\sq \Map(\R^{0|1},G)]\label{eq:grpd01}
\eeq
and 
$$
\SM(\R^{0|1},[M\nsq G])\sq\E^{0|1}\rtimes \C^\times \simeq [\Map(\R^{0|1},M)\times \Omega^1(-\times \R^{0|1};\mf{g})\sq (\E^{0|1}\rtimes \C^\times)\ltimes \Map(\R^{0|1},G)]
$$
where the $\E^{0|1}\rtimes \C^\times$-action on the stack $\SM(\R^{0|1},[M\nsq G])$ is through its precomposition action on $\R^{0|1}$, and the quotient by this action is taken in stacks. 
\end{lem}
\begin{proof}
Consider the map
\beq
&&\Map(\R^{0|1},M)\times \Omega^1(-\times \R^{0|1};\mf{g})\to \Map(\R^{0|1},[M\nsq G])\label{eq:01epi}
\eeq
that sends an $S$-point $\phi\in \Map(\R^{0|1},M)(S)$, $A\in \Omega^1(S\times \R^{0|1};\fg)$ of the source to the trivial $G$-bundle $G\times S\times \R^{0|1}\to S\times \R^{0|1}$ with connection determined by $A$ and $G$-equivariant map  given by the composition 
\beq
G\times S\times \R^{0|1}\stackrel{\id_G\times \phi}{\longrightarrow} G\times M\stackrel{{\rm act}}{\to} M.\label{eq:01maptoM}
\eeq
We claim that~\eqref{eq:01epi} is an epimorphism of stacks (see Definition~\ref{defn:epi}). This follows from the fact $\R^{0|1}$ is contractible, so for any principal $G$-bundle $P\to S\times \R^{0|1}$ there exists an open cover $\{S_\alpha\}$ of~$S$ and trivializations of $G$-bundles $P|_{S_\alpha\times \R^{0|1}}\simeq G\times S\times \R^{0|1}$.

With the epimorphism~\eqref{eq:01epi} in hand, by Lemma~\ref{lem:genatlas} the claimed groupoid presentation~\eqref{eq:grpd01} follows if the pullback of~\eqref{eq:01epi} along itself is $\Map(\R^{0|1},M)\times \Omega^1(-\times \R^{0|1};\mf{g})\times \Map(\R^{0|1},G)$ and the structure maps in this pullback are the projection and action maps. The 2-pullback consists of $F\colon (P,\nabla,\phi)\Rightarrow (P',\nabla',\phi')$ for trivial $G$-bundles $P,P'$ with connections $\nabla,\nabla'$ and $G$-equivariant maps $\phi\colon P\to M$, $\phi'\colon P'\to M$ that are related by an isomorphism $F$ of $G$-bundles with connection compatible with the map to $M$. The data of such an isomorphism of trivial $G$-bundles is an $S$-point of~$\Map(\R^{0|1},G)$. Given a source $S$-point of this isomorphism, the target is determined by acting on the connection by gauge transformations and on the map to $M$ through modifying the map $\id_G$ in~\eqref{eq:01maptoM} by $g\in \Map(\R^{0|1},G)(S)$. Hence, the 2-pullback is as required and this proves the groupoid presentation~\eqref{eq:grpd01}. 

For the second statement, we recall that stackification is a left adjoint. Hence, it preserves 2-colimits and in particular quotients. Therefore it suffices to compute the quotient in generalized super Lie groupoids, where the claimed equivalence is clear.
\ep

Ultimately our goal is to characterize the $(\E^{0|1}\rtimes \C^\times)\ltimes \Map(\R^{0|1},G)$-action on the algebra of functions $
C^\infty(\Map(\R^{0|1},M)\times \Omega^1(-\times \R^{0|1};\fg))$ using the action in the groupoid presentation from Lemma~\ref{lem:01present}. It turns out (see Lemma~\ref{lem:stalk}) that it suffices to calculate this action for functions on a subsheaf associated with the inclusion
\beq
\underline{\fg}\times \Pi \fg&\hookrightarrow& \Omega^1(-\times \R^{0|1},\fg),\label{eq:WZinclusion}\\
 (\underline{\fg}\times \Pi \fg)(S)\ni (X,\chi) &\mapsto& d\theta\otimes \chi+\theta d\theta\otimes X\in \Omega^1(S\times \R^{0|1})\otimes \fg\simeq \Omega^1(S\times \R^{0|1};\fg). \nonumber
\eeq
Using the notation for $S$-points 
$$
(X,\chi)\in (\underline{\fg}\times \Pi \fg)(S), \ (x,\psi)\in \Pi TM(S), \ \eta\in \E^{0|1}(S), \ (\mu,\bar\mu)\in \C^\times(S), \ g\in G(S), \gamma\in \Pi \fg(S)
$$
define the actions
\beq
\E^{0|1}\times \underline{\fg}\times \Pi \fg\to \underline{\fg}\times \Pi \fg,&\quad& (\eta,X,\chi)\mapsto (X,\chi+\eta X) \label{eq:act1}\\
\C^\times\times \underline{\fg}\times \Pi \fg\to \underline{\fg}\times \Pi \fg, &\quad& (\mu,\bar\mu,X,\chi)\mapsto (\bar\mu^{-2}X,\bar\mu^{-1}\chi)\label{eq:act2}\\
G\times \underline{\fg}\times \Pi \fg\to \underline{\fg}\times \Pi \fg,&\quad& (g,X,\chi)\mapsto (\Ad_{g^{-1}}X,\Ad_{g^{-1}}\chi)\label{eq:act3}\\
\Pi\fg\times\underline{\fg}\times \Pi \fg\to \underline{\fg}\times \Pi \fg,&\quad& (\gamma,X,\chi)\mapsto (X-[\gamma,\chi],\chi+\gamma)\label{eq:act4}\\
G\times \Pi TM\to \Pi TM,&\quad& (g,x,\psi)\mapsto (gx,g_*\psi)\label{eq:act5}\\
\Pi\fg\times \Pi TM\to \Pi TM,&\quad& (\gamma,x,\psi)\mapsto (x,\psi+\gamma_M).\label{eq:act6}
\eeq
In~\eqref{eq:act6}, $\gamma_M$ is the odd derivation associated with the infinitesimal $G$-action by $\gamma$.

\begin{lem}\label{lem:01action}

The action of $(\E^{0|1}\rtimes \C^\times)\ltimes \Pi TG$ on $\Pi TM \times \Omega^1(-\times \R^{0|1};\fg)$ from the groupoid in Lemma~\ref{lem:01present} restricts along the inclusion 
$$
\Pi TM \times \underline{\fg}\times \Pi \fg\hookrightarrow \Pi TM \times \underline{\fg}\times \Pi \fg\times \Omega^1(-\fg)\simeq \Pi TM \times \Omega^1(-\times \R^{0|1};\fg)
$$ 
determined by~\eqref{eq:WZinclusion}. On this restriction, the action is given by~\eqref{eq:PiTaction1}-\eqref{eq:PiTaction2} and~\eqref{eq:act1}-\eqref{eq:act6}.
\end{lem}
\bp
Using Remark~\ref{rmk:semi} and standard facts about group actions (e.g., $\Ad_{g^{-1}}[\gamma,\chi]=[\Ad_{g^{-1}}\gamma,\Ad_{g^{-1}}\chi]$), one finds that the claimed formulas do indeed give an action of the semidirect product $(\E^{0|1}\rtimes \C^\times)\ltimes \Map(\R^{0|1},G)$ on $\Pi TM\times \underline{\fg}\times\Pi \fg$. We observe that the actions we wish to compare are diagonal on $\Pi TM\times (\underline{\fg}\times \Pi \fg)$ and $\Pi TM\times \Omega^1(-\times \R^{0|1};\fg)$ so we can compare the actions on each factor. 

To see that the $\E^{0|1}\rtimes \C^\times$-action is the correct one, we invoke Lemma~\ref{lem:PiT} for the action on $\Pi TM$. The action~\eqref{eq:act1} and~\eqref{eq:act2} on $\underline{\fg}\times \Pi \fg$ is the restriction of the action on $\Omega^1(S\times \R^{0|1};\fg)$ by the computation
\beq
&&\phantom{B} \theta d\theta \otimes X+d\theta\otimes \chi\mapsto \bar\mu^{-2}(\theta-\bar\mu\eta)d\theta\otimes X+\bar\mu^{-1} d\theta \otimes \chi=\bar\mu^{-2} \theta d\theta X+\bar\mu^{-1}d\theta (\chi+\eta X). \label{eq:actiononconnection}
\eeq
The action of $\Map(\R^{0|1},G)\simeq G\ltimes \Pi \fg\simeq \Pi TG$ on $\Pi TM$ comes from applying the functor $\Pi T\simeq \Map(\R^{0|1},-)$ to the action map $G\times M\to M$, which recovers the claimed formulas~\eqref{eq:act5} and~\eqref{eq:act6}. Finally, the action on $\underline{\fg}\times \Pi \fg\subset \Omega^1(-\times \R^{0|1};\fg)$ by restriction of gauge transformations can be described as follows. Consider $ge^{\theta\gamma} \in \Pi TG(S)\simeq (G\ltimes\Pi \fg)(S)\simeq \Map(\R^{0|1},G)(S)$ and
$$
(ge^{\theta\gamma})^{-1}d(ge^{\theta\gamma})=d\theta \otimes\gamma,\quad \Ad_{-\theta\gamma}=-\theta[\gamma,-]
$$
where $d$ is the $C^\infty(S)$-linear (or fiberwise) de~Rham differential on $S\times \R^{0|1}$. By a short computation using the gauge transformation formula~\eqref{eq:gaugedef}, we obtain~\eqref{eq:act4} and~\eqref{eq:act5}, proving the lemma. 
\ep

\subsection{A super-geometric description of the Weil model}

\begin{lem}\label{lem:stalk}
There is an isomorphism of algebras
\beq
C^\infty(\Omega^1(-\times \R^{0|1};\fg))\simeq C^\infty(\underline{\fg}\times \Pi \fg)\simeq \mc{O}(\fg_\C)\otimes \Lambda \fg^\vee_\C.\label{eq:functionsonforms}
\eeq
\end{lem}
\bp
Define the sheaf $S\mapsto \Omega^1_S(\R^{0|1};\fg)$ by the exact sequence
\beq
&&0\to\Omega^1(S;\fg)\stackrel{p^*}{\to} \Omega^1(S\times \R^{0|1};\fg)\to \Omega^1_S(\R^{0|1};\fg)\to 0 \label{eq:exact}
\eeq
where $p\colon S\times \R^{0|1}\to S$ is the projection. 
Using the standard coordinate $\theta$ on $\R^{0|1}$, this sequence splits by pulling back along the inclusion $i_0\colon S\hookrightarrow S\times \R^{0|1}$ at $\theta=0$. We observe the isomorphism of sheaves $\Omega^1_{(-)}(\R^{0|1};\fg)\simeq \underline{\fg}\times \Pi\fg$ given by
\beq
(\underline{\fg}\times \Pi\fg)(S)\ni (X,\chi)\mapsto \theta d\theta \otimes X+d\theta\otimes \chi\in \Omega^1_S(\R^{0|1};\fg).\label{eq:chinote}
\eeq
This gives
\beq
&&\Omega^1(-\times \R^{0|1};\fg)\simeq \Omega^1(-;\fg)\times \Omega^1_S(\R^{0|1};\fg)\simeq \Omega^1(-;\fg)\times\underline{\fg}\times \Pi\fg.\nonumber
\eeq
The claimed isomorphism then follows from Lemma~\ref{lem:holodesc} and Corollary~\ref{cor:functionsonOmega}.
%
\ep

\begin{rmk} In the same spirit as Remark~\ref{rmk:useful}, there is an explicit identification between functions~\eqref{eq:functionsonforms} and natural transformations $\Omega^1(-\times \R^{0|1};\fg)\to C^\infty(-)$. It suffices to specify this assignment for $z\in \fg_\C^\vee \subset \mathcal{O}(\fg_\C)\subset C^\infty(\Omega^1(-\times \R^{0|1};\fg))$ and $\zeta\in \Lambda^1\fg^\vee_\C\subset \Lambda\fg_\C^\vee$. We may also assume that the $S$-point of $\Omega^1(-\times\R^{0|1};\fg)$ is of the form $A=d\theta \chi+\theta d\theta X\in \Omega^1(S\times \R^{0|1};\fg)$ using the notation of~\eqref{eq:chinote}. Then $z,\zeta\colon \Omega^1(-\times \R^{0|1};\fg)\to C^\infty(-)$ are given by
$$
z(A)=\langle z,X\rangle\in C^\infty(S)^\ev\qquad \zeta(A)=\langle \zeta,\chi\rangle\in C^\infty(S)^\odd
$$
using from the canonical pairing $\langle-,-\rangle\colon \fg_\C^\vee\otimes \fg_\C\to \C$. 
\end{rmk}

\begin{rmk}
We recall that in the category of supermanifolds with structure sheaves defined over~$\C$, the left-invariant vector fields on a Lie group are a complex Lie super algebra. 
\end{rmk}

\begin{lem}\label{lem:superalgebra}
Given a (real) Lie algebra $\fg$, consider the Lie super algebra over~$\C$ with generators $\iota_X,L_X$ and $d$ for $X\in \fg_\C$ satisfying the relations~\eqref{eq:Gstar}. This Lie super algebra is isomorphic to the Lie super algebra of the normal subgroup
$$
\E^{0|1}\ltimes \Map(\R^{0|1},G)\triangleleft (\E^{0|1}\rtimes \C^\times)\ltimes \Map(\R^{0|1},G).
$$
The Lie super algebra of $(\E^{0|1}\rtimes \C^\times)\ltimes \Map(\R^{0|1},G)$ adds even generators $N$ and $\bar N$, where $N$ is central, and 
$$
[\bar N,d]=d,\quad [\bar N,L_X]=0,\quad [\bar N,\iota_X]=-\iota_X.
$$
\end{lem}

\bp
Using that $\Map(\R^{0|1},G)\simeq \Pi \fg\rtimes G$, we see that the vector space underlying the Lie algebra of $\E^{0|1}\ltimes \Map(\R^{0|1},G)$ is $\Pi \C \oplus \fg_\C \oplus \Pi \fg_\C$, and we take the obvious isomorphism with the vector space underlying the Lie algebra from~\eqref{eq:Gstar}. It remains to check that this isomorphism is compatible with the brackets. We recall that the Lie super algebra of a semidirect product $H\ltimes K$ as a vector space is $\mathfrak{h}\oplus\mathfrak{k}$, but the direct sum bracket is modified by $[h,k]=h\cdot k$, where $h\cdot k$ denotes action of $h\in \mathfrak{h}$ on $k\in \mathfrak{k}$ by a derivation coming from the derivative of the $H$-action on $K$ at the identity. From this it is a straightforward computation to verify the brackets are as claimed in~\eqref{eq:Gstar}. 

For the second statement, the generators $N$ and $\bar N$ correspond to the invariant vector fields $\partial_\mu$ and $\partial_{\bar \mu}$ on $\C^\times$, respectively. The claimed brackets again follow directly from the semidirect product description. 
\ep

Using the groupoid presentation from Lemma~\ref{lem:01present}, the sheaf of functions on the stack $\Map(\R^{0|1},[M\nsq G])$ carries an action by $\E^{0|1}\ltimes \Map(\R^{0|1},G)$. By Lemma~\ref{lem:stalk} and~\ref{lem:superalgebra}, this gives an action of the Lie algebra~\eqref{eq:Gstar} on $\Omega^\bullet(M)\otimes \mc{O}(\fg_\C)\otimes \Lambda \fg^\vee$. It remains to show this action is suitably compatible with the action defining the Weil complex. To state the required compatibility, consider the inclusion
\beq
&&\Omega(M)\otimes W(\fg) \hookrightarrow  \Omega(M)\otimes \mathcal{O}(\fg_\C)\otimes \Lambda\fg_\C^\vee \simeq C^\infty(\Map(\R^{0|1},M)\times \Omega^1(-\times \R^{0|1};\fg))\label{eq:Weilsub}
\eeq
induced by the inclusion of polynomial functions into holomorphic ones, $\Sym(\mf{g}_{\C}^\vee)\subset \mathcal{O}(\fg_\C)$.

\begin{lem}\label{lem:superWeil}
Consider the $(\E^{0|1}\rtimes \C^\times)\ltimes \Map(\R^{0|1},G)$-action on $\Map(\R^{0|1},M)\times \Omega^1(-\times \R^{0|1};\fg)$ from the groupoid in Lemma~\ref{lem:01present}. The $-k$th weight space for the $\C^\times$-action on
\beq
C^\infty(\Map(\R^{0|1},M)\times \Omega^1(-\times \R^{0|1};\fg))\simeq \Omega(M)\otimes \mathcal{O}(\fg_\C)\otimes \Lambda\fg_\C^\vee \label{eq:thealgebra}
\eeq
is equal to the image of the $k$th graded subspace of the Weil complex under~\eqref{eq:Weilsub}. The action of the Lie super algebra~\eqref{eq:Gstar} on the Weil complex is identified under Lemma~\ref{lem:superalgebra} with the infinitesimal action of $\E^{0|1}\ltimes \Map(\R^{0|1},G)$ on $\Map(\R^{0|1},M)\times \Omega^1(-\times \R^{0|1};\fg)$. 
\end{lem}
\bp 
We use Lemma~\ref{lem:01action} to compute the action of $(\E^{0|1}\rtimes \C^\times)\ltimes \Map(\R^{0|1},G)$ on the algebra~\eqref{eq:thealgebra}. Throughout, we shall fix a basis of $\fg$ which determines elements $z_i\in \fg_\C^\vee\subset \mathcal{O}(\fg_\C)$ and $\zeta_i\in \Lambda^1\fg_\C^\vee\subset \Lambda\fg_\C^\vee$ that generate a dense subalgebra of $\mathcal{O}(\fg_\C)\otimes \Lambda\fg_\C^\vee$. 

We start by showing that 
\beq
\bigoplus_{i+j=k} \Omega^i(M)\otimes W^j(\fg)\simeq C^\infty(\Map(\R^{0|1},M)\times \Omega^1(-\times \R^{0|1};\fg))^{\C^{\times}_{-k}}\label{eq:sumofweights}
\eeq
where the superscript ${\C^{\times}_{-k}}$ indicates the subspace on which $\C^\times$ acts by $\bar\mu^{-k}$, i.e., the $-k$th weight space. From~\eqref{eq:PiTaction2}, the $\C^\times$-action on $\Omega^k(M)$ is $\omega\mapsto \bar\mu^{-k}\omega$. From~\eqref{eq:act2}, the action on $\mathcal{O}(\fg_\C)\otimes \Lambda\fg_\C^\vee $ is determined by $z_i\mapsto \bar\mu^{-2} z_i$ and $\zeta_i\mapsto \bar\mu^{-1}\zeta_i$. Hence, the $-k$th weight space of this $\C^\times$-action is the $k$th degree subspace of the Weil complex. 
Furthermore, the infinitesimal $\C^\times$-action restricts to the (negative) grading derivation~$\bar N$ on the Weil complex that acts by $-k$ on a degree $k$ element (the action by $ N$ is trivial). 

Using Lemma~\ref{lem:superalgebra}, we match the Lie super algebra actions on either side of~\eqref{eq:sumofweights} using the construction of this action from~\cite[Ch.~3 and~4]{GuilleminSternberg}. Their description of the Weil complex comes from an action of the Lie super algebra with relations~\eqref{eq:Gstar} on~$\Omega(M)$ and also on~$W(\fg)$, leading to an action on the tensor product $\Omega(M)\otimes W(\fg)$. We observe that the $(\E^{0|1}\rtimes \C^\times)\ltimes \Map(\R^{0|1},G)$-action is diagonal on $\Omega(M)\otimes (\mathcal{O}(\fg_\C)\otimes \Lambda\fg_\C^\vee)$. Hence, it remains to match the infinitesimal action on $\Omega(M)$ and $\mathcal{O}(\fg_\C)\otimes \Lambda\fg_\C^\vee$ with the Lie algebra actions defining the Weil complex.

To start, the $\E^{0|1}$-action on $\Omega^\bullet(M)$ is the same as in Lemma~\ref{lem:PiT}, and hence by (minus) the de~Rham differential. The $\E^{0|1}$-action on $\mathcal{O}(\fg_\C)\otimes \Lambda\fg_\C^\vee$ can be computed using~\eqref{eq:act1}. In terms of our chosen basis of $\fg$, the infinitesimal action is $(z_i,\zeta_i)\mapsto (0,-z_i)$. This restricts to (minus) the Koszul differential on~$W(\fg)$. The de~Rham differential and Koszul differential combine to give the (negative) Weil differential in the description from~\cite[Ch.~4]{GuilleminSternberg}.

The $G\ltimes\Pi \fg\simeq \Map(\R^{0|1},G)$-action on $\Omega^\bullet(M)\simeq C^\infty(\Map(\R^{0|1},M))$ can be computed using the functor of points formulas~\eqref{eq:act5} and~\eqref{eq:act6}. As computed there, the action comes from applying the functor $\Map(\R^{0|1},-)$ to the action map $G\times M\to M$, and hence can be computed in terms of the map on differential forms,
$$
\Omega^\bullet(M)\simeq C^\infty(\Map(\R^{0|1},M))\to C^\infty(\Map(\R^{0|1},G\times M))\simeq \Omega^\bullet(G)\otimes \Omega^\bullet(M).
$$
One finds that~$\fg$ acts on $\Omega^\bullet(M)$ by Lie derivatives along the vector fields associated with the infinitesimal $G$-action, and $\Pi \fg$ (as a Lie algebra with trivial bracket) acts on $\Omega^\bullet(M)$ by contraction with these vector fields. The action on $\mathcal{O}(\fg_\C)\otimes \Lambda\fg_\C^\vee$ comes from~\eqref{eq:act3} and~\eqref{eq:act4}: the infinitesimal $G$-action is by the adjoint representation, and the $\Pi \fg$-action is determined by contraction operators. These restrict along~\eqref{eq:sumofweights} to the required action on the Weil algebra~\cite[Equations 3.6, 3.7 and~3.8]{GuilleminSternberg}. 
\ep


Since $\Map(\R^{0|1},G)$-invariant functions on $\Map(\R^{0|1},M)\times \Omega^1(-\times \R^{0|1};\fg)$ can be identified with basic forms, we obtain the following. 

\begin{cor}\label{cor:Weil}
There is an isomorphism of complexes
$$
((\Omega(M)\otimes W(\mf{g}))_{\rm bas},d_W) \simeq \left(\bigoplus_{k\in \Z}\left((C^\infty(\Map(\R^{0|1},M)\times \Omega^1(-\times \R^{0|1};\fg)))^{\C^\times_{-k}}\right)^{\Map(\R^{0|1},G)},-Q\right)
$$
where $Q$ denotes the infinitesimal action of~$\E^{0|1}$ and $(\Omega(M)\otimes W(\mf{g}))_{\rm bas}$ are the basic forms in the Weil complex. Hence, the right hand side computes the $G$-equivariant de~Rham cohomology of $M$ with complex coefficients. 
\end{cor}

\subsection{The Wess--Zumino gauge and the Cartan model}\label{sec:appenCartan}

Corollary~\ref{cor:Weil} gives a super geometric description of equivariant de~Rham cohomology in the Weil model. There is also a description of Cartan model, using the following restriction of $\fg$-valued 1-forms on $\R^{0|1}$. 
\begin{defn}
Define the subsheaf of $\Omega^1(-\times \R^{0|1};\fg)$
\beq
&&\Omega^1(-\times \R^{0|1};\fg)_\wz:=\underline{\fg}\times \Omega^1(-;\fg)\hookrightarrow \underline{\fg}\times \Pi \fg\times \Omega^1(-;\fg)\simeq \Omega^1(-\times \R^{0|1};\fg)\label{eq:diffformdecomp}
\eeq
where the inclusion is along $0\in \Pi\fg$. An $S$-point $A\in \Omega^1(S\times \R^{0|1};\fg)$ is in the \emph{Wess--Zumino gauge} if it lies in the subsheaf $A\in \Omega^1(S\times \R^{0|1};\fg)_\wz$. 
\end{defn}

Said differently, a $\fg$-valued 1-form $A$ on $S\times \R^{0|1}$ is in the Wess--Zumino gauge if $\chi=0$ in the notation of~\eqref{eq:chinote}. The terminology \emph{Wess--Zumino gauge} comes from supersymmetric gauge theory; e.g., see~\cite[\S4.2]{Wu}. The condition $\chi=0$ is a choice of gauge fixing, meaning any $\fg$-valued 1-form is gauge equivalent to one in the Wess--Zumino gauge:

\begin{lem}\label{lem:WZ01}
Consider the groupoid presentation of $\Map(\R^{0|1},[M\nsq G])\sq \E^{0|1}\rtimes \C^\times$ from Lemma~\ref{lem:01present}. The inclusion of the full subgroupoid with objects in the subsheaf $\Map(\R^{0|1},M)\times \Omega^1(-\times \R^{0|1};\fg)_\wz$ yields an equivalent super Lie groupoid, and hence the composition
\beq
\Map(\R^{0|1},M)\times \Omega^1(-\times \R^{0|1};\fg)_\wz&\hookrightarrow& \Map(\R^{0|1},M)\times \Omega^1(-\times \R^{0|1};\fg)\label{eq:inclusion01}\\
&\to& \Map(\R^{0|1},[M\nsq G])\sq \E^{0|1}\rtimes \C^\times \nonumber
\eeq
is also an atlas for the stack. 
\end{lem}
\bp
One need only check that the inclusion of objects~\eqref{eq:inclusion01} is an essential surjection. This in turn follows from~\eqref{eq:act5}: for $A\in \Omega^1(S\times \R^{0|1};\fg)$, there is a (unique) element $\gamma=-\chi\in \Pi \fg(S)$ so that $\gamma\cdot A\in \Omega^1(S\times \R^{0|1};\fg)_\wz$, where we have used the notation from~\eqref{eq:chinote}. 
\ep

Following Lemma~\ref{lem:stalk}, we obtain:

\begin{cor}\label{cor:Cartanvs}
There is an isomorphism of super algebras,
\beq
&&C^\infty(\Map(\R^{0|1},M)\times \Omega^1(-\times \R^{0|1};\fg)_\wz)\simeq \Omega(M)\otimes \mathcal{O}(\fg_\C). \label{eq:Cartanvss}
\eeq
\end{cor}

The right hand side above is a completion of the graded vector space~\eqref{eq:Cartanvs} underlying the Cartan complex, and we will again extract this vector space by taking a direct sum over $\C^\times$-weight spaces; see Lemma~\ref{lem:superCartan} below. However, the differential is a bit more subtle: the $\E^{0|1}$-action on $\Map(\R^{0|1},M)\times \Omega^1(-\times \R^{0|1};\fg)$ does not preserve the subspace $\Map(\R^{0|1},M)\times \Omega^1(-\times \R^{0|1};\fg)_\wz$, and hence its infinitesimal action cannot be used to define an odd derivation on~\eqref{eq:Cartanvss}. Instead, we only get a well-defined $\E^{0|1}$-action on the $G$-quotient. 

\begin{lem}\label{lem:Cartandiff}
The $\E^{0|1}$-action on $\Map(\R^{0|1},M)\times \Omega^1(-\times \R^{0|1};\fg)$ descends to an action on the subquotient $(\Map(\R^{0|1},M)\times \Omega^1(-\times \R^{0|1};\fg)_\wz)/G$. This action is generated by (minus) the Cartan differential on
\beq
C^\infty(\Map(\R^{0|1},M)\times \Omega^1(-\times \R^{0|1};\fg)_\wz/G)&\simeq &C^\infty(\Map(\R^{0|1},M)\times \Omega^1(-\times \R^{0|1};\fg)_\wz)^{G}\nonumber\\
&\simeq& (\Omega(M)\otimes \mathcal{O}(\fg_\C))^{G}.\nonumber
\eeq
\end{lem}

\bp
From~\eqref{eq:act4}, the $\Pi\fg$-action on $\Omega^1(-\times\R^{0|1};\fg)$ is free, and by~\eqref{eq:diffformdecomp} we get
$$
(\Map(\R^{0|1},M)\times \Omega^1(-\times\R^{0|1};\fg))/\Pi \fg\simeq \Map(\R^{0|1},M)\times\Omega^1(-\times\R^{0|1};\fg)_\wz.
$$
From the inclusions of normal subgroups,
 $$\Pi\fg \triangleleft G\ltimes \Pi\fg\simeq \Map(\R^{0|1},G)\triangleleft \E^{0|1}\ltimes \Map(\R^{0|1},G)
$$
we see that the $\E^{0|1}\ltimes\Map(\R^{0|1},G)$-action on $\Map(\R^{0|1},M)\times\Omega^1(-\times\R^{0|1};\fg)$ gives a $G$ action on the $\Pi\fg$-quotient $\Map(\R^{0|1},M)\times \Omega^1(-\times\R^{0|1};\fg)_\wz$, and an $\E^{0|1}$-action on the further $G$-quotient
\beq\label{eq:WZinvtfuns}
&&\resizebox{.9\textwidth}{!}{$(\Map(\R^{0|1},M)\times \Omega^1(-\times\R^{0|1};\fg)_\wz)/G\simeq (\Map(\R^{0|1},M)\times \Omega^1(-\times\R^{0|1};\fg))/\Map(\R^{0|1},G). $}
\eeq
This shows that the action descends as claimed. For the second part of the statement, we identify functions on the coarse quotient with invariant functions,
\beq
\resizebox{\textwidth}{!}{$
C^\infty((\Map(\R^{0|1},M)\times \Omega^1(-\times\R^{0|1};\fg))/\Map(\R^{0|1},G))\simeq C^\infty(\Map(\R^{0|1},M)\times \Omega^1(-\times\R^{0|1};\fg))^{\Map(\R^{0|1},G)}$}\nonumber
\eeq
so that the $\E^{0|1}$-action is inherited from the previous $\E^{0|1}$-action on $\Map(\R^{0|1},M)\times \Omega^1(-\times\R^{0|1};\fg)$. From Lemma~\ref{lem:superWeil}, $\Map(\R^{0|1},G)$-invariant functions are the horizontal, basic forms for the $G$-action on~$M$. Furthermore, the $\E^{0|1}$-action is generated by (minus) the Weil differential on these horizontal basic forms. Consider the commutative square 
\beq
\resizebox{\textwidth}{!}{$
\begin{tikzpicture}[baseline=(basepoint)];
\node (A) at (0,0) {$(\Omega(M)\otimes W(\mf{g}))_{\rm bas}$};
\node (B) at (8,0) {$(\bigoplus_{i+2j=\bullet} \Omega^i(M)\otimes \Sym^j (\mf{g}_{\C}^\vee))^{G}$};
\node (C) at (0,-1.2) {$C^\infty(\Map(\R^{0|1},M)\times\Omega^1(-\times\R^{0|1};\fg))^{\Map(\R^{0|1},G)}$};
\node (D) at (8,-1.2) {$C^\infty(\Map(\R^{0|1},M)\times\Omega^1(-\times\R^{0|1};\fg))^{G}$}; 
\draw[->] (A) to node [above] {$\sim$} (B);
\draw[->] (A) to  (C);
\draw[->] (C) to node [above] {$\sim$} (D);
\draw[->] (B) to (D);
\path (0,-.75) coordinate (basepoint);
\end{tikzpicture}$}\nonumber
\eeq 
where the vertical arrows come from the dense inclusion~\eqref{eq:Weilsub} and the dense inclusion using Corollary~\ref{cor:Cartanvs}. The upper isomorphism is the standard isomorphism between basic forms and the graded vector space underlying the Cartan model. The lower isomorphism comes from~\eqref{eq:WZinvtfuns}. It is well-known that the upper isomorphism sends the Weil differential to the Cartan differential. Hence, the lower isomorphism does as well. 
\ep

\begin{lem}\label{lem:superCartan}
There is an isomorphism of complexes
$$
((\Omega(M)\otimes \Sym(\fg_\C^\vee))^{G},d_C) \simeq \left(\bigoplus_{k\in \Z}\left((C^\infty(\Map(\R^{0|1},M)\times \Omega^1(-\times \R^{0|1};\fg)_\wz))^{\C^\times_{-k}}\right)^{G},-Q\right).
$$
where $Q$ denotes the infinitesimal action of~$\E^{0|1}$ from Lemma~\ref{lem:Cartandiff} and $((\Omega(M)\otimes \Sym(\fg_\C^\vee))^{G},d_C)$ is the Cartan complex. Hence, the right hand side computes the $G$-equivariant de~Rham cohomology of $M$ with complex coefficients. 
\end{lem}

\bp
Corollary~\ref{cor:Cartanvs} along with the computations going into~\eqref{eq:sumofweights} show that
$$
\bigoplus_{j+2k=\bullet}\Omega^j(M)\otimes \Sym^k(\fg_\C^\vee)\simeq \bigoplus_{\bullet\in \Z}(C^\infty(\Map(\R^{0|1},M)\times \Omega^1(-\times \R^{0|1};\fg)_\wz))^{\C^\times_{-\bullet}}.
$$
Taking $G$-invariants on either side then leads to the isomorphism of graded vector spaces in the statement of the lemma. From Lemma~\ref{lem:Cartandiff}, the $\E^{0|1}$-action is generated by the Cartan differential. \ep

\begin{rmk} 
From~\eqref{eq:act1}-\eqref{eq:act6}, one can also understand the Cartan differential before taking $G$-invariants. Namely in the description $d_C=d-\iota$, the de~Rham differential $d$ is the infinitesimal action by $\E^{0|1}$, which does not preserve the Wess--Zumino gauge. The contraction operator $\iota$ is the (unique) compensating infinitesimal gauge transformation from the proof of Lemma~\ref{lem:WZ01} so that the combination~$d-\iota$ together preserves the Wess--Zumino gauge. Note that $d-\iota$ squares to a Lie derivative operator, which in the super geometric interpretation is an infinitesimal gauge transformation. Hence, taking functions invariant under gauge transformations generated by~$G$ results in $d-\iota$ becoming a square zero operator. 
\end{rmk}

\subsection{Equivariant de~Rham cohomology as a sheaf on a super moduli space}\label{sec:01sheaf}

For a $G$-manifold $M$, the $G$-equivariant map $M\to\pt$ determines a morphism of stacks
\beq
&&\pi\colon \SM(\R^{0|1},[M\nsq G])\sq\E^{0|1}\rtimes \C^\times\to \SM(\R^{0|1},[\pt\nsq G])\sq\E^{0|1}\rtimes \C^\times \label{eq:01pi}
\eeq
where we view the target of this map as the moduli stack of $G$-bundles with connection on the super point~$\R^{0|1}$. By Corollary~\ref{cor:Weil} and Lemma~\ref{lem:superCartan}, the direct image sheaf $\pi_*C^\infty_{\SM(\R^{0|1},[M\nsq G])\sq\E^{0|1}\rtimes \C^\times}$ contains the information of the $G$-equivariant de~Rham cohomology of $M$. However, this direct image sheaf is not a sheaf of chain complexes; it is not even a graded algebra. As a warm-up to the elliptic case, we explain how to define a complex of sheaves on a moduli space of $G$-bundles with connection on $\R^{0|1}$ whose global sections compute the $G$-equivariant de~Rham cohomology of $M$. 

We recall that there is a canonical line bundle on $[\pt\sq \C^\times]$ from the homomorphism~\eqref{eq:GL1rep2}. This can be pulled back along the map induced by the homomorphism $\E^{0|1}\rtimes \C^\times\to \C^\times$,
$$
\Map(\R^{0|1},[\pt\nsq G])\sq \E^{0|1}\rtimes \C^\times\to [\pt\sq \C^\times], 
$$
to produce a line bundle over the source, denoted $\omega^{1/2}$. The $\C^\times$-invariant global sections of $\omega^{k/2}\otimes \pi_*C^\infty_{\SM(\R^{0|1},[M\nsq G])\sq\E^{0|1}\rtimes \C^\times}$ are precisely the weight $k$-subspaces
$$
\Gamma(\omega^{k/2}\otimes \pi_*C^\infty_{\SM(\R^{0|1},[M\nsq G])\sq\E^{0|1}\rtimes \C^\times})^{\C^\times} \simeq C^\infty(\Map(\R^{0|1},M)\times \Omega^1(-\times \R^{0|1};\fg))^{\C^\times_k}
$$
featured in Corollary~\ref{cor:Weil}. Hence, $\bigoplus_k \omega^{k/2}\otimes \pi_*C^\infty_{\SM(\R^{0|1},[M\nsq G])\sq\E^{0|1}\rtimes \C^\times}$ is a sheaf of graded algebras. 

For $Q$ the generator of the $\E^{0|1}$-action to define a differential, we must pass to the coarse quotient of the stack $\SM(\R^{0|1},[\pt\nsq G])\sq\E^{0|1}\rtimes \C^\times$ by gauge transformations. From Corollary~\ref{cor:Weil}, we have
$$
(\bigoplus_{k\in \Z}\Gamma(\omega^{k/2}\otimes \pi_*C^\infty_{\SM(\R^{0|1},[M\nsq G])\sq\E^{0|1}\rtimes \C^\times})^{\Map(\R^{0|1},G)\rtimes \C^\times} ,Q) \simeq ((\Omega(M)\otimes W(\mf{g}))_{\rm bas},d_W).
$$
Using the identification between invariant functions on the stack and functions on the coarse quotient, we equivalently understand the left hand side as global sections of a sheaf of cdgas on the coarse quotient $\SM(\R^{0|1},[\pt\nsq G])/\Map(\R^{0|1},G)\rtimes \C^\times$. This type of construction involving tensor powers of (grading) line bundles and coarse quotients is essential in the equivariant elliptic setting, where the goal is to obtain a sheaf of chain complexes over a certain moduli space rather than just a single chain complex. 

\begin{rmk}
One also obtains a sheaf of chain complexes by taking $\Map(\R^{0|1},G_0)\rtimes \C^\times$-invariant sections, i.e., sections over the coarse quotient $\SM(\R^{0|1},[\pt\nsq G])/\Map(\R^{0|1},G_0)\rtimes \C^\times$ for $G_0<G$ the connected component of the identity. Since $\Map(\R^{0|1},G)/\Map(\R^{0|1},G_0)\simeq \pi_0(G)$, the result computes the $G_0$-equivariant de~Rham cohomology of $M$ as a representation of $\pi_0(G)$; see Remark~\ref{rmk:pi0G}. 
\end{rmk}

\section{The de~Rham model for equivariant elliptic cohomology over $\C$}\label{appen:ell}

In this section all sheaves and stacks are on the site of (ordinary) smooth manifolds. The goal is to review the de~Rham model for equivariant elliptic cohomology over~$\C$ developed in~\cite{BET0}.  In brief, complex analytic equivariant cohomology assigns to any $G$-manifold~$M$ a sheaf of commutative differential graded algebras (cdgas) $\dEll_G(M)$ on the moduli stack $\Bun_G(\EE)$ of flat $G$-bundles on complex analytic elliptic curves. 


Recall the stack of elliptic curves $\Mell \simeq [\HH/\SL_2(\Z)]$. The Hodge line bundles $\omega^{\otimes k}$ have as holomorphic sections functions $F\in \mathcal{O}(\HH)$ satisfying
$$
F\left(\frac{a\tau+b}{c\tau+d}\right)=(c\tau+d)^{k}F(\tau)\qquad {\rm for}\qquad \left[\begin{array}{cc} a & b \\ c & d\end{array}\right]\in \SL_2(\Z).
$$
As stated in the introduction, we always impose meromorphicity at the cusp so that global sections of $\omega^{\otimes k}$ are weakly holomorphic modular forms of weight $k$. We take the following grading convention on the ring of modular forms.

\begin{defn}
Define the graded commutative algebra of modular forms, $\MF=\bigoplus_{k\in \Z} \MF^k$ whose graded subspaces are
$$
\MF^k:=\left\{\begin{array}{cl} \begin{smallmatrix} {\rm weakly\ holomorphic\ modular } \\ {\rm  forms\ of\ weight\ }-k/2 \end{smallmatrix} & k\ {\rm even}\\ 0 & k \ {\rm odd}. \end{array}\right.
$$ 
\end{defn}

We require the following extension of equivariant differential forms to sheaves. 

\begin{defn}\label{defn:valuesinasheaf} 
For $H$ a complex manifold, let $\Omega_{G,M}$ denote the sheaf of commutative differential graded algebras (cdgas) in $\mathcal{O}$-modules on~$H\times \fg_\C$ defined by
\beq
U\mapsto \mathcal{O}(U;\Omega^j(M)[\beta,\beta^{-1}])^G, \qquad U\subset H\times \fg_\C,\ \ |\beta|=-2\label{eq:sheafone}
\eeq
equipped with the Cartan differential $Q=d-\beta^{-1}\iota$ (see~\S\ref{eq:equivariantdeRham}). Define the sheaf on $H$
$$
V\mapsto \Omega^{\bullet}_G(M;\mathcal{O}(V)[\beta,\beta^{-1}]):=(i_0^{-1}\Omega_{G,M})(V)\qquad V\subset H
$$
as the inverse image sheaf of~\eqref{eq:sheafone} along the inclusion at zero,
$$
i_0\colon H\simeq H\times \{0\}\hookrightarrow H\times \fg_\C. 
$$
\end{defn}

\begin{rmk}
When $H=\pt$, we observe that 
$$
\Omega^{\bullet}_G(M;\mathcal{O}(\pt)[\beta,\beta^{-1}])=\mathcal{O}_0(\mf{g}_{\C}; \Omega^j(M;\C[\beta,\beta^{-1}]))^G
$$
is the germ of $G$-invariant holomorphic functions at $0\in \fg_\C$ with values in 2-periodic differential forms. We emphasize that $\mathcal{O}_0(\mf{g}_{\C})^G\subset \mathcal{O}_0(\mf{g}_{\C}; \Omega^j(M;\C[\beta,\beta^{-1}]))^G$ is in degree zero. The standard algebra of equivariant differential forms consists of \emph{polynomial} functions on~$\fg_\C$, where the monomial generators are in degree~2. The Bott class $\beta$ allows us to place these generators in degree zero, and then Definition~\ref{defn:valuesinasheaf} extends this both by a completion to holomorphic functions on $\fg_\C$, and then to a sheaf on $H$. 
\end{rmk}

Next we recall the moduli stack $\Bun_G(\EE)$ of $G$-bundles on elliptic curves. We define this as a stack on the site of smooth manifolds; it has a holomorphic structure constructed in~\cite[\S2]{BET0} that we comment on briefly below. Let $\mc{C}^2(G)$ denote the sheaf on the site of manifolds whose $S$-points are $S$-families of homomorphisms $S \times \Z^2 \to G$; note $G$ acts via postcomposition by conjugation and $\SL_2(\Z)$ acts via precomposition through its standard action on $\Z^2$. These actions commute, so there is an $\SL_2(\Z)$-action on the coarse quotient $\mc{C}^2[G] := \mc{C}^2(G)/G_0$ by the $G_0$-action. Define the sheaf of $G$-bundles on elliptic curves by
\begin{eqnarray}
\widetilde{\Bun}_G(\EE)&:=&\HH \times \mc{C}^2[G] \nonumber \\
\Bun_G(\EE)&:=&[\widetilde{\Bun}_G(\EE)\sq \SL_2(\Z)\times \pi_0(G)], \label{eq:BunG}
\end{eqnarray}
for the usual $\SL_2(\Z)$-action on $\HH$, the $\SL_2(\Z)$-action on $\mathcal{C}^2[G]$ defined above, and $\pi_0(G)\simeq G/G_0$ acting from the residual $G$-action on $\mathcal{C}^2[G]=\mathcal{C}^2(G)/G_0$. When $G=T$ is a torus,
$$
\Bun_G(\EE)\simeq [\HH\times T\times T\sq \SL_2(\Z)]\simeq \EE^{\vee}\times_{\Mell}\cdots \times_{\Mell}\EE^\vee
$$
where the complex structure on $\Bun_G(\EE)$ is inherited from the identification with the ${\rm rk}(T)$-fold fibered product of the universal dual elliptic curve. 

Next we define a convenient family of open subsheaves $U_h^\epsilon \subset \widetilde{\Bun}_G(\EE)$ for $h\colon \Z^2\to G$ and $\epsilon>0$.
 Let $B_\epsilon\subset \fg$ be an $\epsilon$-ball centered at $0\in \fg$ for an $\Ad$-invariant metric on~$\fg$. For $h=(h_1,h_2)\in \mathcal{C}^2(G)\subset G\times G$ a pair of commuting elements, define $B_\epsilon^h$ as the pullback
\beq
\begin{tikzpicture}[baseline=(basepoint)];
\node (A) at (0,0) {$B_\epsilon^h$};
\node (B) at (5,0) {$B_\epsilon\times B_\epsilon$};
\node (C) at (0,-1.2) {$\mathcal{C}^2(G)$};
\node (D) at (5,-1.2) {$G\times G$}; 
\draw[->] (A) to  (B);
\draw[->] (A) to  (C);
\draw[->] (C) to (D);
\draw[->] (B) to node [right] {$\exp_h$} (D);
\path (0,-.75) coordinate (basepoint);
\end{tikzpicture}
\label{eq:GGcovmfld}
\eeq 
where $\exp_h$ is the exponential map shifted by~$h$,
$$
(X_1,X_2)\mapsto (h_1e^{X_1},h_2e^{X_2})\qquad h=(h_1,h_2),\ (X_1,X_2)\in B_\epsilon\times B_\epsilon\subset \fg\times \fg.
$$ 
Observe that $B_\epsilon^h \subset \mathcal{C}^2(G)$ is an open subsheaf, being the pullback of an open subsheaf. Consider now the $G$-action on $\mathcal{C}^2(G)$ by conjugation and take the $G_0$-orbit $\widetilde{U}_h^{\epsilon} := G_0 \cdot B_\epsilon^h$. By Lemma~\ref{lem:orbitopen}, this is a $G_0$-invariant open subsheaf $\widetilde{U}_h^{\epsilon} \subset \mathcal{C}^2(G)$. By Proposition~\ref{prop:opensofquot}, this uniquely specifies an open subsheaf $U_h^{\epsilon}\subset \mathcal{C}^2[G]\simeq \mathcal{C}^2(G)/G_0$ sitting in the pullback diagram
\beq
\begin{tikzpicture}[baseline=(basepoint)];
\node (A) at (0,0) {$\widetilde{U}_h^{\epsilon}$};
\node (B) at (5,0) {$\mathcal{C}^2(G)$};
\node (C) at (0,-1.2) {$U_h^{\epsilon}$};
\node (D) at (5,-1.2) {$\mathcal{C}^2[G].$}; 
\draw[->] (A) to  (B);
\draw[->] (A) to  (C);
\draw[->] (C) to (D);
\draw[->] (B) to (D);
\path (0,-.75) coordinate (basepoint);
\end{tikzpicture}
\label{eq:Utildefn}
\eeq 
In Corollary~\ref{cor:topofCG} we prove that $\{U_h^\epsilon\}$ is a basis for the topology of $\mathcal{C}^2[G]$ and in Corollary~\ref{cor:topofBunG} we prove that opens of the form $V \times U_h^{\epsilon}$ provide a basis for the topology of $\widetilde{\Bun}_G(\EE)$. 


Let $M^h\subset M$ denote the submanifold fixed by the subgroup of $G$ generated by~$h$, and define $G_0^h=G_0\bigcap G^h$ for $G^h<G$ the subgroup fixed by the conjugation action by~$h$ and $G_0<G$ the connected component of the identity. We observe that the Lie algebra of $G_0^h$ is $\mf{g}^h\subset \mf{g}$, the subalgebra invariant under the adjoint action by~$h$. Let $\mf{t}_{\mf{g}^h}\subset \mf{g}^h$ denote a maximal commuting subalgebra. Let $U = V \times Y$ be an open subsheaf of $\widetilde{\Bun}_G(\EE)$ for $V\subset \HH$ an open subset and $Y\subset \mathcal{C}^2[G]$ an open subsheaf.

\begin{defn} \label{defn:ellcocycle}
Define the sheaf of commutative differential graded algebras $\dEll^\bullet_G(M)$ on $\Bun_G(E)$ as having sections $\alpha\in \dEll^\bullet_G(M)(U)$ given by a collection $\{\alpha_{h}\}_{[h]\in Y}$
\beq
\alpha_{h}\in  \Omega^{\bullet}_{G_0^{h}} (M^{h};\mathcal{O}(V)[\beta, \beta^{-1}])\label{eq:cocycledef}
\eeq
for the cdga defined in Definition~\ref{defn:valuesinasheaf} for all $[h]=[h_1,h_2]\in \mathcal{C}^2[G]$ in the image of $Y\to \mathcal{C}^2[G]$. 
These data are required to satisfy:
\begin{enumerate}
\item \emph{Invariance}: for all $g\in G_0$, we have
$$
\alpha_{h}=g^*\alpha_{ghg^{-1}}
$$
where $g^*$ is the pullback along left multiplication by $g$, $M^{h}\to M^{ghg^{-1}}$. 

\item \emph{Analyticity}: for $h'=(h_1e^{X_1},h_2e^{X_2})\in U_{h}^\epsilon\subset Y$ for $\epsilon$ sufficiently small, we have 
\beq
&&\alpha_{h'}(X)={\rm res} (\alpha_{h}(X+(X_1-\tau X_2)))\in \Omega^{\bullet}_{G_0^{h'}}(M^{h'};\O(V)[\beta,\beta^{-1}])\label{eq:analyticity}
\eeq
where ${\rm res}\colon \Omega^{\bullet}_{G_0^h}(M^{h};\O_{\HH}[\beta,\beta^{-1}])(V) \to \Omega^{\bullet}_{G_0^{h'}}(M^{h'};\O(V)[\beta,\beta^{-1}])$ is the restriction map associated to the inclusions $M^{h'}\hookrightarrow M^h$ and $G_0^{h'}<G_0^h$ (see Lemma~\ref{lem:BG}).
\end{enumerate}
There is a compatible differential $Q$ acting on sections given by the Cartan differentials on each~\eqref{eq:cocycledef}. 

We promote this to an $\SL_2(\Z)\times \pi_0(G)$-equivariant sheaf as follows. For $\gamma\in \SL_2(\Z)$, $[g]\in \pi_0(G)$
%
defining an isomorphism $V \times Y=U\to U'=V'\times Y'$ between open subsheaves of $\HH\times\mathcal{C}^2[G]$,
we obtain a map $\dEll^k_G(M)(U')\to \dEll^k_G(M)(U)$ from
\beq
\Omega^{\bullet}_{G_0^{\gamma\cdot ghg^{-1}}} (M^{\gamma\cdot ghg^{-1}};\mathcal{O}(V')[\beta, \beta^{-1}]) {\longrightarrow} \Omega^{\bullet}_{G_0^h} (M^{h};\mathcal{O}(V)[\beta, \beta^{-1}])\label{eq:equiv}
\eeq
for each $h=(h_1,h_2)$. The map~\eqref{eq:equiv} uses the pullback of holomorphic functions~$\mathcal{O}(V')\to  \mathcal{O}(V)$ along the map $V\to V'$ determined by $\gamma$, together with the isomorphisms $M^h\simeq M^{\gamma\cdot ghg^{-1}}$, $G_0^{h}\simeq G_0^{\gamma\cdot ghg^{-1}}$. We then modify this pullback map by rescaling the Lie algebra~$\mf{g}^h$ by $c\tau+d$ (so $z\in (\mf{g}^h)^\vee$ is sent to $\frac{z}{c\tau+d}$), and sending $\beta$ to $\beta/(c\tau+d)$.
\end{defn}


%
%

\subsection{The topology of $\Bun_G(\EE)$}

\begin{prop} For $L \in {\sf Mfld}$, $\ell \in L$, $h=(h_1,h_2) \in \mc{C}^2(G)$, and a $G_0$-invariant open subsheaf $U \subset L \times \mc{C}^2(G)$ such that $(\ell, h)\in U$, there exists an open neighborhood $V \subset L$ containing $\ell$ and some $\epsilon>0$ such that $U \supset V \times \widetilde{U}_h^{\epsilon}$. \end{prop}

\begin{proof}

Consider the surjective morphism $G_0 \times B_{\alpha}(0)^2 \to \widetilde{U}_h^{\alpha}$, where $B_{\alpha}(0)$ denotes the ball of radius $\alpha$ about $0$ in $\mf{t}_{\mf{g}^h}$ and the morphism sends $(g, Y_1, Y_2) \mapsto (g h_1e^{Y_1} g^{-1}, g h_2e^{Y_2} g^{-1})$. Then pulling back $U$ under the similar morphism $L \times G_0 \times B_{\alpha}(0)^2 \to L \times \mc{C}^2(G)$ yields an open submanifold, which we denote $U'$, of the smooth manifold $L \times G_0 \times B_{\alpha}(0)^2$. Since $U$ is $G_0$-invariant (by assumption) its pullback $U'$ must be of the form $G_0 \times U''$, where $U'' \subset L \times B_{\alpha}(0)^2$. Moreover, by assumption, we have that $U'' \ni (\ell, 0)$ so that as $L \times B_{\alpha}(0)^2$ has the product topology. Hence there exists some open $V \subset L$ containing $\ell$ and some $\epsilon > 0$ such that $U'' \supset V \times B_{\epsilon}(0)^2$. Since $U$ is a concrete subsheaf of $L \times \mc{C}^2(G)$, it is determined by its points $U(\pt)\subset \mc{C}^2(G)(\pt)$ as a subset. Hence it necessarily contains $V \times \widetilde{U}_h^{\epsilon}$, as desired. \end{proof}

By Proposition~\ref{prop:opensofquot}, there is an equivalence of categories between $G_0$-invariant opens of $L \times \mc{C}^2(G)$ and opens of $L \times \mc{C}^2[G]$. Applying this to the previous proposition with $L = \pt $ and $L=\HH$, we obtain the following corollaries. 

\begin{cor} \label{cor:topofCG} The open subsheaves $U_h^{\epsilon} \subset \mc{C}^2[G]$ provide a basis for its topology. \end{cor}

\begin{cor} \label{cor:topofBunG} The open subsheaves $V \times U_h^{\epsilon} \subset \widetilde{\Bun}_G(\EE)$ provide a basis for its topology. \end{cor}

\bibliographystyle{amsalpha}
\bibliography{references}

\end{document}